\documentclass[11pt,reqno]{amsart}

\makeatletter
\usepackage{amssymb}
\usepackage{amssymb,amsmath,color,comment}
\usepackage{amsbsy}
\usepackage{amsfonts}
\RequirePackage[colorlinks,linkcolor=blue,citecolor=blue,urlcolor=blue]{hyperref}

\def\marginpar#1{\ignorespaces}

\numberwithin{equation}{section}

\newtheorem{theorem}{Theorem}[section]
\newtheorem{proposition}[theorem]{Proposition}
\newtheorem{lemma}[theorem]{Lemma}
\newtheorem{corollary}[theorem]{Corollary}

\theoremstyle{definition}
\newtheorem{assumption}[theorem]{Assumption}
\newtheorem{definition}[theorem]{Definition}
\newtheorem{remark}[theorem]{Remark}

\newtheorem{example}[theorem]{Example}

\newdimen\AAdi%
\newbox\AAbo%
\def\AAk#1#2{\setbox\AAbo=\hbox{#2}\AAdi=\wd\AAbo\kern#1\AAdi{}}%

\def\eqref#1{(\ref{#1})}
\def\eqlabel#1{\def\@currentlabel{#1}}

\def\formula#1{\def\@tempa{#1}\let\@tempb\theequation\def\theequation{%
\hbox{#1}}\def\@currentlabel{(\theequation)}$$}
\def\endformula{\leqno\hbox{(\@tempa)}$$\@ignoretrue\let\theequation\@tempb}

\def\given{\hskip5\p@\relax\vrule\@width.4\p@\hskip5\p@\relax}

\newcommand{\open}[1]{%
\par\normalfont\topsep6\p@\@plus6\p@\trivlist\item[\hskip\labelsep\itshape#1%
\@addpunct{.}]\ignorespaces}

\DeclareRobustCommand{\close}[1]{%
  \ifmmode 
  \else \leavevmode\unskip\penalty9999 \hbox{}\nobreak\hfill
  \fi
  \quad\hbox{$#1$}}

\newlength{\toskip}

\def\<{\langle}
\def\>{\rangle}

\makeatother

\begin{document}
\date{\today}

\title[Singular mean field]{Entropy on the path space and application to singular diffusions and mean-field models.}

 \author[P. Cattiaux]{\textbf{\quad {Patrick} Cattiaux $^{\spadesuit}$ \, \, }}
\address{{\bf {Patrick} CATTIAUX},\\ Institut de Math\'ematiques de Toulouse. Universit\'e de Toulouse. CNRS UMR 5219. \\ 118 route de Narbonne, F-31062 Toulouse cedex 09.}
\email{patrick.cattiaux@math.univ-toulouse.fr}

\maketitle
 \begin{center}

 \textsc{$^{\spadesuit}$  Universit\'e de Toulouse}
\smallskip

 \end{center}

\begin{abstract}
In this paper we intend to present a unified treatment of a variety of singular interacting particle systems and their McKean-Vlasov limits. This unified approach is based on the use of the relative entropy on the path space in the spirit of our previous works together with C. L\'eonard. 

\noindent We show how it can be used to derive existence and uniqueness for some singular diffusions, in particular linear mean field stochastic particle systems and non linear SDE of McKean-Vlasov type, including $\mathbf L^p-\mathbf L^q$ models, the 2D vortex model associated to the 2D Navier-Stokes equation, sub-Coulombic interactions models or the Patlak-Keller-Segel model. 

\noindent We also show the convergence and propagation of chaos as the number of particles grows to infinity. This is (mainly) obtained at the process level, not only at the Liouville equation (marginals flow) level. The paper thus contains new proofs and extensions of known results, as well as new results.

\noindent The main results are given at the end of the Introduction.
\end{abstract}
\bigskip

\textit{ Key words : singular McKean-Vlasov equations, diffusion processes, Kullback information, propagation of chaos.}  
\bigskip

\textit{ MSC 2010 : 35Q92, 60J60, 60K35.}
\bigskip

\section{Introduction and main results.}\label{Intro}

McKean-Vlasov non linear partial differential equations have been extensively used for modeling the dynamics of complex systems in physical and life sciences in particular. They are a macroscopic description of collective behaviour of particle systems for which the density of particles $\bar \rho(t)$ solves 
\begin{equation}\label{eqmcv}
\partial_t \bar \rho_t(x) =  \Delta_x \, \bar \rho_t(x) + \nabla_x.((b+K*\bar \rho_t)\bar \rho_t)(x) \, .
\end{equation}
Here $x \in \mathbb R^d$, $b$ is some drift term describing a self-interaction and $K$ describes interaction between particles. It is not difficult to see that \eqref{eqmcv} preserves positivity and mass, so that we may assume that $\bar \rho_0$ is a density of probability, i.e. $\bar \rho_0\geq 0$ and $\int \bar \rho_t dx = \int \rho_0 dx =1$. 

Of course we have chosen here to only look at ``diffusive'' particles of Ito type and one can replace the Laplace operator by more general local or non local operators. Notice that up to a linear time change, one can include in this framework the case where the diffusion part is given by $\sigma^2 \, \Delta_x$.

In order to understand the microscopic behaviour of the system one has to introduce a ``linear'' stochastic version of \eqref{eqmcv}, describing the individual behaviour of each particle, namely, for $i=1,...,N$, the system of Stochastic Differential Equations (SDE)
\begin{equation}\label{eqsys}
dX_t^{i,N} = \sqrt 2 \, dB_t^{i,N} \, - \, b(X_t^{i,N}) dt \, - \, \frac {1}{N} \, \sum_{j=1}^N \, K(X_t^{i,N}-X_t^{j,N}) \, dt \, ,
\end{equation}
where $B_.^{i,N}$ is a collection of $N$ i.i.d. standard Brownian motions on $\mathbb R^d$. If one assumes that the initial random vector $X_0^{.,N}$ is a finite exchangeable sequence, it is easy to see that the same is true for $X_t^{.,N}$ for all $t \geq 0$. Assuming that the variables $X_0^{.,N}$ are i.i.d. with common distribution $\mu_0(dx)$, one can thus expect that  the empirical measure $\nu_t^N = \frac 1N \, \sum_{i=1}^N \, \delta_{X_t^{i,N}}$ which is a random variable taking its values in the set of probability measures, satisfies some ``law of large numbers'' i.e. converges in some sense (weakly or strongly) and hopefully that the limit is a non-random probability measure $\mu_t$. This convergence implies that for all $k \in \mathbb N$, the distribution of $(X_t^{1,N},...,X_t^{k,N})$ weakly converges to the tensor product $\mu_t^{\otimes k}$. This property is called since its introduction by Kac, the \emph{propagation of chaos} property, at the level of the marginals. 

Since the interaction term in the drift of \eqref{eqsys} writes as $K*\nu_t^N$ one can then expect that $X_.^{1,N}$ converges in some sense to a \emph{non-linear diffusion process} solution of the so called \emph{non linear SDE},
\begin{eqnarray}\label{eqnldiff}
d\bar X_t &=& \sqrt 2 \, dB_t \, - \, b(\bar X_t) \, dt \, - \, (K*\bar \rho_t)(\bar X_t) \, dt \, ,\\
\mu_t = \bar \rho_t(x) \, dx &=& \mathcal L(\bar X_t) \, , \nonumber
\end{eqnarray}
whose marginals flow satisfies \eqref{eqmcv}. 
\medskip

The previous program was partly or fully completed in several situations during the last decades. The first basic result is described in \cite{Sznit} section 1 and reads as follows
\begin{theorem}\label{thmbasic}
Assume that $b$ and $K$ are global Lipschitz and bounded. Assume in addition that the distribution of the initial condition $X_0^{.,N}$ is given by $\mu_0^{\otimes N}$. 

Then there exists a strongly unique solution for both \eqref{eqsys} and \eqref{eqnldiff}, where for $t>0$ the distributions of $X_t^{.,N}$ and $\bar X_t$ admit a density $\rho_t^{.,N}$ and $\bar \rho_t$. Furthermore $\bar \rho_t$ given by \eqref{eqnldiff} solves \eqref{eqmcv}. 

In addition, building solutions with the same Brownian motion (synchronous coupling), it holds for any $i\geq 1$ and $T>0$, $$\sup_N \, \sqrt N \; \mathbb E[\sup_{t\leq T} \, |X_t^{i,N}-X_t|] = C(T) < +\infty \, ,$$ and propagation of chaos holds true.

In particular, if $Q^{k,N}$, and $\bar Q$ denote respectively the distribution on $C^0([0,T],(\mathbb R^d)^N$ of the processes $(X_.^{1,N},...,X_.^{k,N})$ and $\bar X_.$ on a finite time interval $[0,T]$, the Wasserstein distance $W_1(Q^{k,N},\bar Q^{\otimes k})$ goes to $0$ as $N$ goes to infinity with a rate $\sqrt N$.
\end{theorem}
The last sentence implies the propagation of chaos at the level of the paths of the processes (we shall say \emph{process level} for short), which is of course stronger than propagation of chaos for the marginals.

The existence of a density for the given laws is an application of elliptic regularity for the particle system and a little bit of additional work for the non-linear S.D.E. Notice that the densities $\rho_t^N$ of the marginals flow of $Q^{N,N}$ solve a Fokker-Planck equation, often called the \emph{Liouville equation} in the PDE literature
\begin{equation}\label{eqliouv}
\partial_t \rho^N_t(x) =  \sum_i \, \Delta_{x^i} \, \rho_t^N(x) + \sum_i \, \nabla_{x^i}.((b(x^i)+ \frac 1N \, \sum_{j} K(x^i-x^j))\rho_t^N)(x) \, , 
\end{equation}
and process level propagation of chaos implies the convergence of $\rho_t^{k,N}$ to $\bar \rho_t^{\otimes k}$, that is propagation of chaos at \emph{marginals level}.
\bigskip

Since this time, a lot of works have been devoted to partly or fully extend Theorem \ref{thmbasic} to more general situations. The first immediate extension consists in replacing the independence assumption $\mu_0^N=\mu_0^{\otimes N}$ by $\mu_0^N$ is $\mu_0$ chaotic i.e. is exchangeable and satisfies for all fixed $k$ that the distribution $\mu_0^{(1,...,k),N}$ of $(X_0^{1,N},...,X_0^{k,N})$ weakly converges to $\mu_0^{\otimes N}$ as $N$ goes to infinity.  

Many of these works are using the original scheme of proof: prove existence and uniqueness for \eqref{eqsys}, then for \eqref{eqnldiff} using a fixed point theorem and finally prove propagation of chaos by using a synchronous coupling between the linear and the non-linear SDE's. 

An alternate approach replacing the fixed point theorem consists in proving tightness for the empirical measures using moment bounds, to show that all limit points are solving the martingale problem associated with \eqref{eqnldiff}, show uniqueness for the latter and deduce propagation of chaos from this uniqueness and Proposition 2.2 in \cite{Sznit}. This alternate approach is described (in an even more general framework) in the survey \cite{Mel} where the $W_1$ distance is replaced by the $W_2$ distance. 

We also refer to the more recent surveys \cite{Bol,JW17}, and especially to the very recent massive one \cite{diez1,diez2}. One would find in the bibliography of \cite{JW17} and \cite{RS21} a number of examples and references of mean-field diffusive models we will not add in this paper in order to save place.
\bigskip

The extensions of Theorem \ref{thmbasic} we are interested in, are in particular the ones where $b$ or $K$ are no more bounded, nor global Lipschitz. 

For non bounded examples let us mention aggregation models like the granular media equation (see \cite{CMCV03,CMCV06} for an analytic approach, \cite{BRV,BRTV} for a probabilistic one based on a fixed point argument, \cite{CGM} for a (more general) probabilistic one using the alternate approach and look at the references therein) or some stochastic FitzHugh-Nagumo equations for which we refer to \cite{Colbri} and the references therein.

In the present paper we shall focus on singular kernels.  
\medskip

A particularly large family of interesting kernels is given for $d \geq 2$ by 
\begin{equation}\label{eqkernelcoulomb}
K(x)=\chi \, \frac{x}{|x|^{s+2}}
\end{equation}
 for some $s$, i.e. deriving from a Riesz potential. The constant $\chi$ determines whether the model is \emph{attractive} for $\chi>0$ or \emph{repulsive} for $\chi<0$. This terminology is motivated by the following property : looking at the squared distance between two particles $|X_t^{i,N}-X_t^{j,N}|^2$, the interaction part $- \, \langle K(X_t^{i,N}-X_t^{j,N}),X_t^{i,N}-X_t^{j,N}\rangle$ of the dynamics is negative for $\chi>0$ (resp. positive for $\chi<0$) when particles are close, inducing that this effect tends to make the distance decaying (attraction) or growing (repulsion) in each case.

It is worth noticing that even the existence of a solution for the particle system is far to be clear, so that all the program has to be performed. 
\medskip

For $d=2$, $\chi >0$, $b=0$ and $s=0$ correspond to the celebrated \emph{parabolic-elliptic} (Patlak)-Keller-Segel model introduced in \cite{KS1,KS2} in order to model some chemotaxis phenomena observed in some bacteria populations like the Dictyostelium discoideum for instance. A detailed description of the most important results on this model is given in subsection \ref{subsecKS}. Let us simply indicate here that existence and uniqueness of a \emph{free energy} solution to \eqref{eqmcv} is known under mild assumptions on the initial condition (see \cite{BDP,FM}) provided $\chi \leq 4$, existence and uniqueness for the particle system is studied in \cite{FJ,CP,FT-coll} for $\chi < 4$, propagation of chaos at the level of marginals is studied in \cite{BJW} but replacing the euclidean space $\mathbb R^d$ by the torus $\mathbb T^d$, partly studied in \cite{Tar} in the whole space at the level of marginals and in \cite{FJ} at the process level but for $\chi < 1$. Notice that \cite{BJW} contains quantitative results on propagation of chaos, while the results in \cite{Tar,FJ} are only qualitative. The Keller-Segel model is presumably one of the most difficult to study due to the explosion of solutions for $\chi>4$ and to various phenomena, like for instance the existence of collisions between particles with positive probability for the stochastic process. 

A $\eta$ relaxed Keller-Segel model corresponding to $s=-\eta$ for $\eta>0$ is studied in \cite{GQ}. The main difficulties of the classical one disappear, in particular there are no collisions, and one can use the strategy in \cite{FHM} for the Biot-Savart kernel we shall introduce below, in order to perform the whole program and obtain propagation of chaos at the marginals level. 
\medskip

Still for $d=2$, the 2D-vortex model corresponds to $b=0$ and to the Biot-Savart kernel $$K(x)= \chi \, \left(\frac{x_2}{|x|^2} \, , \, \frac{- x_1}{|x|^2} \right) := \chi \, \frac{x^\perp}{|x|^2}$$ for some constant $\chi$. It is linked to the $2D$ Navier-Stokes equation. One can see the initial \cite{Osa85,Osa86} and the more generals \cite{Mapul,FHM}. The key point for existence and (strong) uniqueness is thus that there are no collisions in finite time, i.e. that for all $T>0$, all $i\neq j$, $$\inf_{0<t\leq T} |X_t^{i,N}-X_t^{j,N}|\neq 0$$ almost surely as soon as the same holds at time $0$. Once this is shown, the alternate strategy of proof is used in \cite{FHM} to get a qualitative analogue of Theorem \ref{thmbasic} replacing convergence in Wasserstein distance by convergence in distribution at the  process level. For an analogue 3D model we refer to \cite{Fontbo-3D}.

Quantitative results at the marginals level have been recently obtained in \cite{GLBM-biot} on the torus and in \cite{FWvortex} on the whole euclidean space. The results \cite{GLBM-biot} are also uniform in time.
\medskip

In higher dimension $d \geq 3$, the repulsive case is studied in \cite{RS21} for \emph{sub-Coulombic} potentials, i.e. $0<s<d-2$ and in the very recent \cite{serfks} on the torus, the latter also containing some results in dimension $2$ for the periodized Keller-Segel model studied in \cite{BJW}. All these works only consider the Liouville equation, hence the marginals level. 
\medskip

Another particularly interesting case is the one of the Dyson Brownian or Ornstein-Uhlenbeck motion, corresponding to $d=1$, $\chi<0$ (repulsive situation), $b=0$ or $b(x)= -  ax$ for some $a>0$ and $K(x)=1/x$ which is connected to the spectrum of random matrices, studied in the 90's in \cite{Chan,Rogshi,Cepa,Cepamult} and very recently in \cite{Bour,GLBM1D}.
\bigskip

The notion of propagation of chaos has been extended in several directions and a very general theory is contained in \cite{HM}. In particular one may reinforce the weak convergence of $\mu_t^{(1,...,k),N}$ to $\mu_0^{\otimes N}$ by adding the convergence of some entropies (see \cite{FHM} definition 2.3 for this notion). This is sometimes called entropic propagation of chaos in the PDE literature.

Recall that for two probability measures $\mu$ and $\nu$ on the same Polish space the relative entropy of $\nu$ w.r.t. $\mu$ (also called Kullback information or Kullback-Leibler divergence) is defined as
\begin{equation}\label{eqdefentrop}
H(\nu|\mu) = \int \, \ln \left(\frac{d\nu}{d\mu}\right) \, d\nu \, = \, \sup_{||f||_\infty <+\infty} \, \left(\int \, f \, d\nu \, - \, \ln\int \, e^f \, d\mu\right) \, ,
\end{equation}
this quantity being infinite if $\nu$ is not absolutely continuous w.r.t. $\mu$. If Kullback information is not a distance it controls the total variation distance $d_{TV}$ thanks to the celebrated Pinsker's inequality
\begin{equation}\label{eqpinsker}
\sup_{||f||_\infty \leq 1} \left|\int \, f (d\nu-d\mu)\right| := d_{TV}(\nu,\mu) \, \leq \, \sqrt{ 2 \, H(\nu|\mu)} \, .
\end{equation}

Entropic propagation of chaos is proved for the vortex and the relaxed Keller-Segel models.
\medskip

Recently other approaches have been proposed for the study of singular mean-field models, in particular in order to obtain quantitative and/or uniform in time propagation of chaos, i.e. results like $\sup_{t\geq 0} \, d(\mu_t^{1,N},\mu_t) \leq \theta(N)$ for some distance $d$ and some $\theta(N)$ going to $0$ as $N$ growths to infinity. 

The first one, initiated in \cite{JW16} and extended in \cite{JW18}, consists in studying the time evolution of $H(\rho_t^{(1,...,N),N}|\bar \rho_t^{\otimes N})$ where $\bar \rho_t$ is the solution of \eqref{eqnldiff}. The Biot-Savart kernel is also studied (as an example in a larger class) in \cite{GLBM-biot} combining the Jabin-Wang approach and coupling methods introduced in \cite{EGZ}. In these papers the dynamics is built on the torus $\mathbb T^{\otimes N}$, in order to ensure some boundedness properties. The second one, explained in \cite{RS21}, is based on the study of the time evolution of the \emph{modulated energy} introduced by Serfaty. Convergence for the modulated energy implies weak convergence (\cite{RS21} Remark 1.5). Both approaches are combined in \cite{BJW} in order to cover more singular kernels including the Patlak-Keller-Segel case for $d=2$, 
\medskip

A new entropic approach was initiated by D. Lacker in \cite{Lack18} and \cite{Lack}, using probabilistic tools and some hierarchy principle for the entropies (BBGKY hierarchy) at the process level. The results in the previous references as well as in the last opus \cite{LackLef} of the method are written with very general assumptions that are difficult to check on explicit singular examples (as the authors themselves are claiming). In \cite{Huang,Han} the authors combine Lacker's approach and Krylov estimates as e.g. in \cite{KR}, that are used to check the assumptions made by Lacker. Propagation of chaos for singular kernels like the ones we have described, is not shown and seems very hard to obtain by using this approach. 
\medskip

Though singular, the kernels $K$ we have discussed, satisfy $K \in \mathbb L^\alpha(\mathbb R^d)$ for some $\alpha >0$. A general theory for S.D.E. $$dX_t = \sqrt 2 \, dB_t + g(t,X_t) dt $$ with a drift $g \in \mathbb L^p([0,T],\mathbb L^q(\mathbb R^n))$ has been developed by Krylov and co-authors. Since we are mainly interested here in homogeneous processes, we shall essentially consider the case $p=+\infty$. However, for non-linear SDE we have to consider general time dependent drifts. Existence and uniqueness for such processes is obtained for $q>n$ (see e.g. \cite{KR}). Some results were obtained in the critical case $q=n$ (see the discussion in Remark \ref{remliterature}). The method has been applied to interacting particles, with $q>d$ in e.g. \cite{Hao,Toma-p}. This limitation however is too strong for the most interesting singular kernels we want to look at, except the $\eta$ relaxed Keller-Segel model. One can also look at \cite{kinzmulti} for another approach based on De Giorgi's method.
\bigskip

The main idea of the present paper is to look at the Kullback information (relative entropy) on the path space, i.e. for probability measures that are the laws of diffusion processes rather than looking at their time marginals. Except this use, our approach does not have many similarities with the one of Lacker (except in section \ref{secchaos1} where we will recall and get tiny improvements of the results in \cite{Lack}). We shall push further this idea including the proof of existence and uniqueness for the particle system and for the non linear diffusion process, and not only the propagation of chaos result. 

In addition, we will show convergence and propagation of chaos at the process level (except in one case), not only at the marginals flow level.
\bigskip

This paper (sometimes partly expository) is thus some kind of propaganda for relative entropy on the path space. This notion offers a unified approach for many of the problems we have introduced, in particular for the singular models we have discussed. It allows us to give new (and often shortest) proofs for known results, to improve several existing results, and to obtain new results, in particular for the Keller-Segel model.
\medskip

Relative entropy on the path space has already been used in several other contexts, for instance the study of the celebrated Schr\"{o}dinger problem (see e.g. \cite{Leo-survey}), stochastic analysis (see e.g. \cite{Leo12,Leo14,LeoPTRF,ACLZ,CF96}), or large time behaviour and functional inequalities (see e.g. \cite{FJaop,Gent-Leo-etc,Catpota,CG-semin,CG-transport,Cattoul,Tschi}).
\bigskip

\textbf{An overview of the contents of the paper.}
\medskip

In the next section \ref{secpath} we will recall first (subsection \ref{subsecpath}) the basics on relative entropy on path spaces and the associated Girsanov transforms, including the use of time reversal in order to connect relative entropy and Fisher information (Proposition \ref{propretour} and Corollary \ref{cordual}). We also recall the main results of \cite{CL1} for the construction of conservative diffusions with a given flow of marginals of finite energy. The stationary situation in connection with Dirichlet forms theory is discussed in subsection \ref{subsecdiric}. We then relate absolute continuity of the law of the particle system (w.r.t. Wiener measure) to the absence of collisions (Theorem \ref{thmcoll}). Finally we apply these results to the construction of some singular particle systems ($\eta$ relaxed Keller-Segel, sub-Coulombic or Dyson). If the results are mainly known, the method of construction is new.
\medskip

In section \ref{secapprox} we first recall the useful Gagliardo-Nirenberg type inequalities (Lemma \ref{lemFHM}) and their application to generalized Hardy-Littlewood-Sobolev inequalities (Lemma \ref{lemFHM2}). We then recall the standard way to relate Fisher information and entropy dissipation for the marginals, and compare with the results of the previous section. The next subsection \ref{subsecdiffsing} contains a new approach and new results about existence and uniqueness of drifted Brownian motions in $\mathbb R^m$ with a possibly singular drift $g \in \mathbb L^p(\mathbb R^m)$ for $p\geq m$ ($p>2$ if $m=2$), see Theorem \ref{thmkryl}. This Theorem completes the huge literature on the topic (see Remark \ref{remliterature}) and is new in the critical case $p=m$. We then apply the method of construction to particle systems (Theorem \ref{thmkrylpartic}) and rapidly discuss similar construction using the possible absence of collisions. We refer to remark \ref{remliterature} for a more complete discussion on the literature.
\medskip

Section \ref{secnonlinear} is devoted to the study of the non linear S.D.E for general (not too degenerate) interaction kernels, including the aforementioned $\mathbb L^p$ case.
\medskip

In Sections \ref{secexample} and section \ref{subsecKS}, we study in details the four examples we are mainly interested in: the $\eta$ relaxed keller-Segel model which is the prototype of the $\mathbb L^p$ case, the sub-Coulombic model which is the prototype of repulsive models, the $2D$ vortex model, possibly adding a confining potential, which is ``neutral'' (i.e. nor attractive, nor repulsive) and finally the Keller-Segel model with or without confining potential, which is the prototype of attractive models. In each case we discuss existence, uniqueness and absolute continuity of the law of the particle system, existence and uniqueness for the non-linear S.D.E. The results for the sub-Coulombic model and the $2D$ vortex model complete the existing one in the literature, in particular by estimating the relative entropy of the law of the particle system w.r.t. a well chosen product measure (see Theorem \ref{thmserfexist}, Theorem \ref{thmbiotexist} and Theorem \ref{thmbiotexistconfine}) as well as the non-linear S.D.E (Theorem \ref{thmnonlincoulomb} and Proposition \ref{propedpvortex}). These results are completed and extended in Section \ref{secchaos2}.
\medskip

Section \ref{subsecKS} is entirely devoted to the Keller-Segel model. We recall the existing results for the P.D.E. \eqref{eqmcv} and for the particle system. In particular we recall that collisions occur with a positive probability, so that the law of the particle system is not absolutely continuous w.r.t. the Wiener measure (or any product measure). We then spend some time to prove a result suggested in \cite{BJW} without any proof, namely that the marginals flow of the particle system is an entropy solution of the associated Liouville equation (see Theorem \ref{propKSconfine}) at least when adding a confining potential. The proof is delicate and technical. We also see how existence and some uniqueness for the non-linear S.D.E. is an immediate consequence of the results of section \ref{secpath} and the known analytic results for the Keller-Segel P.D.E. We also extend the existing uniqueness results on the non linear PDE.
\medskip

Section \ref{secchaos1} is a little bit different. We mainly obtain tiny improvements of the results by Lacker, and then apply them to singular models with cut-off. These results are related to several papers in the literature.
\medskip

Sections \ref{secchaos2} and \ref{secchaos3} are devoted to the proof of convergence and propagation of chaos for the four singular models we are interested in. Based on the material of the first six sections, we build on the usual strategy: proving tightness, thanks to some bounds on relative entropy (or on moments), identify all possible limits as solutions of the non linear S.D.E., prove or use some uniqueness result for the latter. We also derive a slightly original way to obtain asymptotic independence of $k$ particles (i.e. propagation of chaos). Let us give the main results
\medskip

\textbf{Main results.}
\medskip

In all what follows $Q^N$ (resp. $Q^{k,N}$) denotes the law of the particle system defined in \eqref{eqsys} (resp. the law of the first $k$ particles, for $k\leq N$). Existence and uniqueness of the solution of \eqref{eqsys} is in general not as simple, and is part of the Theorems below.
\medskip

We start with the $\mathbb L^p$ case.
\begin{theorem}\label{thmlpfinintro}(see Theorem \ref{thmlpfin} and Theorem \ref{thmkrylorl}.)

Assume that $K \mathbf 1_{|K|>A} \in \mathbb L^p(\mathbb R^d)$ for some $p \geq d$ if $d\geq 3$ or $p>2$ if $d=2$ and some $A>0$ and that the additional drift $b$ is continuous and bounded. Also assume that the initial condition $\mu_0^N=\rho^N \, dx$ is chaotic so that $\mu_0^{k,N} \to (\bar \rho_0 \, dx^1)^{\otimes k}$ and satisfies $H(\mu_0^N|\gamma_0^{\otimes N}) \leq CN$.

Then $Q^N$ is chaotic and for each $k$, $Q^{k,N}$ weakly converges to $(\bar Q)^{\otimes k}$ where $\bar Q$ is the unique solution of the non linear SDE \eqref{eqnldiff} with initial condition $\bar \rho_0 \, dx$.

In the case $d=2$, one can improve the result, defining $K_A= K \mathbf 1_{|K|>A}$ and assuming, $$\int |K_A|^2 \, |\ln(|K_A|)| \, dx \, < \, +\infty \, .$$
\end{theorem}
This result is new in the critical case $d\geq 3$ and $p=d$ and for $d=2$. The proof works for the all range $p\geq d$ and does not use $\mathbb L^q$ bounds for the Girsanov density as previous ones for $p>d$ (the simpler proof of this kind being the one in \cite{Toma-p}). The case $d=2$ requires to deal with fine properties of Orlicz spaces. At this point we may also mention that our method of proof extends to diffusion processes with a non-constant or even a degenerate diffusion matrix (like kinetic models), while the former approaches based on Hasminski's result (see Remark \ref{remliterature}) do not. This will be done in a forthcoming work.
\bigskip

The $2D$ vortex model is studied in details in \cite{FHM}, and despite intermediate new results, we do not obtain new results for propagation of chaos, except if we add some confinement potential (though even in this case one can presumably adapt the proof of \cite{FHM})
\begin{theorem}(see Theorem \ref{thmchaosmulti}.)
\textbf{The 2D vortex case with confinement.} \quad Let $d=2$, $K(x)=\chi \, \frac{x^\perp}{|x|^2}$.

Assume that the initial condition $\mu_0^N=\rho_0^N \, dx$ is chaotic so that $\mu_0^{k,N} \to (\bar \rho_0 \, dx^1)^{\otimes k}$, satisfies $H(\mu_0^N|\gamma_0^{\otimes N}) \leq CN$ and finally satisfies $\int |x|^2 \, \bar \rho_0(x) dx < +\infty$. Also assume that the additional drift $b$ is a confining potential as in Theorem \ref{thmbiotexistconfine}.

Then $Q^N$ is chaotic and for each $k$, $Q^{k,N}$ weakly converges to $(\bar Q)^{\otimes k}$ where $\bar Q$ is the unique solution of the non linear SDE \eqref{eqnldiff} with initial condition $\bar \rho_0 \, dx$ satisfying $\int_0^T \, I(\bar Q\circ \bar \omega_t^{-1}) dt < +\infty$.
\end{theorem}
\bigskip

For the sub-Coulombic model we obtain
\begin{theorem}(see Theorem \ref{thmchaosRSimproved}.)
Consider the sub-Coulombic case:  \quad $d \geq 3$, $\chi <0$, , and $$K(x)= \chi \, \frac{x}{|x|^{s+2}} \, \mathbf 1_{x\neq 0} \; \textrm{ for } \; 0<s \leq d-2 \, ,$$ also assume that the additional drift  $b$ is bounded and Lipschitz. Assume that the initial condition $\mu_0^N=\rho_0^N \, dx$ is chaotic so that $\mu_0^{k,N} \to (\bar \rho_0 \, dx^1)^{\otimes k}$, and satisfies $H(\mu_0^N|\gamma_0^{\otimes N}) \leq CN$. Finally assume that for some $q>d/d-s-2$, $\bar \rho_0 \in \mathbb L^q(\mathbb R^d)$.   
\medskip

Then $Q^N$ is chaotic and for each $k$, $Q^{k,N}$ weakly converges to $(\bar Q)^{\otimes k}$ where $\bar Q$ is the unique solution of the non linear SDE \eqref{eqnldiff} with initial condition $\bar \rho_0 \, dx$ satisfying $\int_0^T \, ||\bar Q\circ \omega_t^{-1}||_{\mathbb L^q} \, dt < + \infty$.
\end{theorem}
At the qualitative level this result extends \cite{RS21} in various directions. First it holds at the process level and not only at the marginals flow level. Second, the conditions on $\bar \rho_0$ are more general, since only the case $\bar \rho_0$ bounded is studied in \cite{RS21}. 
\bigskip

Finally for the Keller-Segel model we get

\begin{theorem}\label{thmalleluiaintro}(see Corollary \ref{corKSorl} and Theorem \ref{thmksfinalchi}.)

Consider the Keller-Segel case $d=2$, $0<\chi<4$ and $$K(x)= \chi \, \frac{x}{|x|^2} \, \mathbf 1_{x \neq 0}$$ with or without a confining potential $U$ such that $\int e^{- U} dx < +\infty$, i.e. an additional drift $b=\nabla U$ and assuming that $b$ is Lipschitz.

Assume that the initial condition $\mu_0^N=\rho_0^N dx$ is chaotic so that $\mu_0^{k,N} \to (\bar \rho_0 \, dx^1)^{\otimes k}$, satisfies $\int \rho_0^N \, |\ln \rho_0^N| dx < C N$ and $\int |x^1|^2 \, \rho_0^{1,N} \, dx^1 \, < +\infty$. Then
\begin{itemize}
\item[(1)] \quad if $0<\chi<2$,  $Q^N$ is chaotic  and for each $k$, $Q^{k,N}$ weakly converges to $(\bar Q)^{\otimes k}$ where $\bar Q$ is the unique solution of the non linear SDE \eqref{eqnldiff} with initial condition $\bar \rho_0 dx$ whose marginals flow is the (unique) free energy solution of the non linear P.D.E. \eqref{eqmcv}.
\item[(2)] \quad if $0<\chi<4$ and $\int |x^1|^p \, \rho_0^{1,N} \, dx^1 \, < +\infty$ for some $p>4\chi$, $\rho^{1,N}_{t \in [0,T]}$, the marginals flow of one particle, weakly converges to the unique free solution of the Keller-Segel equation.
\end{itemize}
\end{theorem}
We also prove several intermediate results. For example, that actually the non linear S.D.E. admits only one solution such that its marginals flow is bounded in entropy uniformly in time.

In \cite{BJW} Theorem 1.1 a similar result is obtained for the Keller-Segel model on the torus, with in addition a quantitative rate of convergence, but at the level of the Liouville equation, for entropy solutions of the latter. Compactness of the state space is essential in the proof, as well as additional assumptions on the non linear PDE. Theorem \ref{thmalleluiaintro} is thus the first one giving the convergence (not only tightness) and identifying the limit. The latter point is delicate and requires to use Orlicz spaces and the extension of Theorem \ref{thmlpfinintro} to this framework.
\bigskip

We emphasize that, except for Theorem \ref{thmalleluiaintro} (2), all the convergence results are obtained at the process level. This is mainly new.

Another important point is that, except for the Keller-Segel case, the results obtained for the non linear SDE in the final sections, do not call upon existing results in the PDE literature for the non linear PDE. Moreover, these results extend many known results for the non linear PDE.
\bigskip

To conclude this long introduction, let us say that we decided to study at the same time all these models in order to compare their behaviour and to exhibit what are the differences (attractive vs repulsive or neutral for instance). It turns out that Theorem \ref{thmlpfinintro} is also useful in some of the other cases. As we briefly said before, the approach developed in the present work, can be adapted to much more general models. This will be done elsewhere. Of course the previous approach only furnishes qualitative results, while part of the mainstream, is now concerned with quantitative results. It is not hard to see that obtaining quantitative results using our approach essentially depends on our ability to get better controls in $N$ for the relative entropy of the full particle system. This is exactly what is done in recent quantitative papers, at the Liouville equation (marginals flow) level. It should be interesting to relax Lacker's transportation assumption in this direction.
\bigskip

\textbf{Acknowledgments.} \quad I would like to heartily thank Nicolas Fournier who pointed out two non negligible mistakes in the first version of this work. I also thank my old friends Christian L\'eonard and Arnaud Guillin who encouraged (obliged) me to write this paper.

\bigskip

\section{Entropy on the path space, Dirichlet forms and applications.}\label{secpath}

\subsection{Kullback information on the path space.\\ \\}\label{subsecpath}

On $\Omega^d_T=C([0,T],\mathbb R^{d})$ equipped with its usual filtration, we introduce the probability measure $P$ which is the law (up to time $T$) of a reversible diffusion process $Y_.=(Y_.^1,...,Y_.^d)$, the $Y_.^i$'s being independent copies of the $1$-dimensional  process, satisfying 
$$y_t = y_0 + \sqrt 2 \, B_t \, - \, \int_0^t  \, V'(y_s) \, ds$$ where $y_0$ is a random variable with density $p_0(y)= Z_V^{-1} \, e^{- V(y)}$. We will assume that existence and strong uniqueness hold for this S.D.E. It is well known that $\gamma_0(dy)=\prod_{i=1}^d \, p_0(y^i) \, dy := \, Z_V^{-d} \, e^{- \bar V(y)} \, dy$ is a reversible measure for this process so that if we define on $\Omega^d_T$ the time reversal operator
\begin{equation}\label{eqdefR}
R(\omega) = \bar \omega \quad \textrm{ where } \quad \bar \omega_t=\omega_{T-t}
\end{equation}
it holds $P\circ R^{-1}=P$.
\medskip

The two main examples we have in mind are the Ornstein-Uhlenbeck process corresponding to $V_2(y)=y^2/2$ and a gaussian reversible distribution, and $V_1(y)= |y|$ or $V_1$ a smooth ($C_b^2$) non-negative function that coincides with $|y|$ for $|y|\geq 1$. The main advantage of the  choice of $V_1$ is that $V_1'$ and $V''_1$ are bounded. Other choices of such $V$'s should also be interesting.
\medskip

\begin{remark}\label{remframe}The choice of a product measure $P$ is arbitrary. All what follows remains true if $\bar V$ is a general potential such that the $e^{-\bar V} dx$ is a bounded measure and $P$ is the reversible diffusion process associated to this potential. The only point is that the dimension dependence is not as explicit (see for instance Lemma \ref{lemmoment1} below). \hfill $\diamondsuit$
\end{remark}

Let $Q$ be another probability distribution on $\Omega^d_T$ with finite relative entropy w.r.t. $P$, i.e. such that $$H(Q|P):=\int \, \ln \left(\frac{dQ}{dP}\right) \, dQ \, < \, +\infty \, .$$ Denote  $\eta_0=Q\circ \omega_0^{-1}$. First, since relative entropy is non-increasing by pushforward, it holds $H(\eta_0|\gamma_0) < +\infty$. If $\rho_0$ denotes the density of probability of $\eta_0$ w.r.t. the Lebesgue measure, the latter writes
\begin{equation}\label{eqent1}
H(\eta_0|\gamma_0) = \int \, (\ln \rho_0(x) + \bar V(x) + d \ln Z) \, \rho_0(dx) \, < \, +\infty.
\end{equation}
Since relative entropy is non-negative we immediately deduce that for any density of probability $\rho$,
\begin{equation}\label{eqent2}
\int \, \rho(x) \ln \rho(x) \, dx \geq - \, d \ln Z_V - \, \int \bar V(x) \, \rho(x) \, dx \, ,
\end{equation} 
so that in particular, as mentioned in \cite{GQ} lemma 2.3, one can bound from below the entropy $h(\rho)$ by the opposite of some moments of $\rho$. 

Remark that if $\rho_0$ is exchangeable, we get 
\begin{equation}\label{eqexchangentrop}
\int \, \rho_0(x) \ln \rho_0(x) \, dx \geq - \, d \ln Z_V - \, d \, \int_{\mathbb R} V(u) \, \rho^1_0(u) \, du \, .
\end{equation}
Notice that in \cite{GQ} the entropy is normalized (by $2/d$ in our framework) explaining why the bound in lemma 2.3 therein, is dimension free.
\medskip

Second, Girsanov transform theory ensures that there exists a $\mathbb R^d$ previsible process $\beta_.$ (we will call the drift of $Q$) such that $dQ/dP$ is given by the exponential (Doleans-Dade) martingale associated to the process $\beta_.$, i.e. 
\begin{equation}\label{eqgirsanov}
\frac{dQ}{dP}|\mathcal F_T= (d\eta_0/d\gamma_0) \, \exp \left(\int_0^T \, \langle \beta_s \, , \, \sqrt 2 \, dB_s\rangle \, -  \, \int_0^T \, |\beta_s|^2 \, ds\right)
\end{equation}
where $B_s= (1\sqrt 2)(\omega_s - \int_0^s \, \nabla \bar V(\omega_u) \, du)$ is a Brownian motion. In addition it holds
\begin{equation}\label{eqentp1}
H(Q|P) \, = \, H(\eta_0|\gamma_0) + \, \int \, \int_0^T \, |\beta_t|^2 \, dt \, dQ \, .
\end{equation}
More generally, if $Q$ is absolutely continuous w.r.t. $P$ with drift $\beta_.$ furnished by the Girsanov transform theory, a necessary and sufficient condition for $H(Q|P)$ to be finite is that the right hand side of \eqref{eqentp1} is finite, and in this case $H(Q|P)$ is given by \eqref{eqentp1}. 

All this is developed in \cite{CL1} (the correction \cite{CL1c} is not useful in the present framework), in connection with the construction of conservative diffusion processes with finite energy, we shall revisit later. 

The key observation is that relative entropy is invariant under time reversal, so that here 
\begin{equation}\label{eqret1}
H(Q\circ R^{-1}|P) = H(Q|P) \, .
\end{equation}
This elementary observation was first made in \cite{Fol1,Fol2} in order to study time reversal on the Wiener space, then developed in the unpublished \cite{CatPet} and pushed forward recently in \cite{CCGL}. We refer to the latter for all the results we shall now describe. Readers who are not familiar with stochastic forward and backward derivatives should have a look at section 3 in \cite{CatPet}. 

Since $Q\circ R^{-1}$ has finite relative entropy we can use the same arguments as before and obtain some ``backward'' drift $\bar \beta_.$. One then obtains (see \cite{CCGL} Theorem 4.9 or \cite{CatPet} Corollary 3.15)
\begin{proposition}\label{propretour}
Assume that $H(Q|P)<+\infty$. Then
\begin{enumerate}
\item For all $t \in [0,T]$, the time marginal law $Q \circ \omega_t^{-1}=\eta_t$ is absolutely continuous w.r.t. the Lebesgue measure, and its density $\rho_t$ satisfies $$\int \, (\ln \rho_t(x) + \bar V(x)) \, \rho_t(x) \, dx \, < \, +\infty.$$
\item  One can find two previsible processes $\beta_.$ and $\bar \beta_.$ such that, $$H(Q|P) \, = \, H(\eta_0|\gamma_0) + \, \int \, \int_0^T \, |\beta_t|^2 \, dt \, dQ \, = \, H(\eta_T|\gamma_0) +  \, \int \, \int_0^T \, |\bar \beta_{T-t}|^2 \, dt \, dQ\circ R^{-1}$$
\item For all $f \in C_c^{\infty}(\mathbb R^d)$, any vector $u \in \mathbb R^d$ and almost any $t>0$ it holds 
\begin{equation}\label{eqdual}
 - \, \int \langle u,\nabla f(\omega_t)\rangle \, dQ \, = \, \int \, \left(\langle u,(\beta_t + \bar \beta_{T-t}\circ R) -  \, \nabla \bar V(\omega_t)\rangle\right) \, f(\omega_t) \, dQ \, .
\end{equation}
\end{enumerate}
\end{proposition}

The final equation \eqref{eqdual} is called the \emph{duality equation}. The accurate reader will check that the previous result is written for ``Markov'' probability measures $Q$ in \cite{CCGL} (but in the general form in \cite{CatPet}). One can of course take conditional expectations w.r.t. $\sigma(\omega_t)$ and obtain ``Markov type'' drifts. It is precisely the aim of section 3 in \cite{CL1} to prove that the obtained probability measure has the same time marginals as the initial one. In any case, taking these conditional expectations we get the following corollary (see \cite{CCGL} (4.12) or \cite{CatPet} lemma 4.9)
\begin{corollary}\label{cordual}
In the situation of proposition \ref{propretour}, the density $\rho_t$ satisfies the following duality equation $$\nabla_x \rho_t = \rho_t \, \left[\left(\mathbb E^Q(\beta_t|\omega_t=x) + \mathbb E^Q(\bar \beta_{T-t}\circ R|\omega_t=x)\right) -  \nabla \bar V(x)\right]$$ for almost all $t \in (0,T)$ in the sense of Schwartz distributions on $\mathbb R^d$. As a consequence, the Fisher information $$I(\rho_t):=\int \, \rho_t \, |\nabla \ln(\rho_t)|^2 \, dx \, ,$$ satisfies, for all $T>0$, all $\lambda>0$, $$2(1+\lambda) H(\rho_T dx|\gamma_0) + \int_0^T \, I(\rho_t) \, dt \, \leq  4(1+\lambda) \int \, \int_0^T \, |\beta_t|^2 \, dt \, dQ  \, +$$ $$  + \left(1+\frac{1}{\lambda}\right)  \,  \int_0^T \, \int \, |\nabla \bar V(x)|^2 \rho_t(x) \, dx \, dt + (1+\lambda) \, \int (\ln \rho_0(x) +\bar V(x)) \, \rho_0(x) \, dx \, .$$
\end{corollary}
\begin{remark}\label{remfish0}
Recall that in the definition of $I(\rho)$, $|\nabla \rho|^2/\rho = 0$ by convention on the set $\rho=0$. \hfill $\diamondsuit$
\end{remark}

The proof of the final bound is an immediate application of $(a+b)^2 \leq (1+\lambda) a^2+(1+(1/\lambda) b^2$, of Cauchy-Schwarz inequality for the conditional expectations and of (2) in proposition \ref{propretour}. Since relative entropy $H(\eta_T|\gamma_0) \geq 0$ we thus get a bound for the Fischer information..

Notice that the same (2) in proposition \ref{propretour} also furnishes
\begin{equation}\label{eqent3}
\int (\ln \rho_T(x) +\bar V(x)) \, \rho_T(x) \, dx \leq \int \, \int_0^T \, |\beta_t|^2 \, dt \, dQ + \int (\ln \rho_0(x) +\bar V(x)) \, \rho_0(x) \, dx \, .
\end{equation}
These two bounds can replace (with different constants) equation (3.8) in \cite{GQ}. They also will be related to \eqref{eqentrop} and the notion of entropy solution in \cite{BJW}.
\medskip

In some situations one can go further. For instance a simple application of Ito's formula furnishes the next lemma.
\begin{lemma}\label{lemmoment1}
In the situation of proposition \ref{propretour}, assume that $V$ is smooth and that, $|V'|$ and $|V''|$ are bounded by $A$. Then $$\int \bar V(x) \, \rho_T(x) \, dx \leq \int \bar V(x) \, \rho_0(x) \, dx + A \, \sqrt {dT} \, \left(\int \, \int_0^T \, |\beta_t|^2 \, dt \, dQ\right)^{\frac 12} + A d T \, .$$ As a consequence
\begin{eqnarray}\label{eqent4}
\int \ln \rho_T \, \rho_T \, dx &\geq& \, - (d\ln Z_V +\int \bar V(x) \, \rho_0(x) \, dx + A \, \sqrt d \, \left(\int \, \int_0^T \, |\beta_t|^2 \, dt \, dQ\right)^{\frac 12} + dA T) \nonumber \\
\int \ln \rho_T \, \rho_T \, dx &\leq& \int \, \int_0^T \, |\beta_t|^2 \, dt \, dQ + \int (\ln \rho_0(x) +\bar V(x)) \, \rho_0(x) \, dx \, ,
\end{eqnarray}
and 
\begin{equation}\label{eqfish}
\int_0^T \, I(\rho_t) \, dt \, \leq 4(1+\lambda) \int \, \int_0^T \, |\beta_t|^2 \, dt \, dQ + (1+(1/\lambda)) \, d \, A^2 \, T + (1+\lambda) \, \int (\ln \rho_0(x) +\bar V(x)) \, \rho_0(x) \, dx \, .
\end{equation}
\end{lemma}
\medskip

The main result of \cite{CL1} is a converse construction. Let $t \mapsto \nu_t$ be a flow of probability measures on $\mathbb R^d$ and $(t,x) \mapsto b(t,x)$ some measurable function.
\begin{definition}\label{defadmis}
We will say that $\nu_.$ is an admissible flow if
\begin{enumerate}
\item[(i)] \quad $\nu_.$ satisfies the $(b,C_b^\infty)$ weak forward equation i.e. for all $f\in C_b^{1,\infty}(\mathbb R^+\times \mathbb R^d)$ and all $0\leq s <t \leq T$, $$\int f(t,x)d\nu_t - \int f(s,x)d\nu_s = \int_s^t \int (\partial_u f+ \Delta_x f+ \langle 2b(u,x)-\nabla \bar V(x),\nabla_x f\rangle) \, d\nu_u \, du \, .$$
\item[(ii)] \quad $b$ is of finite $\nu$-energy, i.e. for all $T\geq t>0$, $$\int_0^t \int |b(s,x)|^2 \nu_s(dx) ds \, < \, +\infty \, .$$
\item[(iii)] \quad $H(\nu_0|\gamma_0) < +\infty$.
\end{enumerate}
\end{definition}
Theorem (4.18) in \cite{CL1} then says
\begin{theorem}\label{thmCL}
Let $\nu$ be an admissible flow. Define $\beta_s= b(s,\omega_s) \, \mathbf 1_{s<\tau}$ where $\tau =\sup_n \tau_n$ and $\tau_n= \inf \{t>0, \int_0^t \, |b|^2(u,\omega_u) \, du \geq n\}$. 

Then $Q$ defined by \eqref{eqgirsanov} is a probability measure satisfying $H(Q|P)<+\infty$ and is a weak solution (up to time $T$) of the stochastic differential equation (SDE) $$dX_t= \sqrt2 \, dB_t \, + \, (2b(t,X_t)-\nabla \bar V(X_t)) dt$$ with initial distribution $\nu_0$. In addition $Q\circ \omega_t^{-1} = \nu_t$.
\end{theorem}
\begin{remark}\label{remCL1}
The proof in \cite{CL1} uses stochastic calculus. In \cite{CL2,CL3} the same result (and more general ones) is obtained by using Sanov theorem, i.e. a large deviation approach. \hfill $\diamondsuit$
\end{remark}

Actually the construction in \cite{CL1} is more general. We may replace $P$ by any measure with initial measure $\nu_0$ which is no more reversible nor invariant, provided $P$ is the unique solution of the associated martingale problem. Theorem \ref{thmCL} then becomes
\begin{theorem}\label{thmCL2}
Assume that $P$ is the unique solution of the martingale problem associated to the generator $L_P=\Delta + 2 \, g \, \nabla$, with initial measure $\theta_0$ (no other specific condition than measurability is imposed on $g$). Let $\nu$ be an admissible flow in the following sense
\begin{enumerate}
\item[(i)] \quad $\nu_.$ satisfies the $(b,C_b^\infty)$ weak forward equation i.e. for all $f\in C_b^{1,\infty}(\mathbb R^+\times \mathbb R^d)$ and all $0\leq s <t \leq T$, $$\int f(t,x)d\nu_t - \int f(s,x)d\nu_s = \int_s^t \int (\partial_u f+ \Delta_x f+ \langle 2b(u,x),\nabla_x f\rangle) \, d\nu_u \, du \, .$$
\item[(ii)] \quad $b-g$ is of finite $\nu$-energy, i.e. for all $T\geq t>0$, $$\int_0^t \int |(b-g)(s,x)|^2 \nu_s(dx) ds \, < \, +\infty \, .$$
\item[(iii)] \quad $H(\nu_0|\theta_0)<+\infty$.
\end{enumerate}
Then $Q$ defined by \eqref{eqgirsanov} and $\beta_s=(b-g)(s,\omega_s)$ is a probability measure satisfying $$H(Q|P) = H(\nu_0|\theta_0) + \int_0^T \int |(b-g)(s,x)|^2 \nu_s(dx) ds \, < \, +\infty \, .$$  $Q$ is a weak solution (up to time $T$) of the stochastic differential equation (SDE) $$dX_t= \sqrt2 \, dB_t \, + \, 2b(t,X_t) dt$$ with initial distribution $\nu_0$. In addition $Q\circ \omega_t^{-1} = \nu_t$ and $\nu_t(dx)=\rho_t(x) \, dx$ for some flow of probability densities.
\end{theorem}
\medskip

One can also use similar time reversal arguments, but this time the duality equation involves the marginals flow of $P$ and is thus more delicate (see \cite{CCGL,CP}), except when $P$ is time reversible, in which case what has previously been done extends immediately. We shall describe an application of this more general statement. 

\begin{remark}\label{remenerglibre}
Let $P$ be as in Theorem \ref{thmCL2}, with a time homogeneous drift $g$ and assume that $P\circ R^{-1}=P$.  Let $Q=\frac{d\nu_0}{d\theta_0}\, P $. $Q$ is thus the solution of $$dY_t=\sqrt 2 \, dB_t + 2 g(Y_t) dt \quad ; \quad Q\circ\omega_0^{-1}=\nu_0 \, .$$ In addition $$H(Q|P) = H(\nu_0|\theta_0)$$ is assumed to be finite. We thus deduce the existence of a time reversed drift $\hat \beta_s$ such that $$H(\nu_0|\theta_0)=H(Q|P)=H(Q\circ R^{-1}|P)= H(\nu_T|\theta_0) + \int \int_0^T |\hat \beta_{T-t}|^2 dt dQ\circ R^{-1}$$ and satisfying (if $\rho_t$ denotes the density of $\nu_t$),
\begin{equation}\label{eqdualrev}
\nabla_x \rho_t = \rho_t \, \left[\mathbb E^Q(\hat \beta_{T-t}\circ R|\omega_t=x) +  2g(x)\right]
\end{equation}
so that
\begin{equation}\label{eqenerglibre}
H(\nu_0|\theta_0) \geq  H(\nu_T|\theta_0) + \int_0^T \, \int \, |\nabla \ln(\rho_t) - 2g|^2 \, \rho_t \, dx \, dt \, .
\end{equation}
When $g$ derives from a potential i.e. $2g=-\nabla U$ (this is natural since $P$ is reversible), one recognizes the \emph{free energy} functional $F(\rho)=\int \rho (\ln(\rho) + \, U) dx$, used in several studies of non linear P.D.E., in particular, \eqref{eqdualrev} can be rewritten $$F(\rho_0) \geq F(\rho_T) +  \int_0^T \, \int \, |\nabla \ln(\rho_t) +  \nabla U|^2 \, \rho_t \, dx \, dt$$ which is similar to (1.2) in \cite{BDP} concerned with the Keller-Segel model. In other words, if it exists an unique reversible solution to the martingale problem associated to $L_P$, the marginals flow, starting from a finite free energy condition, is naturally a \emph{free energy solution} of the associated Liouville equation. We shall see that this property is true in many cases.  
\medskip

In the previous situation we may obtain additional informations on $\rho_t$. Indeed if $d\theta_0=\theta_0(x) \, dx$,  using reversibility we obtain $$\rho_t(x) = \mathbb E_P\left[\frac{\rho_0}{\theta_0}(Y_t)|Y_0=x\right] \, \theta_0(x) \, .$$ In particular $$||\rho_t/\theta_0||_{\mathbb L^q} \leq 
||\rho_0/\theta_0||_{\mathbb L^q}$$ for all $1\leq q \leq +\infty$ and $$\inf (\rho_t/\theta_0) \geq \inf (\rho_0/\theta_0) \, .$$ 
\hfill $\diamondsuit$
\end{remark}
\begin{remark}\label{remac}
It is interesting to notice the difference between finite relative entropy and absolute continuity. It is indeed well known that if a stochastic process $Y_.$ is such that $Y_. - \int_0^. h(Y_s)ds$ is a Brownian motion, then the law $Q$ of $Y_.$ on a finite time interval $[0,T]$ is absolutely continuous w.r.t. the Wiener measure if and only if $\int_0^T \, |h(Y_s)|^2 ds < +\infty$ almost surely. One can of course replace the Wiener measure by any $P$ which is equivalent, as the one we are considering in this section. Passage to finite entropy is thus simply assuming that this quantity has finite expectation.
\hfill $\diamondsuit$
\end{remark}

\begin{remark}\label{remuniqueCL} \quad Let us say a word about weak uniqueness, since this aspect is not explicitly stated in \cite{CL1}. The two previous Theorems \ref{thmCL} and \ref{thmCL2} are saying that a solution $Q$ is built via \eqref{eqgirsanov}. This measure is actually the F\"{o}llmer measure associated to the exponential martingale (see \cite{CL1} 1.13, 1.14 and 1.15 for this notion and references). In particular, contained in the statement of these theorems, is the following fact: $\tau_n \to T$, $Q$ almost surely as $n$ goes to infinity. Now if $Q'$ denotes a solution of the SDE in the theorems, Girsanov theory says that $Q'$ and $Q$ coincide up to the stopping times $\tau_n$. In particular $\tau_n \to T$, $Q'$ almost surely too and $Q=Q'$.
\hfill $\diamondsuit$
\end{remark}

\begin{remark}\label{remtrevisan}
\quad The construction of a diffusion with given marginals flow is an old question. We discovered it in Carlen's paper \cite{carlcons} which was motivated by Nelson's stochastic mechanics. The method used in \cite{carlcons} is purely analytic. In our papers \cite{CL1,CL2,CL3} we have shown that this construction can be done (and extended) using relative entropy on the path space. Note that the forward and backward derivatives used in in \cite{CCGL} are actually an extended version of the ones introduced by Nelson. It seems that this question did not attract interest (except in some works on mathematical finance) during the following twenty years, some of our results being ``rediscovered'' recently in particular in the mean field games theory (see e.g. \cite{Porre}). A new existence result appeared in \cite{Trevi} where the following is shown: 

\begin{theorem}\label{thmtrevisan}
If a flow $\eta_.$ of probability measures satisfies the $(b,C_b^{1,\infty})$ weak forward equation and the integrability condition $\int \int_0^T \, |b(u,x)| \, \eta_u(dx) \, du < +\infty$, one can build a solution $Q$ to the martingale problem such that $\eta_.$ is exactly the marginals flow of $Q$.

This is Theorem 2.5 in \cite{Trevi}, called ``superposition principle'' therein.
\end{theorem}
It is worth to notice that, even if a less demanding $\mathbb L^1$ condition is required (instead of the energy $\mathbb L^2$ integrability), weak uniqueness is not contained in the superposition principle, contrary to the entropic case (Remark \ref{remuniqueCL}). As explained in Remark \ref{remac}, absolute continuity is not ensured for the process built via the previous principle.
\hfill $\diamondsuit$
\end{remark}

\medskip

\subsection{Some useful elements of Dirichlet forms theory. \\ \\}\label{subsecdiric}

As we said the use of Dirichlet forms can be (at least at a ``theoretical'' level, as said in \cite{FJ}) very useful in our context. We shall here recall some useful results obtained in the 90th's. We refer to \cite{MR92,FOT,CF96} and the references in the last mentioned paper for all the material described below.
\medskip

Let $\nu(dx)=\rho(x) \, dx$ be a $\sigma$-finite, non necessarily bounded, positive measure. Consider the symmetric bilinear form $$\mathcal E_\rho(f,g)= \int \, \langle \nabla f,\nabla g\rangle \, d\nu$$ defined on $C_c^\infty(\mathbb R^d)$. If
\begin{equation}\label{eqdiric1}
\rho^{\frac 12} \in H^1_{loc}(\mathbb R^d) \quad \textrm{i.e. for any $R$,} \quad \int_{|y|\leq R} \, |\nabla \ln(\rho)|^2 \,  d\nu < +\infty \, ,
\end{equation}
then the form $\mathcal E_\rho$ is closable and its minimal extension (still denoted by $\mathcal E_\rho$ on its domain) is a Dirichlet form. 

Assume that $\nu$ is bounded (hence a probability measure after renormalization) and $\rho^{\frac 12} \in H^1(dx)$. \eqref{eqdiric1} thus simply becomes the finite energy condition in Theorem \ref{thmCL2} for the stationary flow $\rho_t=\rho$. We recall below the main results proved almost thirty years ago.

\begin{enumerate}
\item \quad If $\rho \neq 0$ $dx$ almost everywhere, then it is also the maximal closed markovian extension, i.e. Markov uniqueness holds.
\item \quad Under the previous assumption, $\mathcal E_\rho$ is regular and local, so that there exists a $\nu$-symmetric diffusion $Q_x$ associated to this form. 
\item \quad $Q_\nu$ is a weak solution of 
\begin{equation}\label{eqdiric}
dY_t =  \sqrt 2 \, dB_t + \nabla (\ln\rho)(Y_t) dt
\end{equation}
with initial distribution equal to $\nu$. Furthermore $(Q_x)_{x \in \mathbb R^d}$ is a solution of the same SDE starting from $x$ for quasi-every $x$.
\item $Q_\nu$ is absolutely continuous w.r.t. the Wiener measure with initial distribution $\nu$ denoted by $W_\nu$, and the density is given by \eqref{eqgirsanov} with $\beta=\nabla \rho/\rho$, hence for all $T>0$, $$H(Q_\nu|W_\nu)<+\infty \, \textrm{ in restriction to $[0,T]$.}$$ For $\nu$ quasi every $x$, $Q_x$ is absolutely continuous w.r.t the Wiener measure starting from $x$ with the same Girsanov density.

\noindent As a consequence, for $\nu$ almost all $x$, $H(Q_x|W_x) < +\infty$ and $$H(Q_\nu|W_\nu) = \int H(Q_x|W_x) \, \nu(dx) \, .$$
\end{enumerate}

An interesting point, already proved in \cite{MZ}, is that, if $\rho$ is a density of probability,  the \emph{nodal set} i.e. $\mathcal N=\{\rho=0\}$ is not attained. According to section 4 in \cite{CF96}, one can ``desintegrate'' $Q_\rho$ in order to get the density of $Q_x$ w.r.t. the Wiener measure starting from $x$, for all $x$ outside some explicit polar set, and the process does not hit the nodal set for those $x$'s, $Q_x$ almost surely. It follows that if $d\nu_0 =\rho_0 \, d\nu$ one can build a solution of \eqref{eqdiric} with initial distribution $\nu_0$. However, unless $\rho_0$ is bounded, we do not know whether this solution has finite relative entropy w.r.t. $W_{\nu_0}$. Nevertheless, we may argue exactly as in Remark \ref{remuniqueCL} and prove uniqueness.
\medskip

If $\rho(x) dx$ is not bounded, some properties are preserved, in particular the last one, provided the process is conservative. An easy way to see it is to consider a family $\rho_k$ of densities such that $\rho_k=\rho$ on the ball $B(0,k)$ and $\rho_K^{\frac 12} \in H^1(dx)$. This is easy to build. The corresponding process coincides with the initial one up to the exit time $\bar \tau_k$ of this ball, so that, if $\sigma$ denotes the hitting time of $\mathcal N$, we have $$Q_x(\bar \tau_k \wedge \sigma<+\infty) =Q^k_x(\bar \tau_k \wedge \sigma<+\infty)=Q^k_x(\bar \tau_k <+\infty) = Q_x(\bar \tau_k <+\infty) \to_{k \to +\infty} 0$$ since the process is conservative. 

Notice that part of these results are the \emph{stationary} version of Theorem \ref{thmCL} and are shown using large deviations results in Theorem 5.5 of \cite{CL3}.
\medskip

\subsection{Absolute continuity and collisions. \\ \\}\label{subsecac}

If we consider $Q^N$ the law solution of \eqref{eqsys}, if it exists, one may ask about the consequences of its absolute continuity w.r.t. $P$. As we already said, absolute continuity is equivalent to the almost sure finiteness of $\int_0^T |\beta_s|^2 ds$. Actually absolute continuity is simply connected to the absence of collisions. We shall state a general result in this direction.

\begin{theorem}\label{thmcoll}
Consider the system \eqref{eqsys} for $N\geq 2$. Assume that $b=b_1+b_2$ where $b_1$ is bounded, $b_2$ is local Lipschitz and the solution of $dY_t=\sqrt 2 \, dB_t - b_2(Y_t)dt$ is conservative, i.e. does not explode in finite time.

Assume in addition that, for any $\varepsilon>0$, $K$ is local Lipschitz and bounded in the set $|x-y|\geq \varepsilon$. Define $$C^c_\varepsilon=\{(x^1,...,x^N)\in \mathbb (R^d)^N \; ; \; \inf_{i,j} |x^i-x^j|\geq \varepsilon\} \, .$$ Let $C_0=\{(x^1,...,x^N)\in \mathbb (R^d)^N \; ; \; x^i=x^j \; \textrm{for at least one pair $i\neq j$}\}$ be the collision set. By convention $K(0)=0$.

Let $Q$ be a solution of \eqref{eqsys} such that $Q(X^N_0 \in C_0)=0$, assumed to exist. Then there is an equivalence between: 
\begin{itemize}
\item for all $T>0$, $Q$ is absolutely continuous w.r.t. the Wiener measure (actually the distribution of $\sqrt 2$ times a $dN$-dimensional Brownian motion) in restriction to $[0,T]$,
\item the hitting time of $C_0$ is $Q$ almost surely infinite, i.e. the solution of \eqref{eqsys} has no collisions.
\end{itemize}
\end{theorem}
\begin{proof}
We may first consider the case $b_1=0$, and then get the result via a Girsanov transform. The first part is then a simple consequence of the fact that for $d\geq 1$, a $2d$-dimensional Brownian motion starting away from $0$ never hits the origin. The same will hold for the distribution of any pair of particles under $Q$ if $Q$ is absolutely continuous. 
Conversely, one can build the solution starting from some $x \notin C_0$ denoted by $Q_x$ up to the exit time $\tau_\varepsilon$ of $C_\varepsilon^c \cap \{|x|\leq \varepsilon^{-1}\}$ thanks to our regularity assumptions. This sequence of stopping times goes to infinity according to our assumptions on $Q$. One can thus build a strong solution up to infinity. In addition in restriction to $T\wedge \tau_\varepsilon$, $Q_x$ is given by the Girsanov density $$Z_T^\varepsilon = \exp \left(- \int_0^{T\wedge \tau_\varepsilon} \langle \beta_s \, \sqrt 2 d\omega_s\rangle \, - \, \frac 12 \, \int_0^{T\wedge \tau_\varepsilon} |\beta_s|^2 ds\right)$$ where $\beta^j_s=b(\omega_s^j)+\sum_k K(\omega^j_s -\omega_s^k)$ and the reference measure is $P_x$. Let $A$ be $\mathcal F_T$ measurable and such that $P_x(A)=0$. Then $P_x(A\cap \{\tau_\varepsilon \geq T\})=0$ so that $Q_x(A\cap \{\tau_\varepsilon \geq T\})=0$.  It remains to use the fact that on $\mathcal F_{T\wedge \tau_\varepsilon}$, $Q=\int Q_x \; Q\circ \omega_0^{-1}(dx)$, so that $Q(A\cap \{\tau_\varepsilon \geq T\})=0$ and since $\{\tau_\varepsilon \geq T\}$ is non-decreasing ($\varepsilon$ is supposed to decrease to $0$) and satisfies $ \lim_\varepsilon Q(\tau_\varepsilon \geq T)=1$,  we get $Q(A)=0$.
\end{proof}
\medskip

\subsection{Applications to some singular models.\\ \\}\label{subsecksdir}

In this subsection we will prove a general existence theorem for the particle system \eqref{eqsys}. In order to understand the assumptions required for this goal, we shall first look at some example.
\medskip

\subsubsection{\textbf{The relaxed Keller-Segel model. }\quad}\label{subsubsecrelaxdir}
Pick $\eta >0$ and consider $$\rho_\eta(x)=\exp \, \left(-\frac{\chi}{N \eta} \, \sum_{1\leq i<j\leq N} \, |x^i-x^j|^\eta \right ) \quad , \quad x=(x^1,...,x^N)\in (\mathbb R^2)^N \, .$$ It is easily seen that $\rho_\eta^{\frac 12} \in H^1_{loc}(dx)$ but $\nu(dx)=\rho(x) dx$ is not bounded. 

In order to apply the Dirichlet forms approach, one can add a confining potential i.e. consider $$\rho^M_\eta(x)=\exp \, \left(-\frac{\chi}{N \eta} \, \sum_{1\leq i<j\leq N} \, |x^i-x^j|^\eta \right) \, \exp( - V_M(x)) \quad , \quad x=(x^1,...,x^N)\in (\mathbb R^2)^N \, ,$$ where $V_M$ is smooth, nonnegative, vanishes on $B(0,M)$ and coincides with $|x|^2$ outside $B(0,M+1)$. This time $\rho^M_\eta dx$ is a bounded measure and $\sqrt{\rho^M_\eta} \in H^1(dx)$. We may thus apply what precedes and get the existence and uniqueness of a  weak solution of 
$$
dX^{i,N,M}_t =\sqrt 2 \, dB^{i,N}_t \, - \, \frac{\chi}{N} \, \sum_{j=1}^N \, \frac{X_t^{i,N,M}-X_t^{j,N,M}}{|X_t^{i,N,M}-X_t^{j,N,M}|^{2-\eta}} \, dt  \, - \nabla_i V_M(X_t^{N,M}) \, dt \quad , \quad i=1,...,N
$$
with initial distribution any probability measure absolutely continuous w.r.t. $\rho^M_\eta$, hence w.r.t. $dx$. Actually one can also consider $\delta_x$ for quasi every $x$. Notice that arguing as in \cite{CP} subsection 3.1, one can prove that the sets of zero capacity are exactly the polar sets for the process. If the initial measure is absolutely continuous w.r.t. Lebesgue's measure, we deduce from what we said in subsection \ref{subsecdiric} that $Q$ is absolutely continuous w.r.t. the Wiener measure, and thanks to Theorem \ref{thmcoll}, that there are no collisions. 

If we start from a non collision point $x$, we may build a solution up to the exit time $\tau$ of a small ball centered at $x$ that does not intersect $C_0$. The law of $X^{N,M}_{\varepsilon \wedge \tau}$ is supported by a non-polar set, so that we can build a solution starting from it. It remains to concatenate both laws to build a solution starting from $x$. Hence we get a solution starting from any non collision point $x$, and this solution never hits $C_0$, so that it is a strong solution as the authors of \cite{GQ} are saying.
\medskip

Since we have strong solutions, we may build a family of solutions for $M \in \mathbb N^*$ with the same Brownian motion and define $X_t^N = X_t^{N,M} \, \mathbf 1_{T_{M-1}<t\leq T_M}$ where $T_M$ is the exit time of $B(0,M)$ for the process $X_.^{N,M}$, hence for $X_.^N$. It remains to show that $\lim_{M \to +\infty} T_M=T_\infty=+\infty$ almost surely, i.e. that the law $Q_x$ of $X^N_.$ is conservative.
 
The proof is standard. First, $\lim_{M \to +\infty} |X^{N}_{T_M}|=+\infty$ almost surely. Second, $$\sum_i \, \left\langle x^i, \sum_j \frac{x^i-x^j}{|x^i-x^j|^{2-\eta}}\right\rangle = \sum_{i,j}  \, |x^i-x^j|^{\eta} \, \geq 0 \, .$$ Hence applying Ito's formula, and provided $\chi \geq 0$, we get 
\begin{equation}\label{eqnonexplosion}
|X^N_{t\wedge T_M}|^2 \leq |X^N_{0}|^2 + 2\sqrt 2 \, \int_0^t \, \mathbf 1_{s\leq T_M} \sum_i \langle X_s^{N,i},dB_s^{N,i}\rangle +2 dN (t\wedge T_M)
\end{equation}
 so that $\mathbb E_x(|X^N_{t\wedge T_M}|^2) \leq |x|^2 +2 dNt$ implying $M^2 \mathbb P(T_M \leq t) \leq |x|^2 +2 dNt$ for all $M$, and finally $T_\infty = +\infty$ almost surely.
 
If $\chi<0$, since $0<\eta<2$, we may use $|x^i-x^j|^{\eta} \leq 1+2|x^i|^2+2|x^j|^2$ so that $$\sum_i \, \left\langle x^i, \sum_j \frac{x^i-x^j}{|x^i-x^j|^{2-\eta}}\right\rangle = \sum_{i,j}  \, |x^i-x^j|^{\eta} \, \leq N (N + 2|x|^2) \, .$$ It follows \begin{eqnarray}\label{eqnonexplosionbis}
|X^N_{t\wedge T_M}|^2 &\leq& |X^N_{0}|^2 + 2\sqrt 2 \, \int_0^t \, \mathbf 1_{s\leq T_M} \sum_i \langle X_s^{N,i},dB_s^{N,i}\rangle +2 dN (t\wedge T_M)\\ && + 2 |\chi| \,  \int_0^t \, \mathbf 1_{s\leq T_M}\,  (N+ 2|X_s^{N}|^2) \, ds \nonumber
\end{eqnarray}
so that $$\mathbb E_x(|X^N_{t\wedge T_M}|^2) \leq  |x|^2 + 2 dN t + \,  2 |\chi| \,  \int_0^t \, \mathbb E_x(\mathbf 1_{s\leq T_M}\,  (N+ 2|X_s^{N}|^2)) \, ds \, .$$ According to Gronwall's lemma we thus deduce that for any fixed $t$, $$\mathbb E_x(|X^N_{t\wedge T_M}|^2) \leq C(t,N,|\chi|)$$ and we may conclude as before.
\medskip

This answers the question of existence and uniqueness for the $\eta$-relaxed model in \cite{GQ}, i.e.
\begin{equation}\label{eqGQ}
dX^{i,N}_t =\sqrt 2 \, dB^{i,N}_t \, - \, \frac{\chi}{N} \, \sum_{j=1}^N \, \frac{X_t^{i,N}-X_t^{j,N}}{|X_t^{i,N}-X_t^{j,N}|^{2-\eta}} \, dt  \quad , \quad i=1,...,N
\end{equation}
and shows at the same time that there are no collisions. Notice that the previous construction also furnishes a solution in  the repulsive case $\chi<0$.

\subsubsection{\textbf{A general result.} \quad}\label{subsubsecexgenedir} We can now introduce the required assumptions

\begin{assumption}\label{assump0}
$b=b_1+b_2$ where $b_1$ is bounded, $b_2$ is local Lipschitz and the solution of $dY_t=\sqrt 2 \, dB_t - b_2(Y_t)dt$ is conservative, i.e. does not explode in finite time.

In addition, for any $\varepsilon>0$, $K$ is local Lipschitz and bounded in the set $|x-y|\geq \varepsilon$. Define
$$C^c_\varepsilon=\{(x^1,...,x^N)\in \mathbb (R^d)^N \; ; \; \inf_{i,j} |x^i-x^j|\geq \varepsilon\} \, .$$ 
Finally $C_0=\{(x^1,...,x^N)\in \mathbb (R^d)^N \; ; \; x^i=x^j \; \textrm{for at least one pair $i\neq j$}\}$ denotes the collision set. By convention $K(0)=0$.
\end{assumption}

\begin{assumption}\label{assump1}
\begin{enumerate}
\item[]
\item[(1)] There exists a potential $U$ such that $K= \chi \, \nabla U$, and a (smooth enough) potential $V$ such that $b_2=\nabla V$.
\item[(2)] The measure $d\nu=\rho dx$ is bounded, where $$\rho(x)=\exp \, \left(-\frac{\chi}{N} \, \sum_{1\leq i<j\leq N} \, U(x^i-x^j) \right) \, \exp \left(- V(x)\right) \, \quad , \quad x=(x^1,...,x^N)\in (\mathbb R^d)^N \, .$$ 
\item[(3)] $\rho^\frac 12 \in H^1(dx)$.
\end{enumerate}
\end{assumption}
and a second one
\begin{assumption}\label{assump2}
\begin{enumerate}
\item[]
\item[(1)] There exists a potential $U$ such that $K=\nabla U$, and a potential $V$ such that $b_2=\nabla V$.
\item[(2)] One can find a sequence $(V_M)_M$ of smooth, nonnegative potentials such that $V_M=0$ on the ball $B(0,M)$, and such that the potentials $K$ and $U+V_M$ satisfy Assumption \ref{assump1}.
\item[(3)] $$ (a) \quad  \sum_i \left \langle x^i, \sum_j K(x^i-x^j)\right \rangle \geq 0 $$ or $$(b) \quad \left|\sum_i \left \langle x^i, \sum_j K(x^i-x^j)\right \rangle\right| \leq c(N)(1+|x|^2)$$ for some $c>0$.
\item[(4)] $$ (a) \quad \langle x,\nabla V(x)\rangle \geq 0 $$ or $$ (b) \quad |\langle x,\nabla V(x)\rangle| \leq c |x|^2$$ for some constant $c>0$.
\end{enumerate}
\end{assumption}
We then have
\begin{theorem}\label{thmexistgene}
Assume that $\nu_0$ is a probability measure on $(\mathbb R^d)^N$ such that $\nu_0(C_0)=0$ and that Assumption \ref{assump0} is fulfilled.  

If Assumption \ref{assump1} is satisfied, there exists a unique weak solution to the system \eqref{eqsys}. The law $Q_{\nu_0}$ of this solution is absolutely continuous w.r.t. the Wiener measure with the same initial condition $W_{\nu_0}$ and there are no collisions. 
Hence if $b_1$ is local Lipschitz too, the solution is strong. In addition for all $T>0$, in restriction to $[0,T]$,  $H(Q_{\nu}|W_{\nu}) < +\infty$, for $\nu$ defined in Assumption \ref{assump1}.

If Assumption \ref{assump2} is satisfied, the same conclusion, except the finiteness of the relative entropy, is true.
\end{theorem}
The proof is exactly the same as in the particular relaxed Keller-Segel case.

\subsubsection{\textbf{Some more examples.}\\ \\ \quad}\label{subsubsecexempdir}

\noindent (1) \textit{\underline{Granular media}}. \quad The granular media equation studied in \cite{CGM} enters the framework of  Theorem \ref{thmexistgene}. Of course the direct proof given in \cite{CGM} is much simpler.
\medskip

(2) \textit{\underline{(sub)-Coulombic potentials in dimension $d\geq 3$}}. \quad Consider $$\rho(x)=\exp \, \left( \frac{\chi}{s N} \, \sum_{1\leq i<j\leq N} \, |x^i- x^j|^{-s}\right) \,  \quad , \quad x=(x^1,...,x^N)\in (\mathbb R^d)^N \, ,$$ i.e. $U(y)= - |y|^{-s}$ for some $s>0$. The measure $\rho(x) dx$ is bounded on compact sets (Radon measure), if and only if $\chi<0$. Hence we can only consider this situation. 

One has $\rho^{\frac 12} \in H^1_{loc}(dx)$ but $K(y)=\frac{\chi \, y}{|y|^{s+2}}$ does not satisfy  Assumption \ref{assump2} (3). If $b_2=\nabla V$ is a confining potential, i.e. if Assumption \ref{assump1} is satisfied, one can apply Theorem \ref{thmexistgene} and get the existence (and uniqueness) of a weak (resp. strong) solution to 
\begin{equation}\label{eqRS}
dX^{i,N}_t =\sqrt 2 \, dB^{i,N}_t \, - \, \frac{\chi}{N} \, \sum_{j=1}^N \, \frac{X_t^{i,N}-X_t^{j,N}}{|X_t^{i,N}-X_t^{j,N}|^{s+2}} \, dt  - \, b(X^{i,N}_t) dt \quad , \quad i=1,...,N
\end{equation}
for any $\chi<0$, $s>0$ and $b=\nabla V + b_1$ where $b_1$ is bounded (resp. Lipschitz and bounded) . In addition no collision occurs.
\medskip

If we do not add a confining potential, one can build solutions up to the exit time $T_M$ of large balls, as we did for the $\eta$ relaxed Keller-Segel model. However proving that $\sup_M T_M=+\infty$ requires some extra computations similar to what we shall do later in order to directly prove existence and uniqueness.
 
In the presence of a confining potential, this result entails the one obtained in section 4 of \cite{RS21} where the range of $s$ is reduced to $0<s<d-2$. We shall see that this restriction is in a sense natural for proving conservativeness without confining potential. Actually this restriction seems to be useless. Indeed the more $s$ is large, the more $C_0$ becomes repulsive (since $\chi<0$), while in $C^c_{1}$ (for $\varepsilon=1$) the drift term becomes smaller. 

By the way, it seems us that the proof of Lemma 4.2 in \cite{RS21} is wrong because inequality (4.13) therein is not correct, $\leq$ has to be changed into $\geq$. In other words Assumption \ref{assump2} (3) (a) is not satisfied. However the statement of Lemma 4.2 in \cite{RS21} (where there is no confining potential) is correct as we shall see later.
\medskip

(3) \textit{\underline{The Dyson processes}}. \quad The Dyson Ornstein-Uhlenbeck is the process associated to the S.D.E.
\begin{equation}\label{eqdysoneq}
dX^{i,N}_t =\sqrt 2 \, dB^{i,N}_t \, - \, \frac{\chi}{N} \, \sum_{j=1}^N \, \frac{1}{X_t^{i,N}-X_t^{j,N}} \, dt  - \, \beta X^{i,N}_t dt \quad , \quad i=1,...,N
\end{equation}
where each particle $X^{i,N} \in \mathbb R$. The associated reversible measure $\nu^N$ is thus given on $\mathbb R^N$ by the density $$\rho(dx)= \prod_{i< j} |x^i-x^j|^{-\chi/N} \, e^{-\beta |x|^2} \, .$$ Assumption \ref{assump1} is satisfied if and only if $\chi/N<-1$ so that we get existence and absence of collisions in this repulsive case. We thus recover the result in \cite{Cepa} (Theorem 3.1 and Proposition 4.1) except that we do not cover the equality case $\chi/N=-1$ for which the same existence result is known (also see \cite{Rogshi}). In addition we know that the law of the process has finite relative entropy w.r.t. the law of the Ornstein-Uhlenbeck process. In the equality case we only know thanks to Theorem \ref{thmcoll} that the  law is absolutely continuous.

We shall come back later to the proofs. Notice that recently \cite{GLBM1D} the authors proposed another proof, for $\chi/N<-1$ (simply use a linear time change to compare the results) based on the existence of a Lyapunov function and Khasminski explosion test.

In what precedes we may replace $e^{-\beta |x|^2}$ by $e^{-\beta \psi(x)}$ for some smooth $\psi$ such that $\psi(x)=0$ if $|x|\leq M$, and $\psi(x)=|x|^2$ for $|x|\geq M+1$. If $T_M$ denotes the first exit time from $\max_i |x^i|\leq M$, this modified Dyson Ornstein-Uhlenbeck process coincides with the Dyson Brownian motion obtained for $\beta=0$, up to time $T_M$. It is thus enough to show that $T_M$ goes to infinity almost surely for the latter. As for \eqref{eqnonexplosionbis} it is enough to remark that  $$\mathbb E(\sum_i |X^i_{T\wedge T_M}|^2) \leq  \mathbb E(\sum_i |X^i_{0}|^2) + 2T (N + |\chi| (N-1))$$ since $\sum_i x^i \sum_{j\neq i} \, \frac{1}{x^i-x^j} = N(N-1)/2$.
\medskip

One thus see that many existing results in the literature enter the entropic framework.
\bigskip

\section{Approximations and singular diffusions.}\label{secapprox}

In this section we shall introduce and use some analytic tools in order to study the existence of singular diffusions without calling upon the Dirichlet forms theory, but still using relative entropy.
\medskip

\subsection{Some analytic bounds. \\ \\}\label{subsecbounds}

\noindent We will collect in this subsection some useful bounds for the sequel.

The finiteness of Fisher information has some regularity consequences. The following is the extension to $\mathbb R^d$ of Lemma 3.2 in \cite{FHM} written for $d=2$:
\begin{lemma}\label{lemFHM}
Let $\rho$ be a density of probability in $\mathbb R^d$, $d\geq 2$. Then denoting by $||.||_p$ the $\mathbb L^p(dx)$ norm, and $I(\rho)$ its Fisher information $I(\rho)=\int \rho \, |\nabla \ln \rho|^2 dx$, it holds
\begin{enumerate}
\item for all $q \in [1,(d/d-1))$, 
\begin{equation}\label{eqFHM1}
||\nabla \rho||_q \leq a^{\frac{d(q-1)}{q}} \, (I(\rho))^{\frac{qd+q-d}{2q}} \, ,
\end{equation}
with $a=\frac{d-1}{d-q} \, \frac{q}{d^{\frac 12}}$,
\item for all $p \in [1,(d/d-2))$, 
\begin{equation}\label{eqFHM2}
||\rho||_p \leq  a^{\frac{d(p-1)}{p}} \, (I(\rho))^{\frac{d(p-1)}{2p}} \, ,
\end{equation}
with $a=\frac{2p(d-1)}{p(d-2)+d} \, \frac{1}{d^\frac 12}$.
\item for $d\geq 3$ \eqref{eqFHM2} (resp. \eqref{eqFHM1}) is still true for $p=d/(d-2)$ (resp. $p=d/d-1$).
\end{enumerate}
\end{lemma}
\begin{proof}
The proof of the two first items is the same as the one in \cite{FHM}. First for any $1\leq q <2$, H\"{o}lder's inequality furnishes $$||\nabla \rho||_q \leq (I(\rho))^{\frac 12} \, ||\rho||_{q/(2-q)}^{\frac 12} \, .$$ Next according to the Sobolev inequality (see e.g. \cite{dav} p-42), $$||\rho||_{q*} \leq \frac{d-1}{d-q} \, \frac{q}{d^{\frac 12}} \, ||\nabla \rho||_q \quad \textrm{ for } \quad q^*=\frac{qd}{d-q} \, .$$ We may thus use interpolation between $\mathbb L^{q/(2-q)}$, $\mathbb L^1$ and $\mathbb L^{q*}$, provided $q^* > q/(2-q)$ i.e. $q<d/d-1$. We obtain since $\rho$ is a density of probability, 
\begin{equation}\label{eqinterpol}
 ||\rho||_{q/(2-q)} \leq \, ||\rho||_{q^*}^{\frac{2d(q-1)}{qd+q-d}}  \leq a^{\frac{2d(q-1)}{qd+q-d}} \, ||\nabla \rho||_q^{\frac{2d(q-1)}{qd+q-d}}
 \end{equation}
and finally \eqref{eqFHM1}. Plugging this inequality in \eqref{eqinterpol} we get \eqref{eqFHM2} if we denote $p=q/(2-q)$.
  
Notice that for $d=2$ and $q=2$, \eqref{eqFHM1} is still true but $a=+\infty$.

For $d\geq 3$ first remark that \eqref{eqFHM1} and \eqref{eqFHM2} extend to the limiting $d/d-1$ and $d/d-2$, first for compactly supported $\rho$. Indeed take some compactly supported $g\geq 0$ and an increasing sequence $p$ such that $p \to p'<+\infty$. On $g \leq 1$, $g^p$ decreases to $g^{p'}$ and on $g >1$ it increases, so that, since $g$ is compactly supported $\int g^p dx \to \int g^{p'} dx$ and the latter is bounded above by any bound of the family $\int g^p dx$. 

Next one can build a non-negative smooth function $\psi$ on $\mathbb R^+$ such that $\psi(u)=1$ for $u \leq 1$, $\psi(u)=0$ if $u\geq 2$ and $|\nabla \psi|^2/\psi$ is bounded by some constant $A$ where by convention we define $|\nabla \psi|^2/\psi=0$ on the set $\psi=0$. Define $\psi_M(x)=\psi(|x|/M)$. 

Defining $\rho_M(x) = z_M \, \rho(x) \, \psi_M(x)$ where $z_M$ is a normalizing constant, we have as $M \to +\infty$, $z_M \to 1$ since $\rho \in \mathbb L^1$, $\rho_M$ increases to $\rho$ so that $\lim_M \int \rho_M^{p'} dx= \int \rho^{p'} dx$ and finally $$I(\rho_M)= \int \frac{|\nabla \rho|^2}{\rho} \, \psi_M \, dx + \, \frac{1}{M^2} \, \int \frac{|\nabla \psi|^2}{\psi}(|x|/M) \, \rho(x) \, dx \, ,$$ so that $I(\rho_M)$  goes to $I(\rho)$, and the proof is complete. 
\end{proof}
\begin{remark}\label{remMTd2}
When $d=2$ one can easily find unbounded densities with finite Fisher information, i.e. the limiting case $d/d-2$ is not allowed in (2). However one can improve on (2) in the previous Lemma by showing the existence of exponential moments. This will be done in section \ref{secchaos3}.
\hfill $\diamondsuit$
\end{remark}
\medskip

An immediate consequence is the following extending \cite{FHM} Lemma 3.3:
\begin{lemma}\label{lemFHM2}
Let $(Y_1,Y_2)$ be a random variable taking values in $\mathbb R^d\times \mathbb R^d$ with distribution $\rho(y_1,y_2) \, dy_1 \, dy_2$. Then for any $\gamma \in (0,2)$ and any $\beta$ such that $\gamma/d <\beta < 2/d$ ($\beta=2/d$ is allowed if $d\geq 3$), there exists $C(\gamma,\beta,d)$ such that $$\mathbb E(|Y_1-Y_2|^{-\gamma}) = \int_{\mathbb R^d\times \mathbb R^d} \, \frac{\rho(y_1,y_2)}{|y_1-y_2|^\gamma} \, dy_1 \, dy_2 \, \leq \, C(\gamma,\beta,d) \, (1+I^{\frac{d\beta}{2}}(\rho)) \, .$$
\end{lemma}
\begin{proof}
The proof is exactly the same as in \cite{FHM} with a modification due to dimension $d$ instead of $2$. One performs the change of variable $(y_1,y_2) \mapsto (y_1-y_2,y_1+y_2)$ so that the integral is with respect to the marginal distribution of $Y_1-Y_2$ denoted by $\tilde \rho$. Hence $$\int_{\mathbb R^d\times \mathbb R^d} \, \frac{\rho(y_1,y_2)}{|y_1-y_2|^\gamma} \, dy_1 \, dy_2 = \int_{\mathbb R^d} \, \frac{\tilde \rho(z)}{|z|^\gamma} \, dz \leq 1 + \int_{|z|\leq 1} \frac{\tilde \rho(z)}{|z|^\gamma} \, dz \, .$$
One then applies H\"{o}lder inequality yielding, provided $\beta>\gamma/d$, $$\int_{|z|\leq 1} \frac{\tilde \rho(z)}{|z|^\gamma} \, dz \leq \left(\int_{|z|\leq 1} |z|^{-\gamma/\beta} dz\right)^\beta \, ||\tilde \rho||_{1/(1-\beta)} \leq C(\gamma,\beta,d) \, ||\tilde \rho||_{1/(1-\beta)} \, .$$ We use next \eqref{lemFHM2}, provided $\beta<2/d$, yielding $||\tilde \rho||_{1/(1-\beta)} \leq C(\beta,d) \, I^{\frac{d \beta}{2}}(\tilde \rho)$. It remains to recall that the Fisher information of a marginal is less than the Fisher information of the initial probability measure (super-additivity, see remark \ref{remcarlen} below), so that $I(\tilde \rho)\leq I(\rho')$ where $\rho'$ is the joint density of $(y_1-y_2,y_1+y_2)$. It is immediately seen that $I(\rho') \leq c(d) \, I(\rho^{1,2})$ where $\rho^{1,2}$ is the joint density of $(y_1,y_2)$ so that we get the result using again the super-additivity of the Fisher information.
\end{proof}
\begin{remark}\label{remint1}
Consider $$\int_{\mathbb R^d\times \mathbb R^d} \, h(y_1-y_2) \, \rho(y_1,y_2) \, dy_1 \, dy_2 \, ,$$ and assume that $h \in \mathbb L^\alpha(\mathbb R^d)$ for $\alpha \geq d/2$ if $d\geq 3$, $\alpha >1$ if $d=2$. Then the same proof shows that $$\int_{\mathbb R^d\times \mathbb R^d} \, |h|(y_1-y_2) \, \rho(y_1,y_2) \, dy_1 \, dy_2  \leq \, C(\alpha,d) \, || h ||_\alpha \, I^{\frac{d}{2 \alpha}}(\rho) \, .$$
\hfill $\diamondsuit$
\end{remark}

\begin{remark}\label{remcarlen}
The super-additivity property of the Fisher information was first proved in \cite{carlen} Theorem 3. It says that for any density of probability $\rho$ on $\mathbb R^{m+n}$ whose marginals on $\mathbb R^m$ and $\mathbb R^n$ respectively are denoted by $\rho_1$ and $\rho_2$, it holds $I(\rho) \geq I(\rho_1)+I(\rho_2)$ with equality if and only if $\rho = \rho_1 \otimes \rho_2$ (i.e. marginals are independent). It is amusing to remark that another proof, in a more general framework, similarly using entropy on the path space, is contained in \cite{Catpota} Proposition 3.5. \hfill $\diamondsuit$
\end{remark}

Another consequence strongly used in the P.D.E literature  is the following result, similar to \cite{GQ} proof of Proposition 3.1. Notice that, up to the constants, this result is exactly the same as what we directly obtained in lemma \ref{lemmoment1} \eqref{eqfish}. We give a proof for the sake of completeness and also to see what kind of analytic arguments are necessary for the derivation of this inequality. 
\begin{lemma}\label{lementropdecay}
Let $b(.,.)$ be some bounded smooth function. Let $t \mapsto \nu_t(dx)=\rho_t(x) dx$ be the  flow of probability laws of the solution of $$dX_t=\sqrt 2 \, dB_t + b(X_t)dt \quad , \quad Law(X_0)=\rho_0(x) dx \, .$$ Assume in addition that for some $\alpha>0$, $$\int \, \ln \rho_0(x) \, \rho_0(x) \, dx \, < \, +\infty \quad \textrm{ and } \quad \int \bar V_\alpha(x) \rho_0(dx) < +\infty$$ where $\bar V_\alpha(x) = \sum_{i=1}^d \, |x_i|^{\alpha}$. Then 
\begin{equation}\label{eqentropdecay}
\int \, \rho_T \, \ln \rho_T  \, dx + \frac 12 \, \left(\int_0^T \, I(\rho_t) dt\right) \leq  \int \, \rho_0 \, \ln \rho_0  \, dx + 2 \, \left(\int_0^T \, \int \, b^2 \, \rho_t \, dx \, dt\right) \, .
\end{equation}
\end{lemma}
\begin{proof}
Since $b$ is smooth and bounded, using elliptic regularity (or Malliavin calculus) so is $\rho_t$ for $t>0$. $\rho_t$ then satisfies the forward equation $$\partial_t \rho_t = - \, div_x(2b \, \rho_t) + \Delta_x \rho_t \, .$$ Let $h(\rho)= \int \rho \ln(\rho) \, dx$. The first point is to show that $-\infty < h(\rho_t) < +\infty$ for all $t>0$.

Since $b$ is bounded, the law $Q$ of the solution of $$dX_t=\sqrt 2 \, dB_t + b(X_t)dt \quad , \quad Law(X_0)=\rho_0(x) dx$$ is absolutely continuous (up to any finite time) w.r.t. the Wiener measure $\tilde P$ with the same initial distribution (recall that $\tilde P$ is the distribution of $\sqrt 2 B_.$). In addition $dQ/dP$ belongs to all the $\mathbb L^p$ for $1\leq p <+\infty$. With our assumption, for all $t$, $\int \bar V_\alpha(\omega_t) \, d\tilde P < +\infty$, so that the same is true for $Q$ for any $\alpha'<\alpha$.
According to \eqref{eqent2}, $h(\rho_t) > -\infty$, while $h(\rho_t)<+\infty$ since $H(Q|\tilde P)<+\infty$.

The goal is now to study the time evolution of $h(\rho_t)$ i.e. to compute 
\begin{eqnarray*}
\partial_t h(\rho_t) &=& \int \, \partial_t \rho_t \, (1 + \ln \rho_t) \, dx \\
&=& \int \, (- div_x(2b\rho_t) + \Delta_x \rho_t) \, \ln \rho_t \, dx \, .
\end{eqnarray*}
In order to justify differentiation under the integral one can for instance recall that the heat kernel associated to a smooth parabolic equation as here, is smooth and that all its derivatives have gaussian tails, i.e. are bounded by $C(m) t^{-d/2} \, e^{- \, \frac{|x-y|^2}{c(m)t}}$ where $m$ denotes the order of the derivatives. It follows that for $t>0$, $$\sup_{s \in [t/2,3t/2]} \, (- div_x(2b\rho_s) + \Delta_x \rho_s) \in \mathbb L^1(dx)$$ justifying the differentiation.

Next, a simple integration by parts furnishes $$\partial_t h(\rho_t) = \int \, 2b \, \nabla_x \rho_t \, dx \, - \, I(\rho_t) \, .$$
Using the finite energy condition we deduce $$\partial_t h(\rho_t) \leq 2 \left(\int \, b^2 \, \rho_t(x) \, dx\right)^{\frac 12} \, I^{\frac 12}(\rho_t) \, - \, I(\rho_t)$$ so that integrating w.r.t. $dt$ and using Cauchy-Schwarz inequality we obtain $$h(\rho_T) - h(\rho_0) \leq 2  \left(\int_0^T \, \int \, b^2 \, \rho_t \, dx \, dt\right)^{\frac 12} \, \left(\int_0^T \, I(\rho_t) dt\right)^{\frac 12} \, - \, \left(\int_0^T \, I(\rho_t) dt\right) \, , $$ from which we deduce the result.
\end{proof}
The natural question is to extend the previous lemma to any admissible flow. In the following remark we explain how to do and compare with the results obtained in lemma \ref{lemmoment1}.

\begin{remark}\label{rementropdecay}

Consider an admissible  flow $t \mapsto \nu_t$ as in theorem \ref{thmCL2}. If $b$ is not smooth or bounded we may find a sequence of smooth and bounded $b_n$ such that $$\int_0^T \, \int \, |b-b_n|^2 \, \rho_t(x) \, dx \, dt \, \to \, 0 \quad \textrm{ as } n \to +\infty \, .$$ (first approximate $b$ in $\mathbb L^2([0,T]\times \mathbb R^d, \rho_t \, dx dt)$ by a continuous and compactly supported $\tilde b_n$, and then using a nice mollifier $\eta_n$ approximate uniformly $\tilde b_n$ by $\tilde b_n * \eta_n$). 

Denoting by $Q_n$ (resp. $Q$ that exists thanks to theorem \ref{thmCL2}) the associated probability measures on the path space with the same initial $\nu_0$, it holds $$H(Q|Q_n) = \int_0^T \, \int \, |b-b_n|^2 \, \rho_t(x) \, dx \, dt $$ that goes to $0$. Since $H(Q|Q_n)$ goes to $0$ so do all time marginals $\rho^n_t$ for $t>0$. 

We first claim  that $\liminf_n I(\rho^n_t)\geq I(\rho_t)$. Many proofs appear in the literature as for instance the general results in \cite{Rocka}.  We give here a direct (idea of) proof. Assume that $\liminf_n I(\rho_t^n) < +\infty$ and choose a subsequence $\bar \rho^n_t$ such that $\lim_n I(\bar \rho_t^n) = \liminf_n I(\rho_t^n)$. According to lemma \ref{lemFHM} we know that $\bar \rho^n_t$ is bounded in some $\mathbb L^p$. Since $\sqrt{\bar \rho^n_t}$ is also bounded in $\mathbb L^2$, it is not too difficult to get a subsequence (still denoted by $\bar \rho^n_t$) converging to some $\rho_t$ weakly in $\mathbb L^p$ and such that $\sqrt{\bar \rho^n_t}$ converges to $\sqrt{\rho_t}$ weakly in $\mathbb L^2$. It follows that $\nabla \sqrt{\bar \rho^n_t}$ converges to $\nabla \sqrt{\rho_t}$ in the set of (Schwartz) distributions, and since this sequence is bounded in $\mathbb L^2$, the limit also belongs to $\mathbb L^2$. It remains to use Fatou's lemma to control the time integral. For a complete rigorous proof in this spirit see Proposition 13.2 in \cite{bobkov-fisher}.

Choose some $V$ such that $|V'|$ and $|V''|$ are bounded by $K$. Taking a subsequence if necessary, we also have that $\int \bar V \rho_T dx \leq \liminf_n \int \bar V \rho^n_T dx$. According to lemma \ref{lemmoment1} we know that $$\left |\int \rho_T \ln \rho_T dx\right| \leq d C(K,T) \left(1 + \ln Z_V + \int_0^T \int |b -\nabla \bar V|^2 dx dt\right) + \int (\ln \rho_0 +\bar V) \rho_0 dx \, .$$  

It follows that, we may control both $\int_0^T \, I(\rho_t) dt$ and $h(\rho_T)$ by the $\liminf_n \int_0^T \int  b_n^2 \, \rho^n_t dx dt$. These results are of the same nature as the ones we directly obtained in lemma \ref{lemmoment1}. 
\medskip

One can also use the euclidean log-Sobolev inequality 
\begin{equation}\label{eqlseuclid}
\int \rho_T \ln \rho_T dx \leq \frac d2 \, \ln \left(\frac{4 I(\rho_T)}{d \pi e} \right)
\end{equation}
which furnishes, for large Fisher information, a better bound.
\hfill $\diamondsuit$
\end{remark}
\medskip

\subsection{Existence and uniqueness of entropic singular diffusions. \\ \\}\label{subsecdiffsing}

In this subsection we shall use all what precedes to study existence and uniqueness for Ito processes with singular drift. Comparison with the existing literature will be made at the end of the subsection.

We thus want to study the S.D.E. in $\mathbb R^m$
\begin{equation}\label{eqdiffnormale}
X_t = X_0 + \sqrt 2 \, B_t + \int_0^t \, 2 \, g(X_s) ds \, .
\end{equation}
The law of $X_0$ is denoted by $\nu_0(dx)=\rho_0(x) dx$. We may first use a cut-off $g^M=g \, \mathbf 1_{|g|\leq M}$ which is bounded. One can thus, using Girsanov theory, build a (weak) solution $Q^M$ which is absolutely continuous w.r.t. $W$ the Wiener measure, and more precisely the law of $X_0 + \sqrt 2 B_t$, on the time interval $[0,T]$ for any $T>0$. We fix once for all such a $T$ and keep the same notation $Q^M$ and $W$. As in subsection \ref{subsecpath}, we choose a smooth $V$ (defined on $\mathbb R$) such that $|V'|$ and $|V''|$ are bounded by $A$, define $\bar V(x)=\sum_{j=1}^m V(x^j)$ and define $P$ as the distribution of $$Y_t=Y_0 + \sqrt 2 \, B_t - \int_0^t \, \nabla \bar V(Y_s) \, ds$$ where $Y_0$ is distributed according to the probability measure $\gamma_0(dy)=Z^{-1} \, e^{- \bar V(y)} \, dy$.

Assume that $$H(\nu_0|\gamma_0)= \int (\ln \rho_0 + \bar V + \ln Z) \rho_0 dx < + \infty \, .$$ We thus have $$H(Q^M|P) = H(\nu_0|\gamma_0) +  \, \int \int_0^T \, |g^M +\nabla \bar V|^2(\omega_t)  \, dt \, dQ^M$$ (recall that $\omega$ is the generic element of the path space). We thus know that the law of $X^M_t$ admits a density $\rho^M_t$ (with respect to Lebesgue measure) so that, $$\int \int_0^T \, |g^M +\nabla \bar V|^2(\omega_t)  \, dt \, dQ^M = \int \int_0^T \, |g^M +\nabla \bar V|^2(x)  \, \rho^M_t(x) \, dt \, dx$$ satisfies 
for all $\lambda >0$ 
\begin{equation}\label{eqentfish}
\int \int_0^T \, |g^M +\nabla \bar V|^2(\omega_t)  \, dt \, dQ^M  \leq (1+\lambda) \, \int_0^T \int \, |g^M|^2 \, \rho^M_t dx dt +  \, \left(1+\frac 1\lambda\right) \, m \, A^2 \, T \, .
\end{equation}
According to Corollary \ref{cordual} we also know that $\int_0^T \, I(\rho_t^M) \, dt < +\infty$ so that we may use \eqref{eqFHM2} for almost all $t \in ]0,T]$. 

In particular for $m\geq 3$ one may use lemma \ref{lemFHM} and H\"{o}lder's inequality with $p=m/(m-2)$ whose dual exponent is $q=m/2$ yielding (see the value of $a$ in Lemma \ref{lemFHM})
\begin{eqnarray}\label{eqkrylov1}
\int \, |g^M|^2 \, \rho^M_t dx  &=& \int \, |g^M \, \mathbf 1_{|g^M|>A}|^2 \, \rho^M_t dx + \int \, |g^M \, \mathbf 1_{|g^M|\leq A}|^2 \, \rho^M_t dx \nonumber \\ &\leq& ||\rho^M_t||_{p} \, |||g^M \, \mathbf 1_{|g^M|> A}|^2||_{q}  + A^2 \nonumber \\ &\leq &   \, a^2 \, I(\rho^M_t) \, ||g^M \, \mathbf 1_{|g^M|> A}||_{m}^2 \, + \, A^2 \nonumber \\ &\leq& \, \frac{(m-1)^2}{m \, (m-2)^2} \, I(\rho^M_t) \, ||g \, \mathbf 1_{|g|> A} ||_{m}^2 \, + A^2 \, .
\end{eqnarray} 
If we plug this inequality into \eqref{eqfish} we thus obtain 
\begin{equation}\label{eqkrylov2}
\int_0^T \, I(\rho^M_t) \, dt \leq 4(1+\lambda)^2 \, \frac{(m-1)^2}{m \, (m-2)^2} \; \; \left(\int_0^T \, I(\rho^M_t) \, dt\right) \, ||g \, \mathbf 1_{|g|> A}||_{m}^2 + C(\bar V,\rho_0,T,\lambda,m,||g \mathbf 1_{|g|> A}||_{m},A) \, .
\end{equation}
Of course \eqref{eqkrylov2} is useful only if $$4 \, ||g \, \mathbf 1_{|g|> A}||_{m}^2 \, \frac{(m-1)^2}{m \, (m-2)^2} < 1$$ in which case we get that $\int I(\rho_t^M) dt \leq C$ where $C$ does not depend on $M$, hence $\sup_M \, \int_0^T \, I(\rho^M_t) \, dt< +\infty$. If $g \, \mathbf 1_{|g|> C} \in \mathbb L^m$ for some $C>0$ we may always find $A>C$ such that the previous is satisfied.

One can also be less demanding and apply H\"{o}lder's inequality for some $q>m/2$, i.e. $p<m/(m-2)$ so that $m(p-1)/2p <1$ i.e. $m(p-1)/2p=1-\varepsilon_p$ for some $\varepsilon_p>0$. We thus obtain, this time for all $m\geq 2$, 
\begin{equation}\label{eqkrylov3}
 \int I(\rho^M_t) dt \leq  \, C(m,p) \; ||g \, \mathbf 1_{|g|> A}||_{2q}^2 \; \int (I(\rho^M_t))^{1-\varepsilon_p} dt \, + C(\bar V,\rho_0,T,q, m,||g \, \mathbf 1_{|g|> A}||_{2q},A)  \, ,
\end{equation}
and again if $||g \, \mathbf 1_{|g|> A}||_{2q} < +\infty$ that $\sup_M \, \int_0^T \, I(\rho^M_t) \, dt < +\infty$.

In both cases, plugging this bound first in \eqref{eqkrylov1} and then the new obtained inequality in \eqref{eqentfish}, we deduce that $\sup_M H(Q^M|P) < +\infty$. This will yield the following result
\begin{theorem}\label{thmkryl}
Let $m \geq 2$. Assume that $H(\nu_0|\gamma_0)<+\infty$ and in addition that $g \, \mathbf 1_{|g|>A} \in \mathbb L^p(\mathbb R^m)$ for some $p \geq m$ if $m\geq 3$ or $p>2$ if $m=2$.  

Then equation \eqref{eqdiffnormale} has an unique weak solution (solution in law) $Q$, and this $Q$ satisfies $H(Q|P)<+\infty$ on $[0,T]$ for all $T>0$.
\end{theorem}
\begin{proof}
According to the discussion before the Theorem's statement, the sequence $Q^M$ is such that $\sup_M \, H(Q^M|P) < +\infty$. It follows that one can find a subsequence which is weakly converging to some $Q$ with $H(Q|P) \leq \liminf_M \, H(Q^M|P) < +\infty$, since relative entropy is lower semi-continuous. Thanks to Dunford-Pettis theorem the convergence of $dQ^M/dP$ to $dQ/dP$ holds for the stronger $\sigma(\mathbb L^1,\mathbb L^\infty)$ topology.

Also remark that we can replace $g^M$ by $g$ in the left hand side of \eqref{eqkrylov1}, so that it holds $\sup_M \, \int_0^T \int |g|^2 \rho_t^M dx dt < +\infty$, and $\lim_{M  \to +\infty} \, \int_0^T \int |g-g^M|^2 \rho_t^M dx dt = 0$ since $|| |g-g^M| ||_{2q} \to 0$.

To prove that $Q$ solves \eqref{eqdiffnormale} it remains to show that it solves the corresponding martingale problem i.e. to show that for all $t>s$, all bounded $H$ defined on $C^0([0,s])$, all smooth $\varphi$ with compact support, it holds
\begin{equation}\label{eqmart0}
\mathbb E^{Q} \left[ \, H(\omega_{v\leq s}) \; \left(\varphi(\omega_t)- \varphi(\omega_s) + \int_s^t \, (\Delta \varphi(\omega_u) + \langle \nabla \varphi(\omega_u), 2g(\omega_u)\rangle) du\right) \right] = 0 \, .
\end{equation}
Since the previous is true replacing $Q$ and $g$ by $Q^M$ and $g^M$, it is enough to prove
\begin{equation}\label{eqmart1}
\lim_M \mathbb E^{Q^M} \left[\left(\int_s^t \, \langle\nabla \varphi(\omega_u) ,g^M(\omega_u)\rangle du\right) \, H(\omega_{v\leq s})\right]=\mathbb E^{Q} \left[\left(\int_s^t \, \langle \nabla \varphi(\omega_u),g(\omega_u)\rangle du\right) \, H(\omega_{v\leq s})\right] \, .
\end{equation}
Since $\lim_{M  \to +\infty} \, \int_s^t \int |g-g^M|^2 \rho_u^M dx du = 0$, thanks to \eqref{eqkrylov1}, $$\lim_M \mathbb E^{Q^M} \left[\left(\int_s^t \, \langle\nabla \varphi(\omega_u),(g^M-g)(\omega_u)\rangle du\right) \, H(\omega_{v\leq s})\right] =0 \, ,$$ while $$\lim_M \mathbb E^{Q^M} \left[\left(\int_s^t \, \langle\nabla \varphi(\omega_u) ,g(\omega_u)\rangle du\right) \, H(\omega_{v\leq s})\right]=\mathbb E^{Q} \left[\left(\int_s^t \, \langle \nabla \varphi(\omega_u),g(\omega_u)\rangle du\right) \, H(\omega_{v\leq s})\right]$$ follows from the $\sigma(\mathbb L^1,\mathbb L^\infty)$ convergence of $Q^M$ to $Q$. Indeed we can first choose $A$ large enough for $\sup_{M} \, \int_s^t \int |g-g^A|^2 \rho_u^M dx du \leq  \varepsilon$ and $\int_s^t \int |g-g^A|^2 \rho_u dx du \leq  \varepsilon$ according to \eqref{eqkrylov1}, and then replacing $g$ by $g^A$ get the desired limit in $M$ so that it only remains to let $\varepsilon$ go to $0$.

Finally we have to prove uniqueness. Define $T_k=\inf \{t\geq 0, \int_0^t |g+\nabla \bar V|^2(\omega_s) \, ds \geq k\}$. Since $H(Q|P) < +\infty$ (on $[0,T]$), Girsanov theory tells us that $T_k\wedge T \to T$, $Q$ almost surely. If $Q'$ is another solution of \eqref{eqdiffnormale}, Girsanov theory again tells us that it coincides with $Q$ on $[0,T\wedge T_k]$, hence $T_k\wedge T \to T$, $Q'$ almost surely too and $Q=Q'$ on $[0,T]$. 
\end{proof}
\medskip

\begin{remark}{\textit{About the literature.}}\label{remliterature}

Theorem \ref{thmkryl} is mainly known, except in the critical case $p=m$. For weak solutions it has been first obtained in \cite{bass} for homogeneous drift, the method being then extended to time dependent drifts in \cite{peng}. A different approach was proposed by N. Krylov and co-authors. For a complete bibliography on the topic see \cite{kinzsurvey}. The proof we have given here can be extended to several other situations: particle system, non linear SDE, more general processes than Ito processes, thanks to the versatility of the entropic approach. This will be done in the next subsection for particle systems, and in another work for the other cases.

We shall recall below, for later use, the arguments allowing to get a weak solution, following Krylov's ideas. The case $g \in \mathbb L^p(\mathbb R^m)$ for some $p>m$ is contained for instance in \cite{KR} including time inhomogeneous drift, while the case $p=m$ is studied and partly solved in e.g. \cite{KrylPTRF,KrylSPA,KrylAOP,beck}. In all these papers existence of a strong solution is the main goal. 

Actually a standard result by Khasminskii recalled in \cite{Han} Proposition 3.1 (also see \cite{KR}) says that for a $\mathbb R^m$ valued standard Brownian motion $w_.$, any $g \in \mathbb L^q([0,T],\mathbb L^p(dx))$ with $\frac mp + \frac 2q <1$, any $x \in \mathbb R^m$ and any $\kappa>0$ it holds 
\begin{equation}\label{eqhas}
\mathbb E \left(\exp \left(\kappa \, \int_0^T \, g^2(s,x+w_s) ds\right)\right) \leq C(\kappa,T,p,q,||g||_{\mathbb L^q([0,T],\mathbb L^p(dx))})< +\infty  \, .
\end{equation}
Notice that this upper bound does not depend on $x$ since the $\mathbb L^p$ norm of $g(s,x+.)$ is equal to the one of $g(s,.)$ thanks to the invariance by translation of Lebesque's measure. It immediately follows from Novikov's criterion, that the Girsanov density associated to $g$ is a true martingale, the existence of a weak solution $Q_x$ of \eqref{eqdiffnormale} starting from $x$ follows. In addition the Girsanov density belongs to all the $\mathbb L^k(W_x)$ where $W$ is Wiener measure starting from $x$. Indeed 
\begin{eqnarray*}
\mathbb E_x\left(e^{k(\int_0^T g\sqrt 2 \, d\omega_s - \int_0^T \, |g|^2 \, ds)}\right) &=&  \mathbb E_x\left(e^{(k \, \int_0^T g\sqrt 2 \, d\omega_s - 2k^2 \int_0^T \, |g|^2 \, ds)} \, e^{(2k^2-k)\, \int_0^T \, |g|^2 \, ds)}\right)\\ &\leq& \mathbb E_x^{\frac 12}\left(e^{\int_0^T 2k \, g\sqrt 2 \, d\omega_s - \int_0^T \, 4 k^2 \, |g|^2 \, ds}\right) \, \mathbb E_x^{\frac 12}\left(e^{2 (2k^2-k) \, \int_0^T  \, |g|^2 \, ds}\right)
\end{eqnarray*} 
thanks to Cauchy-Schwarz inequality. The first factor is equal to $1$ according to Novikov's criterion, hence for $k \geq 1$, 
\begin{equation}\label{eqlpgirs}
\mathbb E_x\left(e^{k(\int_0^T g\sqrt 2 \, d\omega_s - \int_0^T \, |g|^2 \, ds)}\right) \leq \mathbb E_x^{\frac 12}\left(e^{2 (2k^2-k) \, \int_0^T  \, |g|^2 \, ds}\right) \, .
\end{equation}
In what precedes $\mathbb E_x$ is the expectation w.r.t. the Wiener measure starting from $x$. 

It is worth to notice that the previous quantity can be bounded independently of $x$ by $C(T,k,p,q,||g||_{\mathbb L^q([0,T],\mathbb L^p(dx))})$.

It follows the same exponential integrability with respect to $Q_x$ since $$\mathbb E_{Q_x}((dQ_x/dW_x)^k)= \mathbb E_{W_x}((dQ_x/dW_x)^{k+1}) \, .$$ 
It immediately follows that \eqref{eqhas} is still true if one replaces $w_.$ by $X_.$, so that the same argument shows that the inverse of the Girsanov density also belongs to all the $\mathbb L^k$ w.r.t $W$ or $Q$ (this is (ii) in Proposition 3.3 of \cite{Han}). 

Assume for simplicity that $g(s,x)=g(x)$ and that $p>m$. Girsanov theory tells us that $Q_x$ solves \eqref{eqdiffnormale} and that the family $(Q_x)_{x\in \mathbb R^m}$ is strong Markov. As a consequence, for any initial distribution $\nu_0$, $Q_{\nu_0}=\int Q_x \, \nu_0(dx)$ solves \eqref{eqdiffnormale} with initial distribution $\nu_0$. If $\nu_0$ satisfies the hypotheses in Theorem \ref{thmkryl}, $Q_{\nu_0}$ coincides with the solution built in this Theorem.

Notice that it is absolutely continuous w.r.t. $W_{\nu_0}$, and that its density satisfies the same $\mathbb L^k$ controls than $Q_x$.  Also notice that since we can always add or subtract a bounded drift via another Girsanov transform (it does not change the $k$ integrability), the required $\mathbb L^p$ integrability can be restricted to $g \, \mathbf 1_{|g|>A}$ as for our result.

Denote by $\rho_t(x,.)$ the density w.r.t. Lebesgue's measure at time $t$ of $Q_{x}$. Let $h$ be a compactly supported and bounded function. Then for $1/u + 1/v=1$,
\begin{eqnarray*}
\int h(y) \, e^{\frac{|x-y|^2}{4ut}} \, \rho_t(x,y) \, dy &=& \int \, h(\omega_t) \, e^{\frac{|x-\omega_t|^2}{4ut}} \, \frac{dQ_x}{dW_x}(\omega) \, W_x(d\omega) \\ &\leq& \left(\int \, \left(\frac{dQ_x}{dW_x}\right)^v \, W_x(d\omega) \right)^{1/v} \, \left(\int \, h^u(\omega_t) \, e^{\frac{|x-\omega_t|^2}{4t}} \, W_x(d\omega)\right)^{1/u} \\ &\leq& C(T,u,p,||g||_p) \, \left(\int \, h^u(y) \, (4\pi t)^{-m/2} \, dy\right)^{1/u}
\end{eqnarray*}
so that if $h$ is non-negative $$\int h(y) \rho_t(x,y) \, dy \leq C(T,u,p,||g||_p) \, t^{-m/2u} \, ||h||_u \, .$$ Taking the positive and the negative part of $h$ the previous extends to any compactly supported $h$, showing that for any $1 \leq v < +\infty$ and any $x$
\begin{equation}\label{eqkrylint}
||\rho_t(x,.)||_v \leq \, C(T,v,p,||g||_p) \, t^{- \frac{(v-1) m}{2v}}.
\end{equation} 
It immediately follows that the density $\rho_t = \int \, \rho_t(x,.) \, \nu_0(dx)$ satisfies the same property, thanks to H\"{o}lder inequality. Notice that, according to \eqref{eqkrylint}, 
\begin{equation}\label{eqkrylint2}
\rho_. \in \mathbb L^1((0,T],\mathbb L^v(\mathbb R^m)) \quad \textrm{ provided } v<m/(m-2) \, .
\end{equation}
\medskip

Our result (despite the fact we are not looking for a strong solution) seems to be new for $p=m$, for which we (presumably) do no more have a Girsanov density in some $\mathbb L^k$ for some $k>1$ but only in the Orlicz space $\mathbb L^{x \ln x}$.

It is worth noticing that our proof immediately extends to the case of time inhomogeneous drift $g(t,.)$ provided $\sup_{t \in [0,T]} || g(t,.) \, \mathbf 1_{|g|>A} ||_{p}  < +\infty$ for some (time independent) $A$.  Actually using some H\"{o}lder inequality w.r.t. time, it extends to the standard $\mathbb L^p$-$\mathbb L^q$ case. This will be explained for the non linear SDE in the next section. \hfill $\diamondsuit$
\end{remark}

Finally in the case $d=2$ the result will be improved in the final section.
\medskip

\subsection{Application to particle systems. \\ \\}\label{subsecparticexist}

We come back to \eqref{eqsys} where for simplicity we will first assume that $b=0$. We will not use Theorem \ref{thmkryl} since $m=dN$ is too big. Instead we may look at the proof. It is easily seen that all we have to do is to bound $$\int_0^T \, \int \, K^2(\omega_t^i -\omega_t^j) \, dQ^M \, dt$$ so that we may apply \eqref{eqkrylov1} with $m=2d$, replacing $\rho_t^M$ by $\rho_{t,i,j}^M$ the $Q^M$ law of $(\omega_t^i ,\omega_t^j)$ and $||g \, \mathbf 1_{|g|>A}||_m$ by $||K \, \mathbf 1_{|K|>A}||_{2\alpha}$ provided $K \, \mathbf 1_{|K|>A} \in \mathbb L^{2\alpha}(\mathbb R^d)$ with $\alpha \geq d/2$ if $d\geq 3$, $\alpha >1$ if $d=2$ according to Remark \ref{remint1}. 

Using again the super-additivity of Fisher information and the convexity inequality $$\sum_i \left| \frac 1N \sum_{j \neq i} K(x^i-x^j)\right|^2 \leq \frac 1N \sum_i \, \sum_{j \neq i} \, K^2(x^i-x^j) $$we  deduce the analogue of \eqref{eqkrylov3} 
\begin{equation}\label{eqkrylovparticles}
I(\rho^M_t) \leq  \, C(d,\alpha) \; (N-1) \, ||K \, \mathbf 1_{|K|>A}||_{2\alpha}^2 \; (I(\rho^M_t))^{1-\varepsilon_\alpha} \, + N \, C(\bar V,\rho_0,T,\alpha, d,||K \mathbf 1_{|K|>A}||_{2\alpha},A)  \, ,
\end{equation}
for $\alpha$ as before and some $\varepsilon_\alpha >0$ except if $\alpha=d/2$. In the latter case we may again choose a larger $A$ as for the proof of Theorem \ref{thmkryl}.
The remaining part of the proof of Theorem \ref{thmkryl} is unchanged so that we may state 
\begin{theorem}\label{thmkrylpartic}
Assume that $K \mathbf 1_{|K|>A} \in \mathbb L^p(\mathbb R^d)$ for some $p \geq d$ for $d\geq 3$ and $p>2$ for $d=2$, and some $A>0$. Assume in addition that $b=b_1+b_2$ where $b_1$ is bounded and $b_2=\nabla U$. 

Assume that there exists a smooth $V$ defined on $\mathbb R$ such that $|V'|$ and $|V''|$ are bounded and $\int e^{- (\bar V + U)} dx < +\infty$ where as before $\bar V(x)=\sum_i V(x^i)$. Introduce the probability distribution $\gamma_0(dx) = Z^{-1} \, e^{- (\bar V + U)(x)} dx$ and define $P$ as the symmetric diffusion process with reversible measure $\gamma_0$ which is assumed to exist. 

Finally assume that $H(\nu_0|\gamma_0)<+\infty$. Then there exists an unique weak solution $Q$ to \eqref{eqsys} with initial law $\nu_0$. In addition $H(Q|P)<+\infty$.
\end{theorem}
For the proof it is enough to remark that the drift of $Q$ w.r.t. $P$ is simply given by the interaction kernel $K$, $\nabla \bar V$ and $b_1$ since we have included $\nabla U$ in the the drift of $P$, and to recall Remark \ref{remframe}.

\begin{remark}\textit{Discussion and examples.}

Theorem \ref{thmkrylpartic} is also partially already known (except the absolute continuity of $Q$) and is contained (for $p>d$) in \cite{Hao} Theorem 1.1 for instance (also see \cite{Han2} Theorem 1.4). The proof in \cite{Hao} uses deep analytical results on singular PDE's (as for Krylov's approach), while ours is more probabilistic and direct. Notice that if we assume in addition that $K$ is local Lipschitz and bounded in $C^c_\varepsilon$ for all $\varepsilon >0$ (recall Assumption \ref{assump0}), we get a strong (pathwise) solution, as for Theorem \ref{thmexistgene}.

However, in our benchmark of examples the only one to which we may apply Theorem \ref{thmkrylpartic} is the relaxed Keller-Segel model. In dimension $d\geq 3$ we should also consider $K_s(x)=\chi \, x/|x|^{s+1}$ for $s<1$ (which is example 1 (i) in \cite{Hao}).
This is a little bit disappointing in comparison with Theorem \ref{thmexistgene} (except that the previous $K_s$ does not satisfy the assumptions in Theorem \ref{thmexistgene}), but proves that approximations by smooth models is not efficient for very singular models.

Of course one explanation is the following: the sign of $\chi$ has no importance in what we have done in this section, while we know that repulsive or attractive models can have very different behaviours, as the Coulombic type potentials clearly show. \hfill $\diamondsuit$
\end{remark}

\subsection{Existence via the absence of collisions. \\ \\}\label{remsubsecnocollision}

Consider the  particle system in $(\mathbb R^d)^N$,
\begin{equation}\label{eqsysnoncoll}
dX_t^{i,N} = \sqrt 2 \, dB_t^{i,N} \, - \, b(X_t^{i,N}) \, dt \, - \, \frac{\chi}{N} \, \sum_j \, \frac{X_t^{i,N}-X_t^{j,N}}{|X_t^{i,N}-X_t^{j,N}|} \, h'(|X_t^{i,N}-X_t^{j,N}|) \, dt
\end{equation}
for some function $b$ assumed to be $a$-Lipschitz and some $h$ defined and regular on $\mathbb R^+ -\{0\}$, in other words the interaction kernel is given by $\chi \, \nabla h(|x|)$. We define $0 \, h'(0)=0$ and assume that $h$ is non-positive and non-decreasing (for instance $h(u)=- \, 1/u^m$ for some $m>0$ corresponds to a (sub)-Coulombic type interaction). If $\chi<0$ the system is repulsive. We make this choice.

A non explosive pathwise solution exists up to the hitting time $T_0$ of the collision set $C_0$ defined in Theorem \ref{thmcoll}. It is tempting to try to directly show that $T_0$ is almost surely infinite. 

Introduce $$H(x) = - \, \sum_{i \neq j} \, h(|x^i-x^j|) \, .$$ Applying Ito formula we have that for $T < T_0$, 
\begin{eqnarray}\label{eqexnocoll}
H(X^N_T) &=& H(X^N_0) +   \, \frac{\chi}{N} \int_0^T  \sum_{i \neq j}  \left< \frac{X_t^{i,N}-X_t^{j,N}}{|X_t^{i,N}-X_t^{j,N}|} \, h'(X_t^{i,N}-X_t^{j,N}) \, , \, \sum_k (A_{i,k} - A_{j,k})\right> dt  \, \nonumber \\ & & \, - \,  2 \, \int_0^T  \sum_{i\neq j} \frac{(d-1)h'(X_t^{i,N}-X_t^{j,N})+ |X_t^{i,N}-X_t^{j,N}| \, h''(X_t^{i,N}-X_t^{j,N})}{|X_t^{i,N}-X_t^{j,N}|} \, dt \nonumber \\ & & \, + \, \int_0^T  \sum_{i\neq j} \frac{\langle b(X_t^{i,N}) - b(X_t^{j,N}),X_t^{i,N} -X_t^{j,N}\rangle} {|X_t^{i,N}-X_t^{j,N}|} \, h'(X_t^{i,N}-X_t^{j,N}) \, dt  \\ & & \, + \, M_T \nonumber
\end{eqnarray}
where $M_.$ is a local martingale, and a true martingale up to $\tau_\varepsilon$ the first hitting time of $C_\varepsilon$ defined in Theorem \ref{thmcoll}, and $$A_{i,k} = (X_t^{i,N}-X_t^{k,N}) \, \, h'(|X_t^{i,N}-X_t^{k,N}|) \, .$$ 

Exchanging the indices it is easily seen that $$\sum_{i \neq j}  \left< \frac{X_t^{i,N}-X_t^{j,N}}{|X_t^{i,N}-X_t^{j,N}|} \, h'(X_t^{i,N}-X_t^{j,N}) \, , \, \sum_k (A_{i,k} - A_{j,k})\right> = $$ $$ \quad \quad = 2 \, \sum_i \, \left|\sum_{j\neq i} \, \frac{X_t^{i,N}-X_t^{j,N}|}{|X_t^{i,N}-X_t^{j,N}|} \, h'(X_t^{i,N}-X_t^{j,N}\right|^2 \, .$$
Since $\chi<0$ (repulsive case) the first integral term in the right hand side is non-positive. 

For the second term to be non-positive it is enough that $$(d-1)h'(u)+uh''(u) \geq 0 \, .$$ Actually since $h'$ is bounded on any interval $[\varepsilon,+\infty[$ for $\varepsilon >0$, if the previous property is satisfied for small $u$'s, the second term is bounded. This implies that $$h'(u) \leq c \, (1/u)^{d-1}$$ for small $u$. In particular if $h(u)= - u^{-m}$ this implies $m \leq d-2$.

Finally if we assume in addition that $u h'(u) \leq c \, |h(u)|$ the third integral term is controlled by $\int_0^T \, ca \, H(X^N_t) \,  dt$, since $b$ is $a$-Lipschitz. a simple use of Gronwall's lemma shows that $$\mathbb E[H(X^N_{T\wedge \tau_\varepsilon})] \leq \mathbb E[H(X^N_0) + C] \, e^{ \, ac \, T} \, .$$
Hence if $\chi<0$, we deduce that $\tau_0=+\infty$ almost surely for $X_0^N=x \notin C_0$, hence $\tau_0=+\infty$ almost surely if the law of $X_0^N$ is absolutely continuous.
\medskip

The previous thus applies  in the sub-Coulombic situation for any $d-2\geq s>0$, furnishing another proof of strong existence for the particle system with or without a confining potential in this range of $s$'s. This is close to the proof chosen in \cite{RS21} in a more general context (the $g$ therein generalizes the sub-Coulombic potential). Notice that the second line in (4.19) of \cite{RS21} is non-positive and not non-negative as it is said therein. We shall discuss a little bit more on this point in a forthcoming section.
\medskip

Of course the method is more general and relies on the existence of a Lyapunov function $H$ which is smooth outside the collision set, equals $+\infty$ on $C_0$, goes to infinity at infinity and satisfies $$\mathcal L^N \, H \leq D^N \, + \, C^N \, H$$ for some positive constants $C^N$ and $D^N$, where $\mathcal L^N$ is the generator of the process $X^N$. The existence of such a Lyapunov function is known to ensure non-explosion (Khasminskii test). The construction of such a Lyapunov function for the system \eqref{eqsysnoncoll} when adding a self-interaction is made in \cite{GLBM1D} Lemma 2.1. For $d=1$ it is also made with $m=2$ (Dyson processes) provided $|\chi|$ is large enough.
\medskip

\section{The non linear S.D.E.}\label{secnonlinear}

We turn to the study of \eqref{eqnldiff} and its applications. 
\medskip

First consider a function $L:\mathbb R^d \mapsto \mathbb R$ which is assumed to be in $C_b^\infty$, i.e. $C^\infty$, bounded with bounded derivatives of any order, and $b$ be bounded and global Lipschitz. 

We may define the system of S.D.E.'s $$dY_t^{i,N} = \sqrt 2 \, dB_t^{i,N} \, - \, b(Y_t^{N,i}) \, dt \, - \, \frac 1N \, \sum_{j=1}^N \, L(Y_t^{i,N}-Y_t^{j,N}) \, dt \, ,$$ starting from an initial configuration $Y^N_0$ whose law is assumed to be exchangeable. If for all fixed $k$ the distribution of $(Y^{1,N}_0,...,Y^{k,N}_0)$ converges to $(\rho_0(y) dy)^{\otimes k}$ as $N$ goes to infinity, it is well known that, for all $T>0$, the distribution of the process $(Y^{1,N}_t,...,Y^{k,N}_t)_{t\in [0,T]}$ converges to the one of an i.i.d. $k$-uple of processes $(Y^{1}_t,...,Y^{k}_t)_{t\in [0,T]}$ each one being a solution of the ``non linear'' S.D.E.
\begin{eqnarray}\label{eqnldiff2}
dY_t &=& \sqrt 2 \, dB_t \, - \, b(Y_t) \, dt \, - \, (L*\rho_t)(Y_t) \, dt \\
Law (Y_t) &=& \rho_t(y) \,  dy \quad \forall t\in[0,T] \, . \nonumber
\end{eqnarray}
In addition this solution is unique. For all this see e.g. \cite{Sznit,Mel}.

Notice that $L*\rho_t$ is a bounded (time dependent) drift, so that the law $Q_L$ of $Y_.$ restricted to the time interval $[0,T]$, satisfies $H(Q_L|P) < +\infty$ where $P$ is as in subsection \ref{subsecpath} for a $V$ with bounded first and second derivatives. Recall that the drift $\beta_t$ is given by $\beta_t=\frac 12 \, (\nabla \bar V- \, (L*\rho_t))(Y_t)$. According to what we have done in the previous sections we thus know some bounds on the entropy and the Fisher information of $\rho_t$. 
\medskip

Assume that $L \in \mathbb L^\alpha(\mathbb R^d)$ for some $\alpha\geq 1$. We will first get a bound for $H(Q_L|P)$ that only depends on $\alpha$ and $d$. To this end we need to get controls for $$ \int |L*\rho_t|^2 \, \rho_t \, dx \, .$$ Since $H(Q_L|P)<+\infty$ we know that $I(\rho_t)<+\infty$ for almost all $t$ so that $\rho_t \in \mathbb L^p(\mathbb R^d)$ for $p \leq d/d-2$ if $d\geq 3$ and $p<+\infty$ for $d=2$. 

As in the previous section introduce $L_A= L \, \mathbf 1_{|L|\geq A}$. $L_A*\rho_t \in \mathbb L^r(\mathbb R^d)$ for $1 + \frac 1r = \frac 1p + \frac{1}{\alpha}$  and  $||L_A*\rho_t||_r \leq ||L_A||_\alpha \, ||\rho_t||_p$. In addition $$\int |(L-L_A)*\rho_t|^2 \, \rho_t dx \leq A^2 \, .$$ We  have for another $p'\leq d/d-2$ (but $p'<+\infty$), $$\int |L_A*\rho_t|^2 \, \rho_t \, dx \, \leq \, ||\rho_t||_{p'} \, |||L_A*\rho_t|||_{2p'/(p'-1)}^2$$ hence for $r=2p'/(p'-1)$ i.e. $\frac{1}{\alpha}=\frac 32 - \frac{1}{2p'} - \frac{1}{p}$, 
\begin{equation}\label{eqkrylnl0}
\int |L_A*\rho_t|^2 \, \rho_t \, dx \, \leq \, ||\rho_t||_{p'} \, ||\rho_t||^2_{p} \, |||L_A|||^2_\alpha \, .
\end{equation}
We thus obtain the analogue of \eqref{eqkrylov1}, 
\begin{equation}\label{eqkrylnl}
\int |L*\rho_t|^2 \, \rho_t \, dx \, \leq C(d) \, |||L_A|||^2_\alpha \, (I(\rho_t))^{d(\frac 32 - \frac{1}{2p'} - \frac{1}{p})} \, + \, A^2 \, .
\end{equation}

In the context of an a priori bound for $I(\rho_t)$, the previous is interesting only if $d(\frac 32 - \frac{1}{2p'} - \frac{1}{p}) \leq 1$, yielding $\alpha \geq d$ for $d\geq 3$ or $\alpha >2$ if $d=2$. Indeed it can be used for $L=  K^M$, where $K^M$ is a sequence of regular ($C_b^\infty$) kernels converging to $K$ in $\mathbb L^\alpha$ norm. We will thus obtain the following analogue of Theorem \ref{thmkryl} and Theorem \ref{thmkrylpartic}

\begin{theorem}\label{thmexistnl}
Let $V$ be a smooth function such that $e^{-V} \in \mathbb L^1(\mathbb R)$ and such that $V'$ and $V''$ are bounded. Define $\bar V(x) = \sum_{i=1}^d \, V(x^i)$. Let $\rho_0$ be a density of probability such that $\int \rho_0(\ln \rho_0 +\bar V) dx < +\infty$. 

Assume that $b$ is bounded and global Lipschitz and that $K \mathbf 1_{|K|>A} \in \mathbb L^\alpha(\mathbb R^d)$ for some $A>0$. 

Then, there exists a (weak) solution of \eqref{eqnldiff} i.e. 
\begin{eqnarray*}
dX_t &=& \sqrt 2 \, dB_t \, - \, b(X_t) \, dt \, - \, (K*\rho_t)(X_t) \, dt \, ,\\
\nu_t = \rho_t(x) \, dx &=& \mathcal L(X_t) \, , \nonumber
\end{eqnarray*}
provided $\alpha >2$ if $d=2$ or $\alpha \geq d$ for $d \geq 3$.

In addition this solution satisfies $H(Q|P)<+\infty$.
\end{theorem}
\begin{proof}
The proof is very similar to the one of Theorem \ref{thmkryl} and here we may choose for simplicity $p=1$, $p'=\alpha/(\alpha-2)$ in \eqref{eqkrylnl} so that the exponent of $I(\rho_t)$ is $d/\alpha$. First $Q^M$ converges to $Q$ in the $\sigma(\mathbb L^1,\mathbb L^\infty)$ topology with $H(Q|P)< +\infty$.  In particular $\rho_t^M$ weakly converges to $\rho_t$ which is the marginal distribution of $Q$ at time $t$, and $\int_0^T \, I(\rho_t) dt < +\infty$ thanks to \eqref{eqentp1}. Hence \eqref{eqkrylnl} is still true when we replace $L$ by $K$ or by $K-K^M$. 

Since $b$ is Lipschitz and bounded (hence does not play any annoying role), it remains to prove that 
$$
\lim_M \mathbb E^{Q^M} \left[\left(\int_s^t \, \langle\nabla \varphi(\omega_u) ,(K^M*\rho_u^M)(\omega_u)\rangle du\right) \, H(\omega_{v\leq s})\right]= $$ $$ \quad = \mathbb E^{Q} \left[\left(\int_s^t \, \langle \nabla \varphi(\omega_u),(K*\rho_u)(\omega_u)\rangle du\right) \, H(\omega_{v\leq s})\right] \, .$$
Using \eqref{eqkrylnl} with $K^M$ and $\rho_t^M$ in place of $L$ and $\rho_t$, it is easily seen, since $\int_s^t \, I^{d/\alpha}(\rho_u^M) \, du \leq c(s,t)$ where $c(s,t)$ does not depend on $M$ (recall that $d/\alpha \leq 1$), that $$
\lim_M \mathbb E^{Q^M} \left[\left(\int_s^t \, \langle\nabla \varphi(\omega_u) ,((K^M-K)*\rho_u^M)(\omega_u)\rangle du\right) \, H(\omega_{v\leq s})\right]= 0$$ so that $$
\lim_M \mathbb E^{Q^M} \left[\left(\int_s^t \, \langle\nabla \varphi(\omega_u) ,(K^M*\rho_u^M)(\omega_u)\rangle du\right) \, H(\omega_{v\leq s})\right]= $$ $$ \quad = \lim_M \, \mathbb E^{Q^M} \left[\left(\int_s^t \, \langle \nabla \varphi(\omega_u),(K*\rho_u)(\omega_u)\rangle du\right) \, H(\omega_{v\leq s})\right] \, .$$
To show that the latter is what we want we may first choose some $B$ large enough so that for all $M$, $$\mathbb E^{Q^M} \left[\left(\int_s^t \, \langle \nabla \varphi(\omega_u),((K-K^B)*\rho_u)(\omega_u)\rangle du\right) \, H(\omega_{v\leq s})\right] \leq \varepsilon $$ and also $$\mathbb E^{Q} \left[\left(\int_s^t \, \langle \nabla \varphi(\omega_u),((K-K^B)*\rho_u)(\omega_u)\rangle du\right) \, H(\omega_{v\leq s})\right] \leq \varepsilon $$ which is possible again using \eqref{eqkrylnl}, and then use the convergence of $Q^M$ to $Q$ in the $\sigma(\mathbb L^1,\mathbb L^\infty)$ topology to control $$\mathbb E^{Q^M} \left[\left(\int_s^t \, \langle \nabla \varphi(\omega_u),(K^B*\rho_u)(\omega_u)\rangle du\right) \, H(\omega_{v\leq s})\right] \, .$$ It remains to let $\varepsilon$ go to $0$.
\end{proof}
\medskip

\begin{remark}\label{remnl1}
Once again, Theorem \ref{thmexistnl} covers the relaxed Keller-Segel case for which a proof appears in \cite{GQ}. The interested reader can compare both proofs. As a general result it also entails Theorem 1.2 in \cite{Han} for $q_1=+\infty$ in equation (1.15) therein, but with a much weaker integrability condition for the initial measure in our result. Actually, our approach can be generalized without too much efforts to time dependent interaction kernels as in \cite{Han} and then improve on the results therein.
\hfill $\diamondsuit$
\end{remark}
\begin{remark}\label{remtimep}
If $L$ is time dependent, we may similarly introduce $L_A^t$ for a fixed $A$. \eqref{eqkrylnl0} is unchanged so that, using H\"{o}lder inequality w.r.t. time we have to control $$\left(\int_0^T ||L_A^t||_\alpha^{2s} dt\right)^{\frac 1s} \, \left(\int_0^T \, I(\rho_t)^{ds'(\frac 32 - \frac{1}{2p'}-\frac 1p)} \,dt\right)^{\frac{1}{s'}}$$ for $\frac 1s + \frac{1}{s'} = 1$. As before what is required for this to be interesting is $\alpha \geq ds'$ ($>$ if $d=2$), yielding, provided $L \in \mathbb L^{2s}([0,T],\mathbb L^{\alpha}(\mathbb R^d))$, $$\frac{d}{\alpha} + \frac 1s \leq 1$$ i.e. generalizing the time inhomogeneous case up to the critical case (of equality).
\hfill $\diamondsuit$
\end{remark}
\medskip
 
\begin{corollary}\label{cornlinteg}
Under the assumptions of Theorem \ref{thmexistnl} the solution of \eqref{eqnldiff} admits for all $t>0$ a density $\rho_t$ w.r.t. Lebesgue measure. In addition
\begin{enumerate}
\item If $\alpha >d$, $\rho_t \in \mathbb L^p(\mathbb R^d)$ for all $1 \leq p < +\infty$.
\item If $\alpha \geq d \geq 3$, $\rho_. \in \mathbb L^1([0,T],\mathbb L^p(\mathbb R^d))$ for all $1 \leq p \leq d/d-2$.
\end{enumerate}
\end{corollary}
\begin{proof}
Let $Q$ be a solution of \eqref{eqnldiff}. First since $\rho_t \in \mathbb L^1(\mathbb R^d)$, $\beta(t,.)= K*\rho_t \in \mathbb L^\alpha(\mathbb R^d)$. We thus may apply Theorem \ref{thmkryl} and Remark \ref{remliterature} (since $\beta$ is time inhomogeneous) ensuring that the \emph{linear} SDE $$dY_t =\sqrt 2 \, dB_t - \, b(Y_t) dt \, - \, \beta(t,Y_t) dt$$ has an unique solution whose marginals flow satisfies the desired integrability properties.
\end{proof}
Theorem \ref{thmexistnl} also furnishes via Ito formula an existence result for the non linear McKean-Vlasov P.D.E.

\begin{corollary}\label{cornl1}
Under the assumptions of Theorem \ref{thmexistnl}, equation \eqref{eqmcv} admits a weak solution, i.e. a flow of measures $\nu_t=\rho_t(x) dx$ such that for all $f \in C_b^{1,\infty}(\mathbb R^+\times \mathbb R^d)$ and all $0 \leq s < t \leq T$, 
\begin{equation}\label{eqwfnl}
\int f(t,x) d\nu_t - \int f(s,x) d\nu_s = \int_s^t \, \int \, (\partial_u f + \Delta_x f + \langle(b+K*\rho_u),\nabla_x f\rangle \, d\nu_u \, du \, .
\end{equation}
In addition this solution belongs to $\mathbb L^1([0,T],\mathbb L^p(\mathbb R^2))$ for any $p \in (1,d/d-2)$ and $p=d/d-2$ if $d\geq 3$, and is a finite energy solution, i.e. $\int_0^T I(\rho_t) dt < +\infty$ and $\int \rho_t \, |\ln(\rho_t)| dx < +\infty$.
\end{corollary}

A converse statement directly follows from Theorem \ref{thmCL2}. Indeed \eqref{eqwfnl} is exactly the weak forward equation Theorem \ref{thmCL2} with the non homogeneous drift $(b+K*\rho_t)$. We may thus state
\begin{theorem}\label{thmnlback}
Let $V$, $b$, $\rho_0$ be as in Theorem \ref{thmexistnl} and $K$ be a measurable kernel.  Let $t \mapsto \rho_t$ be a flow of probability densities, weak solution of \eqref{eqwfnl} such that $\int_0^T \, \int |K*\rho_t|^2 \, \rho_t \, dx \, dt < +\infty$. Then there exists a weak solution $Q$ of \eqref{eqnldiff} and $H(Q|P) < +\infty$. 
\end{theorem}

\begin{example}\label{exnl1}

Let us describe two situations (in addition to the one in Theorem \ref{thmexistnl}) where one can check the energy condition of Theorem \ref{thmnlback}

\begin{itemize}
\item[(1)] \quad Assume that $K \in \mathbb L^\alpha(\mathbb R^d)$ for some $\alpha \geq 1$. Choose $p=p'=\frac{3\alpha}{3\alpha-2}$, provided $\frac{3\alpha}{3\alpha-2} \leq \frac{d}{d-2}$ i.e. $\alpha \geq d/3$. \eqref{eqkrylnl} furnishes $$\int |K*\rho_t|^2 \, \rho_t \, dx \leq C(d) \, ||K||_\alpha^2 \, (I(\rho_t))^{\frac{d}{\alpha}} $$ so that if $I(\rho_t) \in \mathbb L^{d/\alpha}([0,T])$ (in particular if $\sup_{t\in ]0,T]} I(\rho_t) < +\infty$) we may apply Theorem \ref{thmnlback}.
\item[(2)] \quad Let $K$ and $p=\frac{3\alpha}{3\alpha-2}$ as in the previous example. If $||\rho_t|| \, \in \mathbb L^3([0,T],\mathbb L^{3\alpha/(3\alpha -2)}(\mathbb R^d))$ (in particular if $\sup_{t\in ]0,T]} ||\rho_t||_{3\alpha/(3\alpha -2)} < +\infty$) we may directly apply \eqref{eqkrylnl0} to show that Theorem \ref{thmnlback} applies.
\end{itemize}
We shall discuss explicit examples satisfying in particular the second condition in the next sections.
\end{example}
\medskip

For uniqueness we will have to discuss separately each example. However the next result covers some cases
\begin{proposition}\label{propuniquenl}
Assume that the assumptions of Theorem \ref{thmexistnl} are fulfilled for $\alpha \geq d$ ($>2$ if $d=2$). Then there is only one weak solution of \eqref{eqnldiff}.
\end{proposition}
\begin{proof}
Let $Q$ be a solution of \eqref{eqnldiff}. First since $\rho_t \in \mathbb L^1(\mathbb R^d)$, $\beta(t,.)= K*\rho_t \in \mathbb L^\alpha(\mathbb R^d)$. We thus may apply Theorem \ref{thmkryl} and the final comment of Remark \ref{remliterature} (since $\beta$ is time inhomogeneous) ensuring that the \emph{linear} SDE $$dY_t =\sqrt 2 \, dB_t - \, b(Y_t) dt \, - \, \beta(t,Y_t) dt$$ has an unique solution. Since $Q$ is a solution, it is the only one and $H(Q|P)< +\infty$. It follows that $\int_0^T \, I(\rho_t) \, dt < +\infty$.  

Let $Q'$ be another solution of \eqref{eqnldiff}, with marginals $\rho'_t$ and same initial $\rho_0$. Denote by $Q_t$ and $Q'_t$ the restriction of $Q$ and $Q'$ to $[0,t]$. Applying Pinsker inequality, it holds 
\begin{eqnarray*}
\sup_{u\leq t} \, ||\rho_u' -\rho_u||_1^2 &\leq& 2 \, \sup_{u\leq t} \, H(\rho_u|\rho'_u)  \, \leq \, 2 \, H(Q_t|Q'_t) \\ &\leq&  \, 2 \, \int_0^t \, \int \, |K*(\rho_s-\rho'_s)|^2 \, \rho_s \, dx \, ds \\ &\leq& \, 2 \, ||K \, \mathbf 1_{|K|>A}||_\alpha^2 \, \sup_{u \leq t} ||\rho_u' -\rho_u||_p^2 \, \int_0^t \, ||\rho_s||_{p'} \, ds + \, 2 \, A^2 \, t \,  \sup_{u \leq t} ||\rho_u' -\rho_u||_1^2 
\end{eqnarray*}
for any $p$ and $p'$ in $[1,d/d-2]$ if $d\geq 3$, $[1,+\infty)$ if $d=2$, and satisfying $\frac 32 = \frac 1\alpha + \frac{1}{2p'}+ \frac 1p$ (recall \eqref{eqkrylnl0}). Choose $p=1$ and $p'=\alpha/(\alpha -2) \leq d/(d-2)$. We deduce from Lemma \ref{lemFHM} (2) and (3) that $$\int_0^t \, ||\rho_s||_{p'} \, ds \, \leq \, C(p') \, \left(T + \int_0^T \, I(\rho_s) ds\right) \, \leq \, C'(p',T,H(Q|P)) \, ,$$ where $M \mapsto C'(p',T,M)$ is non-decreasing. First choose $A$ such that $$2 \, ||K \, \mathbf 1_{|K|>A}||_\alpha^2 \, C'(p',T,H(Q|P)) \leq \frac 12 \, .$$ We thus have $$\sup_{u\leq t} \, ||\rho_u' -\rho_u||_1^2  \leq 4 A^2 \, t \, \sup_{u\leq t} \, ||\rho_u' -\rho_u||_1^2$$ so that if $t$ is small enough, satisfying $4A^2 t<1$ we deduce $\rho_u=\rho'_u$ for all $u \leq t$. 

We may now iterate the procedure, i.e. 
\begin{eqnarray*}
\sup_{u\leq 2t} \, ||\rho_u' -\rho_u||_1^2 &\leq& 2 \, \int_0^{2t} \, \int \, |K*(\rho_s-\rho'_s)|^2 \, \rho_s \, dx \, ds = 2 \, \int_t^{2t} \, \int \, |K*(\rho_s-\rho'_s)|^2 \, \rho_s \, dx \, ds \\ &\leq& 2 \, ||K \, \mathbf 1_{|K|>A}||_\alpha^2 \, \sup_{u \leq 2t} ||\rho_u' -\rho_u||_1^2 \, \int_t^{2t} \, ||\rho_s||_{p'} \, ds + \, 2 \, A^2 \, t \,  \sup_{ u \leq 2t} ||\rho_u' -\rho_u||_1^2 
\end{eqnarray*}
and deduce as before $\rho_u=\rho'_u$ for all $u \leq 2t$ and finally for $0\leq u\leq T$. It follows that $Q$ and $Q'$ are solutions of the same \emph{linear} SDE as before, so that as we already said that $Q=Q'$.
\end{proof}
The previous result extends \cite{Hao} Theorem 4.8 to the critical case $\alpha=d$ and to a larger class of initial conditions. One can also see \cite{RoZ}.
\begin{remark}\label{remnlunique}
The previous proof actually shows that, in full generality, if $K \mathbf 1_{|K|>A} \in \mathbb L^\alpha(\mathbb R^d)$ for some $\alpha \geq 2$, uniqueness holds in the set of solutions such that $\int_0^T \, ||\rho_s||_{p'} \, ds < +\infty$, for $p'\geq \alpha/(\alpha -2)$.
\hfill $\diamondsuit$
\end{remark}
\medskip

In some cases one can give a direct proof of uniqueness
\begin{proposition}\label{propuniquenl1}
Consider \eqref{eqnldiff}, assume that $b$ is $L$-Lipschitz and that $K$ satisfies $|K(x)-K(y)| \leq C \, |x-y| \, (\frac{1}{|x|^m} + \frac{1}{|y|^m})$ for some $0 \leq m < d$. Then there exists at most one solution of \eqref{eqnldiff} which is pathwise unique in the set of Probability measures whose marginals flow satisfies $\rho_. \in \mathbb L^1([0,T],\mathbb L^p(\mathbb R^d))$ in particular if $\sup_{t \in (0,T]} ||\rho_t||_p \leq M$ for some $p>d/(d-m)$.
\end{proposition}
The proof is similar to the proof of Theorem 1.7 in \cite{GQ} p.981-982 written for $d=2$, the key being the estimate (5.1) therein. 

\begin{remark}\label{remuniqueGQ}
It is worth to notice that if $K(x)=x/|x|^m$ for some $m>0$, the previous inequality is satisfied with the same $m$ and $C=m+1$. Look at the first two lines in the proof of Lemma 2.5 in \cite{GQ}, and consider the case $|x|\leq |y|$. \hfill $\diamondsuit$
\end{remark}
\medskip

\section{SDE and non linear SDE for some singular examples.}\label{secexample}

In this section we shall review the models we have quoted in the introduction, and give results of existence and uniqueness for both the SDE system of particles and the non linear SDE.
\medskip

\subsection{The $\eta$ relaxed Keller-Segel model. \\ \\}\label{subsecrelax}

Recall that this model corresponds to $K_\eta(x)= \chi \, \frac{x}{|x|^{2-\eta}} \, \mathbf 1_{x\neq 0}$ for $x \in \mathbb R^2$. 

The existence of a strong solution for the particle system can be shown either by using Dirichlet forms as in subsubsection \ref{subsubsecrelaxdir} or approximations i.e. Theorem \ref{thmkrylpartic} (and no collisions to get a strong and not only weak solution). The second proof shows that this solution is of finite entropy.

The existence of a solution for the non-linear SDE is a consequence of Theorem \ref{thmexistnl}. This solution has finite entropy. The marginals flow of this solution thus satisfies $\int_0^T I(\rho_t) dt < +\infty$ so that, according to \eqref{eqFHM2}, $\rho_. \in \mathbb L^1([0,T],\mathbb L^p(\mathbb R^2))$ for any $1\leq p < +\infty$. According to Remark \ref{remuniqueGQ}  we may apply Proposition \ref{propuniquenl1} to get strong uniqueness in the corresponding set of Probability measures. 

Actually we have strong uniqueness without restrictions. Indeed if $Q'$ is another solution of the non linear SDE, it is also a solution of the \emph{linear} SDE with drift $b+K*\rho_t$. Hence according to Remark \ref{remliterature} (see \eqref{eqkrylint2}) its marginals flow $\rho'_. \in \mathbb L^1([0,T],\mathbb L^p(\mathbb R^2))$ for all $p<+\infty$, so that we may apply the previous uniqueness result.

The $\eta$ relaxed Keller-Segel model is thus the prototype to which all we have done before applies. Notice that we may take $\chi>0$ or $\chi<0$, as we said before the entropic approach cannot separate repulsive and attractive situations.
\medskip

\subsection{The sub-Coulombic repulsive model. \\ \\}\label{subseccoul}

We directly state a first result for the particle system
\begin{theorem}\label{thmserfexist}
For $d\geq 3$ look at $K(x) = \chi \, \frac{x}{|x|^{s+2}} \, \mathbf 1_{x \neq 0}$ for $d-2 \geq s>0$ and $\chi <0$, i.e. the repulsive situation in \cite{RS21}. Assume in addition that $b$ is $a$-Lipschitz. 

Finally assume the following property: for all $i \neq j$, the distribution $\tilde \rho_{0,i,j}$ of $(X_0^i-X_0^j)$ satisfies $\tilde \rho_{0,i,j} \in \mathbb L^q(\mathbb R^d)$ for some $q > d/(d-s)$.

Then the particle system \eqref{eqsys} admits a unique strong solution whose distribution $Q$ satisfies $H(Q|\frac{\rho_0}{\gamma_0} \, P)<+\infty$, where $P$ is the product measure introduced in subsection \ref{subsecpath} with $|V'|$ bounded. 

Accordingly, if $\int \rho_0 \, \ln(\rho_0/\gamma_0) \, dx < +\infty$, the marginals flow $\rho_t$ satisfies $\int_0^T \, I(\rho_t) dt < +\infty$ so that $\rho_. \in \mathbb L^1([0,T],\mathbb L^p(\mathbb R^{Nd}))$ for $p\in [1,Nd/Nd-2]$ and each marginal $\rho^j_. \in \mathbb L^1([0,T],\mathbb L^p(\mathbb R^{d}))$ for $p \leq d/d-2$.
\end{theorem}
Notice that, as explained in the proof of Lemma \ref{lemFHM2} the additional property is satisfied as soon as $I(\rho_{0,i,j})<+\infty$ where $\rho_{0,i,j}$ is the joint density of $(X_0^i,X_0^j)$. 

\begin{proof}
If $b$ is a confining potential, according to subsubsection \ref{subsubsecexempdir} we know that the particle system admits a unique strong solution which is absolutely continuous with respect to the Wiener measure with the same initial distribution, i.e. there is no collision.  In addition the stationary solution has finite relative entropy. If we want to consider more general initial conditions we have to use the method in subsection \ref{remsubsecnocollision}.
\medskip

Come back to \eqref{eqexnocoll} with $h(a)=-(1/a)^s$. If $0<s\leq d-2$, we have on one hand $$\mathbb E(H(X_T^N)) \leq \mathbb E(H(X_0^N)) + sa \, \int_0^T \, \mathbb E(H(X_t^N)) \, dt$$ for any stopping time $T$ showing that on one hand there are no collisions, on the other hand $$\mathbb E(H(X_T^N)) \leq \mathbb E(H(X_0^N)) \\, e^{saT}$$ for any fixed time $T$. Arguing as in the proof of Lemma \ref{lemFHM2}, the assumption on $\rho_0$ ensures that $\mathbb E(H(X^N_0))< +\infty$. 

\eqref{eqexnocoll} shows in particular (recall the $H \geq 0$) that 
\begin{eqnarray*}
\mathbb E\left[\frac{2 |\chi| s}{N} \int_0^T  \sum_i \left|\sum_j \frac{X_t^{i,N}-X_t^{j,N}}{|X_t^{i,N}-X_t^{j,N}|^{s+2}}\right|^2  dt\right]  &\leq& \mathbb E(H(X^N_0))  \\ & & \, + \, s \, a \,  \int_0^T  \, \mathbb E(H(X_t^N)) \, dt \\ &\leq& \mathbb E(H(X^N_0)) \, (1 + e^{saT})  
\end{eqnarray*}
since $\mathbb E(H(X_t^N)) \leq \mathbb E(H(X_0^N)) \, e^{sat}$. Hence if $\mathbb E(H(X^N_0)) < +\infty$, the left hand side is also finite. If $b=0$, the left hand side is nothing else but $$2N \, s \, |\chi|^{-1} \, H(Q|W)$$ where $H(Q|W)$ is the relative entropy of the law $Q$ of $X_.^N$ w.r.t. the Wiener measure with the same initial condition. Adding a bounded $b$ does not change the finiteness of relative entropy, so that $H(Q| \frac{\rho_0}{\gamma_0} \, P)$ is also finite, where $P$ is the product measure introduced in subsection \ref{subsecpath} with $|V'|$ bounded as in Lemma \ref{lemmoment1}.

If $b$ is only Lipschitz and no more bounded, one sees that the same line of reasoning can be used provided one can control the second moments of $X_t^{i,N}$ ($|b(x)| \leq c + a|x|$). Using Ito's formula this control amounts to the control of $\mathbb E \left(\frac{1}{|X_t^{i,N}-X_t^{j,N}|^{s}}\right)$ i.e. of $\mathbb E(H(X_t^N))$ which was already done.  

The others statements in the Theorem follow from the discussion on relative entropy, Lemma \ref{lemFHM} and sub-additivity of the Fisher information.
\end{proof}

Hence the particle system has an entropic solution with or without confinement. This has to be related to some comments in the introduction (subsection 1.4 of \cite{RS21}) about the confinement force. 
\begin{remark}\label{remcoulN}
In the previous Theorem, assume that $\max_{i,j} ||\tilde \rho_{0,i,j}||_q = C_q$. Then $\mathbb E(H(X^N_0)) \leq C(s,a,T,C_q) \, N^2$. It follows that $$H(Q| \frac{\rho_0}{\gamma_0} \, P) \leq C(s,a,T,|\chi|,C_q) \, N \, .$$ This will be crucial in the sequel.

It is also worth noticing that the integrability of $\left|\sum_j \frac{X_t^{i,N}-X_t^{j,N}}{|X_t^{i,N}-X_t^{j,N}|^{s+2}}\right|^2$ is much weaker than the one of $$\sum_j \left|\frac{X_t^{i,N}-X_t^{j,N}}{|X_t^{i,N}-X_t^{j,N}|^{s+2}}\right|^2 = \sum_j \frac{1}{|X_t^{i,N}-X_t^{j,N}|^{2s+2}}$$ which is presumably not true. 
\hfill $\diamondsuit$
\end{remark}
\medskip

Clearly $K \notin \mathbb L^\alpha(\mathbb R^d)$ for any $\alpha \geq d$ so that the entropic approach of the non-linear SDE fails. Fortunately the non-linear PDE \eqref{eqwfnl} has some good properties (for our purpose). Let us recall the following result shown in section 3 of \cite{RS21} (we emphasize that this result assumes that there is no confining potential, i.e. $b=0$, see the discussion in Remark 1.6 of \cite{RS21})
\begin{proposition}\label{propedpRS}
Suppose $b=0$, $0<s<d-2$ and that $\rho_0$ is a probability density which is bounded. Then there exists a flow $t \mapsto \rho_t$ of probability densities, satisfying \eqref{eqwfnl} and belonging to $\mathbb C^0([0,T],\mathbb L^\infty(\mathbb R^d))$. This flow belongs to $C^0([0,T],\mathbb L^\infty(\mathbb R^d))$ and satisfies $||\rho_t||_\infty \leq ||\rho_0||_\infty$. 
It is the unique solution  of \eqref{eqwfnl} in the set $\mathbb C^0([0,T],\mathbb L^\infty(\mathbb R^d))$.
\end{proposition}
Here are the precise references in \cite{RS21}. Remark 3.7 shows that the non-linear PDE is positivity preserving and Remark 3.4 that mass is preserved. Proposition 3.1 shows existence of a bounded solution on small enough time interval, Remark 3.4 shows that the $\mathbb L^\infty$ norm of $\rho_t$ is non-increasing in time, so that one can conclude (see Remark 3.5) to the existence of a global (in time) bounded solution. Uniqueness is part of Proposition 3.1. 
\medskip

Notice that even if $\rho_0$ is bounded we do not know (at this stage) whether the density of the particle system at time $t$ is bounded or not, but we can apply Theorem \ref{thmserfexist}.
\medskip

We can thus apply Theorem \ref{thmnlback} since for $A$ large enough, $K \, \mathbf 1_{|K|>A} \in \mathbb L^1(\mathbb R^d)$ (this is true as soon as $s<d-1$) so that $$\int_0^T \, \int |K*\rho_t|^2 \, \rho_t \, dx \, dt \, \leq \, 2T \, (\sup_{0\leq t\leq T}||\rho_t||_\infty^2 \, ||K \, \mathbf 1_{|K|>A}||_1^2 \, + \, A^2 ) \, .$$ We thus have obtained the existence of a weak solution to the non-linear SDE, which is of finite entropy. 
Of course the previous bound is obtained via 
\begin{equation}\label{eqboundedserf}
|K*\rho_t| \, \leq \, A \, + \, \sup_{0\leq t\leq T}||\rho_t||_\infty \, ||K \, \mathbf 1_{|K|>A}||_1
\end{equation} 
showing that the ``drift'' $|K*\rho_t|$ is bounded. The existence of a \emph{linear} diffusion process with this drift is thus standard via Girsanov theory. That this diffusion process has marginals $\rho_t$ is more or less classical and it seems that calling upon the results of subsection \ref{subsecpath}, if presumably not necessary, is not as strange. Nevertheless this remark shows that $dQ/dP$ belongs to all the $\mathbb L^p$'s for $1\leq p < +\infty$ and not only $L^{x ln x}$. Finally one can apply the uniqueness part in Proposition \ref{propedpRS}.
\medskip

Let us gather the results we have just described
\begin{theorem}\label{thmnonlincoulomb}
For $d\geq 3$ look at $K(x)=\chi \, \frac{x}{|x|^{s+2}} \, \mathbf 1_{x \neq 0}$ for $d-2 > s >0$ and $\chi<0$. Finally assume that $\rho_0$ is a bounded density of probability.

Then there exists a solution $\bar Q$ of the non linear SDE \eqref{eqnldiff}. This solution is absolutely continuous w.r.t. $(\rho_0/\gamma_0) P$ with a density $Z_T \in \bigcap_{1\leq p<+\infty} \, \mathbb L^p((\rho_0/\gamma_0) P)$ (in particular with finite entropy) and such that  the marginals flow satisfies $\sup_{0 \leq t \leq T} ||\rho_t||_\infty < +\infty$. 

It is the pathwise unique solution in the set of Probabilities on the path space such that the marginals flow satisfies $\sup_{0 \leq t \leq T} ||\rho_t||_\infty < +\infty$.
\end{theorem}

Notice that the limitation to the repulsive case seems to be necessary for Proposition \ref{propedpRS} too. 
\medskip

We shall see in section \ref{secchaos2} another approach of existence. A natural question will thus be to get uniqueness in more general situations. To this end we may see how to apply Remark \ref{remnlunique} or Proposition \ref{propuniquenl1} and Remark \ref{remuniqueGQ}. Since $K \in \mathbb L^\alpha(\mathbb R^d)$ for $\alpha<d/(1+s)$, in order to apply Remark \ref{remnlunique} we need $d>2(1+s)$ and $\rho_. \in \mathbb L^1([0,T],\mathbb L^{d/(d-2s-2)}(\mathbb R^d))$ which is of course worse than the $d/(d-s-2)$ obtained in Proposition \ref{propuniquenl1}. 

We are nevertheless led to look at a priori $q$ integrability of the marginals of solutions to the non linear SDE. This will be also done in section \ref{secchaos2}.

\medskip

\subsection{The $2D$ vortex model.\\ \\}\label{subsecvortex}

In dimension $d=2$ we consider the Biot-Savart kernel $K(x)= \chi \, \frac{x^\perp}{|x|^2} \, \mathbf 1_{x\neq 0}$ where for $x=(x_1,x_2)$ we define $x^\perp=(-x_2,x_1)$. The case of $\chi<0$ is studied in \cite{Mapul}, we will also consider the case $\chi>0$. Notice that in \cite{FHM} a more general situation with some random weights $\chi_{j}$ is also considered.

Look first at the particle system. 

When $b=0$, the natural candidate as invariant measure is given by the density 
\begin{equation}\label{eqinvbiot}
\rho_\infty(x) = \prod_{i < j} \, e^{- \, \frac{\chi}{N} \, \arctan((x_2^i -x_2^j)/(x_1^i -x_1^j))}
\end{equation}
 which is bounded (we may define $\arctan(0/0)=0$ if necessary), but does not satisfy \eqref{eqdiric1}, so that we may not use the strategy of subsection \ref{subsecdiric}. For $b\neq 0$ the same is true.

Nevertheless one can use the strategy in subsection \ref{remsubsecnocollision} and find a Lyapunov function proving the absence of collisions. This is done in \cite{Taka}. Actually, we may use the function $$H(x)= - \sum_{i \neq j} \, \ln(|x^i-x^j|^2)$$ as in subsection \ref{remsubsecnocollision}.  

Up to a multiplicative constant the first part of the drift is given by $$\sum_{i,j,k} \left \langle \frac{X_t^{i,N} -X_t^{j,N}}{|X_t^{i,N}-X_t^{j,N}|^{2}}, \left(\frac{(X_t^{i,N} -X_t^{k,N})^\perp}{|X_t^{i,N}-X_t^{k,N}|^{2}} - \frac{(X_t^{j,N} -X_t^{k,N})^\perp}{|X_t^{j,N}-X_t^{k,N}|^{2}}\right) \right \rangle$$ i.e. $$\sum_{i,j,k} \left \langle \frac{X_t^{i,N} -X_t^{j,N}}{|X_t^{i,N}-X_t^{j,N}|^{2}},\frac{(X_t^{i,N} -X_t^{k,N})^\perp}{|X_t^{j,N}-X_t^{k,N}|^{2}}\right \rangle \, - \, \sum_{i,j,k} \left \langle \frac{X_t^{i,N} -X_t^{j,N}}{|X_t^{i,N}-X_t^{j,N}|^{2}},\frac{(X_t^{j,N} -X_t^{k,N})^\perp}{|X_t^{j,N}-X_t^{k,N}|^{2}}\right \rangle \, .$$ Exchanging the role of $i$ and $j$ we see that both terms in the difference are opposite, so that this difference is twice the first term. Now exchanging the role of $j$ and $k$ we have $$\sum_{i,j,k} \left \langle \frac{X_t^{i,N} -X_t^{j,N}}{|X_t^{i,N}-X_t^{j,N}|^{2}},\frac{(X_t^{i,N} -X_t^{k,N})^\perp}{|X_t^{j,N}-X_t^{k,N}|^{2}}\right \rangle = \sum_{i,j,k} \left \langle \frac{X_t^{i,N} -X_t^{k,N}}{|X_t^{i,N}-X_t^{k,N}|^{2}},\frac{(X_t^{i,N} -X_t^{j,N})^\perp}{|X_t^{j,N}-X_t^{k,N}|^{2}}\right \rangle$$ so that using that $\langle u,v^\perp\rangle = - \langle u^\perp,v\rangle$ the previous sum equals $0$. 

Finally, using that $\Delta_x \ln(|x|^2)=0$ for $x \neq 0$ \eqref{eqexnocoll} becomes 
\begin{equation}\label{eqexnocollbis}
H(X^N_u) - H(X^N_0) = 2 \, \int_0^u  \sum_{i\neq j} \frac{\langle b(X_t^{i,N}) - b(X_t^{j,N}),X_t^{i,N} -X_t^{j,N}\rangle}{|X_t^{i,N}-X_t^{j,N}|^{2}} \, dt  \, + \, M_u 
\end{equation}
for all stopping time $u$ with $u<\tau_0$ ($\tau_\varepsilon$ denotes the hitting time of the set $C_\varepsilon$ defined in Theorem \ref{thmcoll}), ensuring again that there are no collisions starting from a non collision $x$ arguing as follows. Since $H$ is not everywhere positive we have to introduce $S_L$ the first time $H(X_.^N)$ becomes less than $-L$, for a positive $L$. We thus have $$\mathbb E[H(X_{T\wedge \tau_\varepsilon\wedge S_L}^N)] \leq \mathbb E[H(X_{0}^N)] + 2 \, T \, a$$ so that if $X_0^N=x$ a non collision point, $\tau_0 \geq T\wedge S_L$ almost surely for all $L$. It remains to make $L$ go to infinity.

Remark that the result does not depend on the sign of $\chi$.

According to Theorem \ref{thmcoll} the distribution $Q^N$ of $X_.^N$ on $C([0,T],(\mathbb R^2)^N)$ is thus absolutely continuous with respect to the corresponding Wiener measure or w.r.t $P$ provided the distribution at time $0$ is absolutely continuous w.r.t. $\gamma_0$. 
\medskip

Contrary to the sub-Coulombic case, we cannot use \eqref{eqexnocollbis} for studying the relative entropy. One can however obtain interesting informations. Indeed, we may explicitly write the martingale part $M_t$. It holds 
\begin{eqnarray*}
M_t &=& - \, 2\sqrt 2 \, \int_0^t \, \sum_{i\neq j} \, \frac{\langle X_s^{i,N}-X_s^{j,N},dB_s^i-dB_s^j\rangle}{|X_s^{i,N}-X_s^{j,N}|^2} \\ &=& - 4 \sqrt 2 \, \sum_{i} \, \, \int_0^t \, \left \langle \sum_{j \neq i} \frac{X_s^{i,N}-X_s^{j,N}}{|X_s^{i,N}-X_s^{j,N}|^2},dB_s^i \right\rangle
\end{eqnarray*}
the second equality being obtained by exchanging the roles of $i$ and $j$, so that $$\mathbb E(M_T^2) = C \, \sum_i \, \int_0^T \, \left|\sum_{j \neq i} \frac{X_t^{i,N}-X_t^{j,N}}{|X_t^{i,N}-X_t^{j,N}|^2}\right|^2 \, dt \, .$$ 
Since $|U^\perp|=|U|$, once again, up to a constant and to the additional regular drift $b$ we recognize $H(Q|\frac{\rho_0}{\gamma_0} P)$. More precisely 
\begin{equation}\label{eqentropdrifbiot}
H(Q|\frac{\rho_0}{\gamma_0} P) \leq   \frac{\chi^2}{32 \, N^2} \, \mathbb E(M_T^2) + N \, ||b||_\infty^2 \, .
\end{equation}
We are thus led to estimate $\mathbb E(H^2(X_u^N))$ for $u=0,T$ since the expectation of the square of remaining term involving $b$ in \eqref{eqexnocollbis} is less than $4 a^2 T \, N^4$. We will use the rough estimate $|\sum_{i\neq j} a_{i,j}|^2 \leq N^2 \, \sum_{i\neq j} |a_{i,j}|^2$ and the elementary $$\ln^2(v) \leq C + v + \frac 1v$$ for some constant $C$ and all $v>0$. It allows us to reduce the problem to the study of  $G_\alpha(x)= \sum_{i\neq j} |x^i-x^j|^\alpha$ for $\alpha=1$ and $\alpha=-1$. 

Using Ito's formula, the same manipulations as for $H$ show that 
\begin{eqnarray*}
G_\alpha(X^N_u) - G_\alpha(X^N_0) &=& \alpha \, \int_0^u  \sum_{i\neq j} \langle b(X_t^{i,N}) - b(X_t^{j,N}),X_t^{i,N} -X_t^{j,N}\rangle \, |X_t^{i,N}-X_t^{j,N}|^{\alpha -2} \, dt \\ & & \, + \,  2 \, \alpha^2 \, \int_0^u  \sum_{i\neq j} |X_t^{i,N}-X_t^{j,N}|^{\alpha -2} \, dt \, + \, M_{u,\alpha}
\end{eqnarray*}
for all $u<\tau_0$. We immediately deduce using Gronwall lemma that for $\alpha=2$, and $t\leq T$,$$\mathbb E(G_2(X^N_t)) \leq \mathbb E(G_2(X^N_0)) + C(a,T) \, N^2 \, .$$ For any $0<\alpha\leq 2$, using $|v|^\alpha \leq 1 + |v|^2|$, it follows $$\mathbb E(G_\alpha(X^N_t)) \leq \mathbb E(G_2(X^N_0))) + C'(a,T) N^2 \, .$$  We thus deduce
\begin{eqnarray}\label{eqbiotpower}
\int_0^T  \sum_{i\neq j} \mathbb E(|X_t^{i,N}-X_t^{j,N}|^{\alpha -2}) \, dt &\leq& \frac{a}{2\alpha} \, \int_0^T \, G_\alpha(X^N_t) \, dt \, + \, (1/2\alpha^2) \, \mathbb E(G_\alpha(X^N_T)) \nonumber \\ &\leq& C(T,a,\alpha) \, (N^2+\mathbb E(G_2(X^N_0))) \, .
\end{eqnarray}
and finally
\begin{equation}\label{eqbiotpower2}
\int_0^T \, \mathbb E(H^2(X_t^N)) \, dt \, \leq \, C(T) \, N^2 \, \left( N^2 \, + \, \int_0^T \, \mathbb E(G_1(X_t^N)) \, dt \, + \, \int_0^T  \sum_{i\neq j} \mathbb E(|X_t^{i,N}-X_t^{j,N}|^{- 1}) \, dt\right)
\end{equation}
so that using  \eqref{eqbiotpower} with $\alpha =1$ and what precedes
\begin{equation}\label{eqbiotpower3}
\int_0^T \, \mathbb E(H^2(X_t^N)) \, dt \, \leq \, C(T,a) \, N^2 \, (N^2+\mathbb E(G_2(X^N_0))) \, .
\end{equation}
From \eqref{eqexnocollbis} we deduce for all $t\leq T$, 
\begin{eqnarray*}
\mathbb E(M_t^2) &\leq& 3 \, \mathbb E(H^2(X_t^N)) + \, 3 \, \mathbb E(H^2(X_0^N)) + \, c \, a^2 \, T^2 \, N^4 \, \\ &\leq& 3 \, \mathbb E(H^2(X_t^N)) + C(a,T)N^2(N^2+\mathbb E(G_2(X^N_0))) + \, c \, a^2 \, T^2 \,  N^4 \, .
\end{eqnarray*}
so that integrating with respect to $t$ it holds for all time $T$, 
\begin{equation}\label{eqbiotpower4}
\int_0^{2T} \, \mathbb E(M_t^2) \, dt \, \leq \, C(a,T) \, N^2(N^2+\mathbb E(G_2(X^N_0))) \, + \, c \, a^2 \, T^2 \,  N^4 \, .
\end{equation}
Using that for $T\leq t\leq 2T$, $\mathbb E(M_t^2)\geq \mathbb E(M_T^2)$, we deduce from \eqref{eqbiotpower4} that
\begin{equation}\label{eqbiotpower5}
\mathbb E(M_T^2) \leq T^{-1} \, C(a,T) \, N^2(N^2+\mathbb E(G_2(X^N_0))) \,  .
\end{equation}
so that we have obtained the analogue of Theorem \ref{thmserfexist}
\begin{theorem}\label{thmbiotexist}
For $d=2$ look at $K(x) = \chi \, \frac{x^\perp}{|x|^{2}} \, \mathbf 1_{x \neq 0}$ and assume in addition that $b$ is bounded and $a$-Lipschitz. 

If $\int \, \sum_{i,j} |x^i-x^j|^2 \, \rho_0(dx) < +\infty$, then the particle system \eqref{eqsys} admits a unique strong solution whose distribution $Q$ satisfies $H(Q|\frac{\rho_0}{\gamma_0} \, P)<+\infty$, where $P$ is the product measure introduced in subsection \ref{subsecpath} with $|V'|$ bounded. 

Accordingly, if $\int \rho_0 \, \ln(\rho_0/\gamma_0) \, dx < +\infty$, the marginals flow $\rho_t$ satisfies $\int_0^T \, I(\rho_t) dt < +\infty$ so that $\rho_. \in \mathbb L^1([0,T],\mathbb L^p(\mathbb R^d))$ for $p\in [1,2N/2(N-1)]$.
\end{theorem}

\begin{remark}\label{rembiotN}
The previous calculation shows that $$H(Q|\frac{\rho_0}{\gamma_0} \, P) \leq C(a,T) N^2 \, \left(1+ \sup_{i,j} \int \, |x^i -x^j|^2 \, \rho_0(dx)\right) \, + N \, ||b||^2_\infty \, .$$  This bound will be unfortunately insufficient for studying propagation of chaos.
\hfill $\diamondsuit$
\end{remark}

One can nevertheless obtain a somehow better bound for the free energy functional, assuming that the additional drift $b=\nabla \tilde U$ as in subsection \ref{subsecpath}, with $\tilde U(x) = \sum_j U(x^j)$ and $\int e^{- U} < +\infty$, so that $\theta_0 = Z^{-1} \, \rho_\infty \, e^{-\tilde U} \, dx$ is a reversible probability measure for the particle system. We will assume that $U'$ is Lipschitz, but non necessarily bounded in order to allow a gaussian confinement.

We may then use Remark \ref{remenerglibre} and obtain $$H(\rho_0 \, dx|\theta_0) \geq H(\rho_T \, dx|\theta_0) + \int_0^T \, \int \, \sum_{i=1}^N \, \left|\nabla_i \ln(\rho_t) + b(x) + \frac{\chi}{2N} \, \sum_{j \neq i} \, \frac{(x^i-x^j)^\perp}{|x^i-x^j|^2} \right|^2 \, \rho_t \, dx \, dt \, .$$ Since $$H(\rho_0 \, dx|\theta_0) = \int \rho_0 \, \left(\ln(\rho_0) \, + \sum_i \, U(x^i)  \,  + \, \frac{\chi}{2N} \, \sum_{i \neq j} \, \arctan((x_2^i -x_2^j)/(x_1^i -x_1^j)) \right) \, dx $$ we obtain that provided $\int \rho_0 |\ln(\rho_0)| dx = O(N)$ and $\max_j \, \int \, |U(x^j)| \, \rho_0 \, dx < +\infty$, the free energy dissipation $$\int_0^T \, \sum_i \, \int \, \left|\nabla_i \ln(\rho_t) + b(x) + \frac{\chi}{2N} \, \sum_{j \neq i} \, \frac{(x^i-x^j)^\perp}{|x^i-x^j|^2} \right|^2 \, \rho_t \, dx \, dt$$ is of order $N$.

According to the discussion in subsection \ref{subsecpath}, the marginals flow of the solution is (up to some constants) almost an entropy solution of the Liouville equation in the sense of \cite{JW18} definition 2 (so that we partly recover Proposition 1 therein). 

First using $|a+b|^2 \geq \frac 12 \, |a|^2 \, - \, |b|^2$ we have $$H(\rho_0 \, dx|\theta_0) + \, \int_0^T \, \int |b|^2 \, \rho_t dx \, dt \geq H(\rho_T \, dx|\theta_0) + $$ $$ \quad \frac 12 \, \int_0^T \, \sum_i \, \int \, \left|\nabla_i \ln(\rho_t)  + \frac{\chi}{2N} \, \sum_{j \neq i} \, \frac{(x^i-x^j)^\perp}{|x^i-x^j|^2} \right|^2 \, \rho_t \, dx \, dt$$
Expanding the squared norm in the right hand side, using a truncation $\psi(x)$ with compact support, the product term becomes $$\frac{2\chi}{N} \, \int \, \left\langle \nabla \rho , \psi \, \sum_{i \neq j} \, \frac{(x^i-x^j)^\perp}{|x^i-x^j|^2} \right\rangle \, dx$$ one can integrate by parts. Since the divergence of the Biot-Savart kernel vanishes it remains to look at $$\int \, \rho \, \left\langle \sum_{i \neq j} \, \frac{(x^i-x^j)^\perp}{|x^i-x^j|^2},\nabla \psi \right\rangle \, dx$$ which goes to $0$ by correctly choosing a sequence of $\psi$'s, since $$\int \, \left|\sum_{i \neq j} \, \frac{(x^i-x^j)^\perp}{|x^i-x^j|^2}\right| \, \rho \, dx < +\infty \, .$$ For the left hand side we remark that $|b|^2$ is at most quadratic. Using one more time the specificity of the kernel, it is immediately seen, since $b$ is Lipschitz, that $$\mathbb E(|X_t|^2) \leq \mathbb E(|X_0|^2) + c\left(t + \int_0^t \, \mathbb E(|X_s|^2) ds\right)$$ so that $\mathbb E(|X_t|^2) \leq c(T) \, \mathbb E(|X_0|^2) $ by using Gronwall.

Since $H(\rho_T dx|\theta_0) \geq 0$, we finally have obtained a better result
\begin{theorem}\label{thmbiotexistconfine}
Assume that $\tilde U(x) = \sum_j U(x^j)$ is a confining potential such that $\nabla U$ is Lipschitz, and that $\int |x|^2 \rho_0(x) dx < +\infty$. Then there exists a constant $c(T)$ such that $$H(\rho_0 \, dx|\theta_0) + c(T) \, \int |x|^2 \rho_0(x) dx \, \geq H(\rho_T \, dx|\theta_0) \, + \, \frac 12 \, \int_0^T \, I(\rho_t) \, dt \, + \, \frac 18 \, H(Q|\frac{\rho_0}{\gamma_0} P) \, .$$ Recall that here $\gamma_0(dx)=Z^{-1} \, e^{- \tilde U(x)} dx$ and $P$ is the distribution of the $\gamma_0$ symmetric diffusion process $dY_t=dB_t - \nabla \tilde U(Y_t)dt$.

In particular if $H(\rho_0 \, dx|\theta_0) = O(N)$ and $\int |x|^2 \rho_0(x) dx = O(N)$, $H(Q|\frac{\rho_0}{\gamma_0} P)=O(N)$.
\end{theorem}
\begin{remark}\label{remFHMbol}
In the case without confinement, using the calculations in the proof of Lemma \ref{lementropdecay} and an ad-hoc regularization $K_\varepsilon$ of $K$ such that $div K_\varepsilon =0$, one obtains $$\int \rho_T \, \ln(\rho_T) \, dx + \int_0^T \, I(\rho_t) \, dt \, =  \, \int \rho_0 \, \ln(\rho_0) \, dx \, ,$$ i.e. formula (5.7) in \cite{FHM}. 

As said in \cite{FHM} this equation is interesting provided one can get a lower bound for the Boltzmann entropy $\int \rho_T \, \ln(\rho_T) \, dx$ (which is no more necessarily non-negative). Such a (uniform) lower bound is obtained provided one can find an upper bound for $\int |x|^k \, d\rho_T$ (see \cite{FHM} Lemma 3.1). That is why, the whole paper \cite{FHM} is studying together polynomial moments, Boltzmann entropy and Fisher information. Notice that provided $\int |x|^k \, d\rho_0 = O(N)$ for some fixed $k > 0$, then, according to Proposition 5.1 in \cite{FHM}, as soon as $\int \rho_0 \, |\ln(\rho_0)| dx = O(N)$, $\int_0^T \, I(\rho_t) \, dt = O(N)$.
\hfill $\diamondsuit$
\end{remark}
\medskip

In the sequel we assume that $b=0$. It is shown in \cite{Bena} Theorem B, that for any initial density of probability $\rho_0$, the nonlinear \eqref{eqwfnl} has a solution belonging to $C^0((0,T],\mathbb L^\infty(\mathbb R^2))$. If in addition $\rho_0$ is bounded, then $||\rho_t||_\infty \leq ||\rho_0||_\infty$. Assuming some more regularity on $\rho_0$ furnishes higher regularity for $\rho_t$ see \cite{Bena} Theorem A or \cite{FWvortex} Lemma 2.2. In addition this solution is unique in the set of bounded functions.

Another existence result, using convergence of the particle system is shown in \cite{FHM} Lemma 3.5 and Theorem 2.5 (one can also look at the former \cite{Mapul}).
\medskip

In \cite{FHM} section 7 another uniqueness result is shown in the set of solutions such that $\nabla \rho_. \in \mathbb L^{2q/(3q-2)}([0,T],\mathbb L^q(\mathbb R^2))$ for all $1\leq q <2$. If $\int_0^T \, I(\rho_t) \, dt < + \infty$, the latter condition is satisfied according to \ref{lemFHM} (1). This uniqueness result is self contained i.e. the proof does not use the particle approximation, contrary to the existence part.

We thus have the analogue of proposition \ref{propedpRS}

\begin{proposition}\label{propedpvortex}
Assume that $\rho_0$ is a probability density such that $\int \rho_0 \, (\ln (\rho_0)+|x|^2) dx < +\infty$. Then, when $b=0$, if $\rho_0$ is bounded, there exists a flow $t \mapsto \rho_t$ of probability densities, satisfying \eqref{eqwfnl} and belonging to $\mathbb C^0([0,T],\mathbb L^\infty(\mathbb R^d))$. 

Even if $\rho_0$ is not bounded, there exists at most one solution in the set of probability densities flows such that $\int_0^T \, I(\rho_t) \, dt < + \infty$.

Accordingly there exists a solution $\bar Q$ of the non linear SDE, satisfying the same properties as in Theorem \ref{thmnonlincoulomb} and at most one weak solution in the set of probability measures such that $\int_0^T \, I(Q\circ (X_t)^{-1}) \, dt < + \infty$.
\end{proposition}

As we said, the proof of existence for this solution given in \cite{FHM} is totally different, based on the relative compactness (tightness) of the law of $X_.^{1,N}$. Notice that the former \cite{Fontmel} also contains results in this direction. We shall look at this approach later.
\begin{remark}\label{remBSdrift}
We claim that the above result is still true if we add a confining potential as before. To show this, it is enough to extend Ben Artzi's proof following the presentation made in the proof of Theorem 2 in \cite{GLBM-biot}. We will not give the details, so that this result is still a claim, and we shall see later another approach for uniqueness.
\hfill $\diamondsuit$
\end{remark}
\medskip

\subsection{Dyson processes. \\ \\}\label{subsecdysongene}

To complete the picture, let us say a word on Dyson type models. 

We gave an existence proof in subsubsection \ref{subsubsecexempdir} for the particle system, recovering the results in \cite{Cepa,GLBM1D} for $\chi/N < -1$. The existence of the particle system was obtained in \cite{Cepa} Theorem 3.1 for $\chi<0$ using the theory of multivalued SDE's developed by C\'epa in his PHD Thesis. Transporting the system on the torus, the authors have shown in \cite{Cepacercle} that for $0>\chi/N>-1$ collisions of two particles always occur, while for $\chi/N \leq -1$ they do not (remark that the equality case $\chi/N=-1$ is out of reach of our approach). In \cite{Cepamult} they also have shown that no $k$-collisions occur for $k \geq 3$. Actually the proof in \cite{CP,FT-coll} using Dirichlet forms presumably also works for the Dyson model, for $\chi<0$.

The non linear SDE is not studied in \cite{Cepa,GLBM1D} where the limit for a large number of particles is studied for a varying $\chi=\chi_N  \to +\infty$ after a linear time change (vanishing noise). 
\medskip

\section{The $2D$ parabolic-elliptic (Patlak)-Keller-Segel attractive model.}\label{subsecKS}

The (Patlak) Keller-Segel system introduced in \cite{KS1}, is a tentative model to describe chemo-taxis phenomenon, an attractive chemical phenomenon between organisms. Since we do not yet speak of this model, that contains new interesting features, we will give a more complete description of the situation.

In two dimensions, the classical 2-D parabolic-elliptic Keller-Segel model reduces to the  single non linear P.D.E., 
\begin{equation}\label{eqKS}
\partial_t \rho_t(x) =  \Delta_x \, \rho_t(x) + \chi \, \nabla_x.((K*\rho_t)\rho_t)(x) 
\end{equation}
with some initial $\rho_0$. It is not difficult to see that \eqref{eqKS} preserves positivity and mass, so that we may assume that $\rho_0$ is a density of probability i.e. the model enters the framework of this work with $K(x)=\chi \, \frac{x}{|x|^2}$  defined on $\mathbb R^2$. $K$ is the gradient of the harmonic kernel, i.e. $K(x) = \nabla \, \log(|x|)$. In order to compare our model to the usual formulation, the reader can think that the parameter $\chi$ is actually given by $$\chi \, = \, \chi_0 \, \frac{\alpha \, m}{2\pi \, D}$$ where $\chi_0$ is the chemotactic sensitivity, $\alpha$ is the rate of production of chemoattractant by the cells, $m$ is the total mass and $D$ is the product of the diffusivities.\\ As usual, $\rho$ is modeling a density of cells, and $c_t=K*\rho_t$ is (up to some constant) the concentration of chemo-attractant. 
\medskip

A very interesting property of such an equation is a blow-up phenomenon. The following is easily obtained by looking at the time evolution of the variance of a solution (see \cite{BDP}):
\begin{proposition}\label{blow-up}
If $\chi >4$, the maximal time interval of existence of a classical solution of \eqref{eqKS} is $[0,T^*)$ with $$T^* \, \leq \, \frac{1}{2\pi \, \chi \, (\chi-4)} \, \int |x|^2 \, \rho_0(x) \, dx \, .$$ 
\end{proposition}
The existence of a solution, up to this explosion time, is more delicate. One classically says that $\chi<4$ is the sub-critical case, $\chi=4$ the critical case and $\chi>4$ the super-critical one.

As in \cite{BDP} one can consider a weak version of \eqref{eqKS}, i.e. looking for a continuous flow $s \mapsto \mu_s$ of probability measures on $\mathbb R^2$ satisfying for all smooth function $\varphi$ and all $t>0$
\begin{eqnarray}\label{eqKS2}
\int \, \varphi \, d\mu_t &=& \int \, \varphi \, d\mu_0 + \, \int_0^t \, \int \, \Delta \varphi \, d\mu_s \, ds +  \nonumber \\ && \; - \,  \frac{\chi}{2} \, \int_0^t \, \int \, \int \, \langle K(x-y),\nabla \varphi(x)-\nabla \varphi(y)\rangle \, \mu_s(dx) \mu_s(dy) \, ds \, .
\end{eqnarray}
It is worth to remark that the integrand in the last term in \eqref{eqKS2}, which is obtained thanks to the symmetry of $K$, is bounded, so that the integral is well defined for any probability distribution $\mu$. The oddness of $K$ is a key point for writing the last term in this convenient form.

We then have (see \cite{BDP} Theorem 1.1 for $\chi<4$ and \cite{BCM} Theorem 1.3 for $\chi=4$, also see \cite{Liuca})
\begin{theorem}\label{thm-exist-edp}
For $\chi \leq 4$, assume that $\mu_0(dx)=\rho_0(x) \, dx$ is an initial probability distribution. If $$\int \, (|x|^2 + |\ln \rho_0|(x)) \, \mu_0(dx) < +\infty \, ,$$ then there exists a flow of probability densities $t \mapsto \rho_t(x)$ defined for all $t\geq 0$, such that $\mu_t(dx)=\rho_t(dx) dx$ is a solution of \eqref{eqKS2}. 
\end{theorem}
Actually the solution built in these papers satisfies additional regularity properties. First for any $T>0$, 
\begin{equation}\label{eqsolentrop}
\sup_{t \in[0,T]} \, \int \, (|x|^2+|\ln \rho_t|(x)) \, \rho_t(x) \,dx \, < +\infty \, .
\end{equation}
Next, if we introduce the free energy of a density of probability $\rho$, defined by
\begin{equation}\label{eqfree}
F(\rho) \, = \, \int \, \rho \, \ln \rho \, dx \, + \, \frac{\chi}{2} \, \int \, c(\rho) \, \rho \, dx \, ,
\end{equation}
where $$c(\rho)(x) = \int \,  \ln(|x-y)|) \, \rho(y) \,  dy \, $$
then the solution in theorem \ref{thm-exist-edp} satisfies
\begin{equation}\label{eqfree2}
F(\rho_t) \, + \, \int_0^t \, \int \, (|\nabla \ln \rho_s(x) + \chi \nabla c(\rho_s)(x)|^2 ) \, \rho_s(x) \, dx \, ds \, \leq \, F(\rho_0) \, .
\end{equation}
Such a solution is called a \emph{free energy} solution. This notion is particularly relevant for uniqueness, for which we have the following (see \cite{FM} Theorem 1.3)
\begin{theorem}\label{thm-unique}
If $\chi \leq 4$, and $\rho_0$ satisfies the assumption in theorem \ref{thm-exist-edp}, there exists at most one free energy solution of \eqref{eqKS2}, i.e. satisfying \eqref{eqsolentrop} and \eqref{eqfree2}.
\end{theorem}
\begin{remark}\label{rembetter}
Actually \cite{BDP} contains slightly better results. Indeed according to Lemma 2.11 and the proof of Lemma 2.12 therein, for all $t>0$, $$\int_0^t \, \int \, |\nabla \ln \rho_s(x)|^2 \, \rho_s(x) \, dx \, ds < +\infty$$ and $$\int_0^t \, \int \, |\chi \nabla c(\rho_s)(x)|^2  \, \rho_s(x) \, dx \, ds < +\infty$$ too.

Actually this point is automatic, as shown in \cite{FM} Lemma 2.2. Denote by 
\begin{equation}\label{eqfreeD}
D_t(F)= \int_0^t \, \int \, (|\nabla \ln \rho_s(x) + \chi \nabla c(\rho_s)(x)|^2 ) \, \rho_s(x) \, dx \, ds \, .
\end{equation}
Using integration by parts and $\Delta c(\rho)= 2\pi \, \rho$, any solution of \eqref{eqKS} satisfies $$\int_0^t \, I(\rho_s) ds + \int_0^t \, \int \,  |\chi \nabla c(\rho_s)(x)|^2  \, \rho_s(x) \, dx \, ds \leq \, D_t(F)  +  \, 4\pi \, \chi \, \int  \int_0^t \, \rho_s^2(x) \,dx ds \, .$$  $\int_0^t \, I(\rho_s) ds < +\infty$ and $\int_0^t \, \int \,  |\chi \nabla c(\rho_s)(x)|^2  \, \rho_s(x) \, dx \, ds < +\infty$ follow from Lemma \ref{lemFHM} (2) and some clever manipulations, see the end of the proof of Lemma 2.2 in \cite{FM}. Conversely if these two quantities are finite, $D_t(F) < +\infty$ is immediate. Remark that this can be rewritten using $$\int_0^t \, \int \,  |\chi \nabla c(\rho_s)(x)|^2  \, \rho_s(x) \, dx \, ds = \int_0^t \, \int \, |K*\rho_s|^2 \, \rho_s \, dx \, ds < +\infty \, .$$ 
At least formally one has $F(\rho_t)+D_tF(\rho) = F(\rho_0)$, provided $D_tF(\rho)<+\infty$.  However a rigorous proof requires some additional regularity on $\rho_.$ (see e.g. \cite{BDP} Lemma 2.2). This regularity is shown in \cite{FM}, in particular Lemmas 2.6, 2.7, 2.8. This regularity holds on time intervals $[t_0,T]$ but for $t_0>0$. An accurate reading shows that assuming \eqref{eqfree2} is not used in these proofs. One can thus deduce for $t \geq t_0 >0$, $F(\rho_t)+D_tF(\rho) = F(\rho_{t_0})+ D_{t_0}F(\rho)$. Step 2 in the proof of Theorem 1.4 p. 1176 of \cite{FM} together with Lemma 2.9 therein (lower semi continuity of the free energy) furnishes that \eqref{eqfree2} is actually satisfies as soon as $D_tF < +\infty$. The price to pay is that equality becomes an inequality in \eqref{eqfree2} because of lower semi-continuity.

Hence uniqueness holds once, in addition to \eqref{eqsolentrop}, $$\int_0^t \, \int \, (|\nabla \ln \rho_s(x) + \chi \nabla c(\rho_s)(x)|^2 ) \, \rho_s(x) \, dx \, ds < +\infty \, .$$ This will be of particular interest in the sequel. Notice that according to the results in subsection \ref{subsecbounds}, as said in Lemma 2.4 in \cite{FM}, 
\begin{equation}\label{eqregul}
\rho_. \in \mathbb L^{p/(p-1)}(]0,T], \mathbb L^p(\mathbb R^2)) \quad \textrm{ for all $1<p<+\infty$.}
\end{equation}
Stronger regularity is known since, according to \cite{FM} Lemma 2.8, for all $\varepsilon >0$, $\rho_. \in C_b^\infty([\varepsilon,T]\times \mathbb R^2)$.

When $\rho_0$ is bounded, this result is shown in to Theorem 1.3 in \cite{CLM}, for bounded solutions. If $\rho_0 \in \mathbb L^p$ it is shown in \cite{FM} Lemma 2.7, that the free energy solution built in Theorem \ref{thm-exist-edp} satisfies $\sup_{0\leq t \leq T} ||\rho_t||_{\mathbb L^p} < + \infty$. 
\hfill$\diamondsuit$
\end{remark}
\begin{remark}\label{remmild}
Another notion of solution called ``mild'' solution is discussed in \cite{Bedmas,wei}. Existence for general initial data, and uniqueness in the set of such solutions is obtained therein. A mild solution is a solution that can be expressed via Duhamel formula for the heat semi-group. In particular an a priori $\mathbb L^{\frac 43}$ bound is required. It is shown in \cite{FM} section 3 that the (unique) free energy solution is a mild solution.
\hfill$\diamondsuit$
\end{remark}

\medskip

The assumptions in theorem \ref{thm-exist-edp} can be relaxed, when $\chi<4$, the following is proved in the recent \cite{FT-explo} (Theorem 2 therein), where the moment condition \eqref{eqFT2} is shown
\begin{theorem}\label{thmexFT}
If $\chi<4$, for any initial probability measure $\mu_0$, there exists a flow of probability measures  $t \mapsto \mu_t$ defined for all $t>0$ and solution  of \eqref{eqKS2} . Furthermore, for all $T>0$, and all $2 \geq \gamma > \frac{\chi}{2}$, 
\begin{equation}\label{eqFT2}
\int_0^T \, \int \int \, |x-y|^{\gamma \, - \, 2} \, \mu_s(dx) \, \mu_s(dy) \, \leq \, C(\gamma,\chi) \, (1+T) \, .
\end{equation}
\end{theorem}
Unfortunately, no uniqueness result is known in this more general context. Notice that if $\rho_0$ satisfies the assumptions in Theorem \ref{thm-exist-edp}, according to Remark \ref{rembetter} and Lemma \ref{lemFHM2} the previous integrability \eqref{eqFT2} holds for all $0<\gamma < 2$. 
\medskip

\begin{remark}\label{remadvection}
In the sequel we will also be interested by the model with an added drift $b(x)=\nabla \tilde U(x)=\sum_{j=1}^2 \, U'(x^j)$ for some nice $U$ we will call a confinement potential and the corresponding drift an advection term. As for the $2D$ vortex model we will assume that $U''$ is bounded. When $U$ is convex,  $- \langle x,\nabla \tilde U\rangle \leq 0$ inducing a true confinement. There is no difficulty to add a second bounded drift term, but for an easier reading we will not introduce this additional term.

Despite our tentatives, we did not find in the jungle of the literature on Keller-Segel models, the analogue of the results recalled before for $\alpha=0$, except one case: a change of coordinates in self similar variables naturally introduces a quadratic confinement potential (see e.g \cite{FM} p.1164 formula (1.16) and what precedes). One can thus transfer the previous results to this situation. 
\medskip

We nevertheless claim that all that we recalled before is still true if we add a confinement as we defined before. The only way we found to see it is to rewrite the proofs in \cite{BDP} and \cite{FM}. We shall not give all the details here, only explain how to modify some notations, and why the confinement does not perturb the heart of the problem.
\medskip

First, we may look at the second moment $$\int |x|^2 \rho_T dx = \int |x|^2 \rho_0 dx +(4-\chi)T - \, 2 \int_0^T \, \int \, \langle x, \nabla \tilde U\rangle \, \rho_t \, dx \, dt$$ where the last term is non-positive as soon as $\tilde U$ is convex. It follows that, as for the usual Keller-Segel model, existence of a global solution is only possible for $\chi \leq 4$. In this case one also sees that the second moment is finite at time $T$ if it is finite at time $0$.

In order to mimic the proofs in \cite{BDP} it is enough to modify the free energy functional by adding $\frac 12 \, \int \, \tilde U \, \rho \, dx$ i.e. replace $\chi c$ by $\chi c + \tilde U$. We follow the numbering in \cite{BDP}. Lemma 2.3 is still true, up to this modification, and since $\int f \, \tilde U \, dx \geq 0$ for a non-negative $f$, adding this term in the Hardy-Littlewood-Sobolev inequality (2.6) does not modify the lower bound, and Lemma 2.5 is unchanged. The approximation method in 2.5.3 is modified, replacing $\chi c^\varepsilon$ by $\theta^\varepsilon= \chi c^\varepsilon + \tilde U$. The key lemma 2.11 is also unchanged, once one remarks that $\Delta c^\varepsilon$ is replaced by $\Delta c^\varepsilon + \Delta \tilde U$ with a bounded $\Delta \tilde U$, so that for instance in the proof of (iv) p.14, one simply has to add a constant in the right hand side of the first inequality. We stop here, but the reader can easily follow line by line \cite{BDP} and see that, in the worst case, only constant additional terms appear and do not modify the main statements, the most important results being Lemma 2.12, proposition 3.3 or Lemma 3.4, where for instance the dissipation (or production) of the free energy is a simple consequence of the non positivity of its time derivative, as in Lemma 2.3.
\medskip

One can do exactly the same with \cite{FM}. The modification of $\chi c$ does not introduce any difference in the proofs. For instance, Step 4 on p. 1177 uses the self similar coordinates, introducing a quadratic term $|x|^2/2$ in formula (2.27). This term is changed into $|x|^2/2 + h(x)$ for an at most quadratic $h$ and all the estimates starting from (2.29) are still true. 
\hfill $\diamondsuit$
\end{remark}
\medskip

From the point of view of the non linear PDE the situation is close to the case of the sub-Coulombic potential. However for the particle system, the situation is very different. Indeed one cannot use nor subsection \ref{subsecksdir} with the invariant measure candidate $$G_N(dx) = \prod_{1\leq i<j\leq N} \; |x^i-x^j|^{- \, \frac{\chi}{N}} \, dx \, ,$$ nor subsection \ref{subsecparticexist} since $K \notin \mathbb L^2(\mathbb R^2)$. 

One can nevertheless prove the following
\begin{theorem}\label{thmexistN}
Assume that $b=0$ and $K(x) = \chi \, \frac{x}{|x|^2} \, \mathbf 1_{x\neq 0}$ for some $\chi>0$. 

Let $M^N=\{x \in (\mathbb R^2)^N \, \textrm{there exists no triple $i\neq j \neq k$ such that $x^i=x^j=x^k$}\}$. Then for $N \geq 2$ and $\chi<4$, there exists a non explosive solution $Q_x$ of \eqref{eqsys} starting from any $x\in M^N$. Moreover the process is strong Markov and admits a symmetric $\sigma$-finite invariant  measure given by $$G_N(dx) = \prod_{1\leq i<j\leq N} \; |x^i-x^j|^{- \, \frac{\chi}{N}} \, dx \, .$$ In addition, for all $T>0$, 
\begin{equation}\label{eqnul}
\sum_{i \neq j} \, \int_0^T \, \mathbf 1_{\omega_t^{i,N}=\omega_t^{j,N}} \, dt \, = \, 0 \, ,  \quad Q_x \; a.s.
\end{equation}
If $\chi>4$ the solution explodes in finite time.

If $\chi < 4 \, \frac{N-2}{N-1}:= \chi_N$ the process lives in $M^N$ i.e. there are no $k$-collisions for $k\geq 3$ and is unique (in distribution). 
\end{theorem}
The first part in this general form is due to \cite{FT-coll} (see proposition 2 and proposition 3 therein) and uses Dirichlet forms theory (due to the normalizing coefficients our $\chi$ is $2 \theta$ in \cite{FT-coll}). The use of Dirichlet forms theory was previously introduced in \cite{CP} where a similar result is claimed for $N\geq 4$ and $\chi < 4 \, \frac{N-2}{N-1}$. Actually the (too simple) proof of the absence of $k$ collisions given in \cite{CP} subsection 2.5.1 is incomplete (for the interested readers one also has to consider the case $\varepsilon=0$ therein). A complete proof of this fact is contained in \cite{FJ} Lemma 15 and its proof. \cite{FJ} also provides us with a proof of existence for $\chi < 4 \, \frac{N-2}{N-1}$, using some compactness argument we shall revisit later on. 

The (non trivial) fact that $G_N$ is locally bounded on the whole space is shown in Appendix A of \cite{FT-coll}, while \cite{CP} only considered this measure restricted to $M^N$.

For $\chi\geq 4$ explosion is shown in \cite{CP} subsection 2.4 and a very precise description of how multiple collisions occur before explosion is done in \cite{FT-coll} Theorem 5.
\medskip

Uniqueness in the sense of Hunt processes is also shown in \cite{CP} Theorem 3.2 and the same argument works for the more general case of \cite{FT-coll}. It follows that one can build a unique solution with any initial distribution $\rho_0^N(x) dx$, since $M^N$ has Lebesgue measure $0$, and this solution satisfies \eqref{eqnul}. We are using here the notation $G_N$ which is the one used in the recent \cite{BJW}. The latter property \eqref{eqnul} is Lemma 3.4 in \cite{CP}. 
\medskip

Uniqueness starting from a non $3$-collision point such that in addition no multiple $2$-collisions occur is shown in \cite{CP} and uses an explicit computation made in \cite{FJ} Lemmata 19 and 20. If several pair of coordinates are equal (two by two) it can be easily extended. This more precise result will be useless in the sequel.
\medskip

Another crucial result is that $2$-collisions are not only allowed but always occur with positive probability as shown in \cite{FJ} Proposition 4 (with an initial exchangeable condition) or in \cite{CP} subsubsection 2.5.2
\begin{proposition}\label{propkscoll}
For any solution of \eqref{eqsys} in the Keller-Segel setting, and any $t>0$, $$\mathbb P(\exists s \in [0,t], \exists i\neq j \; ; \; X_s^{i,N}=X_s^{j,N}) > 0 \, .$$ In particular a solution $Q$ to \eqref{eqsys} cannot be absolutely continuous w.r.t. $P$ or to the Wiener measure with the same initial condition.
\end{proposition}
For the consequence stated above, recall that if $B$ and $B'$ are two independent $2$-dimensional Brownian motions, $Z=|B-B'|^2$ is a squared Bessel process of dimension $2$, so that the origin is polar, i.e. the hitting time of $0$ for $Z$ is almost surely infinite. This shows that for a collection of independent $2$-dimensional particles whose law is absolutely continuous w.r.t. the Wiener measure, collisions (starting from a non-collision point) never occur.
\begin{remark}\label{remKSmicro}
It is interesting to see that Keller and Segel themselves proposed in \cite{KS2} a \emph{microscopic} description of their aggregation model, which is nothing else (in a modern formulation) than the particle system \eqref{eqsys}.
\hfill $\diamondsuit$
\end{remark}
\begin{remark}\label{remconfinepartic}
In all what precedes one can add a linear confinement potential as in Remark \ref{remadvection} simply using a Girsanov transform, since the additional drift is bounded. We shall explain below that one can also add a quadratic confinement potential.
\hfill $\diamondsuit$
\end{remark}
\medskip

In the recent \cite{BJW}, the authors have studied the Liouville equation (recall \eqref{eqliouv}) associated to \eqref{eqsys} for the Keller-Segel model, i.e.
\begin{equation}\label{eqkolm}
\partial_t \mu_t^N = \sum_{i=1}^N \, \nabla_{x^i}. \left(\mu_t^N \, \frac{\chi}{N} \, \sum_{j=1}^N \, \frac{x^i - x^j}{|x^i -x^j|^2}\right) + \Delta_x \, \mu_t^N
\end{equation}
where as before $0=0/0$. The authors introduce in their definition 2.1 the notion of (relative) entropy solution, as follows:
\begin{definition}\label{def-entrop}
A solution (in the sense of Schwartz distributions) $\mu_t^N$ of \eqref{eqkolm} is said to be an entropy solution if for all $T>0$, 
\begin{equation}\label{eqentrop}
\int_0^T \, I(\mu_t^N|G_N) \, dt \, \leq \, H(\mu_0^N|G_N) - H(\mu_T^N|G_N)
\end{equation}
where as usual $H(\mu|\nu)$ and $I(\mu|\nu)$ denote the relative entropy (Kullback-Leibler information) and the relative Fisher information, i.e. $$I(\mu|\nu)= \int \left|\nabla \ln\left(\frac{d\mu}{d\nu}\right)\right|^2 \, d\mu \, .$$ 
\end{definition}
Since $G_N$ is not bounded one has to be careful with these definitions. It is immediate that both hand sides in \eqref{eqentrop} are unchanged if we replace $G_N$ by $c \, G_N$ for any nonnegative constant. Remark that one can always choose an appropriate $c$ for the relative entropy $H$ to be nonnegative as in the usual probabilistic situation, but $c$ depends on $\rho$.

Actually the authors of \cite{BJW} replace $\mathbb R^2$ by the $2$ dimensional torus $\Pi$, and consider periodic solutions after periodizing the potential $\ln(|x|^2)$. The meaning and definition of the stochastic system \eqref{eqsys} has to be reformulated. The existence of an entropy solution of \eqref{eqkolm} is shown in \cite{BJW} Proposition 4.1. 
\medskip

Notice that, formally in $(\mathbb R^2)^N$ 
\begin{equation}\label{eqinfKSpartic}
\int_0^T \, I(\mu_t^N|G_N) \, dt =\int_0^T \, \int \, \sum_i \, \left(\left|\nabla_i \ln \rho^N_t + \frac{\chi}{N} \, \sum_{j\neq i} \, \frac{x^i -x^j}{|x^i-x^j|^2} \right|^2\right) \, \rho^N_t(x) \, dx \, dt
\end{equation}
 so that \eqref{eqentrop} is similar to \eqref{eqfree2} in the context of the particle system. Recall that \cite{BJW} works in a $\Pi^N$ the $N$-tensor product of the $2D$ torus, so that the previous equality has to be modified.
\medskip

Another point is concerned with a sentence at the top of p.4 of \cite{BJW} i.e. ``Of course any strong solution to \eqref{eqsys} (in the probabilistic sense) would also yield an entropy solution to \eqref{eqkolm}.'' Since the notion of \emph{strong solution in the probabilistic sense} is not defined in \cite{BJW}, the meaning of this sentence is only speculating. However, the proof of the existence of an entropy solution given in the Appendix of \cite{BJW}, is based on a $\varepsilon$ regularization different from, but similar to the one used in \cite{FJ} for proving the existence of a solution of the particle system. Since we know the (weak) uniqueness of a solution at the particles level, it indicates that, provided one can extend this result to our situation, $\rho_.^{N,U}$ will be an entropy solution. 

Actually some key uniforms (in $\varepsilon$) estimates in the \cite{BJW} proof are obtained via a large deviation estimate  (Proposition 2.1 therein) inspired by \cite{JW18}. If replacing the Lebesgue measure which is bounded on the torus by $e^{- \tilde U} dx$ on the whole space is presumably what has to be done to extend the argument in \cite{BJW}, we confess that we were not able to understand all the steps of the proof in sections 2 and 3 of \cite{BJW}. 
\medskip

We shall thus give a proof of a precise statement. The (potential) reader in a hurry can skip what follows up to the statement of Theorem \ref{propKSconfine}.

Instead of working on the torus we will add a confining smooth potential $U$ defined on $\mathbb R$, as for the 2D vortex model. 

Next, we replace $G_N$ by the measure $G_N^{U}(dx) =  e^{- \tilde U(x)} \, G_N(dx)$. The existence and uniqueness (for $\chi < \chi_N$) of a $G_N^{U}$ symmetric diffusion process 
\begin{equation}\label{eqKSconfine}
dX_t^{i,N,U} = \sqrt 2 \, dB_t^{i,N} \, - \, \nabla \tilde U(X_t^{i,N,U}) - \, \frac{\chi}{N} \, \sum_{j \neq i} \, \frac{X_t^{i,N,U}-X_t^{j,N,U}}{|X_t^{i,N,U}-X_t^{j,N,U}|^2} \, dt 
\end{equation}
can be shown exactly as for Theorem \ref{thmexistN} in the same range $\chi < 4$. If $|\nabla \tilde U|$ is bounded, one can also use a Girsanov transformation.  

Our goal is to use Remark \ref{remenerglibre}. The main difficulty is that if $N$ is large, $G_N^{U}$ is \emph{not} a bounded measure. To overcome this difficulty, first consider the approximate model $Q^{N,U,\varepsilon}$ where we replace $|x-y|^2$ by $\varepsilon +|x-y|^2$. This time the associated $$G_N^{U,\varepsilon}= \prod_{1\leq i<j\leq N} \; (\varepsilon +|x^i-x^j|^2)^{- \, \frac{\chi}{2N}} \, e^{- \tilde U(x)} \, dx $$ is bounded with a normalization constant denoted by $Z_\varepsilon$. Applying \eqref{eqenerglibre} we thus have 
\begin{equation}\label{eqKSliouvpresque}
H(\rho_0^N dx| Z_\varepsilon^{-1} \, G_N^{U,\varepsilon}) - H(\rho_T^{N,U,\varepsilon} dx| Z_\varepsilon^{-1} \, G_N^{U,\varepsilon}) \geq 
\end{equation}
$$ \geq \int_0^T \, \int \, \left(\sum_i \, \left|\nabla_i \ln \rho^{N,U,\varepsilon}_t + \nabla \tilde U \, + \frac{\chi}{N} \, \sum_{j\neq i} \, \frac{x^i -x^j}{\varepsilon +|x^i-x^j|^2} \right|^2\right) \, \rho^{N,U,\varepsilon}_t(x) \, dx \, dt \, .$$ The key is that the normalization constant, that already disappeared in the right hand side thanks to the gradient, also disappears in the left hand side which is equal to $$\int \ln(\rho^{N}_0/G_N^{U,\varepsilon}) \, \rho_0^{N} dx - \int \ln(\rho^{N,U,\varepsilon}_T/G_N^{U,\varepsilon}) \, \rho_T^{N,U,\varepsilon} dx \, (\geq 0) \, .$$ A natural idea is to pass to the limit as $\varepsilon \to 0$ and expect first that limits exist, and second, that they include the marginals flow $\rho_.^{N,U}$.
\medskip

We shall follow the previous program by letting first $\varepsilon$ go to $0$. This is the analogue of subsection 4.2 in \cite{BJW}. We will however give a complete proof.

We shall first control the left hand side in \eqref{eqKSliouvpresque}. Recall that the normalizing constant disappears, so that the first term in this left hand side reduces to 
\begin{equation}\label{eqinitialKS0}
H(\rho_0^N dx|e^{-\tilde U} dx) + \int  \, \rho_0^N \, \frac{\chi}{2N} \, \sum_{i< j} \, \ln(\varepsilon+|x^i-x^j|^2) \, dx
\end{equation}
 the second one being similar replacing $\rho_0^N$ by $\rho_T^{N,U,\varepsilon}$. We will thus assume, first that $\tilde U$ is normalized so that $e^{-\tilde U}$ is a density of probability and in addition that
\begin{equation}\label{eqcontrolentropKSpartic0}
H(\rho_0^N dx|e^{-\tilde U} dx) < +\infty \,  \textrm{ and }
\int \, \rho_0^N \, \sum_{i\neq j} \, |\ln(|x^i-x^j|)| \, dx < +\infty \, ,  
\end{equation}
so that we may use Lebesgue's theorem and get that the sum of the two terms in \eqref{eqinitialKS0} goes to $\int \, \rho_0^N \, \ln(G_N^{U}) dx$ as $\varepsilon \to 0$.
\medskip

It remains to bound $$H(\rho_T^{N,U,\varepsilon} dx|e^{-\tilde U} dx) + \int  \, \rho_T^{N,U,\varepsilon} \, \frac{\chi}{2N} \, \sum_{i< j} \, \ln(\varepsilon+|x^i-x^j|^2) \, dx \, ,$$ from below. Since the first term is non-negative, it is enough to look at the negative part of the second one, i.e. replace the $\ln$ by $- \ln^-$ its negative part.

To simplify the argument we will assume that $\rho_0^N$ hence $\rho_T^{N,U,\varepsilon}$ is exchangeable, so that we have 
\begin{eqnarray}\label{eqcontrolentropKS}
H(\rho_T^{N,U,\varepsilon} dx|e^{- \tilde U}dx) &\leq&  H(\rho_0^{N} dx|e^{- \tilde U}dx)  + \int \rho_0^{N}(x) \,  \frac{\chi (N-1)}{4N} \,  \sum_{j>1} \left|\ln \left(\varepsilon+|x^1-x^j|^2\right)\right| dx \nonumber \\ && \quad + \, \int \rho_T^{N,U,\varepsilon}(x) \,  \frac{\chi (N-1)}{4N} \,  \sum_{j >1} \ln^- \left(\varepsilon+|x^1-x^j|^2\right) dx \, , 
\end{eqnarray}
\medskip

According to the variational formulation of relative entropy, and after a standard approximation by a bounded function, we have for any $\beta>0$
\begin{equation}\label{eqKSvariat}
\beta \, \int \rho_T^{N,U,\varepsilon}(x) \, \sum_{j>1} \ln^- \left(\varepsilon+|x^1-x^j|^2\right) dx \leq H(\rho_T^{N,U,\varepsilon} dx|e^{- \tilde U}dx)  \, + 
\end{equation}
$$ \quad + \ln \left(\int e^{\beta \sum_{j>1} \ln^-\left(\varepsilon+|x^1-x^j|^2\right)} \prod_{j>1} e^{-U(x^j)} \, e^{-U(x^1)}  dx^2 \, dx^N dx^1 \right).$$ The last term is equal to $$(N-1) \, \ln \left(\int \frac{1}{(\varepsilon+|x^1-x^2|^2)^\beta} \, \mathbf 1_{\varepsilon+|x^1-x^2|^2\leq 1} \, e^{-U(x^1)} e^{-U(x^2)} dx^1 \, dx^2\right) \leq (N-1) C(\beta) $$ where $C(\beta)$ does not depend on $\varepsilon$, provided $\beta<1$. 

We have obtained 
\begin{equation}\label{eqentropboundliouville}
\left(1 - \frac{\chi (N-1)}{4\beta N}\right) \, H(\rho_T^{N,U,\varepsilon} dx |e^{- \tilde U}dx) \leq C(\beta,\rho_0^{N}) \, N \, .
\end{equation}
 If $\chi<4$ we thus get a first desired result, namely 
\begin{equation}\label{eqcontrolentropKSpartic}
\sup_\varepsilon \, \sup_{0\leq t \leq T} H(\rho_t^{N,U,\varepsilon} dx |e^{- \tilde U} dx) \leq C(\chi,\rho_0^{N},T) \, N \, .
\end{equation}
 Plugging this result in \eqref{eqKSvariat} we also have $$\sup_\varepsilon \, \sup_{0\leq t \leq T} \, \int \rho_t^{N,U,\varepsilon}(x) \, \sum_{j>1} \ln^- \left(\varepsilon+|x^1-x^j|^2\right) dx \, \leq \, C'(\chi,\rho_0^{N},T) \, N \, .$$
\medskip

An immediate consequence of \eqref{eqcontrolentropKSpartic} is the following
\begin{lemma}\label{lemtightKSapprox}
If \eqref{eqcontrolentropKSpartic0} is satisfied and $\rho_0^N$ is exchangeable, for all $\chi<4$ the family of flows $$(t \in [0,T] \mapsto \rho_t^{N,U,\varepsilon})_{\varepsilon \in (0,1)}$$ is tight. In addition any weak limit $\rho_.^N$ satisfies $$\sup_{0\leq t \leq T} H(\rho_t^{N} dx |e^{- \tilde U} dx) \leq C(\chi,\rho_0^{N}) \, N \, ,$$ and $$\sup_{0\leq t \leq T} \, \int \rho_t^{N,U}(x) \, \sum_{j>1} \ln^- \left(|x^1-x^j|^2\right) dx \, \leq \, C'(\chi,\rho_0^{N},T) \, N \, .$$ In the previous results, constants of type $C(\chi,\rho_0^N,T)$ only depends on $\rho_0^N$ through the distribution of the first two coordinates $(x^1,x^2)$ (or any pair thanks to exchangeability), hence can be chosen independent of $N$ if $\rho_0^N$ is chaotic.
\end{lemma}
\begin{proof}
The first inequality is a consequence of the lower semi continuity of relative entropy. For the second one remark that for $\varepsilon \leq \varepsilon_0$, $\ln^- \left(\varepsilon+|x^1-x^j|^2\right) \geq \ln^- \left(\varepsilon_0+|x^1-x^j|^2\right)$ so that $$\int \rho_t^{N,U,\varepsilon}(x) \, \sum_{j>1} \ln^- \left(\varepsilon_0+|x^1-x^j|^2\right) dx \leq C'(\chi,\rho_0^{N},T) \, N \, .$$ One can thus make $\varepsilon$ go to $0$, using weak convergence, and then let $\varepsilon_0$ go to $0$ using monotone convergence theorem.
\end{proof}

\medskip

\begin{remark}\label{rementropks}
The assumption $\int \, |\ln(|x^i-x^j|)| \, \rho_0^{N} \, dx \, < +\infty$ is satisfied, according to the proof of Lemma \ref{lemFHM2}, as soon as $\tilde \rho_0^{i,j,N}:=\rho_0^{N}\circ (x^i,x^j)^{-1} \in \mathbb L^q(\mathbb R^2)$ for some $q>1$. Actually since $z \mapsto |\ln (|z|)|$ belongs to the Orlicz space $\mathbb L_\Phi(B(0,1))$, where $B(0,1)$ is the unit ball of $\mathbb R^2$ and $\Phi(u)=e^{|u|}-1$, Orlicz-H\"{o}lder inequality ensures that $\int \, |\ln(|x^i-x^j|)| \, \rho_0^{N} \, dx \, < +\infty$ as soon as $\int \, \rho_0^{N} \, |\ln(\rho_0^{N})| dx \, < +\infty$. 

Similarly, it is easy to see that, if $\nu$ and $h \nu$ are probability measures, $$\int h \, |\ln h| d\nu \leq \int h \ln h \, d\nu + 2e^{-1} \, .$$ It follows that for a density of probability $\rho$, and $U$ as before, assuming that $e^{- \tilde U}$ is normalized as a density of probability, $$\int \rho \, |\ln \rho| dx \leq \int \rho \, |\ln (\rho/e^{- \tilde U})| dx + \int |\tilde U| \rho dx \leq \int \rho \ln (\rho/e^{- \tilde U}) \, dx +  \, \int |\tilde U| \rho dx \, + \, 2e^{-1} \, .$$ It means that the second assumption on $\rho_0^{N}$ in \eqref{eqcontrolentropKSpartic0} is satisfied as soon as the first one is satisfied and $\tilde U \in \mathbb L^1(\rho dx)$.
\hfill $\diamondsuit$
\end{remark}
\medskip

Notice that $\rho_.^{N,U,\varepsilon}$ is smooth on $(0,T]\times (\mathbb R^d)^N$ thanks to ellipticity results (probabilists can call upon Malliavin calculus), since the coefficients of the generator $L^{N,U,\varepsilon}$ are smooth.
\medskip

Let $t \mapsto \rho_t$ be any weak limit of the previous family. Introduce now
\begin{equation}\label{eqloinbord}
A_l^N = \{x\in (\mathbb R^2)^{\otimes N} ; \, \min_{i \neq j} |x^i-x^j|>1/l\} \, .
\end{equation}
on this set. One has $(\partial_t +L^{N,U})\rho_.=0$ on any $[s,T]\times A_l^N$ for any $s>0$ in the sense of Schwartz distributions, hence $\rho_.$ is also smooth on this set for the same reason.

Actually we can say much more, namely that 
\begin{equation}\label{eqborneepsilon}
||\rho_t^{N,U,\varepsilon}||_{C_b^j(A_l^N\cap B(0,R))} \leq C(j,V,l,R) \, t^{-\alpha(j)}
\end{equation}
for any $j \in \mathbb N^*$ and some $\alpha(j)>0$ where $B(0,R)$ is the euclidean ball of radius $R$. What is important here is that the previous bound is uniform in $\varepsilon$, because all the coefficients and their derivatives are bounded on $A_l^N\cap B(0,R)$ uniformly in $\varepsilon$. A similar bound is true for $\rho_t^{N,U}$. The same is true for $\rho_t$. Once again, aficionados of Malliavin calculus will find a proof in \cite{Cathypo1} Theorem 1.15 and its proof, and the above statement in \cite{Cathypo2} Theorem 1.5.(i), in a more general hypoelliptic framework. In order to take into account the initial density (the results in \cite{Cathypo1,Cathypo2} are concerned with the density kernels, i.e. an initial Dirac measure) it is enough to differentiate under the integral sign.
\medskip

It is then standard to show that $\rho_t^{N,U,\varepsilon}$ and $\nabla \rho_t^{N,U,\varepsilon}$ weakly converge to $\rho_t$ and $\nabla \rho_t$ in $\mathbb L^2(A_l^N\cap B(0,R))$ so that one easily gets, for $s>0$,  both $$\int_s^T \, \int_{A_l^N\cap B(0,R)} \, \left(\left|\frac{\chi}{N} \, \sum_{j\neq i} \, \frac{x^i -x^j}{\varepsilon +|x^i-x^j|^2} \right|^2\right) \, \rho^{N,U,\varepsilon}_t(x) \, dx \, dt \to$$ $$ \quad \to \,  \int_s^T \, \int_{A_l^N\cap B(0,R)} \, \left(\left|\frac{\chi}{N} \, \sum_{j\neq i} \, \frac{x^i -x^j}{|x^i-x^j|^2} \right|^2\right) \, \rho^{N,U}_t(x) \, dx \, dt \, ,$$
and 
$$\int_s^T \, \int_{A_l^N\cap B(0,R)} \, \left\langle \nabla_i \rho^{N,U,\varepsilon}_t , \frac{\chi}{N} \, \sum_{j\neq i} \, \frac{x^i -x^j}{\varepsilon +|x^i-x^j|^2} \right\rangle  \, dx \, dt  \to $$ $$ \quad \to \int_s^T \, \int_{A_l^N\cap B(0,R)} \, \left\langle \nabla_i \rho^{N,U}_t , \frac{\chi}{N} \, \sum_{j\neq i} \, \frac{x^i -x^j}{|x^i-x^j|^2} \right\rangle \,  dx \, dt \, ,$$ as $\varepsilon \to 0$. The same holds with the terms involving $\tilde U$ if $\tilde U$ is smooth.

In addition, at least for $l$ and $R$ large enough,
$$\sup_\varepsilon \, \int_s^T \, \int_{A_l^N\cap B(0,R)} \, |\nabla \ln(\rho_t^{N,U,\varepsilon})|^2 \, \rho_t^{N,U,\varepsilon} dx dt < +\infty \, .$$ This bound follows from the fact that 
\begin{equation}\label{eqlowereps}
\inf_\varepsilon \, \inf_{s \leq t \leq T, y\in A_l^N\cap B(0,R)} \, \rho_t^{N,U,\varepsilon} \, = \, c(l,s,T,R,U,N) \, > 0 \, .
\end{equation}
 Here is a proof with a probabilistic flavour. First, since the initial measure has a density, one may find $l_0>0$ and $R_0$ such that $\rho_0^N(A_{l_0}^N\cap B(0,R_0)) > 0$. If $l>2l_0$ and $R>2R_0$ we may consider the processes killed when they exit $A_{2l_0}^N\cap B(0,2R_0)$. Since their laws are equivalent to the one of a similarly killed Brownian motion (with variance $2t$), with densities and inverse densities bounded in all the $\mathbb L^p$'s uniformly in $\varepsilon$, all these laws are equivalent with densities and inverse densities bounded in all the $\mathbb L^p$'s uniformly in $\varepsilon$. The uniform in $\varepsilon$ lower bound \eqref{eqlowereps} follows from a similar uniform lower bound for the densities of the killed processes, which itself follows from the previous argument once a strictly positive lower bound is obtained for a given $\varepsilon_0$. The latter is standard see e.g \cite{Cathypo4} p.612-613. 

Consider random variables $Z_t^{N,U,\varepsilon}$ supported by $A_l^N\cap B(0,R)$ with densities proportional to $\rho_t^{N,U,\varepsilon}$ in this set. It is easily seen that they converge in distribution to a random variable $Z_t^{N,U}$ with density proportional to $\rho_t$ and that the normalizing constants in $\varepsilon$ also converge to the normalizing constant for the limiting distribution. According to Proposition 13.2 in \cite{bobkov-fisher} again and to Fatou's lemma, we deduce that $$\liminf_\varepsilon \int_s^T \, \int_{A_l^N\cap B(0,R)} \, |\nabla \ln(\rho_t^{N,U,\varepsilon})|^2 \, \rho_t^{N,U,\varepsilon} dx dt  \geq $$ $$ \quad \quad \quad \geq \int_s^T \, \int_{A_l^N\cap B(0,R)} \, |\nabla \ln(\rho_t)|^2 \, \rho_t dx dt \, .$$ We have thus shown
\begin{equation}\label{eqconvksepsfree}
\liminf_\varepsilon \int_s^T \, \int_{A_l^N\cap B(0,R)} \, \left(\sum_i \, \left|\nabla_i \ln \rho^{N,U,\varepsilon}_t + \nabla \tilde U +\frac{\chi}{N} \, \sum_{j\neq i} \, \frac{x^i -x^j}{\varepsilon +|x^i-x^j|^2} \right|^2\right) \, \rho^{N,U,\varepsilon}_t(x) \, dx \, dt \geq 
\end{equation}
$$ \quad \geq \int_s^T \, \int_{A_l^N\cap B(0,R)} \, \left(\sum_i \, \left|\nabla_i \ln \rho_t + \nabla \tilde U +\frac{\chi}{N} \, \sum_{j\neq i} \, \frac{x^i -x^j}{|x^i-x^j|^2} \right|^2\right) \, \rho_t(x) \, dx \, dt \, ,$$ and finally $$\liminf_\varepsilon \int_0^T \, \int \, \left(\sum_i \left|\nabla_i \ln \rho^{N,U,\varepsilon}_t + \nabla \tilde U +\frac{\chi}{N} \, \sum_{j\neq i} \, \frac{x^i -x^j}{\varepsilon +|x^i-x^j|^2} \right|^2\right) \, \rho^{N,U,\varepsilon}_t(x) \, dx \, dt \geq $$ $$\liminf_\varepsilon \int_s^T \, \int_ {A_l^N\cap B(0,R)}\, \left(\sum_i \, \left|\nabla_i \ln \rho^{N,U,\varepsilon}_t + \nabla \tilde U +\frac{\chi}{N} \, \sum_{j \neq i} \, \frac{x^i -x^j}{\varepsilon +|x^i-x^j|^2} \right|^2\right) \, \rho^{N,U,\varepsilon}_t(x) \, dx \, dt \geq $$ $$\quad \geq \int_s^T \, \int_{A_l^N\cap B(0,R)} \, \left(\sum_i \, \left|\nabla_i \ln \rho_t + \nabla \tilde U +\frac{\chi}{N} \, \sum_{j \neq i} \, \frac{x^i -x^j}{|x^i-x^j|^2} \right|^2\right) \, \rho_t(x) \, dx \, dt \, .$$ 
It remains to use the increasing limit as $l$ and $R$ go to infinity, since $A_l^N\cap B(0,R)$ grows to the whole space, and then as $s$ goes to $0$. We already know the existence of the density $\rho_t$ which is almost everywhere defined as well as $\nabla \rho_t$. We have thus obtained
\begin{lemma}\label{lemBJWentrop}
Assume that \eqref{eqcontrolentropKSpartic0} is satisfied and that $\rho_0^N$ is exchangeable. Then for all $\chi<4$ there exists $C(N,\chi,\rho_0^N)$ such that any weak limit $\rho$ of the tight sequence $\rho^{N,U,\varepsilon}$ satisfies 
$$H(\rho_T \,  dx|e^{-\tilde U} dx) +   \int_0^T \, \int \, \left(\sum_i \, \left|\nabla_i \ln \rho_t + \nabla \tilde U \, + \frac{\chi}{N} \, \sum_{j\neq i} \, \frac{x^i -x^j}{|x^i-x^j|^2} \right|^2\right) \, \rho_t(x) \, dx \, dt \, \leq $$
$$\quad \leq C(N,\chi,\rho_0^N) \, \left(H(\rho_0^N \, dx |e^{- \tilde U} dx) +  \, \int \rho_0^N \, \frac{\chi}{N} \, \sum_{i<j} \, |\ln( |x^i-x^j|) dx\right) \, .$$ 
\end{lemma}
\begin{remark}\label{remBJWbound}
Notice that in the previous result the constant $C(N,\chi,\rho_0^N)$ goes to infinity as $\chi \to 4$ (recall \eqref{eqentropboundliouville}), while in \cite{BJW} (2.5), hence in Proposition 4.1 (choosing $\sigma=1$ therein, hence $\lambda=\chi<4$) this constant  actually equals $1$.  
\hfill $\diamondsuit$
\end{remark}

To complete the picture we have to prove
\begin{lemma}\label{lemexistliouvks}
Assume that \eqref{eqcontrolentropKSpartic0} is satisfied and that $\rho_0^N$ is exchangeable. Then for all $\chi<4$ any weak limit $\rho$ of the tight sequence $\rho^{N,U,\varepsilon}$ solves the Liouville equation, provided $$\sup_\varepsilon \, \sup_t \, \int |\tilde U| \rho_t^{N,U,\varepsilon} \, dx \, < \, + \infty \, .$$
\end{lemma}
The proof is similar to the one in \cite{BJW}, the key being (4.9) therein. Let us briefly recall the argument. First, as we already used,  $\rho$ solves the Liouville equation in the set of Schwartz distributions on the open complement of the collision set, and is smooth in this set, i.e. all its derivatives are almost everywhere well defined. It is thus enough to consider $B_l^N=(A_l^N)^c \cap B(0,R)$ for some $R>0$, and to control $$M= \int_u^s \, \int_{B_l^N} \, \rho_t^{N,U,\varepsilon} \, |\nabla \ln(G_N^{U,\varepsilon}/\rho_t^{N,U,\varepsilon})| \, dx \, dt \, ,$$ for all $0<u<s\leq T$. Applying Cauchy-Schwarz inequality we get $$M \leq \left(\int_u^s \, \int |\nabla \ln(G_N^{U,\varepsilon}/\rho_t^{N,U,\varepsilon})|^2 \, \rho_t^{N,U,\varepsilon} \, dx \, dt\right)^{\frac 12} \; \left(\int_u^s \, \int \mathbf 1_{B_l^N}(x) \, \rho_t^{N,U,\varepsilon} \, dx \, dt\right)^{\frac 12} \, .$$ The first term in the product is uniformly (w.r.t. $\varepsilon$) bounded according to Lemma \ref{lemBJWentrop}. For the second one we again use Orlicz-H\"{o}lder inequality with $\Phi(u)=e^{u}-1$ and its conjugate $\Phi^*(u)=(u\ln(u) -u +1)\mathbf 1_{u\geq 1}$. On one hand the $\mathbb L_\Phi$ norm of $\mathbf 1_{B_l^N}$ is less than $c \, \ln^{-1}(1/vol(B_l^N))$ which goes to $0$ as $l \to + \infty$ for a fixed $R$ (it is of order $1/\ln l$ as mentioned in \cite{BJW}). On the other hand the $\mathbb L_{\Phi^*}$ norm of $\rho_t^{N,U,\varepsilon}$ is less than $c \, \int \, \rho_t^{N,U,\varepsilon} \, |\ln(\rho_t^{N,U,\varepsilon})| \, dx$ the latter being less than $$c(H(\rho_t^{N,U,\varepsilon}dx|e^{- \tilde U}dx) + \int |\tilde U| \rho_t^{N,U,\varepsilon} dx + 2e^{-1})$$ according to remark \ref{rementropks}. We deduce that for all fixed $R$, the second term in the product goes to $0$ as $l \to +\infty$, uniformly w.r.t. $\varepsilon$.  
\medskip

The three previous Lemmata furnish the analogue of Proposition 4.1 in \cite{BJW}.
\medskip

The point now is to know whether for (almost) all $t$, $\rho_t^{N,U,\varepsilon} \to \rho_t^{N,U,}$ as $\varepsilon \to 0$, or at least if the previous holds for some subsequence $\varepsilon_n$, the same for almost all $t$.
\medskip

A first partial answer is given by the following Lemma
\begin{lemma}\label{lemKSFJ}
Let $X_.^N$ be the particle system of the Keller-Segel model introduced in Theorem \ref{thmexistN} and Theorem \ref{thmexistN} with or without an additional the confinement potential (in this case recall that the confinement potential is at most quadratic). For all $0 \leq \varepsilon \leq 1$ ($1$ is arbitrary and we emphasize that $\varepsilon=0$ is allowed), it holds
\begin{enumerate}
\item[(1)] \quad For all $t>0$, $\sum_{j=1}^N \,  \mathbb E(|X_t^{j,N,\varepsilon}|^2) \, \leq \,  c(T,a) (\sum_{j=1}^N \, \mathbb E(|X_0^{j,N}|^2) \, + \, NT) \, ,$ where $a$ is the Lipschitz constant of $\nabla U$.
\item[(2)] \quad For all $0<\gamma< 2 - (\chi/2)$, $$ \int_0^T \, \sum_{i \neq j} \, \mathbb E\left(\frac{1}{|X_t^{j,N,\varepsilon}-X_t^{i,N,\varepsilon}|^{\gamma}}\right) \, dt \, \leq C N \; (\sum_{j=1}^N \, \mathbb E(|X_0^{j,N}|^2)+N) \, ,$$ where $C$ is a constant that only depends on  $T$, $a$ (the Lipschitz constant of $\nabla  U$), $\gamma$ and $\chi$.
\end{enumerate}
\end{lemma}
This Lemma is mainly contained in \cite{Tar} Lemma 16 and Proposition 10 in the case $U=0$. The proofs in \cite{Tar} are written under the assumption that the initial distribution is exchangeable and only concern the process $X_.^N$, not the approximations. We shall give here a (slightly simplified but very similar) complete proof taking care of the dependence in $\varepsilon$. The Lemma has an important consequence
\begin{corollary}\label{corapproxKSstoch}
For $\chi < 2$, $Q^{N,U,\varepsilon}$ weakly converges to $Q^{N,U}$ as $\varepsilon \to 0$. 
\end{corollary}  
The Corollary extends Theorem 6 in \cite{FJ} obtained for $\chi <1$. 

A standard proof for $\chi<2$, is that since (2) holds for $\gamma >1$ the family $Q^{N,U,\varepsilon}$ is tight. A proof of this fact will be given in subsection \ref{subsectight} in the framework of large $N$ instead of small $\varepsilon$. One can also directly follow the proof of Theorem 6 in \cite{FJ} (a small improvement is contained in our derivation in subsection \ref{subsectight}). It is then standard to show (as in \cite{FJ}) that any weak limit solves the martingale problem with $\varepsilon=0$. 
\medskip

We recall (and claim) that the methodology of \cite{FJ} precisely consists in showing that $Q^{N,U,\varepsilon}$ weakly converges to $Q^{N,U}$ on $[0,S_l]$ where $S_l$ denotes the first time where the minimal sum of the distances between three particles becomes less than $1/l$ for large $l$ i.e.
\begin{equation}\label{eqdnl}
D_l^N = \{x\in (\mathbb R^2)^{\otimes N} ; \, \textrm{for all distinct indices $i,j,k$} \; , \; |x^i-x^j|+|x^j-x^k|+|x^k-x^i|>1/l\} \, ,
\end{equation}
using the same approximation in $\varepsilon$ we are using. This is explained in Lemmata 12,13,14 p.
2827-2830 of \cite{FJ} and in the proof of their Theorem 7, where it is shown that any limit law solves the particle system SDE. Since we know that there is only one weak solution, the claim is proven. It is not difficult to see that the same is true for the confined system.
\bigskip

A consequence of Corollary \ref{corapproxKSstoch} is of course that, as expected $\rho_t^{N,U,\varepsilon}$ weakly converges to $\rho_t^{N,U}$ as $\varepsilon \to 0$. 

Another proof of the latter, directly using Lemma \ref{lemKSFJ} (and not the Corollary) is using the \textit{superposition principle} recalled in Remark \ref{remtrevisan}. Indeed, each limiting $\rho_t$ solves the Liouville equation ((1) in the Lemma ensures the integrability of $\tilde U$) according to what precedes. (2) for $\chi<2$ ensures that the drift is integrable w.r.t. the flow $\rho_.$. Thus, there exists a solution $\tilde Q$ to the particle system, with initial distribution $\rho_0^{N}$ and marginals  flow $\rho_.$. Since there exists only one (weak) solution to the particle system, $\tilde Q$ coincides with this solution and its marginals flow with $\rho_t^{N,U}$.
\smallskip

We turn to the proof of Lemma \ref{lemKSFJ}.
\begin{proof}{\textit{of lemma \ref{lemKSFJ}.}}
Proof of (1). Using Ito's formula and the usual care (localizing first with stopping times),  one obtains
\begin{eqnarray*}
\sum_{j} \, |X_t^{j,N,\varepsilon}|^2 &=& \sum_{j} \, |X_0^{j,N}|^2 \, + \, M_t \, + \, 4Nt \, - \sum_j \, \int_0^t \, 2 \langle X_s^{j,N,\varepsilon},\nabla U(X_s^{j,N,\varepsilon})\rangle \, ds \\ & & - \, \frac{2\chi}{N} \sum_{i\neq j} \, \int_0^t \, \frac{\langle X_s^{j,N,\varepsilon},X_s^{j,N,\varepsilon}-X_s^{i,N,\varepsilon}\rangle}{\varepsilon + |X_s^{j,N,\varepsilon}-X_s^{i,N,\varepsilon}|^2} \, ds \, ,
\end{eqnarray*}
where $M_.$ is a local martingale. Exchanging again $i$ and $j$ in the final term, we see that it is equal to $$\frac 12 \, \sum_{i\neq j} \, \frac{\langle X_s^{j,N,\varepsilon}-X_s^{i,N,\varepsilon},X_s^{j,N,\varepsilon}-X_s^{i,N,\varepsilon}\rangle}{\varepsilon + |X_s^{j,N,\varepsilon}-X_s^{i,N,\varepsilon}|^2} \leq  N(N-1)/2$$ and in addition is non-negative, so that finally, using again localization for some stopping time in order to control the martingale term 
\begin{equation}\label{eqsquareKS}
\sum_{j} \, \mathbb E(|X_t^{j,N,\varepsilon}|^2) \leq \sum_{j} \, \mathbb E(|X_0^{j,N}|^2) + CNT + a  \, \int_0^t \, \sum_{j} \, \mathbb E(|X_s^{j,N,\varepsilon}|^2) \, ds. 
\end{equation}
Here we have used $|\langle x,\nabla \tilde U(x)\rangle| \leq c + a |x|^2$ for some constant $c$. We may conclude thanks to Gronwall's lemma. When $U$ is convex, $\langle \nabla \tilde U(y),y\rangle \geq 0$ and the bound does not depend on $U$.
\medskip

For the proof of (2) consider some function $g: \mathbb R^+ \to \mathbb R$ and $G(x)=\sum_{i\neq j} \, g(|x^i-x^j|^2)$. We assume that $g$ is everywhere defined and smooth. Applying Ito formula we thus have $$G(X_t^{N,\varepsilon})=G(X_0^{N,\varepsilon}) + M_t + S_t^1 + S_t^2 + S_t^3$$ where $M_.$ is a martingale the three other terms being of the form $S_t^j=\int_0^t \, R_s^j \, ds$ with
\begin{eqnarray*}
R_s^1 &=& - 2 \, \sum_{i\neq j} \, g'(|X_s^{i,N,\varepsilon}-X_s^{j,N,\varepsilon}|^2) \, \langle X_s^{i,N,\varepsilon}-X_s^{j,N,\varepsilon},\nabla  U(X_s^{i,N,\varepsilon}) - \nabla   U(X_s^{j,N,\varepsilon})\rangle \\ R_s^2 &=& - \, \frac{2 \chi}{N} \, \sum_{i\neq j} \,  g'(|X_s^{i,N,\varepsilon}-X_s^{j,N,\varepsilon}|^2)\, \left(\sum_k \, (A_{i,k}-A_{j,k})\right) \\ & & \textrm{where } \quad A_{i,k} = \frac{\langle X_s^{i,N,\varepsilon}-X_s^{j,N,\varepsilon},X_s^{i,N,\varepsilon}-X_s^{k,N,\varepsilon}\rangle}{\varepsilon+ |X_s^{i,N,\varepsilon}-X_s^{k,N,\varepsilon}|^2}\\ R_s^3 &=& 8 \, \sum_{i\neq j} \, (g'(|X_s^{i,N,\varepsilon}-X_s^{j,N,\varepsilon}|^2) + |X_s^{i,N,\varepsilon}-X_s^{j,N,\varepsilon}|^2 \, g''(|X_s^{i,N,\varepsilon}-X_s^{j,N,\varepsilon}|^2)) \, .
\end{eqnarray*}
For the last term recall that the martingale part of $X^i-X^j$ is $2$ times a Brownian motion so that there is a $2 \Delta$ yielding the $8$.

The most tricky term is $R^2$. In order to get some control we will assume that $$\textrm{ $g'$ is non negative and non increasing .}$$ Denote $$u=X_s^{i,N,\varepsilon}-X_s^{j,N,\varepsilon} \, , \, v=X_s^{j,N,\varepsilon}-X_s^{k,N,\varepsilon} \, , \, w=X_s^{k,N,\varepsilon}-X_s^{i,N,\varepsilon} \, .$$ Each term of the sum is written (up to a circular permutation on $u,v,w$) as $$ - \, g'(|u|^2) \, \left\langle u,\left(\frac{w}{\varepsilon +|w|^2} \, + \, \frac{v}{\varepsilon +|v|^2}\right)\right \rangle$$ so that exchanging the role of the indices one gets 
$$R_s^2 =  \frac{2 \chi}{3 N} \sum_{i,j,k} $$
\begin{eqnarray*}
&& \quad g'(|u|^2) \, \left\langle u,\left(\frac{w}{\varepsilon +|w|^2} \, + \, \frac{v}{\varepsilon +|v|^2}\right)\right \rangle\\ && \quad g'(|v|^2) \, \left\langle v,\left(\frac{w}{\varepsilon +|w|^2} \, + \, \frac{u}{\varepsilon +|u|^2}\right)\right \rangle\\ && \quad g'(|w|^2) \, \left\langle w,\left(\frac{u}{\varepsilon +|u|^2} \, + \, \frac{v}{\varepsilon +|v|^2}\right)\right \rangle \, .
\end{eqnarray*}
Since $g'$ is non increasing on $\mathbb R^+$, so is $u \mapsto g'(u^2)=\varphi(u)$. Thus we may apply Lemma 9 in \cite{Tar} with $(\varphi,\psi)$, where  $\psi(u) = (\varepsilon +u^2)^{-1}$. The sum of the three terms above is larger than $$
- \, g'(|u|^2) \, \frac{|u|^2}{\varepsilon +|u|^2} \, - \, g'(|v|^2) \, \frac{|v|^2}{\varepsilon +|v|^2} \, -  \, g'(|w|^2) \, \frac{|w|^2}{\varepsilon +|w|^2} \, .
$$ Summing up over $(i,j,k)$ (recall that in the sum we only consider the cases where the indices are different), we obtain 
\begin{eqnarray}\label{eqst2}
R_s^2 &\geq& - \, \frac{2 \chi \, (N-2)}{N} \,  \sum_{i\neq j} \,  g'(|X_s^{i,N,\varepsilon}-X_s^{j,N,\varepsilon}|^2) \, \frac{|X_s^{i,N,\varepsilon}-X_s^{j,N,\varepsilon}|^2}{(\varepsilon + |X_s^{i,N,\varepsilon}-X_s^{j,N,\varepsilon}|^2)}   \nonumber \\ &\geq& - \, 2  \chi \, \sum_{i\neq j} \,  g'(|X_s^{i,N,\varepsilon}-X_s^{j,N,\varepsilon}|^2)  \, .
\end{eqnarray}
Summing up, and using that $\nabla U$ is $a$ Lipschitz, we obtain for $$u_{i,j}=|X_s^{i,N,\varepsilon}-X_s^{j,N,\varepsilon}|^2 \, ,$$
\begin{equation}\label{eqst3}
R_s^1+R_s^2+R_s^3 \geq \sum_{i\neq j}  \left[(8-2\chi-2a\,  u_{i,j}) \, g'(u_{i,j}) \, +  \, 8 \, u_{i,j} \, g''(u_{i,j})\right]  \, .
\end{equation}
For $0<\alpha<1$, $\eta>0$ and $r\geq 0$, choose $g(r)=(\eta + r)^{\alpha}$. We have $$g'(r)=\alpha (\eta+r)^{\alpha-1} \quad \textrm{and} \quad g''(r)=\alpha (\alpha -1) \, (\eta+r)^{\alpha-2} \, .$$ 
We thus have $$R_s^1+R_s^2+R_s^3 \geq \sum_{i\neq j} \alpha(8\alpha-2\chi - 2a u_{i,j})(\eta + u_{i,j})^{\alpha-1}$$ so that taking expectations we obtain $$\mathbb E\left(\sum_{i \neq j} (\eta+|X_T^{i,N,\varepsilon}-X_T^{j,N,\varepsilon}|^2)^\alpha\right) \geq $$ $$ \quad \quad \geq \, \int_0^T \mathbb E\left(\sum_{i\neq j} \alpha(8\alpha-2\chi - 2a |X_s^{i,N,\varepsilon}-X_s^{j,N,\varepsilon}|^2 )(\eta + |X_s^{i,N,\varepsilon}-X_s^{j,N,\varepsilon}|^2 )^{\alpha-1}\right) \, ds \, .$$ The left hand side is bounded w.r.t $\varepsilon$ according to the first part of the lemma. 

If $\alpha >\chi/4$ and $2au \leq 4\alpha -  \chi$ one may bound the right hand side from below by $$
\int_0^T \mathbb E\left(\sum_{i\neq j} \alpha(4\alpha-\chi)(\eta + |X_s^{i,N,\varepsilon}-X_s^{j,N,\varepsilon}|^2 )^{\alpha-1}\right) \, \mathbf 1_{|X_s^{i,N,\varepsilon}-X_s^{j,N,\varepsilon}|^2 \leq u} \, ds \, .$$ We deduce that $$\int_0^T \mathbb E\left(\sum_{i\neq j} (\eta + |X_s^{i,N,\varepsilon}-X_s^{j,N,\varepsilon}|^2 )^{\alpha-1}\right) \, \mathbf 1_{|X_s^{i,N,\varepsilon}-X_s^{j,N,\varepsilon}|^2 \leq u} \, ds$$ is bounded uniformly w.r.t. $\varepsilon$ and we can pass to the limit $\eta \to 0$ using monotone convergence. For $|X_s^{i,N,\varepsilon}-X_s^{j,N,\varepsilon}|^2 \geq u$ the term under the integral sign is bounded, so that we get (2) in the Lemma.

The dependence in $N$ is easy to trace.
\end{proof}
\medskip

If Corollary \ref{corapproxKSstoch} furnishes the expected answer for $\chi<2$, for $\chi \geq 2$ we do not know about the weak convergence of the distributions $Q^{N,U,\varepsilon}$. We shall however see that a weaker convergence holds true, implying the desired convergence of the marginals flow.

To this end we will again call upon Dirichlet forms theory, and more precisely upon convergence for Dirichlet forms. Indeed, each $Q^{N,U,\varepsilon}$ for $\varepsilon \geq 0$ is associated to a Dirichlet form $$\mathcal E_\varepsilon (f,g) = \int_{M^N} \, \langle \nabla f,\nabla g\rangle \, dG_N^{U,\varepsilon}$$ defined on $\mathbb L^2(M^N,G_N^{U,\varepsilon}\, dx)$ with domain $D(\mathcal E_\varepsilon)$ and core $C_c^\infty(M^N)$ where $M^N$ is introduced in Theorem \ref{thmexistN}. Notice that on $M^N$, $G_N^U \in \mathbb L^p_{loc}(M^N)$ for $p<2N/\chi$.

For $\varepsilon>0$ the reader can be surprised. Indeed the natural Dirichlet form is defined on the whole $(\mathbb R^2)^N$. Actually the symmetric forms we have defined on $C_c^\infty(M^N)$ are closable, and their smallest closed extension are regular and local. They are thus associated to some Hunt process coinciding with $X_.^{N,U,\varepsilon}$ up to the exit time of $M^N$. This exit time is almost surely infinite for the standard Brownian motion ($0$ is polar for a $2$ dimensional Brownian motion), and since $Q^{N,U,\varepsilon}$ is absolutely continuous w.r.t. the Wiener measure, this time is also almost surely infinite for $X_.^{N,U,\varepsilon}$. In other words the Dirichlet forms $(\mathcal E_\varepsilon,D(\mathcal E_\varepsilon))$ we have defined as closure of the forms defined on their core, coincide with the ones built on the whole space. For $\varepsilon=0$ we recall that the Dirichlet form was introduced in \cite{CP} subsection 3.1 and studied in \cite{FT-coll} section 11.

We will show that this family of Dirichlet forms is convergent in Mosco sense (\cite{Mosco}). Actually, since we are dealing with varying Hilbert spaces $H_\varepsilon:=\mathbb L^2(M^N,G_N^{U,\varepsilon}\, dx)$, we have to use an extension of Mosco's results to this setting. This extension is done in \cite{KuwaeS} (also see \cite{lobus13,lobus15}). In what follows, limits w.r.t. $\varepsilon$ are limits along some sequence $\varepsilon_n$ going to $0$, and as usual $\mathcal E(f)$ denotes $\mathcal E(f,f)$.

First of all, notice that $\mathcal C=C_c^\infty(M^N)$ is a dense subset of all the $H_\varepsilon$. In addition, the sequence of Hilbert spaces $H_\varepsilon$ is converging to $H_0$ in the sense of \cite{KuwaeS} p.611, since if $f \in H_0$, $f$ belongs to all $H_\varepsilon$ and $\lim_{\varepsilon \to 0} ||f||_{H_\varepsilon} = ||f||_{H_0}$ using monotone convergence. Indeed $G_N^{U,\varepsilon}$ increasingly converges to $G_N^U$.

Next we may define the convergence of sequences, adapting the ones in \cite{KuwaeS} to our situation
\begin{definition}\label{defkuwae}
We say that
\begin{itemize}
\item[(1)] \quad  $f_\varepsilon \in H_\varepsilon$ strongly converges to $f \in H_0$, if there exists a sequence $\tilde f_\varepsilon \in \mathcal C$ such that $$\lim_ \varepsilon \, ||\tilde f_\varepsilon - f||_{H_0} = 0 \quad \textrm{ and } \quad \lim_ \varepsilon ||\tilde f_\varepsilon - f_\varepsilon||_{H_\varepsilon} = 0 \, ,$$
\item[(2)] \quad  $f_\varepsilon \in H_\varepsilon$ weakly converges to $f \in H_0$, if $$\lim_\varepsilon \, \langle f_\varepsilon,g_\varepsilon\rangle_{H_\varepsilon} = \langle f,g\rangle_{H_0}$$ for any sequence $g_\varepsilon \in H_\varepsilon$ strongly convergent to $g \in H_0$.
\end{itemize}
\end{definition}
Recall that if $f_\varepsilon$ strongly converges to $f$, $||f_\varepsilon||_{H_\varepsilon} \to ||f||_{H_0}$ in particular this sequence is bounded and that if $||f_\varepsilon - \tilde f_\varepsilon||_{H_\varepsilon}  \to \, 0$, the sequence $\tilde f_\varepsilon$ also strongly converges to $f$.

If $f_\varepsilon$ weakly converges to $f$, $||f_\varepsilon||_{H_\varepsilon} $ is bounded and $\liminf_\varepsilon ||f_\varepsilon||_{H_\varepsilon} \geq ||f||_{H_0}$. If the latter $\liminf$ is a limit and the inequality is an equality, $f_\varepsilon$ strongly converges to $f$. See \cite{KuwaeS} Lemmata 2.1 and 2.3.
\medskip

We can then define the $\Gamma$-convergence and the Mosco convergence (or strong $\Gamma$-convergence) following definitions 2.8 and 2.11 in \cite{KuwaeS} where by convention $\mathcal E(f)=+\infty$ if $f \notin D(\mathcal E)$,
\begin{definition}\label{defmosco}
The sequence $\mathcal E_\varepsilon$ Mosco converges to $\mathcal E$ if
\begin{itemize}
\item[(M1)] \quad for any sequence $f_\varepsilon \in H_\varepsilon$ weakly converging to $f \in H_0$, it holds $$\mathcal E(f) \, \leq \, \liminf_\varepsilon \; \mathcal E_\varepsilon(f_\varepsilon) \, ,$$
\item[(M2)] \quad for any $f \in D(\mathcal E)$ there exists a sequence $f_\varepsilon \in H_\varepsilon$ such that $f_\varepsilon$ strongly converges to $f$ and $$\mathcal E(f) = \lim_\varepsilon \, \mathcal E_\varepsilon(f_\varepsilon) \, .$$
\end{itemize}
The sequence $\mathcal E_\varepsilon$ $\Gamma$-converges if (M2) is satisfied and (M1') is satisfied where (M1') is similar to (M1) simply replacing the weak convergence by the strong convergence.
\end{definition}
\medskip

$\Gamma$-convergence is thus weaker than Mosco convergence.
\medskip

\begin{proposition}\label{propmosco}
The sequence $\mathcal E_\varepsilon$ Mosco converges to $\mathcal E_0$.
\end{proposition}
\begin{proof}
(M2) is almost immediate. If $f \in D(\mathcal E_0)$, $\nabla f$ is dx almost everywhere defined and $\mathcal E_\varepsilon(f) \leq \mathcal E_0(f)$. We may thus choose $f_\varepsilon=f$, and use again the monotone convergence theorem.

We now study (M1). If $f_\varepsilon \in D(\mathcal E_\varepsilon)$ one can find $\tilde f_\varepsilon \in \mathcal C$ such that $\mathcal E_\varepsilon(\tilde f_\varepsilon - f_\varepsilon) + ||\tilde f_\varepsilon - f_\varepsilon||^2_{H_\varepsilon} \leq \varepsilon^2$. If $f_\varepsilon$ weakly converges to $f \in H_0$, so does $\tilde f_\varepsilon$ and $\liminf_\varepsilon \; \mathcal E_\varepsilon(f_\varepsilon)=\liminf_\varepsilon \; \mathcal E_\varepsilon(\tilde f_\varepsilon)$. We may thus assume that $f_\varepsilon \in \mathcal C$.

Next fix some $\eta>0$. For $\varepsilon \leq \eta$, $f_\varepsilon \in H_\eta$ as well as $f \in H_\eta$. Let $g \in H_0 \; (\subset H_\eta)$. We claim that $$\lim_\varepsilon \, \langle f_\varepsilon,g\rangle_{H_\eta} = \langle f,g\rangle_{H_\eta} \, .$$ Indeed $$\langle f_\varepsilon,g\rangle_{H_\eta}=\left\langle f_\varepsilon,g \, \frac{G_N^{U,\eta}}{G_N^{U,\varepsilon}}\right\rangle_{H_\varepsilon} \, \to \, \left \langle f,g \, \frac{G_N^{U,\eta}}{G_N^U}\right\rangle_{H_0} \, = \, \langle f,g\rangle_{H_\eta}$$ because $g \, \frac{G_N^{U,\eta}}{G_N^{U,\varepsilon}}$ strongly converges to $g \, \frac{G_N^{U,\eta}}{G_N^{U}}$. To see the latter simply remark that $\frac{G_N^{U,\eta}}{G_N^{U,\varepsilon}}$ is uniformly (w.r.t $\varepsilon$) bounded and converges almost surely to (the also bounded) $\frac{G_N^{U,\eta}}{G_N^{U}}$.
\medskip

Let $g \in (\mathcal C^2)^N$, we have 
\begin{eqnarray*}
\mathcal E_\varepsilon(f_\varepsilon) &\geq & \mathcal E_\eta(f_\varepsilon) \, \geq \,  \frac{1}{||g||_{H_\eta}} \, \langle \nabla f_\varepsilon, g\rangle_{H_\eta} \, = \,  - \, \frac{1}{||g||_{H_\eta}} \, \langle  f_\varepsilon, \nabla g\rangle_{H_\eta} \, - \, \frac{1}{||g||_{H_\eta}} \, \left \langle f_\varepsilon, g \nabla \ln(G_N^{U,\eta})\right \rangle_{H_\eta}  .
\end{eqnarray*}
Since $\nabla \ln(G_N^{U,\eta})$ is smooth and bounded, we may take the limit w.r.t. $\varepsilon$ in the right hand side and the $\liminf$ in the left hand side, so that $$\liminf_\varepsilon \mathcal E_\varepsilon(f_\varepsilon) \geq - \, \frac{1}{||g||_{H_\eta}} \, \langle  f, \nabla g\rangle_{H_\eta} \, - \, \frac{1}{||g||_{H_\eta}} \, \left \langle f, g \nabla \ln(G_N^{U,\eta})\right \rangle_{H_\eta} \, = \, \frac{1}{||g||_{H_\eta}} \, \langle \nabla f, g\rangle_{H_\eta} \, .$$
Taking the supremum over all $g$ in the right hand side we thus have $$\liminf_\varepsilon \mathcal E_\varepsilon(f_\varepsilon) \geq  \mathcal E_\eta(f)$$ for all $\eta$ and the desired result follows by taking the (increasing) limit w.r.t. $\eta$.
\end{proof}
\medskip

According to Theorem 2.4 in \cite{KuwaeS} we deduce from the previous Proposition that the semi-group $P_.^{N,U,\varepsilon}$ associated to $\mathcal E_\varepsilon$ strongly converges to the semi-group $P_.^{N,U}$ associated to $\mathcal E_0$, meaning that:
\begin{center}
if $f_\varepsilon$ strongly converges to $f$ then $P_.^{N,U,\varepsilon}f_\varepsilon$ strongly converges to $P_.^{N,U}f$.
\end{center}
Pick some $g \in C_b^0((\mathbb R^2)^N)$. We claim that for all $x \notin C_0$ (recall that $C_0$ is the collision set) and all $t>0$, 
\begin{center}
$P_t^{N,U,\varepsilon}g(x)$ converges to $P_t^{N,U}g(x)$.
\end{center}
Indeed one can find an open bounded neighborhood $\mathcal V(x)$ of $x$ included in $C_0^c \subset M^N$, hence for all smooth $h$ compactly supported in $\mathcal V(x)$, $$\int \, h(y) \, P_t^{N,U,\varepsilon}g(y) \, G_N^{U,\varepsilon}(y) \, dy \, \to \, \int \, h(y) \, P_t^{N,U}g(y) \, G_N^{U}(y) \, dy \, .$$ On $\mathcal V(x)$, $G_N^{U,\varepsilon}$ and $G_N^{U}$ are smooth, bounded with all their derivatives bounded, uniformly w.r.t. $\varepsilon$. 

Recall the discussion after  Remark \ref{rementropks}. We know that the above semi-groups admit smooth density kernels $\rho_t^{N,U,\varepsilon}(y,z)$ and $\rho_t^{N,U}(y,z)$ and using this time proposition 1.12 (1) in  \cite{Cathypo2} one knows that, for any $\varepsilon_0>0$,  $$\sup_{\varepsilon_0 \geq \varepsilon \geq 0} \, ||\rho_t^{N,U,\varepsilon}(.,.)||_{C_b^j(\mathcal V(x)\times (\mathbb R^2)^N)} \, \leq C(j,\mathcal V(x),t) \, .$$ We deduce that $y \mapsto P_t^{N,U,\varepsilon}g(y)$ is smooth, bounded with bounded derivatives uniformly in $\varepsilon$ and it is now an easy exercise to show our claim.

It is now enough to integrate with respect to the initial density, to remember that $C_0$ is $dx$ negligible and to use Lebesgue's bounded convergence theorem to get
\begin{proposition}\label{propvarepsgrandchi}
For any $\chi<4$, any $t>0$, $\rho_t^{N,U,\varepsilon}$ weakly converges to $\rho_t^{N,U}$ as $\varepsilon \to 0$.
\end{proposition}
\bigskip

Gathering Lemma \ref{lemBJWentrop}, Lemma \ref{lemexistliouvks}, Corollary \ref{corapproxKSstoch} or Proposition \ref{propvarepsgrandchi}, we have thus obtained the following result 
\begin{theorem}\label{propKSconfine}
When adding a confining potential such that $\nabla U$ is Lipschitz, the marginals flow $\rho_.^{N,U}$ of the unique solution of Theorem \ref{thmexistN} with $\chi<4$, is (almost) an entropy solution, i.e.  satisfies $$H(\rho^{N,U}_T \,  dx|e^{-\tilde U} dx) +   \int_0^T \, \int \, \left(\sum_i \, \left|\nabla_i \ln \rho^{N,U}_t + \nabla \tilde U \, + \frac{\chi}{N} \, \sum_{j\neq i} \, \frac{x^i -x^j}{|x^i-x^j|^2} \right|^2\right) \, \rho^{N,U}_t(x) \, dx \, dt \, \leq $$
$$\quad \leq C(N,\chi,\rho_0^N) \, \left(H(\rho_0^N \, dx |e^{- \tilde U} dx) +  \, \int \rho_0^N \, \frac{\chi}{N} \, \sum_{i<j} \, |\ln( |x^i-x^j|)| dx\right) \, ,$$ for some $C(N,\chi,\rho_0^N)$, as soon as the initial condition is exchangeable and satisfies $$\int \, \rho_0^{N} \, |\ln \rho_0^{N}| \,  dx < +\infty  \, , \quad \sup_{i\neq j} \int \, \rho_0^{N} \, |\ln (|x^i-x^j|)| \,  dx < +\infty  \quad \textrm{and} \quad \int |x|^2 \, \rho_0^{N} \, dx < +\infty \,  .$$
\end{theorem} 

This (partially) confirms the \cite{BJW} prediction when a confining potential is added. 

Notice that similarly to Lemma \ref{lemtightKSapprox}, constants of type $C(\chi,\rho_0^N,T)$ only depends on $\rho_0^N$ through the distribution of the first two coordinates $(x^1,x^2)$ (or any pair thanks to exchangeability), hence can be chosen independent of $N$ if $\rho_0^N$ is chaotic.
\medskip

\begin{remark}\label{remIPPpart}
Recall  that, contrary to the solution of the non linear PDE, the density $\rho_t^{N,U}$, which exists, cannot satisfy $\int_0^T \, I(\rho_t^{N,U}) dt < +\infty$. Indeed, otherwise, the interaction drift would be square integrable, according to the previous Theorem, and Theorem \ref{thmCL} would tell us that the law of the particle system has finite relative entropy, hence is absolutely continuous w.r.t. a probability $P$ equivalent to the Wiener measure. This is impossible since we know that $2$-collisions occur with a strictly positive probability. 

This shows that some (mysterious) cancellations have to appear in \eqref{eqinfKSpartic}. As for the non linear PDE one should, formally, develop  
$$
\int \, \sum_i \, \left|\nabla_i \ln \rho_t^{N,U} + \frac{\chi}{N} \sum_{j \neq i} \frac{x^i-x^j}{|x^i-x^j|^2}\right|^2 \, \rho_t^{N,U} dx = I(\rho_t^{N,U}) $$ $$ \quad + \, \int \, \sum_i \, \left|\frac{\chi}{N} \sum_{j \neq i} \frac{x^i-x^j}{|x^i-x^j|^2}\right|^2 \, \rho_t^{N,U} dx + \int \, \sum_i \left \langle \rho_t^{N,U},\frac{2\chi}{N} \sum_{j \neq i} \frac{x^i-x^j}{|x^i-x^j|^2}\right\rangle dx $$ so that integrating by parts the scalar product 
\begin{eqnarray*}
\int \, \sum_i \left \langle \nabla_i \rho_t^{N,U},\frac{2\chi}{N} \sum_{j \neq i} \frac{x^i-x^j}{|x^i-x^j|^2}\right\rangle dx &=& - \, \int \, \sum_i \, \rho_t^{N,U} \, \frac{4 \pi \chi}{N} \sum_{j \neq i} \delta_{x^i=x^j}(dx)\\ &=& - \, \frac{4 \pi \chi}{N} \, \sum_{i\neq j} \, \int \rho_t^{N,U}(x_{i,j}) d\tilde x_{i,j}
\end{eqnarray*}
where $\tilde x_{i,j}=(x^1,...,x^{j-1},x^{j+1},...,x^N)$ and $x_{i,j}=(x^1,...,x^{j-1},x^i,x^{j+1},...,x^N)$. Since the squared terms are infinite, the latter has also to be infinite.
\hfill $\diamondsuit$
\end{remark}
\medskip

Finally let us look at the $2D$ Keller-Segel non-linear SDE.

Actually one can directly use theorem \ref{thm-exist-edp} and its proof in \cite{BDP}, since as recalled in remark \ref{rembetter}, the drift $b_t=K*\rho_t= - \, \chi \, \nabla c$ satisfies the finite energy condition $$\int_0^T \, \int \, |b_t|^2 \, \rho_t \, dx \, dt \,  < + \infty \, .$$ We may thus use theorem \ref{thmnlback} in order to get
\begin{theorem}\label{thmexistnl1}
Assume that $\int (\ln \rho_0(x)+|x|^2) \, \rho_0(x) dx < +\infty$ and $0<\chi < 4$. If $t \mapsto \rho_t$ is the solution of \eqref{eqKS2} built in theorem \ref{thm-exist-edp} for $K(x)=\frac{\chi \, x}{|x|^2}$, then there exists a (weak) solution of $$dX_t = \sqrt 2 dB_t - \, (K*\rho_t)(X_t) \, dt \quad \textrm{ with } \quad \mathcal L(X_t)=\rho_t(x) \, dx \, .$$
In addition the law $Q$ of this solution restricted to $[0,T]$ satisfies $H(Q|P)<+\infty$ where $P$ denotes the law of a standard $2$-D stationary Ornstein-Uhlenbeck process, i.e. with invariant probability measure the centered gaussian distribution with covariance matrix $a \, Id$ for any $a>0$ and $T>0$.

Finally this solution is unique in the set of solutions such that their time marginals satisfy \eqref{eqsolentrop} and $$\int_0^T \, \int \, |\nabla \ln \rho_t + K*\rho_t|^2(x) \, \rho_t(x) dx dt < +\infty \, .$$ 
\end{theorem}
For uniqueness it is enough to use Theorem \ref{thm-unique}. Accordingly, since for two solutions $Q_1$ and $Q_2$, their marginals flow satisfy \eqref{eqKS2}, these marginals flow are the same denoted by $\rho_t$. $Q^1$ and $Q^2$ are thus both solutions of the linear S.D.E. $$dY_t=\sqrt 2 \, dB_t - (K*\rho_t)(Y_t) \, dt$$ with finite relative entropy w.r.t. $P$, hence are given by the same Girsanov density.

\begin{remark}\label{remKSconfine}
As we already discussed, the same holds true for the Keller-Segel model with confinement. In the special case of $U(v)=|v|^2$ this result is a simple consequence of the correspondence with the so called ``self similar'' coordinates. Indeed if we define $$g_t(x)=\frac{1}{1+2t} \, \rho_{\ln(\sqrt{1+2t})}\left(\frac{x}{\sqrt{1+2t}}\right) \, ,$$ as soon as $\rho$ is a solution to the Keller-Segel P.D.E., $g$ is a solution to the confined Keller-Segel P.D.E. with confinement given by the quadratic $U$ above. 
\hfill $\diamondsuit$
\end{remark}

Thanks to what we have done before, uniqueness in Theorem \ref{thmexistnl1} can be improved. Indeed $K \mathbf 1_{K>A} \in \mathbb L^p(\mathbb R^2)$ for any $p<2$ so that if $\sup_{t \in [0,T]} ||\rho_t||_q < +\infty$ for some $q>1$, $$\sup_{t \in [0,T]} ||K \mathbf 1_{K>A} * \rho_t
||_r < +\infty \quad \textrm{ for some } r>2 \, .$$ We may thus apply Theorem \ref{thmkryl} and get that the \emph{linear} SDE $$dY_t=\sqrt 2 \, dB_t -  \, b(Y_t) dt \, - \, (K*\rho_t)(Y_t) dt$$ has a unique weak solution $\tilde Q$ and $H(\tilde Q|P) < +\infty$ provided $$\int \, \int_0^T \, |x|^2 \, \rho_t(x) \, dt \, dx \, < +\infty \, .$$ Recall that $P$ is the law of some Ornstein-Uhlenbeck process explaining the necessity of the previous moment condition. If $Q$ solves the non linear SDE \eqref{eqnldiff}, it is a weak solution of the previous linear one. Accordingly $Q=\tilde Q$ satisfies $H(Q|P)< +\infty$ and we have shown
\begin{corollary}\label{corexistnl1}
Under the assumptions of Theorem \ref{thmexistnl1}, there exists at most one solution of the non-linear SDE such that its marginals flow satisfies $\sup_{t \in [0,T]} ||\rho_t||_q < +\infty$ for some $q>1$ and $\int  \int_0^T \, |x|^2 \, \rho_t(x) \, dt \, dx \, < +\infty$.
\end{corollary}
Actually we can go a step further. Indeed if $$\sup_{t \in [0,T]} ||K \mathbf 1_{K>A} * \rho_t
||_r < +\infty \quad \textrm{ for some } r>2 \, , $$ and $$\sup_{t \in [0,T]} ||\rho_t||_q < +\infty \, , $$ the product satisfies $$\int \, \int_0^T \, (K  * \rho_t)(x) \, \rho_t(x) \, dt \,  dx < +\infty$$ provided for some $\varepsilon >0$,  $$\frac{1}{2-\varepsilon} + \frac 1q = 1 + \frac 1r \quad \textrm{ and } \quad \frac 1q + \frac 1r = 1 \, .$$ Using interpolation the previous equalities are satisfied as soon as  $\sup_{t \in [0,T]} ||\rho_t||_{q'} < +\infty$ for some $q'>4/3$.

 Note that in the specific Keller-Segel case we may use the Hardy-Littlewood-Sobolev inequality $$||\frac{1}{|z|}*\rho||_{2s/2-s} \leq C_s \, ||\rho||_s \quad \textrm{ for } s \in (1,2) \, ,$$ so that the previous line of reasoning extends to $q'=4/3$.

According to the superposition principle recalled in Remark \ref{remtrevisan}, there exists a solution to the non linear SDE. According to the previous Corollary, this solution has finite relative entropy, so that its marginals flow, which is $\rho_.$, satisfies \eqref{eqsolentrop}. We may thus apply the uniqueness Theorem for free energy solution of the Keller-Segel equation. In conclusion
\begin{corollary}\label{corexistnl2}
Under the assumptions of Theorem \ref{thmexistnl1}, there exists at most one solution of the Keller-Segel equation, such that $\sup_{t \in [0,T]} ||\rho_t||_q < +\infty$ for some $q \geq 4/3$ and $\int  \int_0^T \, |x|^2 \, \rho_t(x) \, dt \, dx \, < +\infty$.
\end{corollary}

\begin{remark}\label{remcornl1}
The previous result is almost the same as the uniqueness part of Theorem 2 in \cite{Bedmas} (condition (1.3) therein allows some explosion of the $\mathbb L^{4/3}$ norm at the origin). Our proof shows that this solution is the finite energy solution. Existence  follows from the estimates on the finite energy solution in \cite{FM}.  \hfill $\diamondsuit$
\end{remark}

\begin{remark}\label{remkslpq}
For simplicity, we only considered uniform in time estimates in the previous results. Using the results by Krylov et al it is easily seen that one can replace in Corollary \ref{corexistnl1}, $\sup_{t \in [0,T]} ||\rho_t||_q < +\infty $ for some $q>1$,  by $$ \rho_. \in \mathbb L^p([0,T],\mathbb L^q) \; \textrm{ for some $2 \geq q >1$ and $p$ satisfying } \; \frac 1p + \frac 1q < 1 \, .$$ Of course if integrability holds for some $q>2$, it also holds for $q=2$ since $\rho_.$ is in $\mathbb L^1$.

Similarly, using in addition the superposition principle, we can assume that the same holds for $2\geq q \geq 4/3$ in Corollary \ref{corexistnl2}. 
\hfill $\diamondsuit$
\end{remark}

\begin{remark}\label{remBCF}
One can also be interested by the \emph{repulsive} Keller-Segel model with $\chi<0$. This process is studied in \cite{BCF} under the name \emph{dynamics of a planar Coulomb gas} with a gaussian confinement potential, generalizing to dimension $d=2$ the Ornstein-Uhlenbeck Dyson process. Existence and uniqueness for the particle system can be shown by using a Lyapunov function for proving the absence of collisions. Non linear SDE and propagation of chaos is mentioned in \cite{BCF} subsubsection 1.4.6 but does not seem to have received a complete treatment since that time.
\hfill $\diamondsuit$
\end{remark}
\medskip

\section{Large number of particles: convergence in relative entropy and chaos.}\label{secchaos1}

We come to the large number of particles problem. In this section we denote by $Q^N$ (resp. $Q^{k,N}$ and $\bar Q^{\otimes j}$) the law of the solution of the particle system \eqref{eqsys} (resp. the law of $(X_.^{1,N},...,X^{k,N}_.)$ in the system \eqref{eqsys} and the $j$ tensor product of the law of the solution of the non linear S.D.E. \eqref{eqnldiff}), when they are well defined. For simplicity we will sometimes denote by $\eta^N$ and $\bar \eta^j$ the full drifts, including interactions and self-interaction, of $Q^N$ and $\bar Q^{\otimes j}$, i.e. for instance $$\eta^{i,N}(t) = b(X^{i,N}_t) + \frac 1N \, \sum_j \,  K(X^{i,N}_t-X^{j,N}_t) \, .$$ In the whole section, the initial distribution \emph{$\mu_0^N$ is assumed to be exchangeable} and to converge, in a sense and with a rate to define, to some $\mu_0^{\otimes N}$. 

We also introduce an additional definition
\begin{definition}\label{deflinnonlin}
Let $\bar Q$ a solution of the non linear SDE \eqref{eqnldiff}. Denote by $\bar \rho_.$ its marginals flow. The \emph{linearized} McKean-Vlasov equation associated to $\bar Q$ is the linear (time inhomogeneous) SDE $$dY_t = \sqrt 2 \, dW_t - b(Y_t) dt - (K*\bar \rho_t)(Y_t) dt \, .$$
\end{definition}
\medskip

As explained in the introduction, it is expected that $Q^{k,N}$ (which is also exchangeable) converges to $\bar Q^{\otimes k}$ as $N$ goes to infinity for any fixed $k$, in some sense (total variation, Wasserstein distance, relative entropy) implying weak convergence. Notice that convergence in relative entropy is stronger than the entropic convergence described in the introduction and shown in \cite{FHM,GQ}. We shall first recall some existing results (and methods) concerned with the relative entropy convergence and then apply them to the situations we are interested in. 

Exchangeability and the variational equivalent definition of relative entropy in \eqref{eqdefentrop} immediately show that for $k \leq N-1$, 
\begin{equation}\label{eqentropproj}
H(Q^{k,N}|\bar Q^{\otimes k}) \leq \frac{1}{\lfloor N/k\rfloor} \, H(Q^{N}|\bar Q^{\otimes N}) \leq \frac{k}{N-k} \, H(Q^{N}|\bar Q^{\otimes N}) \, .
\end{equation} 

It is thus tempting to try to get some upper bound for $H(Q^{N}|\bar Q^{\otimes N})$, ideally some bound that does not depend on $N$. It is the approach developed some times ago in \cite{Benofer} where the authors are using a description of $Q^N$ as a Gibbs measure on the path space, provided $\eta^N$ is of gradient type (for a similar description in a non-mean field framework one can look at \cite{CRZgibbs}). A uniform in $N$ bound is obtained in Theorem 1 of \cite{Benofer} and applied to the particle system in their Theorem 3. The gibbsian description requires $\eta^N$ to be a bounded smooth gradient. 

As remarked by Lacker (this remark is part of the ``old'' folklore, and was already made by F\"{o}llmer, but seemingly not published), looking at the ``reversed'' relative entropy $H(\bar Q^{\otimes N}|Q^{N})$ is interesting because, with less stringent assumptions than before, an uniform in $N$ upper bound is true. However, the ``projected'' $H(\bar Q^{\otimes k}|Q^{k,N})$ does no more satisfy \eqref{eqentropproj} so that the previous bound, if interesting, is far to be enough.

Notice that in our singular framework, looking at the reversed relative entropy has another interest. When $Q^N$ allows collisions (as for the Keller-Segel model), $\bar Q^{\otimes N}$ does not allow collisions (due to the independence of its components). In particular $Q^N$ is not absolutely continuous w.r.t. $\bar Q^{\otimes N}$ and the relative entropy is thus infinite, while $\bar Q^{\otimes N}$ has a chance to be absolutely continuous w.r.t. $Q^N$.

\subsection{\textbf{Lacker's approach.}\\ \\}\label{subseclack} \quad
 Recently Lacker proposed in \cite{Lack} a new method in order to show the convergence in relative entropy (and entropic chaos at the same time). We shall briefly recall his approach, in order to introduce some slight modifications.

Assume that both \eqref{eqsys} and \eqref{eqnldiff} have a weak solution $Q^N$ and $\bar Q$. The key in Lacker's proof is that $Q^{k,N}$ is still an Ito process whose drift is obtained by conditioning w.r.t. $(X_{.\leq t}^{1,N},...,X_{. \leq t}^{k,N})$, i.e is given by the non markovian drift
\begin{eqnarray}\label{eqdriftproj}
\hat b^{i,k}(t,X_{.\leq t}^{1,N},...,X_{.\leq t}^{k,N}) &=& b(X_t^{i,N}) + \frac 1N \, \sum_{i\neq j, j=1,...,k} K(X_t^{i,N}-X_t^{j,N}) + \nonumber \\ && + \, \mathbb E\left[\frac 1N \, \sum_{j=k+1,...,N} K(X_t^{i,N}-X_t^{j,N}) \Big |X_{.\leq t}^{1,N},...,X_{.\leq t}^{k,N}\right] \nonumber  \\ &=& b(X_t^{i,N}) + \frac 1N \, \sum_{i\neq j, j=1,...,k} K(X_t^{i,N}-X_t^{j,N}) + \nonumber \\ && + \, \frac{N-k}{N} \, E\left[K(X_t^{i,N}-X_t^{k+1,N}) \Big |X_{.\leq t}^{1,N},...,X_{.\leq t}^{k,N}\right] \, ,
\end{eqnarray}
since almost surely, for all $j \geq k+1$,
\begin{equation}\label{eqcondex}
E\left[K(X_t^{i,N}-X_t^{k+1,N}) \Big |X_{.\leq t}^{1,N},...,X_{.\leq t}^{k,N}\right]=E\left[K(X_t^{i,N}-X_t^{j,N}) \Big |X_{.\leq t}^{1,N},...,X_{.\leq t}^{k,N}\right]
\end{equation}
the latter equality being a consequence of exchangeability. Of course here $|X_{.\leq t}^{i,N}$ denotes the conditioning w.r.t. the whole path of $X_{.}^{i,N}$ up to time $t$.  

Introduce the $\bar Q^{\otimes k}$ exponential local martingale $$Z_t^{k,N} = \frac{d\mu_0^{k,N}}{d\mu_0^{\otimes k}} \; \exp \left(\int_0^t \, \langle \beta^{k,N}_s, \sqrt 2 \, dB_s\rangle - \int_0^t \, |\beta_s^{k,N}|^2 \, ds\right)$$
where $$\beta_s^{i,k,N}= \hat b^{i,k}(s,\omega^{\leq k,N}) \, - \, b(\omega_s^{i,N}) \, - \, (K*\bar \rho_s)(\omega_s^{i,N})$$ $\omega^{\leq k,N}$ being the generic element in $C^0([0,t],(\mathbb R^d)^k)$ and $\sqrt 2 \, B_.$ being the $\bar Q^{\otimes k}$ $dk$ dimensional Brownian motion (with variance $2t$) $$\omega_.^{k,N}-\omega_0^{k,N} + \int_0^. \, (b(\omega_s^{k,N}) + ( K*\bar \rho_s)(\omega_s^{k,N})) ds \, .$$
If we assume that the linearized McKean-Vlasov associated to $\bar Q$ has a unique weak solution, and that 
\begin{equation}\label{eqMLack}
I^N_T = \int_0^T \, M^N_t \, dt < +\infty \, ; \, \textrm{ where } \; M^N_t=\mathbb E[|K(X_t^{1,N}-X_t^{2,N}) \, - \, (K*\bar \rho_t)(X_t^{1,N})|^2] \, ,
\end{equation}
(where thanks to exchangeability, we may replace the indices $(1,2)$ by any pair $(i,j)$ with $i\neq j$), then the entropic Girsanov theory introduced in subsection \ref{subsecpath} and detailed in this more general context in \cite{Leo12,CL1}, tells us that $Q^N = Z_T^{N,N} \, \bar Q^{\otimes N}$ and $H(Q^N|\bar Q^{\otimes N})<+\infty$. In particular $Q^N$ is the unique weak solution of \eqref{eqsys}. 

Contraction of relative entropy w.r.t. measurable mapping says that for all $k \leq N$, $$H(Q^{k,N}|\bar Q^{\otimes k})\leq H(Q^N|\bar Q^{\otimes N})<+\infty \, .$$ Also remark that for any $k=1,...,N$, all $i \leq k$ and $j > k$, $$\mathbb E \left[|E[K(X_t^{i,N}-X_t^{j,N})|X_{.\leq t}^{1,N},...,X_{.\leq t}^{k,N}]  - (K*\bar \rho_t)(X_t^{i,N})|^2\right] \, \leq \, M^N_t \, .$$
\medskip

Define $$H_t^{k,N} = H(Q_{|\mathcal F_t}^{k,N}|\bar Q_{|\mathcal F_t}^{\otimes k}) \, .$$ Thanks to exchangeability, it holds
\begin{equation}\label{eqentropyloclac}
\frac{d}{dt} \, H_t^{k,N} \leq \frac{k(k-1)^2}{2N^2} \, M_t^N + \frac k2 \, \mathbb E\left[|\mathbb E[K(X_t^{1,N}-X_t^{N,N})|X_{.\leq t}^{1,N},...,X_{.\leq t}^{k,N}]  - (K*\bar \rho_t)(X_t^{1,N})|^2\right] \, .
\end{equation}

Lacker's idea is to get some \emph{hierarchy} for these $H_t^k$, assuming that, 
\begin{eqnarray}\label{eqtransp}
&& |\mathbb E[K(X_t^{1,N}-X_t^{N,N})|X_{.\leq t}^{1,N},...,X_{.\leq t}^{k,N}]  - (K*\bar \rho_t)(X_t^{1,N})|^2 \nonumber \\ && \quad \quad \quad  \leq\,  \gamma(K) \, H([Q_t^{k+1,N}]_{|X_{.\leq t}^{1,N},...,X_{.\leq t}^{k,N}}|\bar Q_t)
\end{eqnarray}
where $[Q_t^{k+1,N}]_{|X_{.\leq t}^{1,N},...,X_{.\leq t}^{k,N}}$ is the conditional distribution of $X_{.\leq t}^{N,N}$ knowing $X_{.\leq t}^{1,N},...,X_{.\leq t}^{k,N}$. Notice that both hand side are random, depending in particular on $X_t^{1,N}$, so that \eqref{eqtransp} has to hold almost surely for all value of $X_t^{1,N}$. It is thus sufficient (and much more tractable) to replace \eqref{eqtransp} by the stronger (still random)
\begin{eqnarray}\label{eqtransp2}
&& \sup_{s \mapsto x_s} \, |\mathbb E[K(x_t-X_t^{N,N})|X_{.\leq t}^{1,N},...,X_{.\leq t}^{k,N}]  - (K*\bar \rho_t)(x_t)|^2 \nonumber \\ && \quad \quad \quad  \leq\,  \gamma(K) \, H([Q_t^{k+1,N}]_{|X_{.\leq t}^{1,N},...,X_{.\leq t}^{k,N}}|\bar Q_t)
\end{eqnarray}
which will be called the Lacker transportation assumption, (LTA) for short. 

If \eqref{eqtransp} is satisfied, then 
\begin{equation}\label{eqderiventlac}
\frac{d}{dt} \, H^{k,N}_t \, \leq \, \frac{k(k-1)^2}{2 N^2} \, M_t^N \, + \,  \frac k2 \, \gamma(K) \, (H^{k+1,N}_t - H^{k,N}_t) \, ,
\end{equation}
 since $$\mathbb E (H([Q_t^{k+1,N}]_{|X_{.\leq t}^{1,N},...,X_{.\leq t}^{k,N}}|\bar Q_t))= H_t^{k+1,N} - H_t^{k,N} \, .$$ The constants are slightly different due to the $\sqrt 2 \, dW_t$ in our equations. Gronwall's lemma yields 
\begin{eqnarray*} 
H_t^{k,N} &\leq& e^{-(\gamma(K)/2)kt} \, H_0^{k,N} + \frac{k(k-1)^2}{2N^2} \, \int_0^t \, e^{- (\gamma(K)/2) k (t-s)} \, M^N_s \, ds + \\ &+& \frac 12 \, \gamma(K)  \, k \int_0^t \, e^{- (\gamma(K)/2) k (t-s)} \, H_s^{k+1,N} \, ds 
\end{eqnarray*}
so that 
\begin{equation}\label{eqlackmodifcont}
H_t^{k,N} \leq e^{-(\gamma(K)/2)kt} \, H_0^{k,N} + \frac{k(k-1)^2}{2N^2} \, I_T^N + \frac 12 \, \gamma(K)  \, k \int_0^t \, e^{- (\gamma(K)/2) k (t-s)} \, H_s^{k+1,N} \, ds \, ,
\end{equation}
where we used a crude upper bound for the second term in the sum and $I_t^N \leq I_T^N$.

Iterating the previous we get 
 $$H_T^{k,N} \leq \sum_{l=k}^{N-1} B_k^l(T) \, H_0^{l,N} + \sum_{l=k}^{N-1} \frac{k(l-1)^2 I^N_T}{2 N^2} \, A_{k+1}^l(T) + A_k^{N-1}(T) \, H_T^{N,N}$$ where $$A_j^l(T) = \left(\prod_{i=j}^l i (\gamma(K)/2) \right) \, \int_0^T \int_0^{t_{j+1}} ... \int_0^{t_l} e^{-\sum_{i=j}^l i(\gamma(K)/2) (t_i-t_{i+1})} \, dt_{l+1} ... dt_{j+1}$$ and  $$B_j^l(T) = \left(\prod_{i=j}^{l-1} i (\gamma(K)/2) \right) \, \int_0^T \int_0^{t_{j+1}} ... \int_0^{t_{l-1}} e^{-l (\gamma(K)/2) t_l \, -\sum_{i=j}^{l-1} i(\gamma(K)/2) (t_i-t_{i+1})} \, dt_{l} ... dt_{j+1}$$ if $N \geq l\geq j$, $A_{k+1}^k (T)= 1$ and $B_k^k(T)=e^{-k (\gamma(K)/2) T}$.
 
Compared with (4.18) in \cite{Lack} we have to shift the subscript of the $A's$ due to the expression in \eqref{eqlackmodifcont}. 
 
If we assume that $H(Q_0^{l,N}|\bar Q_0^{\otimes l}) \leq C_0 \, l^2 \, \varepsilon_N$, and after noticing that $H_T^{N,N} \leq H_0^{N,N}+ \frac N4 \, I^N_T$, one obtains, thanks to Lemma 4.8 in \cite{Lack}, the analogue of (4.24) in \cite{Lack}, 
\begin{eqnarray}\label{eqlackimproved}
H_T^{k,N} &\leq& C_0 \, \varepsilon_N \, 2k^2 \, e^{\gamma(K) T} + \frac{I^N_T (k+3)^3}{6 N^2} \, e^{3\gamma(K) T/2} +  \frac{k^3 \, I_T^N}{2 N^2} \nonumber \\ &+& (C_0 N^2 \varepsilon_N + \frac 14  N I^N_T) \, \exp \left(-2N(e^{-\gamma (K)T/2}-\frac kN)_+^2\right) \, .
\end{eqnarray}
Since we are not interested in optimizing the dependence in $k$ (i.e. replace $k^3$ by $k^2$ as in \cite{Lack}), we may stop here and state the following version of \cite{Lack} Theorem 2.2 
\begin{theorem}{\textbf{Lacker's theorem}}\label{thmlack}

Assume that both \eqref{eqsys} and \eqref{eqnldiff} have a weak solution $Q^N$ and $\bar Q$. Also assume  that the linearized McKean-Vlasov equation  associated to $\bar Q$ of Definition \ref{deflinnonlin} has a unique weak solution. Assume in addition that \eqref{eqMLack} and \eqref{eqtransp} are satisfied. 

Then if for all $k \leq N$, $H(\mu^{k,N}_0|\mu_0^{\otimes k}) \leq C_0 \, k^2 \, \varepsilon_N$ for some $C_0<+\infty$, and of course $\mu_0^N$ is exchangeable, then for any $k\leq N$, $H(Q^{k,N}|\bar Q^{\otimes k})=H_T^{k,N}$ satisfies \eqref{eqlackimproved}.
\end{theorem}
Compared with Theorem 2.2 in \cite{Lack} we simply have relaxed the assumption on the initial condition, and slightly improved the assumption (2) therein which is $\sup_t M_t^N = M^N < +\infty$. 
\medskip

Of course \eqref{eqtransp} and \eqref{eqtransp2} are consequence of the more restrictive: there exists $0<\gamma(K)<+\infty$ such that, for all $t\in (0,T)$, all $\omega \in C^0([0,T],\mathbb R^d)$ and all probability measure $Q'$ on $C^0([0,T],\mathbb R^d)$ such that $\int \, |K(\omega_t - \omega'_t)| Q'(d\omega') < +\infty$, it holds for all $\omega$, 
\begin{equation}\label{eqlacgene}
\left|\int \, K(\omega_t - \omega'_t) (\bar Q - Q')(d\omega') \right|^2 \, \leq \, \gamma(K) \, H(Q'|\bar Q) \, .
\end{equation}
Such an inequality seems very difficult to prove (due to the uniformity w.r.t. $\omega$) except if $K$ is bounded, where it amounts to Pinsker inequality \eqref{eqpinsker} i.e. one can choose $$\gamma(K) = 2 \, ||K||_\infty^2 \, .$$ In \cite{Han2} the author proposes to use the improvement of Pinsker inequality shown in \cite{Bolvil} $$\left(\int  g \, d|\mu-\nu|\right)^2 \leq 4C H(\mu|\nu) \quad \textrm{ with} \quad C=\frac 16 \, \mu(g^2) + \frac 13 \, \nu(g^2) \, .$$ This inequality has to be used with $\mu=[Q_t^{k+1,N}]_{|X_{.\leq t}^{1,N},...,X_{.\leq t}^{k,N}}$, hence once again one needs some uniform (in $\omega$) control which is not easy to get. Notice that another weighted Pinsker inequality is shown in \cite{Bolvil} namely $$\left|\int g  \, d|\mu - \nu|\right|^2 \, \leq \, 2 \left(1+ \ln \int \, e^{g^2} d\nu\right) \, H(\mu|\nu)$$ allowing us to slightly improve on the boundedness assumption of $K$ replacing it by $$ \sup_{\omega'} \, \int \, e^{K^2(\omega'-\omega)} \, \bar Q(d\omega) < +\infty \, ,$$ which is quite difficult to check, but this time only depends on $\bar Q$.

\begin{remark}\label{remgirsbound}
Why is the uniqueness assumption on the linearized McKean-Vlasov equation important ? Because it implies that the Girsanov transform up to the time $\tau_n$ introduced in the Remark \ref{remuniqueCL} furnishes the unique solution of the SDE with drift $\hat b$ up to this time. Uniqueness of the solution up to $T$ follows by the same argument as in Remark \ref{remuniqueCL}, so that the Girsanov transform of $\bar Q^{\otimes k}$ is exactly $Q^{k,N}$. \hfill $\diamondsuit$
\end{remark}
\begin{remark}\label{remlacbad}

If $K$ is bounded, the main weakness of Lacker's results is the exponential behaviour w.r.t. the $\mathbb L^\infty$ squared norm of $K$. In order to use these results for an approximating sequence of $K$ by some  bounded $K_N$, it is necessary to use a logarithmic cut-off (as in \cite{Liucut}).

It is worth noticing that for a bounded $K$, a very similar quantitative bound (involving some $e^{c ||K||^2_\infty}$) is obtained in \cite{Hao} Theorem 1.1 (ii), where the relative entropy is replaced by some squared Wasserstein distance (see (1.19) in \cite{Hao}).  Recall that relative entropy controls the squared total variation distance.
It thus seems that obtaining a better bound in a general bounded framework is not as easy.

Also notice that in the usual Lipschitz framework of McKean's Theorem \ref{thmbasic} the squared $W_2$ Wasserstein distance is bounded by the initial squared distance plus $$C_1 \, T(||b||^2_{Lip}+||K||^2_{Lip}) \, e^{(C_2 \, T \, (1+||b||^2_{Lip}+||K||^2_{Lip}))} \; \frac kN \, .$$ Hence here again some approximation method will require some logarithmic cut-off.
\hfill $\diamondsuit$
\end{remark}

\subsection{\textbf{Reversing the entropy.}\\ \\}\label{subsecreverse}

\quad Remark that during the discussion we have shown the following intermediate result: if \eqref{eqMLack} is satisfied and the linearized McKean-Vlasov equation associated to $\bar Q$ has a unique weak solution, then for all $k=1,...,N$ the SDE in $(\mathbb R^d)^k$, $$dY^k_t = \sqrt 2 \, dW_t + \hat b^{.,k}(Y_{.\leq t}) \, dt$$ also has a unique solution, which is thus $Q^{k,N}$. This point was not known by Lacker in his study of the reverse entropy, so that he only looked at the case $K$ bounded.
\medskip

Assume now 
\begin{equation}\label{eqMLackrev}
\bar I_T = \int_0^T \, \bar M_t \, dt < +\infty \, ; \, \textrm{ where } \; \bar M_t=\mathbb E^{\bar Q^{\otimes 2}}[|K(\omega_t^{1}-\omega_t^{2}) \, - \, (K*\bar \rho_t)(\omega_t^{1})|^2] \, ,
\end{equation}
i.e. the analogue of \eqref{eqMLack} replacing $Q^{2,N}$ by $\bar Q^{\otimes 2}$. Notice that this time these quantities do not depend on $N$. This time, it holds
\begin{equation}\label{eqreversetotal}
H(\bar Q^{\otimes N}|Q^{N}) \leq H(\bar Q_0^{\otimes N}|Q_0^{N}) + \bar I_T \, .
\end{equation}
If \eqref{eqMLackrev} is satisfied we thus have $$H(\bar Q^{\otimes k}|Q^{k,N}) \leq H(\bar Q^{\otimes N}|Q^{N}) < +\infty \, .$$

Since $\bar Q^{\otimes N} \ll Q^{N,N}$, \eqref{eqcondex}, written in the form $$E^{Q^{N,N}}\left[K(\omega_t^{i,N}-\omega_t^{k+1,N}) \Big |\omega_{.\leq t}^{1,N},...,\omega_{.\leq t}^{k,N}\right]=E^{Q^{N,N}}\left[K(\omega_t^{i,N}-\omega_t^{j,N}) \Big |\omega_{.\leq t}^{1,N},...,\omega_{.\leq t}^{k,N}\right]$$ for $j \geq k+1$ is still true $\bar Q^{\otimes N}$ almost surely.

Assuming now that, $\bar Q^{\otimes N}$ almost surely
\begin{eqnarray}\label{eqtransprev}
&& |\mathbb E^{Q^{N,N}}[K(\omega_t^{1,N}-\omega_t^{k+1,N})|\omega_{.\leq t}^{1,N},...,\omega_{.\leq t}^{k,N}]  - (K*\bar \rho_t)(\omega_t^{1,N})|^2 \nonumber \\ && \quad \quad \quad  \leq\,  \gamma(K) \, H(\bar Q_t|[Q_t^{k+1,N}]_{|\omega_{.\leq t}^{1,N},...,\omega_{.\leq t}^{k,N}})
\end{eqnarray}
we may follow the lines of the proof of Lacker's Theorem, defining $\bar H_t^{k,N}= H(\bar Q^{\otimes k}_{|\mathcal F_t}|Q^{k,N}_{|\mathcal F_t})$ and starting with
\begin{equation}\label{eqderiventrev}
\frac{d}{dt} \, \bar H_t^{k,N} \leq \frac{k(k-1)}{2N^2} \, \bar M_t \, + \, \frac k2 \, \gamma(K) \, (\bar H_t^{k+1,N} - \bar H_t^{k,N}) \, ,
\end{equation}
we can reverse the entropy and get the following slight generalization of Lacker's reverse Theorem 2.14 in \cite{Lack}
\begin{theorem}\label{thmlackrev}{\textbf{Lacker's reverse theorem}}

Under the assumptions of Theorem \ref{thmlack}, but this time with $H(\mu_0^{\otimes l}|\mu^{l,N}_0) \leq C_0 \, l^2 \, \varepsilon_N$,  for any $l\leq N$, replacing  \eqref{eqtransp} by \eqref{eqtransprev} and adding \eqref{eqMLack}, it holds
\begin{eqnarray*}
H(\bar Q^{\otimes k}|Q^{k,N}) &\leq& C_0 \, \varepsilon_N \, 2k^2 \, e^{\gamma(K) T} +\frac{\bar I_T (k+2)^2}{4 N^2} \, e^{\gamma(K) T} +  \frac{k^2 \, \bar I_T}{2 N^2}\\ && + (C_0 N^2 \varepsilon_N +  \bar I_T) \, \exp \left(-2N(e^{-\gamma (K)T/2}-\frac kN)_+^2\right) \, .
\end{eqnarray*} 
The previous result holds in particular if $K$ is bounded, with $\gamma(K)=2 ||K||_\infty^2$ and bounding $\bar I_T$ by $4 ||K||_\infty^2 \, T$.
\end{theorem}
\medskip

\subsection{\textbf{Lacker's method and chaos for singular interactions.}\\ \\}\label{subseclacsing}

Since it is delicate to verify \eqref{eqtransp} unless $K$ is bounded, one can try to consider a sequence $K_N$ of bounded kernels converging in some sense to $K$ and the associated $Q_N^N$ and $\bar Q_N$. Notice that $$\max(I_{T,N}^N,\bar I_T^N) \, \leq \, 4 T \, ||K_N||_\infty^2 \, .$$

Let $k$ be fixed. In order to simplify the discussion we will assume that 
\begin{equation}\label{eqepsilcut}
\varepsilon_N \leq c \, \frac{T \, ||K_N||_\infty^2}{N^2}
\end{equation}
  so that according to Theorem \ref{thmlackrev}, $H(\bar Q_N^{\otimes k}|Q_N^{k,N})$ will go to $0$ provided , as $N \to +\infty$
\begin{equation}\label{eqcondborne}
\ln N \, - \, ||K_N||_\infty^2 \, T \, -  \, \ln(||K_N||_\infty^2 \, T) \, \to \, +\infty  \, .
\end{equation}

It remains to control some distance between $Q_N^{k,N}$ and $Q^{k,N}$ on one hand, $\bar Q_N^{\otimes k}$ and $\bar Q^{\otimes k}$ on the other hand. The easiest part is the one concerned with the non linear S.D.E.'s. Actually the part concerned with the particle systems is much more intricate, and will require other tools. An alternate approach will be proposed in the next sections.
\medskip

\textbf{The $\mathbb L^q$ case.} \quad Consider first the situation of Theorem \ref{thmexistnl}. The  proof is precisely based on the convergence of $\bar Q_N$ to $\bar Q$ in the $\sigma(\mathbb L^1,\mathbb L^\infty)$ topology. Extension to the $k^{th}$ tensor product is immediate. Actually it is not difficult to reinforce the considered topology, because under the assumptions of Theorem \ref{thmexistnl}, we may replace $K^M-K$ by $|K^M-K|^2$, i.e. prove that $H(\bar Q|\bar Q_N)$ goes to $0$. Some rates can be obtained in the various examples, but since $||K_N||_\infty$ is bounded by $\ln^{\frac 12} N$, these rates are disastrous. We may apply this result to the $\eta$ relaxed Keller-Segel model. Using again Pinsker inequality we deduce converge in total variation distance for which we may apply the triangle inequality. We have thus obtained:

\begin{proposition}\label{propconvtronc}
In the situation of Theorem \ref{thmexistnl}, for an initial condition satisfying \eqref{eqepsilcut}, for any cut-off $K\wedge A_N$ where $A_N$ goes to infinity and satisfies \eqref{eqcondborne}, $d_{TV}(Q_N^{k,N},\bar Q^{\otimes k})$ goes to $0$ as $N$ goes to infinity.
\end{proposition}
This result has to be compared with \cite{ORT} where some other cut-off is introduced, and with Theorem 1.1 in \cite{Hao} where convergence in Wasserstein distance $W_p$ ($1 \leq p <2$) is obtained without cut-off (and also without rate). We shall come back to this later. \hfill $\diamondsuit$
\bigskip

More generally, it is enough to show that 
\begin{equation}\label{eqentropcutoff}
\int_0^T \, \int \, |K_N*\rho_{t,N} - K*\rho_{t,N}|^2 \, \rho_t \, dx \, dt \, + \, \int_0^T \, \int \, |K*\rho_{t,N} - K*\rho_t|^2 \, \rho_t \, dx \, dt \, \to \, 0
\end{equation}
 in order to prove that $H(\bar Q|\bar Q_N) \, \to \, 0$. Notice that even if we are able to show the previous convergence, it only furnishes a convergence in total variation distance for $\bar Q_N^{\otimes k}$ to $\bar Q^{\otimes k}$ as explained before. For this purpose it is actually enough to show convergence in total variation distance for $\bar Q_N^{\otimes k}$ to $\bar Q^{\otimes k}$. We shall study first the Keller-Segel model. 
\medskip

\textbf{The Keller-Segel model (with or without confining potential).} \quad For the cut-off $K_N$ we may use the one introduced in \cite{BDP} subsection 2.5 (starting p.10) with the changes of notation corresponding to the present paper (mainly $K_N$ here corresponds to $\nabla K^\varepsilon$ in \cite{BDP}). We assume that the assumptions in Theorem \ref{thm-exist-edp} are satisfied.

According to \cite{BDP} Lemma 2.11, almost all bounds one can think about are uniform in $N$. In particular (see (vii) in Lemma 2.11 of \cite{BDP}) $$\sup_N \, \int_0^T \, |K_N*\rho_{t,N}|^2 \, \rho_{t,N} \, dx \, dt = \sup_N \, H(\bar Q_N|P) \, < +\infty \, .$$ The latter implies that the family $(\bar Q_N)_N$ is relatively compact for the weak topology $\sigma(\mathbb L^1,\mathbb L^\infty)$ and that any weak limit $Q^*$ satisfies $H(Q^*|P)< +\infty$ thanks to the lower semi-continuity of relative entropy. 

In addition according to \cite{BDP} Lemma 2.12 any weak limit (in $\mathbb L^p(\mathbb R^+\times \mathbb R^2)$) of the marginals flow $\rho_{.,N}$ of $\bar Q_N$ is a free energy solution of the Keller-Segel equation \eqref{eqKS2}. According to Theorem \ref{thm-unique} this solution is unique, so that the marginals flow $\rho_{.,N}$ weakly (in the previous sense) converges to $\rho_.$.

Consider $F(\omega)=\int_0^T \, f(\omega_t) \, dt$ for a bounded $f$. Since some subsequence ($\eta(N)$) of $\bar Q_N$ converges to $Q^*$ it holds $$\mathbb E^{Q^*}(F(\omega)) = \lim_N \, \mathbb E^{\bar Q_{\eta(N)}} (F(\omega)) = \lim_N \, \int_0^T \, \int \, f(x) \, \rho_{t,\eta(N)}(x) \, dx \, dt = \int_0^T \, \int \, f(x) \, \rho_{t}(x) \, dx \, dt \, .$$ This shows that the marginals flow of $Q^*$ is $\rho_.$. A consequence is that for any $t>0$, $\rho_{t,N}$ converges to $\rho_t$ for the $\sigma(\mathbb L^1,\mathbb L^\infty)$ topology, slightly improving on the result in \cite{BDP}.
\medskip

In order to show that $Q^*=\bar Q$, it remains to show that $Q^*$ is a solution of the linear SDE, $$dY_t = \sqrt 2 \, dB_t - (K*\rho_t) (Y_t) dt \, .$$ As for the proof of Theorem \ref{thmkryl} this amounts to prove that 
\begin{equation}\label{eqkrylKS}
\lim_N \, \mathbb E^{\bar Q_N} \left[\left(\int_s^t \, \langle \nabla \varphi(\omega_u),(K_N*\rho_{u,N}-K*\rho_{u})(\omega_u)\rangle  \, du \right) \, H(\omega_{v\leq s})\right] = 0 \, ,
\end{equation}
for smooth $\varphi$ and bounded $H$. We will not follow this way and come back to \eqref{eqentropcutoff}.
\medskip

Thanks to Lemma 2.11 in \cite{BDP}, one can easily check that one can use the results of \cite{FM} replacing $\rho_t$ by $\rho_{t,N}$. In particular Lemma 2.7 in \cite{FM} tells us that, for any $p\in [2,+\infty)$, as soon as $\rho_0 \in \mathbb L^p(\mathbb R^2)$ then $$\sup_N \, \sup_{t \in [0,T]} \, ||\rho_{t,N}||_p  \, < +\infty \, .$$  By interpolation, since we are dealing with probability densities, the same holds with any $1\leq p'\leq p$. We will assume that $\rho_0$ satisfies such a condition in the sequel.
\medskip

Once remarked that $|K_N-K|= |K_N-K|\, \mathbf 1_{|K|\geq A_N}$ for some $A_N$ going to infinity, we see that $K_N-K \in \mathbb L^r(\mathbb R^2)$ for all $1\leq r<2$ and that $||K_N -K||_r \to 0$ as $N$ growths to infinity. Using $r=1$, it follows that $||(K_N-K)*\rho_{t,N}||_{p'} \to 0$ for all $1\leq p'\leq p$. Taking $p'=2+2\alpha$ for some $\alpha>0$ and using H\^{o}lder inequality we have 
\begin{eqnarray*}
\int_0^T \, \int \, |K_N*\rho_{t,N} - K*\rho_{t,N}|^2 \, \rho_t \, dx \, dt \, &\leq& \, T \, ||(K_N-K)*\rho_{t,N}||^2_{p'} \, \sup_{0\leq t\leq T}||\rho_t||_{1+(1/\alpha)} \\  &\leq&  \, T \, ||(K_N-K)||^2_1 \, \sup_{N,t}||\rho_{t,N}||^2_{p'} \, \sup_{t}||\rho_t||_{1+(1/\alpha)}
\end{eqnarray*}
that goes to $0$ provided $\frac{1+\alpha}{\alpha} \leq p$ and $2(1+\alpha) \leq p$, which is always possible as soon as $p\geq 3$. 

This tackles the first term in \eqref{eqentropcutoff}.
\medskip

For the second term, first choose some $A$  such that $||K \, \mathbf 1_{|K|>A}||_1 \, \leq \varepsilon$. As for the first term we have, taking $\alpha=1/2$ in the computations, $$\int_0^T \, \int \, |K \mathbf 1_{|K|>A} \, *(\rho_{t,N} - \rho_t)|^2 \, \rho_t \, dx \, dt \, \leq \, T \, \varepsilon^2 \, 4(\sup_{N,t}||\rho_{t,N}||^2_{3} + \sup_{t}||\rho_t||^2_{3}) \, \sup_{t}||\rho_t||^2_{3}$$ which is a constant (independent of $N$ and $\varepsilon$ of course) times $\varepsilon^2$. 

If we look now at the remaining term $$\int_0^T \, \int \, |K \mathbf 1_{|K|\leq A} \, *(\rho_{t,N} - \rho_t)|^2 \, \rho_t \, dx \, dt$$ we may argue as follows. First, since for each $t$, $\rho_{t,N}$ converges to $\rho_t$ for the $\sigma(\mathbb L^1,\mathbb L^\infty)$ topology, $(K \mathbf 1_{|K|\leq A} \, *(\rho_{t,N} - \rho_t))(x)$ goes to $0$ for all $(t,x)$. In addition, since $K \in \mathbb L^1$, the same estimate as before shows that $|K \mathbf 1_{|K|\leq A} \, *(\rho_{t,N} - \rho_t)|^2$ is bounded in $\mathbb L^{3/2}(\rho_t \, dx \, dt)$, hence uniformly integrable. We may thus apply Vitali's convergence theorem to get that $$\lim_N \, \int_0^T \, \int \, |K \mathbf 1_{|K|\leq A} \, *(\rho_{t,N} - \rho_t)|^2 \, \rho_t \, dx \, dt \, = \, 0 \, .$$

We have finally obtained $$\limsup_N \, \int_0^T \, \int \, |K \, *(\rho_{t,N} - \rho_t)|^2 \, \rho_t \, dx \, dt \, \leq \, C \, \varepsilon^2$$ from which we deduce the convergence to $0$ since $\varepsilon$ is arbitrary. We have thus shown:

\begin{proposition}\label{propconvtroncKS}
In the situation of Theorem \ref{thm-exist-edp}, for an initial condition satisfying \eqref{eqepsilcut}, for the cut-off $K_N$ introduced in \cite{BDP} and satisfying \eqref{eqcondborne}, if in addition $\rho_0 \in \mathbb L^3(\mathbb R^2)$, $d_{TV}(Q_N^{k,N},\bar Q^{\otimes k})$ goes to $0$ as $N$ goes to infinity.
\end{proposition}

\begin{remark}\label{remKStronc}
The weak convergence to $0$ of $\rho_{t,N}$ for all $t$ is a key element of the proof, since we have to manage the square of $|K \, *(\rho_{t,N} - \rho_t)|$. It does not follow from the results in \cite{BDP} nor the boundedness in an $\mathbb L^p$ as shown in \cite{FM}, the latter furnishing for each $t$ convergence of subsequences.
\medskip

We have shown convergence in entropy to get at the end some convergence in total variation distance, which is stronger than the weak convergence one can expect if we follow the proof of Theorem \ref{thmkryl} as we mentioned in \eqref{eqkrylKS}. In addition the fact that the expectation to consider is the one w.r.t. $\bar Q_N$ introduces some intricacies.
\medskip

The second part of the proof prevents us from getting a quantitative rate of convergence. 
\hfill $\diamondsuit$
\end{remark}

\begin{remark}\label{remKStroncref}
Proposition \ref{propconvtroncKS} can be compared with \cite{ORT} but also, more directly with \cite{Liucut}. In the latter reference a very similar result is obtained in Theorem 1.3, but in Wasserstein $W_1$ distance. It requires in particular to show that the non linear S.D.E. in Theorem \ref{thmexistnl1} has a strong solution, which is shown by using the same cut-off as the one we are using and convergence. The main difference is that $\rho_0$ is assumed to be bounded in \cite{Liucut} and it is shown that under this assumption $\sup_t ||\rho_t||_\infty < +\infty$ (Thm 1.1 (ii)). We confess that we do not understand all the arguments in the Appendix of \cite{Liucut} in order to get this result,  which is crucial for this approach, even if we can easily believe that the result is true. Nevertheless, it is shown in \cite{FM} Lemma 2.8 that $\rho_. \in C_b^\infty([\varepsilon,T]\times \mathbb R^2)$ for all $\varepsilon > 0$, so that all the job consists in the understanding of what happens for small $t$'s.
\hfill $\diamondsuit$
\end{remark}
\medskip

\textbf{The 2D vortex model.} \quad Under the simplifying assumptions of Proposition \ref{propedpvortex} it is know that $\sup_t \, ||\rho_t||_\infty \leq ||\rho_0||_\infty$ (see \cite{GLBM-biot} proof of Theorem 2 for a clever argument). Actually in \cite{Mapul} the authors are using a regularizing cut-off (see the top of p.487 therein), say $K_N$, and prove that the same result is true for the regularized equation i.e. $$ \sup_N \, \sup_t \, ||\rho_{t,N}||_\infty \leq ||\rho_0||_\infty \, .$$ One can thus easily adapt the proof we have done for the Keller-Segel model, since this time we know the uniqueness of the solution for the non linear P.D.E. in the set of bounded solutions (see e.g \cite{GLBM-biot} Theorem 2). Details are left to the reader. It follows
\begin{proposition}\label{propconvtroncbiot}
In the situation of Proposition  \ref{propedpvortex}, for an initial condition satisfying \eqref{eqepsilcut}, for the cut-off $K_N$ introduced in \cite{Mapul} and satisfying \eqref{eqcondborne}, $d_{TV}(Q_N^{k,N},\bar Q^{\otimes k})$ goes to $0$ as $N$ goes to infinity.
\end{proposition} 
\begin{remark}\label{remconvtroncbiot}
Compared with the existing literature the previous proposition only improves on the convergence type, since we are looking at the full law of the process and not only at the marginals flow as in \cite{FHM,GLBM-biot,FWvortex}. It is however weaker for at least three reasons: we are using a cut-off and not looking at the \emph{true} particle system, the initial condition is more general in \cite{FHM}, \cite{GLBM-biot,FWvortex} contain quantitative results.
\hfill $\diamondsuit$
\end{remark}
\medskip

We firmly believe that the situation is exactly the same for the sub-Coulombic repulsive model. 
\medskip

\section{Large number of particles: more on convergence and chaos.}
\label{secchaos2}

We continue with the notations of the previous section. In the latter section, the point was to obtain convergence to $0$ for the relative entropy $H(Q^{k,N}|\bar Q^{\otimes k})$. Another standard approach is first to show the weaker tightness of the family $(Q^{k,N})_N$, then to show that there is only one possible weak limit and identify this weak limit as $\bar Q^{\otimes k}$. This is the strategy adopted for instance in \cite{Toma-p} for the $\mathbb L^p-\mathbb L^q$ case,  
in \cite{FHM} at the level of the marginals flow and in \cite{FJ,Tar} for the Keller-Segel model. In these two references only tightness is shown together with a partial identification of the possible limits. This will be discussed later.
\medskip

\subsection{\textbf{Tightness.} }\label{subsectight}

\begin{theorem}\label{thmtight}
Let $P$ denotes the reversible probability measure on $\mathbb R^d$, with invariant probability measure $\gamma_0(dx)= Z^{-1} \, e^{- \sum_{j=1}^d \, V(x^j)} dx$. Here $V$ is smooth, non-negative, and $|V'|$ is assumed to be bounded. Assume that $H(\mu_0^{N}|\gamma_0^{\otimes N}) \leq C N$ ($\mu_0^{N}$ being exchangeable as before). 

In each of the following situations
\begin{enumerate}
\item[(1)] \; \textbf{$\mathbb L^p$-case.} \quad $K \mathbf 1_{|K|>A} \in \mathbb L^p(\mathbb R^d)$ for some $p \geq d$ if $d\geq 3$ or $p>2$ if $d=2$ and some $A>0$. The additional drift $b$ is bounded.
\item[(2)] \; \textbf{sub-Coulombic case.} \quad $d \geq 3$, $\chi <0$, $b$ is bounded and Lipschitz, and $$K(x)= \chi \, \frac{x}{|x|^{s+2}} \, \mathbf 1_{x\neq 0} \; \textrm{ for } \; 0<s \leq d-2 \, .$$ In addition for some $q>d/(d-s)$, $\sup_N ||\tilde \rho_{0,1,2}^N||_q < +\infty$, where $\tilde \rho_{0,1,2}^N$ is the density of $X_0^{1,N}-X_0^{2,N}$.
\item[(3)] \; \textbf{The 2D vortex case with confinement.} \quad $d=2$, $K(x)=\chi \, \frac{x^\perp}{|x|^2}$ and $b=-\nabla \tilde V$ derives from the confining potential as in Theorem \ref{thmbiotexistconfine}.
\end{enumerate}
Then for all fixed $k$, $\sup_{N>k} \, H(Q^{k,N}|\bar P^{\otimes k}) \leq \, C \, k(k+1) < +\infty$ for some constant $C$ that does not depend on $k$. 

Consequently, the family $(Q^{k,N})_N$ is tight.
\end{theorem}
\begin{proof}
It holds $$H(Q^{N}|P) = H(Q^{N}|\frac{d\mu_0^N}{\gamma_0^{\otimes N}} \, P) + H(\mu_0^N|\gamma_0^{\otimes N})$$ so that, in each situation, thanks to the assumption on the initial distribution and to the calculation in the proof of Lacker's theorem for case (1) (since \eqref{eqMLack} is satisfied in this case), to Remark \ref{remcoulN} for case (2) and to Theorem \ref{thmbiotexistconfine} for case (3), we have $H(Q^{N}|P) \leq C N$ for some $C$ that does not depend on $N$. Since $P$ is a product measure we may apply the trick \eqref{eqentropproj} yielding $H(Q^{k,N}|\bar P^{\otimes k}) \leq C \, \frac{kN}{N-k} \leq C \, k(k+1)$.
\end{proof}
In the $\mathbb L^p$-case, this result is new for the critical $p=d$. Another nice proof, based on the $\mathbb L^q$ (for any $q<+\infty$) integrability of some Girsanov density for $p>d$, is made in \cite{Toma-p} Lemma 3.2 (see Remark \ref{remtight1} for a comment on this result). 

In the 2D vortex case without confinement, the previous result fails since we only knows that $H(Q^N|P) \leq C N^2$. Of course when collisions occur, as for the Keller-Segel model, the previous approach also fails. One can thus come back to the standard approach using Prokhorov's tightness criterion. Namely, tightness of $(\mathbb Q^{1,N})_N$ will follow from the following estimates: for some $N_0$,
\begin{eqnarray}\label{eqprokh}
(i) \, && \, \lim_{a \to +\infty} \, \limsup_{N\geq N_0} \,  \mathbb P(|X_0^{1,N}|>a)= \, 0 \, , \nonumber \\ (ii) \, && \, \forall \varepsilon >0 \, , \lim_{\eta \to 0} \,  \limsup_{N\geq N_0} \, \mathbb P\left(\sup_{|t-s|<\eta} \, |X_t^{1,N}-X_s^{1,N}| \geq \varepsilon\right) = \, 0 \, .
\end{eqnarray}
An analogue is true for $Q_N^{k,N}$, $k$ being fixed and $N_0 \geq k$. 
\medskip

For condition (ii) we may start with 
\begin{eqnarray*}
\mathbb P\left(\sup_{|t-s|<\eta} \, |X_t^{1,N}-X_s^{1,N}| \geq \varepsilon\right) &\leq &\mathbb P\left(\sup_{|t-s|<\eta} \, |B_t^{1,N}-B_s^{1,N}| \geq \varepsilon/3\sqrt 2\right) \\ && + \, \mathbb P\left(\sup_{|t-s|<\eta} \left(\int_s^t \, |b(X_u^{1,N})| du\right) \geq \varepsilon/3\right) \\ && \, + \, \mathbb P\left(\sup_{|t-s|<\eta} \left|\int_s^t \, \frac 1N \, \sum_{j=1}^N \, K(X_u^{1,N}-X_u^{j,N}) \, du\right| \geq \varepsilon/3\right) \, .
\end{eqnarray*}
We then have
\begin{eqnarray*}
\mathbb P\left(\sup_{|t-s|<\eta} \, |B_t^{1,N}-B_s^{1,N}| \geq \varepsilon/3\sqrt 2\right)
&\leq& \frac{3\sqrt 2}{\varepsilon} \, \mathbb E\left(\sup_{|t-s|<\eta} \, |B_t^{1,N}-B_s^{1,N}|\right) \\ &\leq&  \frac{3\sqrt 2}{\varepsilon} \, c_{univ} \, (\eta \, \ln(T/\eta))^\frac 12 \, .
\end{eqnarray*}
provided $\eta < T/e$, where $c_{univ}$ is an universal constant. The last inequality follows from $|z| \leq |z^1| + |z^2|$ in $\mathbb R^2$ and \cite{modulus} Lemma 3 or Lemma 4 where two different proofs are given for the moments controls of the modulus of continuity of a linear Brownian motion. 

Next
\begin{eqnarray*}
\mathbb P\left(\sup_{|t-s|<\eta} \left(\int_s^t \, |b(X_u^{1,N})| du\right) \geq \varepsilon/3\right) \leq \, \mathbf 1_{3 ||b||_\infty \eta > \varepsilon}
\end{eqnarray*}
when $b$ is bounded (in particular $b=0$ for the $2D$ vortex model). If $b$ is $a$-Lipschitz, non necessarily bounded, we have $|b(y)| \leq |b(0)| + a |y|$ so that, for any $p>1$, $$\sup_{|t-s|<\eta} \left(\int_s^t \, |b(X_u^{1,N})| du\right) \leq \, |b(0)| \, \eta + \, a \, \eta^{\frac {p-1}{p}} \, \left(\int_0^T \, |X_u^{1,N})|^p \,  du\right)^{\frac 1p} \, .$$
We may thus write as soon as $\eta |b(0)| < \varepsilon /6$
\begin{eqnarray*}
\mathbb P\left(\sup_{|t-s|<\eta} \left(\int_s^t \, |b(X_u^{1,N})| du\right) \geq \varepsilon/3\right)&\leq& \left(\frac{6 a}{\varepsilon}\right)^p \, \eta^{p-1} \,  \sup_N \, \mathbb E\left(\int_0^T \, |X_u^{1,N})|^p \,  du\right).
\end{eqnarray*}
The desired control will be satisfied as soon as the last term in the right hand side is bounded.
\medskip

For the last term $$\mathbb P\left(\sup_{|t-s|<\eta} \left|\int_s^t \, \frac 1N \, \sum_{j=1}^N \, K(X_u^{1,N}-X_u^{j,N}) \, du\right| \geq \varepsilon/3\right) \leq \frac{3}{\varepsilon} \, \mathbb E\left(\sup_{|t-s|<\eta} \, \left|\frac 1N \, \sum_{j=1}^N  \,  \int_s^t \, K(X_u^{1,N}-X_u^{j,N}) \, du\right|\right) \, .$$
For $1<\gamma$, $$\int_s^t \, \left|\frac 1N \, \sum_{j=1}^N K(X_u^{1,N}-X_u^{j,N})\right| \, du \, \leq \, (\frac{t-s}{T})^{\frac{\gamma-1}{\gamma}} \, \left(\int_0^T \, \left|\frac 1N \, \sum_{j=1}^N \, K(X_u^{1,N}-X_u^{j,N})\right|^\gamma \, du\right)^{\frac 1\gamma}$$ so that, if 
\begin{equation}\label{eqtightentrop}
\int_0^T \, \mathbb E\left(\left|\frac 1N \, \sum_{j=1}^N \, K(X_u^{1,N}-X_u^{j,N})\right|^\gamma\right) \, du \leq M(\gamma),
\end{equation}
in particular, thanks to convexity and exchangeability, if 
\begin{equation}\label{eqtightentrop2}
\int_0^T \, \mathbb E \left[\left|K(X_u^{1,N}-X_u^{j,N})\right|^\gamma\right] \, du \leq M'(\gamma)
\end{equation}
for all $N>N_0$ large enough, condition (ii) in \eqref{eqprokh} is satisfied. Notice that this condition is stronger than the $\mathbb L^1$ integrability required to control the second term for a Lipschitz additional drift.
\medskip

It is immediate to see that the previous discussion extends to $(X_.^{1,N},...,X_.^{k,N})$ for all fixed $k$.

\begin{remark}\label{remtight1}
The previous proof applies in all the cases of Theorem \ref{thmtight} since condition \eqref{eqtightentrop}  is satisfied with any $\gamma<d/(d-1)$. There are some differences on the assumptions on the initial condition (only tightness of the initial $\mu_0^N$ is required here). Of course the relative compactness in Theorem \ref{thmtight} holds for the weak $\sigma(\mathbb L^1,\mathbb L^\infty)$ topology and is thus a little bit stronger.

Notice that in Lemma 3.2 in \cite{Toma-p}, the supremum over $(s,t)$ is outside the expectation, but its size $(t-s)^2$ is small enough thanks to Kolmogorov continuity criterion. We cannot use this tightness result here since the power of $|t-s|$ in our controls can be smaller than $1$.

The line of reasoning used for this derivation is simpler but close to the one in \cite{FHM} proof of Lemma 5.2. 
\hfill $\diamondsuit$
\end{remark}

We may thus state
\begin{theorem}\label{thmtight2}
Assume that the initial measure is exchangeable and satisfies $$ \sup_{N \geq N_0} \, \int \, |x^1|^2 \, \rho_0^{N}(x) \, dx \, < + \, \infty \, .$$ Also assume that $b$ is bounded or Lipschitz. In each of the following situations
\begin{enumerate}
\item[(1)] \; \textbf{The 2D vortex case.} \quad $d=2$, $K(x)=\chi \, \frac{x^\perp}{|x|^2}$,
\item[(2)] \; \textbf{The Keller-Segel case.} \quad $d=2$, $K(x)=\chi \, \frac{x}{|x|^2}$,   $\chi<2$, 
\end{enumerate}
the family $(Q^{k,N})_N$ is tight.
\end{theorem}
\begin{proof}
The Keller-Segel case is an immediate consequence of the previous line of reasoning and Lemma \ref{lemKSFJ}, once remarked that the additional drift only needs to be bounded or Lipschitz and not necessary gradient for this Lemma. 

For the Biot-Savart kernel, it is enough to look at the calculations of subsection \ref{subsecvortex}, in particular equation \eqref{eqbiotpower} who shows that $$\sup_{N \geq N_0} \, \int_0^T \, \mathbb E \left[\left|K(X_u^{1,N}-X_u^{2,N})\right|^{\gamma}\right] \, du \, < \, + \infty$$ for all $0\leq \gamma <2$.
\end{proof}
\begin{remark}\label{remtight2}
For the Keller-Segel model a similar result is stated in \cite{FJ} for $\chi<1$. Despite the fact that the authors are proving at the same time convergence of the regularized model we discussed in section \ref{subsecKS} and tightness, the proofs are very close to what we have done (and conversely), see e.g. the first inequality at the top of p. 2821 therein. Our use of \cite{modulus} simplifies the argument, and the extension of Lemma \ref{lemKSFJ} allows us to cover the larger range $\chi<2$.

For the 2D vortex model, tightness is shown in Lemma 5.2 of \cite{FHM} also using very similar arguments. 
\hfill $\diamondsuit$
\end{remark}

It remains a gap to fill for the Keller-Segel model, namely $2\leq \chi<4$. The situation is easier at the marginals flow level. The following is Theorem 4 (i) in \cite{Tar}
\begin{theorem}\label{thmtardy}
Assume that the initial empirical measure $\bar \mu_0^N = \frac 1N \, \sum_{j=1}^N \, \delta_{X_0^{j,N}}$ weakly converges to some $\bar \rho_0 \, dx$ in probability. 

Then, in the Keller-Segel model, for all $\chi<4$, the sequence $(\bar \mu_t^N)_{t \in [0,T]} = (\frac 1N \, \sum_{j=1}^N \, \delta_{X_t^{j,N}})_{t \in [0,T]}$ is tight.

Consequently, the sequence $(\rho_t^{1,N})_{t \in [0,T]}$ (first marginal flow) is also tight.
\end{theorem}
The last statement is a consequence of \cite{Sznit} Proposition 2.2 (ii).

Let us recall some definitions in the previous statement. 
\begin{enumerate}
\item[(1)] \quad A sequence of random measures (like $\bar \mu_0^N$) converges to a given deterministic measure $\nu_0$ in probability, if for all $\varepsilon > 0$, $$\mathbb P(\delta(\bar \mu_0^N,\nu_0)>\varepsilon) \to \, 0 \; \textrm{ as } \; N \to +\infty$$ where $\delta$ is any distance associated to the weak convergence of probability measures on $\mathbb R^2$.
\item[(2)] \quad A sequence of  flows of random probability measures (like $t \mapsto \bar \mu_t^N$ for $t\in [0,T]$) is tight, if for all $\varepsilon >0$, one can find a compact subset $A_\varepsilon$ of $C^0([0,T],\mathcal M^1(\mathbb R^2))$, such that $\mathbb P(A_\varepsilon^c) \leq \varepsilon$, where $\mathcal M^1(\mathbb R^2)$ denotes the set of probability measures on $\mathbb R^2$.
\end{enumerate}
Since $C^0([0,T],\mathcal M^1(\mathbb R^2))$ is Polish, the tightness of $(\bar \mu_t^N)_{t \in [0,T]}$ is equivalent to the tightness of the law of $(X_t^{1,N})_{t \in [0,T]}$ (\cite{Sznit} proposition 2.2 (ii)), as well as the law of $(X_t^{1,N},...,X_t^{k,N})_{t \in [0,T]}$ for any fixed $k$.

Let us give the main ingredients of the proof. 

In order to prove the above tightness, it is sufficient to show that, there exists some non-decreasing function $\psi : \mathbb R^+ \mapsto \mathbb R^+$ with $\psi(0)=0$, such that for any $\varphi \in C_b^2(\mathbb R^2)$ such that $$||\varphi||_\infty + ||\nabla \varphi||_\infty + ||\nabla^2 \varphi||_\infty \leq 1$$ it holds 
\begin{equation}\label{eqtightflow}
\sup_{N\geq N_0} \, \mathbb E\left[\sup_{0\leq s\leq t \leq T, t-s\leq \eta} \, \left|\int \varphi \, d\bar \mu_t^N - \int \varphi \, d\bar \mu_s^N\right| \right] \, \leq \, \psi(\eta) \, .
\end{equation}
For the details of this claim see \cite{Tar} proof of Theorem 4 (i) p. 12.

Now $$\left|\int \varphi \, d\bar \mu_t^N - \int \varphi \, d\bar \mu_s^N\right| \leq U_s^t+V_s^t+ W_s^t +T_s^t$$ where $$U_s^t = \frac 1N \, \left|\sum_i \, \int_s^t \, \langle \nabla \varphi(X_u^{i,N}),dB_u^{i,N}\rangle\right| \, ,$$ $$V_s^t=\frac 1N \, \left|\sum_i \, \int_s^t \, \langle \nabla \varphi(X_u^{i,N}),b(X_u^{i,N})\rangle du \right| \, ,$$ 
$$W_s^t=\frac 1N \, \left|\sum_i \, \int_s^t \, \Delta \, \varphi(X_u^{i,N}) \, du \right| $$ and $$T_s^t = \frac {1}{N} \, \left|\sum_{i\neq j} \, \int_s^t \, \frac 12 \, \frac{\langle \nabla \varphi(X_u^{i,N})- \nabla \varphi(X_u^{j,N}), X_u^{i,N}-X_u^{j,N}\rangle}{|X_u^{i,N}-X_u^{j,N}|^2} \, du \right| \, .$$ As before we obtain the last term by exchanging the roles of $i$ and $j$. Using the assumptions on $\nabla \varphi$ and $\nabla^2 \varphi$, $$U_s^t+V_s^t+T_s^t \leq (||b||_\infty + 1 + 1/2) \, |t-s|$$ while, according to \cite{modulus} Theorem 1, $$\mathbb E\left[\sup_{0\leq s\leq t \leq T, t-s\leq \eta} \, W_s^t\right] \leq c \, (\eta \, \ln(T/\eta))^{\frac 12} \, ,$$ yielding the result with $\psi(\eta)=C \, (\eta \, \ln(T/\eta))^{\frac 12}$. 
\medskip

\medskip

\subsection{\textbf{Towards consistency.} \\ \\}\label{subsectowards}

In this subsection we shall show that any weak limit in Theorems \ref{thmtight2}, \ref{thmtight}, is a solution of the corresponding non linear equation. The method of proof is almost standard (one can find some similar arguments in \cite{FHM} for instance). We shall nevertheless give the main elements of proof, in order to fill some gaps in previous works, and to show how the use of entropy on the path space simplifies many arguments.
\begin{proposition}\label{propconvnl}
In any of the situations of Theorem \ref{thmtight} and Theorem \ref{thmtight2} (assuming in addition that $b$ is continuous in the $\mathbb L^p$ case), introduce $$\mathcal Q^N = \frac 1N \, \sum_{j=1}^N \, \delta_{X^{j,N}_.}$$ where $X^{.,N}_.$ denotes the whole path of the process on the time interval $[0,T]$. Then $\mathcal Q^N$ is tight and any weak limit $\mathcal Q$ is almost surely a solution of the non linear S.D.E. \eqref{eqnldiff}.
\end{proposition}
\begin{proof}
The statement on tightness follows from proposition 2.2 (ii) in \cite{Sznit}. 

Take some subsequence, still denoted by $\mathcal Q^N$, that weakly converges to some $\mathcal Q$. Recall that $\mathcal Q$ is a random variable taking values in the set of probability measures on the path space. In order to prove that $\mathcal Q$ almost surely solves \eqref{eqnldiff}, it is enough to look at the associated martingale problem. Take some $0\leq s <t \leq T$. Following a now standard method (see \cite{BoTal96} for example)  for any given continuous and bounded $h$ defined on $C^0([0,s],\mathbb R^d)$, any $\varphi \in C^0_b(\mathbb R^d)$ and any probability measure $\nu$ defined on $C^0([0,T],\mathbb R^d)$ we introduce $$\mathcal F(\nu) = \, \int \int \, \left[\varphi(\omega_t)-\varphi(\omega_s) - \int_s^t \, \Delta \varphi(\omega_u) du \, - \, \int_s^t \, \langle b(\omega_u)+K(\omega_u - \omega'_u),\nabla \varphi(\omega_u)\rangle  \, du\right]$$ $$ h(\omega_.) \, \nu(d\omega) \, \nu(d\omega') \, .$$ The statement of the Theorem amounts to 
\begin{equation}\label{eqmartin}
\mathcal F(\mathcal Q)=0 \, \textrm{ almost surely .}
\end{equation}
In order to prove \eqref{eqmartin} the first step is that
\begin{equation}\label{eqmartin2}
\mathbb E[(\mathcal F(\mathcal Q^N))^2] \leq C(\varphi,h,s,t) \, \frac 1N \, .
\end{equation}
It turns out that $$\mathcal F(\mathcal Q^N)= \frac 1N \, \left(\sum_{j=1}^N \, h(X_.^{.,N}) \, (A^{j,N}(t)-A^{j,N}(s))\right)$$ where 
\begin{eqnarray*}
A^{j,N}(t)-A^{j,N}(s)&=&\varphi(X_t^{j,N})-\varphi(X_s^{j,N}) - \int_s^t \, \Delta \varphi(X_u^{j,N}) du \\ && \, - \, \int_s^t \left\langle b(X_u^{j,N}) + \frac 1N \, \sum_{i \neq j} K(X_u^{j,N} - X_u^{i,N}),\nabla \varphi (X_u^{j,N})\right\rangle du \\ &=& \int_s^t \, \langle \nabla \varphi (X_u^{j,N}),dB_u^{j,N}\rangle
\end{eqnarray*}
is thus a square integrable martingale increment. \eqref{eqmartin2} immediately follows using standard stochastic calculus.

Since we always assume that $b$ is bounded and continuous, if $K$ is also bounded and continuous, then $\mathcal F$ is continuous and bounded for the topology of weak convergence, so that $\mathbb E(|\mathcal F(\mathcal Q)|)=\lim_N \, \mathbb E(|\mathcal F(\mathcal Q^N)|) = 0$. But in all the cases we are interested in, $K$ is not bounded. 

For $\varepsilon >0$, let $\psi_\varepsilon(u)$ be a continuous function on $\mathbb R^+$ such that  $\mathbf 1_{u\leq 2 \varepsilon} \leq \psi_\varepsilon(u) \leq \mathbf 1_{u\leq \varepsilon}$. In all the cases except the $\mathbb L^p$ case, we replace $K$ by $K_\varepsilon=K \, \psi_\varepsilon(1/|K|)$ in order to define $\mathcal F_\varepsilon$. In the $\mathbb L^p$ case we have first to regularize $K$ into a continuous function, this regularization inducing a small error term in all what follows and then pass to the limit. We will not detail this step, assume that $K$ is always continuous in the sequel, and make the previous approximation. 

Since $\mathcal F_\varepsilon$ is continuous and bounded, it  holds
\begin{equation}\label{eqmartin3}
\mathbb E(|\mathcal F_\varepsilon(\mathcal Q)|)=\lim_N \, \mathbb E(|\mathcal F_\varepsilon(\mathcal Q^N)|) \, .
\end{equation}

It remains to control $$\mathbb E(|\mathcal F(\mathcal Q^N) - \mathcal F_\varepsilon(\mathcal Q^N)|) \; \textrm{ and } \; \mathbb E(|\mathcal F(\mathcal Q) - \mathcal F_\varepsilon(\mathcal Q)|) \, .$$ As we already remarked, in all these cases, there exists some $\gamma>1$ such that $$\sup_N \int_0^T \int \int  |K|^\gamma(\omega_u - \omega'_u) \, \mathcal Q^N(d\omega_u) \, \mathcal Q^N(d\omega'_u) \, du \, < +\infty \, .$$ Using H\"{o}lder and Markov inequalities, we deduce that
$$ \sup_N \int_0^T \int \int  |K - K_\varepsilon|(\omega_u - \omega'_u) \, \mathcal Q^N(d\omega_u) \, \mathcal Q^N(d\omega'_u) \, du \, \leq \,  C(T,\gamma) \, \varepsilon^{\gamma-1} \, .$$

The same is true for $|K_\eta-K_\varepsilon|$ for any $\eta>0$, so that taking the limit in $N$, it is still true for $|K_\eta-K_\varepsilon|$ replacing $\mathcal Q^N$ by $\mathcal Q$, and finally for $|K-K_\varepsilon|$ using monotone convergence.

We have thus shown that both $\sup_N \, \mathbb E(|\mathcal F(\mathcal Q^N) - \mathcal F_\varepsilon(\mathcal Q^N)|)$ and $\mathbb E(|\mathcal F(\mathcal Q) - \mathcal F_\varepsilon(\mathcal Q)|)$ go to $0$ as $\varepsilon \to 0$. 

Finally, using the triangle inequality $$\mathbb E(|\mathcal F(\mathcal Q)| \leq C \, (\frac{1}{\sqrt N}+ 2 \varepsilon^\theta + \mathbb E(|\mathcal F_\varepsilon(\mathcal Q^N) - \mathcal F_\varepsilon(\mathcal Q)|)$$ so that choosing first $\varepsilon$ small enough, and then $N$ large enough, the left hand side is arbitrarily small, i.e. $\mathbb E(|\mathcal F(\mathcal Q)|)=0$ so that $\mathcal F(\mathcal Q)=0$ almost surely.
\end{proof}
\begin{remark}\label{remboundconv}
In order to simplify several proofs we almost always assumed that $b$ is bounded. Actually, provided $\sup_t \, \mathbb E(X_t^{1,N}) \leq C < +\infty$ where $C$ does not depend on $N$, we may skip the bounded assumption and only assume that $b$ is Lipschitz (hence with linear growth), using another cut-off if necessary. Part of the arguments have been developped at the end of the previous subsection. This is the case in particular for the Keller-Segel and the 2D vortex models with a gaussian confining potential. 
\hfill $\diamondsuit$
\end{remark}
As for Theorem \ref{thmtardy}, in the general Keller-Segel model we have the analogue result at the marginals flow level
\begin{proposition}\label{proptardy2}
In the situation of Theorem \ref{thmtardy} any weak limit of $(\bar \mu_t^N)_{t \in [0,T]}$ is almost surely a solution of the equation \eqref{eqmcv}. In particular, any weak limit of $(\rho_t^{1,N})_{t \in [0,T]}$ is a solution of \eqref{eqmcv}.
\end{proposition}
As for the proof of Theorem \ref{thmtardy} the key is that the ``difficult'' terms can be written as $$\int_0^T \, \int \int \, \langle K(x-y),\nabla \varphi(x)-\nabla \varphi(y)\rangle \, \bar \mu_s^N(dx) \, \bar \mu_s^N(dy) \, ds$$ where $$(x,y) \mapsto \langle K(x-y),\nabla \varphi(x)-\nabla \varphi(y)\rangle$$ is bounded, and is continuous outside $x=y$ which is of zero measure w.r.t. $\bar \mu_s\otimes \bar \mu_s \, ds$, where $\bar \mu_.$ is the corresponding limiting measure (see \cite{Tar} step 2 of the proof of Theorem 4 (ii) for more details).
\bigskip

Proposition \ref{propconvnl} contains an interesting existence corollary, extending some existence results we have already recalled or proved
\begin{corollary}\label{corexistnlraf}
Assume that $H(\bar \rho_0 dx|\gamma_0)<+\infty$ where $\bar \rho_{0}$ is the density of $\bar X_0$.  The non linear S.D.E. \eqref{eqnldiff} has a solution, with initial distribution $\bar \rho_0 dx$  in the following cases
\begin{enumerate}
\item[(1)] \; \textbf{sub-Coulombic case.} \quad $d \geq 3$, $\chi <0$, , and $$K(x)= \chi \, \frac{x}{|x|^{s+2}} \, \mathbf 1_{x\neq 0} \; \textrm{ for } \; 0<s \leq d-2 \, .$$ In addition for some $q>d/(d-s)$, $||\bar \rho_{0}||_q < +\infty$ and and $b$ is bounded and Lipschitz.
\item[(2)] \; \textbf{The 2D vortex case with confinement.} \quad $d=2$, $K(x)=\chi \, \frac{x^\perp}{|x|^2}$, the additional drift $b$ is a confining potential as in Theorem \ref{thmbiotexistconfine} and $\int |x|^2 \, \bar \rho_0(x) dx < +\infty$.
\item[(3)] \; \textbf{The 2D vortex case.} \quad $d=2$, $K(x)=\chi \, \frac{x^\perp}{|x|^2}$, and $\int |x|^2 \, \bar \rho_0(x) dx < +\infty$.
\end{enumerate}
Consequently the non linear P.D.E. \eqref{eqmcv} has a solution with initial condition $\bar \rho_0$ in all these cases.
\end{corollary}
It is enough to use what we achieved to do with an initial $\rho_0^N=\bar \rho_0^{\otimes N}$, since any limiting $\mathcal Q$ will do the job. Recall that in subsection \ref{subseccoul} in the sub-coulombic situation, or in Proposition \ref{propedpvortex} we assumed that $\bar \rho_0$ is bounded.

Notice that in the $\mathbb L^p$ case and for the Keller-Segel model (with $\chi<4$), we also recover the existence part of Theorem \ref{thmexistnl} and Theorem \ref{thmexistnl1}. In the latter case, we have thus obtained the existence of a solution to the Keller-Segel equation (with $\chi<4$), but we do not know (at this stage) that it is a free energy solution as in Theorem \ref{thm-exist-edp}.
\medskip

\subsection{\textbf{Convergence.} \\ \\}\label{subsecconvergence}

We start by tackling the $\mathbb L^p$ situation.
\begin{theorem}\label{thmlpfin}
Assume that $K \mathbf 1_{|K|>A} \in \mathbb L^p(\mathbb R^d)$ for some $p \geq d$ if $d\geq 3$ or $p>2$ if $d=2$ and some $A>0$ and that the additional drift $b$ is continuous and bounded. Also asume that the initial condition $\mu_0^N=\rho^N \, dx$ is chaotic so that $\mu_0^{k,N} \to (\bar \rho_0 \, dx^1)^{\otimes k}$ and satisfies $H(\mu_0^N|\gamma_0^{\otimes N}) \leq CN$.

Then $Q^N$ is chaotic and for each $k$, $Q^{k,N}$ weakly converges to $(\bar Q)^{\otimes k}$ where $\bar Q$ is the unique solution of the non linear SDE \eqref{eqnldiff} with initial condition $\bar \rho_0 \, dx$.
\end{theorem}
\begin{proof}
According to Proposition \ref{propuniquenl} \eqref{eqnldiff} has a unique weak solution $\bar Q$. Any weak limit in Proposition \ref{propconvnl} is thus equal to the deterministic measure $\bar Q$. According to \cite{Sznit} Proposition 2.2. (i), $Q^N$ is thus chaotic.
\end{proof}
In the framework of a \emph{convolution kernel} $K$, this result extends Theorem 1.6 in \cite{Toma-p} to more general initial conditions and to the critical case $d=p$. Still for $p<d$, \cite{Hao} Theorem 1.1 (i) contains a similar statement, but replacing weak convergence by convergence in Wasserstein $W_2$ distance for one particle, while propagation of chaos is shown in \cite{Hao} Theorem 5.1 by the same method as \cite{Toma-p} (itself inspired by \cite{BoTal96}). The fact that the Girsanov density belongs to all the $\mathbb L^q$ is crucial in these proofs.
\medskip

In order to study the other singular models we start with a simple remark. For $N=2n$, consider the pair empirical measure $$\mathcal Q_2^N= \frac{1}{N} \, \sum_{i=1}^n \, (\delta_{(X_.^{2i-1,N},X_.^{2i,N})} + \, \delta_{(X_.^{2i,N},X_.^{2i-1,N})})$$ which is a symmetric probability measure on $C^0([0,T],\mathbb R^d\otimes \mathbb R^d)$. In a sense we are looking at pairs of coordinates. Of course, if $N$ is odd, we are loosing one term (the last $X^N$ for instance), but it is immediate to see that this error term will disappear in the limit in what follows.

Introduce  the pair functional defined for probability measures $\nu$ on $C^0([0,T],\mathbb R^d\otimes \mathbb R^d)$ by $$\mathcal F_2(\nu) = \,  \int \, \left[\varphi(\omega_t,\eta_t)-\varphi(\omega_s,\eta_s) - \int_s^t \, \Delta \varphi(\omega_u,\eta_u) du \, - \, \int_s^t \, \langle b(\omega_u,\eta_u),\nabla \varphi(\omega_u,\eta_u)\rangle  \, du\right]$$ $$ \quad \quad  h(\omega_.,\eta_.) \, \nu(d\omega,d\eta)  \, - \, $$ $$ - \int \int \, \left[\int_s^t \, \left(\langle K(\omega_u - \omega'_u),\nabla_1 \varphi(\omega_u,\eta_u)\rangle + \langle K(\eta_u - \eta'_u),\nabla_2 \varphi(\omega_u,\eta_u)\rangle\right) \, du\right]$$ $$ \quad \quad h(\omega_.,\eta_.) \, \nu(d\omega,d\eta) \nu(d\omega',d\eta') \, .$$ Notice that $$\int K(\omega_u-\omega'_u) \mathcal Q_2^N(d \omega',d\eta') =  \frac 1N \, \sum_j \, K(\omega_u-X^{j,N}_u)$$ so that the same calculations as in the proof of Proposition \ref{propconvnl},  show that $$\mathcal F_2(\mathcal Q_2^N)= h(X_.^N) \, \frac{\sqrt 2}{N} \, \sum_{i=1}^n \, \int_s^t \, \langle \nabla_1 \varphi(X_u^{2i-1,N},X_u^{2i,N}) + \nabla_2 \varphi(X_u^{2i,N},X_u^{2i-1,N}), dB_u^{2i-1,N}\rangle $$ $$ \quad + h(X_.^N) \, \frac{\sqrt 2}{N} \, \sum_{i=1}^n \, \int_s^t \, \langle \nabla_2 \varphi(X_u^{2i-1,N},X_u^{2i,N}) + \nabla_1 \varphi(X_u^{2i,N},X_u^{2i-1,N}), dB_u^{2i,N}\rangle \, .$$ Hence $\mathbb E[(\mathcal F_2(\mathcal Q_2^N))^2] \leq \, C(\varphi,h,s,t) \, \frac 1N$, and we may thus argue as in the proof of Proposition \ref{propconvnl}. 

The tightness of the distribution $\mathcal{LQ}_2^N$ of $\mathcal Q_2^N$ is a consequence of the one of $Q^{2,N}$ since the intensity measure $I(\mathcal{LQ}_2^N)$ is given by $$\int \, f \, dI(\mathcal{LQ}_2^N) = \mathbb E\left[\frac{1}{N} \, \sum_{i=1}^n \, (f(X_.^{2i-1,N},X_.^{2i,N})+f(X_.^{2i,N},X_.^{2i-1,N}))\right] = \mathbb E[f(X_.^{1,N},X_.^{2,N})] \, ,$$ and as shown in the proof of Proposition 2.2 (ii) in \cite{Sznit} tightness of $\mathcal{LQ}_2^N$ is equivalent to tightness of $I(\mathcal{LQ}_2^N)$. We then deduce as before that any weak limit $\mathcal Q_2$ satisfies $\mathcal F_2(\mathcal Q_2)=0$ almost surely. 

Arguing as in the proof of Proposition 2.2 (i) in \cite{Sznit}, introducing the same truncature $K_\varepsilon$ as before,  we have
\begin{eqnarray*}
|\mathcal F_{2,\varepsilon}(Q^{2,N}) - \mathbb E[\mathcal F_{2,\varepsilon}(\mathcal Q_2)]| &\leq&  |\mathcal F_{2,\varepsilon}(Q^{2,N}) - \mathbb E[\mathcal F_{2,\varepsilon}(\mathcal Q_2^N)]| + |\mathbb E[\mathcal F_{2,\varepsilon}(\mathcal Q_2)] - \mathbb E[\mathcal F_{2,\varepsilon}(\mathcal Q_2^N)]|
\end{eqnarray*}
where the second term goes to $0$ using convergence and the first one goes to $0$ using symmetry (see \cite{Sznit} (2.3)). Taking another subsequence if necessary, we may assume that $Q^{2,N}$ weakly converges to $\bar Q_2$. The same argument as the one used at the end of the proof of Proposition \ref{propconvnl} shows that $|\mathcal F_{2,\varepsilon}(\bar Q_2)-\mathcal F_{2}(\bar Q_2)|$ goes to $0$ as $\varepsilon \to 0$. hence taking limits first w.r.t. $N$ and second w.r.t. $\varepsilon$ we have obtained that $$\mathcal F_{2}(\bar Q_2)=0 \, .$$ We similarly have $$\mathcal F(\bar Q)=0 \, ,$$ for any weak limit of $Q^{1,N}$, i.e. any weak limit of $Q^{1,N}$ solves the non linear S.D.E.
\medskip

Since $\bar Q_2$ solves the associated martingale problem, one classically deduce the existence of a standard Brownian motion $(\bar B_.^1,\bar B_.^2)$ such that $\bar Q_2$ is the law of 
\begin{eqnarray*}
d\bar X^1_t &=& \sqrt 2 \, d\bar B^1_t \, - \, b(\bar X^1_t) \, dt \, - \, (K*\bar \rho^1_t)(\bar X^1_t) \, dt \, ,\\
\bar \rho^1_t(x) \, dx &=& \mathcal L(\bar X^1_t) \, , \\
d\bar X^2_t &=& \sqrt 2 \, d\bar B^2_t \, - \, b(\bar X^2_t) \, dt \, - \, (K*\bar \rho^2_t)(\bar X^2_t) \, dt \, ,\\
\bar \rho^2_t(x) \, dx &=& \mathcal L(\bar X^2_t) \, , 
\end{eqnarray*}
with initial condition $\bar \rho_0(x_1) \, dx^1\otimes \bar \rho_0(x_2) \, dx^2$ provided the initial condition $\rho_0^N$ is chaotic.

Even if the two non-linear SDE's are \emph{autonomous}, it is not clear that $\bar X^1_.$ and $\bar X^2_.$ are independent. Thanks to exchangeability we know that $\tilde Q_1=\tilde Q_2$, when $\tilde Q_i$ denotes the law of $\bar X_.^i$. In particular $\tilde Q_1 \otimes \tilde Q_2$ is a solution of the previous system, so that if we know that this system admits only one (weak) solution, any weak limit of $Q^{2,N}$ is given by the law of two independent processes. 

Notice that if $\bar Q_2$ is absolutely continuous w.r.t. the Wiener measure (on the product path space), Girsanov theory tells us that its drift is given by the sum of the two drifts $g_1(t,x^1)+g_2(t,x^2)$ so that the Girsanov density factorizes. Since the Wiener measure itself is a product measure, we deduce as before that $\bar Q_2$ is a product measure. According to what we have already recalled $\tilde Q_1 \ll P$ iff $\int_0^T \, |K*\bar \rho^1_t|^2(\bar X^1_t) \, dt$ is almost surely finite. Since both marginals of $\bar Q_2$  are equal to $\tilde Q_1$ as soon as $\int_0^T \, |K*\bar \rho^1_t|^2(\bar X^1_t) \, dt$ is almost surely finite, $\int_0^T \, (|K*\bar \rho^1_t|^2+ K*\bar \rho^2_t|^2)(\bar \omega_t) \, dt$ is also $\bar Q_2$ almost surely finite, so that $\bar Q_2$ is absolutely continuous, i.e. the latter property is equivalent to the absolute continuity of the marginal $\tilde Q_1$.

We may state

\begin{theorem}\label{thmchaosmulti}
\textbf{The 2D vortex case with confinement.} \quad Let $d=2$, $K(x)=\chi \, \frac{x^\perp}{|x|^2}$.

Assume that the initial condition $\mu_0^N=\rho_0^N \, dx$ is chaotic so that $\mu_0^{k,N} \to (\bar \rho_0 \, dx^1)^{\otimes k}$, satisfies $H(\mu_0^N|\gamma_0^{\otimes N}) \leq CN$ and finally satisfies $\int |x|^2 \, \bar \rho_0(x) dx < +\infty$. Also assume that the additional drift $b$ is a confining potential as in Theorem \ref{thmbiotexistconfine}.

Then $Q^N$ is chaotic and for each $k$, $Q^{k,N}$ weakly converges to $(\bar Q)^{\otimes k}$ where $\bar Q$ is the unique solution of the non linear SDE \eqref{eqnldiff} with initial condition $\bar \rho_0 \, dx$ satisfying $\int_0^T \, I(\bar Q\circ \bar \omega_t^{-1}) dt < +\infty$.
\end{theorem}
\begin{proof}
Conditions on $b$ and the initial measure ensure that $ \sup_N \, H(Q^{2,N}|\bar P^{\otimes 2}) < +\infty$ according to Theorem \ref{thmtight}, and we deduce from the lower semi-continuity of the relative entropy that $H(\bar Q^{2}|\bar P^{\otimes 2}) < +\infty$. It follows that $\int_0^T I(\bar \rho_t) dt < +\infty$. Uniqueness of the solution is thus given by Proposition \ref{propedpvortex} with a slight difference: if there is no confinement one has to assume that the Boltzmann entropy of $\bar \rho_0$ is finite, while here, lower semi continuity of the relative entropy shows that $H(\bar  \rho_0|\gamma_0) < +\infty$, which is enough for our purpose. Notice that this uniqueness shows that $\rho^1=\rho^2$ in the pair particles non-linear SDE. It is also classical that chaos for the first two coordinates implies chaos for any finite set of coordinates.
\end{proof}

In the Biot-Savart situation the proof we have given is simpler than the one in \cite{FHM}, but we have to add a confinement potential.
\medskip

Actually, in what precedes, we used relative entropy to get uniqueness for the non linear SDE. In the 2D vortex case, the uniqueness part in Proposition \ref{propedpvortex} only requires that $\int_0^T I(\bar \rho_s) ds < +\infty$. Again, according to super-additivity of the Fisher information, and lower semi continuity, this condition will be satisfied provided $$\int_0^T \, I(\rho^N_s) ds \leq C(T) N \, .$$ This bound is exactly what is obtained in Proposition 5.1 of \cite{FHM} but, as we discussed in Remark \ref{remFHMbol} it requires to prove some uniform (in $N$) lower bound for the normalized Boltzmann entropy. We thus have the following result, which is the one obtained in \cite{FHM}
 
\begin{theorem}\label{thmchaosmultiBS}
For the 2D vortex case (without confinement), under the same assumptions, the same conclusion as in Theorem \ref{thmchaosmulti} is true if the initial condition satisfies the additional moment condition $\sup_{N \geq N_0} \int |x^1|^2 \, \rho_0^N(x) dx < +\infty$.
\end{theorem}
\bigskip

We turn now to the sub-Coulombic model. All what we have done for the $2D$ vortex model in the proof of Theorem \ref{thmchaosmulti} is still available, assuming this time the hypotheses in Corollary \ref{corexistnlraf} (1), except the final uniqueness result for the non linear S.D.E. As explained at the end of subsection \ref{subseccoul}, our goal will be to find conditions for the flow of densities $\bar \rho_. \in \mathbb L^1([0,T],\mathbb L^{d/d-s-2}(\mathbb R^d))$ and apply Proposition \ref{propuniquenl1}.  
\medskip

We start with a very natural Lemma we found in \cite{monls}. We recall the proof for completeness.
\begin{lemma}\label{lemmareg1}
Let $\tilde Q$ be the law of a diffusion process $$X_t=X_0+\sqrt 2 B_t+ \int_0^t \, g(s,X_s) \, ds \, .$$ Assume that $g \in C_b^\infty([0,T]\times \mathbb R^d)$. Denote by $m_t$ the density of $X_t$ for $t\geq 0$. Assume that $m_0 \in \mathbb L^q(\mathbb R^d)$. Then for $1\leq q \leq +\infty$, $$||m_t||_q \leq e^{\frac{q-1}{q} \, t \, \sup_{u \leq t}||(\nabla.g(u,.))_-||_\infty} \, ||m_0||_q \, .$$ 
\end{lemma}
\begin{proof}
Assume first that $m_0$ is bounded with compact support and $q<+\infty$. Due to our assumptions we know that $m_t \in C^1(]0,T],C^\infty(\mathbb R^d))\cap \mathbb L^q(\mathbb R^d)$, has gaussian tails  and that $m_t \to m_0$ as $t \to 0$ in $\mathbb L^q(\mathbb R^d)$. The following calculation is thus rigorous for $t>0$
\begin{eqnarray*}
\frac{d}{dt} \int \, m_t^q \, dx &=& \int \, q \, m_t^{q-1} \, \partial_t \, m_t \, dx \\ &=& q \, \int \, m_t^{q-1} \, (\Delta m_t - \nabla.(g(t,.) m_t)) \, dx \\ &=& - \, q(q-1) \, \int \, m_t^{q-2} \, |\nabla m_t|^2 \, dx \, + \, q(q-1) \, \int m_t^{q-1} \langle \nabla m_t , g(t,.)\rangle \, dx \\ &=& - \, q(q-1) \, \int \, m_t^{q-2} \, |\nabla m_t|^2 \, dx \, - \, (q-1) \, \int \, m_t^q \, \nabla.g(t,.) \, dx \\ &\leq& (q-1) \, ||(\nabla.g(t,.))_-||_\infty \, \int \, m_t^q \, dx \, .
\end{eqnarray*}
By Gronwall's lemma, for any $t\geq t_0>0$ we thus have $$||m_t||_q \leq e^{\frac{q-1}{q} \sup_{t_0\leq u \leq t}||(\nabla.g(u,.))_-||_\infty \, (t-t_0)} \, ||m_{t_0}||_q$$ and we may let $t_0 \to 0$ thanks to the continuity at the origin of the $\mathbb L^q$ norm. Extension to a general initial $m_0$ in $\mathbb L^q$ follows using approximation. For $q=+\infty$, we have $$ ((1-\varepsilon)||m||_\infty)^{\frac{q-1}{q}}||m||_q \, m^{1/q}(\mathbf 1_{|m|\geq (1-\varepsilon)||m||_\infty}) \, \leq ||m||_\infty^{\frac{q-1}{q}} \, ||m||_1$$ and since for all $\varepsilon>0$, $m(\mathbf 1_{m\geq (1-\varepsilon)||m||_\infty})>0$, the result is obtained by using first $q \to +\infty$ and then $\varepsilon \to 0$, the left hand side of the previous inequality with $m_t$ and the right hand side with $m_0$.
\end{proof}
An immediate consequence is the following. Let $\tilde Q$ be the law of a diffusion process as before for a general $g$. If there exists a sequence $g_n$ of smooth drifts as before such that, on one hand the associated $\tilde Q_n$ weakly converges to $\tilde Q$ (or more generally the convergence holds true at the level of the marginals), on the other hand $$\sup_n \sup_{u \leq t}||(\nabla.g_n(u,.))_-||_\infty = S_t < +\infty \, ,$$ then the time marginals of $\tilde Q$ have densities $m_.$ satisfying $$||m_t||_q \leq e^{\frac{q-1}{q} \, t \, S_t} \, ||m_0||_q \, .$$ This is an immediate consequence of Riesz representation theorem and the density of continuous and bounded functions in $\mathbb L^{q/q-1}$ for $q>1$, the result for $q=1$ being obvious.
\medskip

We will apply this result with $g(t,x)=- (K*\bar \rho_t)(x) - b(x)$ and an approximating sequence $$g_n(t,x) = - (K_n*\bar \rho_t)(x) - b(x)$$ with $$K_n(y) = \chi \; \frac{y}{|y|^{s+2} + \psi_n(|y|^2)}$$ where $\psi_n(u)= \frac 1n \, \psi(nu)$ for some smooth non-increasing and non-negative $\psi$ with $\psi(0)=1$ and $\psi(1)=0$. We thus have $$\nabla.K_n(y) = \chi \; \frac{(d-s-2)|y|^{s+2} + \psi_n(|y|^2) - 2 \psi'_n(|y|^2) \, |y|^2}{(|y|^{s+2} + \psi_n(|y|^2))^2} \, \leq \, 0$$ if $\chi <0$. It follows that $$||(\nabla.g_n(u,.))_-||_\infty \, \leq \, ||(\nabla.b)_+||_\infty \, ,$$ and since $K_n$, hence $g_n$ is smooth that for all $n$ and all $1\leq q \leq +\infty$,
\begin{equation}\label{eqlqapproxm}
||m_{n,t}||_q \leq e^{\frac{q-1}{q} \, t \, ||(\nabla.b)_+||_\infty} \, ||m_0||_q \, .
\end{equation}
It is worth to notice that here $m_{n,t}$ denotes the density of the solution of the \emph{linear} S.D.E with drift $g_n$.

If we denote by $\tilde Q_n$ the distribution of the solution of this linear S.D.E., it remains to show that $\tilde Q_n$ converges to the solution of the non-linear S.D.E. $\bar Q$ whose marginals flow is $\bar \rho_.$. Recall that, under the assumptions of Corollary \ref{corexistnlraf} we know that $H(\bar Q|P) < +\infty$, in particular $$\mathbb E^{\bar Q} \left(\int_0^T \, |K*\bar \rho_t|^2(\omega_t) \, dt \right) < + \infty \, .$$ It is immediate that $|K_n|\leq |K|$ and $K_n \to K$ almost everywhere as $n \to +\infty$. Applying Lebesgue bounded convergence theorem we deduce that $$\mathbb E^{\bar Q} \left(\int_0^T \, |(K-K_n)*\bar \rho_t|^2(\omega_t) \, dt \right) \to 0$$ i.e. that $H(\bar Q|\tilde Q_n) \to 0$. Thanks to Pinsker inequality again, we thus have that $\tilde Q_n$ goes to $\bar Q$ in total variation distance, hence weakly.

Gathering all we have done, we deduce that any weak limit of $Q^{k,N}$ satisfies the non linear S.D.E. \eqref{eqnldiff}, and, provided the initial density $\bar \rho_0 \in \mathbb L^q$, has a marginals flow $\rho_. \in \mathbb L^1([0,T],\mathbb L^q(\mathbb R^d))$. According to Proposition \ref{propuniquenl1} if $q>d/d-s-2$, there exists a unique solution of \eqref{eqnldiff} such that its marginals flow satisfies this integrability condition, so that there exists a unique weak limit of $Q^{k,N}$. Let us summarize 
\begin{theorem}\label{thmchaosRSimproved}
Consider the sub-Coulombic case:  \quad $d \geq 3$, $\chi <0$, , and $$K(x)= \chi \, \frac{x}{|x|^{s+2}} \, \mathbf 1_{x\neq 0} \; \textrm{ for } \; 0<s \leq d-2 \, ,$$ also assume that the additional drift  $b$ is bounded and Lipschitz. Assume that the initial condition $\mu_0^N=\rho_0^N \, dx$ is chaotic so that $\mu_0^{k,N} \to (\bar \rho_0 \, dx^1)^{\otimes k}$, and satisfies $H(\mu_0^N|\gamma_0^{\otimes N}) \leq CN$. Finally assume that for some $q>d/d-s-2$, $\bar \rho_0 \in \mathbb L^q(\mathbb R^d)$.   
\medskip

Then $Q^N$ is chaotic and for each $k$, $Q^{k,N}$ weakly converges to $(\bar Q)^{\otimes k}$ where $\bar Q$ is the unique solution of the non linear SDE \eqref{eqnldiff} with initial condition $\bar \rho_0 \, dx$ satisfying $\int_0^T \, ||\bar Q\circ \omega_t^{-1}||_{\mathbb L^q} \, dt < + \infty$.
\end{theorem}
\medskip

\section{More on the $2D$ situation and  the Keller-Segel model.}
\label{secchaos3}

First consider the Keller-Segel model with confinement, $d=2$, $K(x)=\chi \, \frac{x}{|x|^2}$, $\chi<4$. We have seen in the previous section that, if $\chi<2$, the family $Q^{k,N}$ is tight and that any weak limit $\bar Q$ of $Q^{1,N}$ satisfies the non linear SDE \eqref{eqnldiff}. For $\chi<4$, Theorem \ref{thmtardy} and Proposition \ref{proptardy2} can similarly be completed, and a  similar result is true for the marginals flow and the limiting Keller-Segel PDE. The only remaining problem is to get some uniqueness for the limits. 

In all this section we assume (at least) that $\rho_0^{N}$ is exchangeable and satisfies the integrability assumptions in  Theorem \ref{propKSconfine}. We generically denote by $\bar \rho_.$ the marginals flow of some weak limit. 

If we add a confining potential, Lemma \ref{lemtightKSapprox} together with Corollary \ref{corapproxKSstoch} or Proposition \ref{propvarepsgrandchi} furnish in our context $$\sup_{0 \leq t \leq T} \, H(\rho_t^{N,U}dx|e^{- \tilde U}dx) \leq C(\chi,U) \, N \, .$$ As usual we deduce that $$
\sup_{0 \leq t \leq T} \, H(\bar \rho_t \, dx|e^{- U}dx) \leq  \, C(\chi,U) \, , $$ and since similarly $$\int |U| \, \bar \rho_t \, dx^1 < + \infty$$ taking again $\liminf_N$, using Lemma \ref{lemKSFJ} (1) and Fatou's lemma, we finally obtain $$\sup_{t \in [0,T]} \, \int \, \bar \rho_t |\ln (\bar \rho_t)| \, dx^1  < +\infty \, .$$ 
\medskip

The previous result is still far from the required $\mathbb L^q$ bound for $q>1$ ensuring uniqueness of the non linear SDE in Corollary \ref{corexistnl1}. Remember that the proof of this Corollary is based on the study of the \emph{linear} SDE $$dX_t = \sqrt 2 \, dB_t - \nabla U(X_t) dt - (K*\bar \rho_t)(X_t) dt$$ and the application of Theorem \ref{thmkryl} in the time dependent situation $K*\bar \rho_t \in \mathbb L^{\infty}([0,T],\mathbb L^p(\mathbb R^2))$ for some $p>2$.

In the general situation we are facing (where $\bar \rho$ is a weak limit of the marginals of the particle system), one may ask about the integrability of $K*\bar \rho_t$. Since $\bar \rho_t$ belongs to the Orlicz space $\mathbb L_{x \ln x}$ we are led to work in some Orlicz spaces. We shall briefly recall the material of this theory we need for the sequel.
\medskip

\subsection{The (super) basics for Orlicz spaces.\\ \\} 

A function $\Phi:\mathbb R^+ \to \mathbb R$ is a Young function if it is non-decreasing, convex, satisfies $\Phi(0)=0$ and $\lim_{u \to +\infty} \Phi(u)/u = +\infty$. The convex conjugate $\Phi^*(a) = \sup_{u>0} \left(au - \Phi(u)\right)$ is also a Young function.

The associated Orlicz space $\mathbb L_\Phi(\mathbb R^n)$ is the set of functions $f$ defined on $\mathbb R^n$ and satisfying $\int \Phi(|f|/c) dx < +\infty$ for some $c>0$. It is a vector space and can be equipped with two norms: the Orlicz norm (or dual norm) is defined by  $$|f|_\Phi = \sup\{\int |g \, f| \, dx \; ; \; \int \Phi^*(|g|) \, dx \leq 1 \} \, ,$$ and the Luxemburg norm (gauge norm) is defined as $$N_\Phi(f)= \inf \{c>0 \; ; \; \int \Phi(|f|/c) dx \, \leq 1 \, \} \, .$$ These norms are equivalent, more precisely $N_\Phi(f) \leq |f|_\Phi \leq 2 \, N_\Phi(f)$. $\mathbb L_\phi$ equipped with these norms is a Banach space.

We shall say that a function $\Phi$ is \emph{moderate} if, for any $\alpha >0$ there exist $c(\alpha)$ and $C(\alpha)$ such that $c(\alpha) \Phi(u) \leq \Phi(\alpha \, u) \leq C(\alpha) \Phi(u)$. It is immediate that for a moderate Young function and $\alpha>0$, $\mathbb L_\Phi=\mathbb L_{\alpha \, \Phi}$ and that the norms $N_{\alpha \Phi}$ and $N_\Phi$ are equivalent.

To any $C^2$ function $\Phi$ only convex at infinity and such that $\Phi(u)/u \to + \infty$ at infinity, one can associate several Young functions as follows.

First, if $\Phi$ is convex on $[c_0,+\infty)$ (we may always assume that $c_0 \geq 1$) using convexity and Taylor formula we have for $u \geq c_0$, $\Phi(c_0) - \Phi(u) \geq (c_0-u) \, \Phi'(u)$ so that $\Phi'(u) \geq \frac{\Phi(u)}{u} + \frac 1u \, (c_0 \Phi'(u) - \Phi(c_0))$. Since $\Phi(u)/u \to + \infty$, $\Phi'(u) \to  +\infty$ as $u \to +\infty$, so that one can find $c_1 \geq c_0$ such that $c_0 \, \Phi'(c_1) \geq \Phi(c_0)$. It follows that for all $u \geq c_1$, $\Phi'(u) \geq \frac{\Phi(u)}{u}$. We can thus define for all $c \geq c_1$,
\begin{equation}\label{eqtilde}
\tilde \Phi_c(u)= \frac{\Phi(c)}{c} \, u \, \mathbf 1_{u\leq c} + \Phi(u) \, \mathbf 1_{u\geq c}
\end{equation}
which is a Young function. If $\Phi$ is moderate, so is $\tilde \Phi_c$. 
\medskip

Another natural way consists in first choosing $c_0 \geq 1$ such that $\Phi'(c_0)\geq 1$ and introduce for any $c \geq c_0$
\begin{equation}\label{eqtildebis}
\bar \Phi_c(u)=  u \, \mathbf 1_{u\leq c} + (\Phi(u)-\Phi(c)) \, \mathbf 1_{u\geq c}
\end{equation}
which is again a Young function, moderate if $\Phi$ is.

If we choose the same large enough $c$ such that $\Phi(c)/c \geq 1$ one of course has $\bar \Phi_c \leq \tilde \Phi_c$ so that $$N_{\bar \Phi_c} \leq N_{\tilde \Phi_c}  \, .$$ Conversely if we choose $c$ large enough such that $\Phi(c)/c \geq 2$ one has $(\Phi(c)/c) \, \bar \Phi_c \geq \tilde \Phi_c$. If $\bar \Phi_c$ is moderate, one can find $\alpha(c)$ and $\beta(c)$ such that $\bar \Phi_c(u/\alpha(c)) \leq (\Phi(c)/c) \, \bar \Phi_c(u) \leq \bar \Phi_c(u/\beta(c))$ so that the norms $N_{(\Phi(c)/c) \bar \Phi_c}$ and $N_{\bar \Phi_c}$ are equivalent, hence, $$\textrm{ for some } a(c)>0 \, , \; a(c) \, N_{\bar \Phi_c} \geq N_{\tilde \Phi_c}  \, \textrm{ provided $\Phi$ is moderate .}$$

Orlicz spaces are generalizing $\mathbb L^p$ spaces (corresponding to $\Phi(u)=u^p/p$ and $\Phi^*(a)= a^q/q$ with $\frac 1p+\frac 1q=1$). One can thus expect that several properties are preserved. The first one immediately follows from the above construction, namely the Orlicz-H\"{o}lder inequality $$\int |f \, g| \, dx \leq |f|_\Phi \, N_{\Phi^*}(g) \, \leq \, 2 \, N_\Phi(f)  \, N_{\Phi^*}(g) \, .$$ An interesting generalization is proved in Theorem 2.3 of \cite{Oneil}, namely 
\begin{equation}\label{eqorlhold}
\textrm{If for all $u$, } \Phi^{-1}(u) \, \Psi^{-1}(u) \, \leq \, \Theta^{-1}(u) \; \textrm{ then } \; N_\Theta(fg) \leq 2 \, N_\Phi(f) \, N_\Psi(g) \, .
\end{equation}

In the same spirit, one can extend Young's inequality for a convolution product (Theorem 2.5 in \cite{Oneil})
\begin{equation}\label{eqyoung}
\textrm{If for all $u$, } \Phi^{-1}(u) \, \Psi^{-1}(u) \, \leq \, u \, \Theta^{-1}(u) \; \textrm{ then } \; N_\Theta(f*g) \leq 2 \, N_\Phi(f) \, N_\Psi(g) \, .
\end{equation}

Let now $\Phi$,$\Psi$ and $\Theta$ be convex at infinity and satisfying the condition in \eqref{eqyoung} for $u$ large i.e $u\geq c$.  This condition is trivially  satisfied for $\bar \Phi_c$, $\bar \Psi_c$ and $\bar \Theta_c$ for $u\leq c$ where all inverse functions are identity.

If now condition \eqref{eqorlhold} is satisfied for $u\geq c$, we have for $u \leq c$, $\bar \Phi_c^{-1}(u) \, \bar \Psi_c^{-1}(u) \, \leq \, c \, \bar \Theta_c^{-1}(u)$, so that $$N_{\bar \Theta_c}(fg) \leq \kappa(c) \, N_{\bar \Phi_c}(f) \, N_{\bar \Psi_c}(g) \; \textrm{ for some $\kappa(c)$, provided $\Psi$ or $\Phi$ or $\Theta$ is moderate.}$$  

In the sequel, by abuse of notation, if $\Phi$ is convex at infinity such that $\Phi(u)/u \to +\infty$ at infinity, we shall denote $\mathbb L_\Phi$ the Orlicz space associated to either $\tilde \Phi$ or $\bar \Phi$.
\medskip

Convolution with Riesz kernels was studied in a huge number of papers under the denomination \emph{Fractional integration}. Let we discuss this in the two dimensional case.

Recall the Hardy-Littlewood-Sobolev inequality we already used for proving Corollary \ref{corexistnl2} says that 
\begin{equation}\label{eqHLS}
||\frac{1}{|z|}*\rho||_{2r/2-r} \leq C_r \, ||\rho||_r \quad \textrm{ for } r\in (1,2) \, .
\end{equation}
At a formal level, up to the constant $C_r$ (in particular because  $C_1=+\infty$), this result says that, even if $1/|z|$ does not belong to $\mathbb L^2(\mathbb R^2)$ but belongs to all $\mathbb L^{2-\varepsilon}(\mathbb R^2)$ for $\varepsilon >0$, the integrability of the convolution product is the same as the one given by Young's convolution inequality with $2$ and not only $2-\varepsilon$. Working in Orlicz spaces is not enough because there is no optimal space to which $z \mapsto 1/|z|$ belongs.

However \eqref{eqHLS} is a simple consequence of the analogue of Young's convolution inequality in Lorentz spaces. Indeed, $1/|z|$ belong to the \emph{ weak $\mathbb L^2(\mathbb R^2)$ space}, denoted by $\mathbb L^{2,\infty}$ which is a particular Lorentz space. We will not recall the basics on these spaces here, since we will only use the previous result. Convolution in Lorentz spaces was first studied by O'Neil in \cite{Oneil2} in general form (see Theorem 2.6 therein) and applied to Hardy-Littlewood-Sobolev inequality in Theorem \emph{Fractional integration on $n$-space} at the bottom of p.139 of \cite{Oneil2}. It is an alternate approach to the more popular proof based on the Hardy-Littlewood maximal inequality. 

Later on, O'Neil in \cite{Oneil} extended \eqref{eqHLS} to $\rho$ in Orlicz spaces. Actually he introduced in section 3 a notion of Orlicz-Lorentz space (denoted $M_A$) associated to a Young function $A$ and proved in Theorem 4.7 a general convolution inequality for two functions belonging respectively to some $M_A$ and some Orlicz space (denoted $L_B$ therein). Theorem 4.7 applies to a modification near the origin of $B(u)=u \ln(1+u^2)$ and $A(u)=u^2$. The more refined version shown in Theorem 5.2 directly applies to the previous $B$ (also see Theorem 3.8 in  \cite{Sharpley}).

A different proof based on the maximal function is given in Theorem 6.8 of \cite{Hasto1} (also see Lemma 4.3 and Remark 3.1 in \cite{Karppi}). Another approach is used in \cite{NakaiK} in this and more general framework.

In any case it furnishes the following result, useful in the Keller-Segel context
\begin{proposition}\label{prophasto}
In $\mathbb R^2$, let $\rho$ be a density of probability. Let $g=\rho*(1/|z|)$. There exists some constant $C$ such that, if $\int \, \rho \, \ln(1+\rho^2) \, dx = M < +\infty$ then $$\int \, g^2 \, \ln(1+g^2) \, dx \leq C (1+M) < +\infty \, .$$
\end{proposition}

\medskip

\subsection{$2D$ singular models and Orlicz spaces. \\ \\}

The next step is to extend Theorem \ref{thmkryl} in dimension $d=2$ to the case where the drift $g$ is such that $g^2 \in \mathbb L_\kappa(\mathbb R^2)$ for some $\kappa(u) \gg u$ but $\ll u^p$ for $p>1$, and large $u$'s. 
\medskip

To this end, following the proof of this Theorem we see that it is enough to choose $A$ such that  $\sup_{0<t\leq T} \, N_\kappa(|g_t|^2 \, \mathbf 1_{|g_t|>A})$ is small enough (depending on universal constants like the one in the equation below) and to prove that for any probability density $\rho$
\begin{equation}\label{eqextendFHM}
N_{\kappa^*}(\rho) \, \leq \, C \, (I(\rho)+1) \, .
\end{equation}
 
It remains to come back to Lemma \ref{lemFHM} and its proof to get the desired \eqref{eqextendFHM}. 

We can mimic the proof of Lemma \ref{lemFHM}, using the so called  Orlicz-Sobolev inequalities (see \cite{donaldtrue,kone,} and the more recent and sharp results in \cite{Cianchi96,Cianchi04}) and then the H\"{o}lder-Orlicz inequality.

Actually such a result is already known in a sharp form, as a consequence of the Trudinger-Moser inequality in the whole $\mathbb R^2$ obtained in \cite{Ruf}. We thank Nicolas Fournier for indicating us this reference and, at the same time, for pointing out a mistake in a first version of the present paper. Here is Theorem 1.1 in \cite{Ruf}
\begin{theorem}\label{thmruf}
Define $$||f||_S=\left(\int_{\mathbb R^2} \, (|\nabla f|^2+|f|^2) \, dx \right)^{\frac 12} \, .$$
Then for $\alpha >0$ it holds
\begin{enumerate}
\item[(1)] \quad if $\alpha \leq 4\pi$, $\sup_{||f||_S\leq 1} \, \int_{\mathbb R^2} \, (e^{\alpha f^2}-1) \, dx \, \leq \, C(\alpha) < +\infty$,
\item[(2)] \quad if $\alpha > 4\pi$, $\sup_{||f||_S\leq 1} \, \int_{\mathbb R^2} \, (e^{\alpha f^2}-1) \, dx \, = \, +\infty$.
\end{enumerate}
\end{theorem}
\begin{corollary}\label{corruf}
Consider the Young function $\Theta(u)=e^u-1$. There exists a constant $C(\Theta)$ such that, for any density of probability $\rho$ on $\mathbb R^2$ it holds $$N_\Theta(\rho) \leq C(\Theta) \; (1+I(\rho)) \, .$$
\end{corollary}
\begin{proof}
Let $\rho=f^2$ be a positive function, normalized such that $$1=||f||^2_{S}=\left(\frac 14 \, I(\rho) + ||\rho||_1\right) \, .$$  For $\alpha>0$, introduce $h(\alpha)=\int (e^{\alpha \rho} - 1) \, dx$. We want to find $\alpha_0>0$ such that $ h(\alpha_0) \leq \, 1$ for all such $\rho$. 

To this end, since $h = h_1+h_2$ defined below, it is enough to find $\alpha_0$ such that both $$h_1(\alpha_0)=\int (e^{\alpha_0 \rho} - 1) \, \mathbf 1_{\rho \leq 1} \, dx \leq \frac 12 \quad \textrm{ and } \quad  h_2(\alpha_0)=\int \, (e^{\alpha _0\rho}-1)  \, \mathbf 1_{\rho > 1} \, dx \leq \frac 12 \, .$$ The functions $h_j$ are non-decreasing, with $h_j(0)=0$ and $h_j(\alpha) \leq C(4\pi)$ for $\alpha<4\pi$. In addition they are smooth for $\alpha \leq 1 < 4\pi -1$ and, since $\int \rho dx \leq 1$,
$$h_1'(\alpha)= \int \, \rho \, e^{\alpha \rho} \, \mathbf 1_{\rho \leq 1} \, dx \leq \, e^\alpha \leq \, e$$ while 
\begin{eqnarray*}
h_2'(\alpha)&=& \int \, \rho \, e^{\alpha \rho} \, \mathbf 1_{\rho > 1} \, dx \leq \, \int \, e^{(\alpha+1) \rho} \, \mathbf 1_{\rho>1} \, dx \\ &\leq& \int \, (e^{(\alpha+1) \rho}-1) \, \mathbf 1_{\rho > 1} \, dx \, + \, \int \, \mathbf 1_{\rho>1} \, dx \, \leq \, C(2) + \int \rho dx \leq  C(2) +1 \, .
\end{eqnarray*}
The existence of $\alpha_0$ follows.
\medskip

We deduce that for any probability density, $N_\Theta(\rho/(1 + (I(\rho)/4)) \leq 1/\alpha_0$ and the result follows.
\end{proof}
\begin{remark}\label{remtrudi}
The second part of Theorem \ref{thmruf} indicates that $\Theta$ is in some sense the \textit{largest} function such that Corollary \ref{corruf} is satisfied. Instead of giving a general statement we shall give a counterexample in a specific case, namely look at the case $\Theta_\beta(u)= e^{u \, \ln^\beta(u)}$ for some $\beta>0$ and large $u$'s. This counterexample is also due to N. Fournier.

Consider some $\rho$ behaving like $\ln(1/|x|^2) \, |\ln_2|^{-\gamma}(1/|x|^2)$ for $x$ close to the origin. It is easy to see that $|\nabla \rho|^2/\rho$ behaves like $1/(|x|^2 \, \ln(1/|x|^2) \, |\ln_2|^\gamma(1/|x|^2))$ which is integrable (recall that $d=2$) provided $\gamma >1$. For such a $\gamma$ we thus have $I(\rho) < +\infty$. However for $\beta>\gamma$, $\Theta_\beta(c \, \rho)(x)$ behaves like $\exp\left(c \ln(1/|x|^2) \, |\ln_2|^{\beta - \gamma}(1/|x|^2)\right)$ which is not integrable for any $c>0$.
\hfill $\diamondsuit$
\end{remark}

\begin{remark}\label{remruf}
The proof of Theorem \ref{thmruf} given in \cite{Ruf} is very different from the line we suggested before. It uses two main tricks. First a symmetrization argument allowing to control what happens outside some large ball, then the usual Trudinger-Moser inequality in the remaining ball.\hfill $\diamondsuit$
\end{remark}
\bigskip

We may thus state the following improvement of Theorem \ref{thmkryl}
\begin{theorem}\label{thmkrylorl}
In dimension $d=2$ consider the SDE, $X_t=X_0+\sqrt 2 \, B_t + \int_0^t 2g_s(X_s) ds \, .$ 

Assume that $H(\nu_0|\gamma_0) < +\infty$ (recall Theorem \ref{thmkryl}) and that for $A$ large enough, $$\sup_{0<t\leq T} \, N_{\Phi}(g_t \, \mathbf 1_{|g_t|>A}) < +\infty \, $$ where $\Phi$ is a Young function behaving like $u \mapsto u^2 \, \ln(1+u^2)$ at infinity, that is, there exists $A'$ such that $$\sup_{0<t\leq T} \, \int \, g^2_t \, \ln(|g_t|) \, \mathbf 1_{|g_t|>A'}  < +\infty \, . $$ Then the previous SDE has a unique weak solution $Q$ and this $Q$ satisfies $H(Q|P)<+\infty$.

As a consequence the results in Theorem \ref{thmlpfin} extend for $d=2$ if the interaction kernel satisfies the previous integrability condition and the additional drift $b$ is bounded.
\end{theorem}
\begin{proof} 
In the proof of Theorem \ref{thmkryl} we write for $M>A$,
\begin{eqnarray*}
\int |g_t^M|^2 \rho_t^M dx &=& \, \int |g_t^M \, \mathbf 1_{|g_t^M|>A}|^2 \rho_t^M dx \, + \, \int \,  |g_t^M \, \mathbf 1_{|g_t^M| \leq A}|^2 \rho_t^M dx \\ &\leq&  A^2 + C \, N_{\Phi}(g_t^M \, \mathbf 1_{|g_t^M|>A}) \, N_{\Theta}(\rho_t^M)
\end{eqnarray*}
apply corollary \ref{corruf} and conclude as for the proof of Theorem \ref{thmkryl}.
\end{proof}
\bigskip

\begin{remark}\label{remorlUI}{\textit{The Keller-Segel case.}} \quad

Denote $g_t=K*\bar \rho_t$. Then $g_t^M=(K \, \mathbf 1_{|K|\leq M})*\bar \rho_t$ can be used in the proof of Theorem \ref{thmkryl}. One can however slightly modify the proof of \eqref{eqkrylov1} in order to get uniform estimates in $t$. Indeed for $M>A$, 
\begin{eqnarray*}
\int |g_t^M|^2 \rho_t^M dx &\leq& 2 \, \int |g_t^A|^2 \rho_t^M dx \, + \, 2 \, \int \, |(K \, \mathbf 1_{A<|K|\leq M})*\bar \rho_t|^2 \rho_t^M \, dx \\ &\leq& 2 A^2 + C \, N_{\Phi}((K \, \mathbf 1_{A<|K|\leq M})*\bar \rho_t) \, N_{\Theta}(\rho_t^M)
\end{eqnarray*}
$C$ being a universal constant. It thus suffices to choose $A$ in such a way that $N_{\Phi}(K \, \mathbf 1_{A<|K|})$ is small enough to get the desired uniform in $t$ bound. \hfill $\diamondsuit$
\end{remark}
\bigskip

We may now state an immediate consequence for the Keller-Segel model
\begin{theorem}\label{thmksconvorl}
If $K(z)= \chi \, z/|z|^2 \, \mathbf 1_{z \neq 0}$ in $\mathbb R^2$, the non linear SDE \eqref{eqnldiff} has at most one solution satisfying $\sup_{0\leq t\leq T} \, \int \, \bar \rho_t \, \ln(1+\bar \rho_t) \, dz < +\infty$.
\end{theorem}
\medskip

\begin{corollary}\label{corKSorl}
Consider the Keller-Segel case $d=2$, $0<\chi<2$ with a confining potential $U$ such that $\int e^{- U} dx < +\infty$, i.e. an additional drift $b=\nabla U$ and assuming that $b$ is Lipschitz.

Assume that the initial condition $\mu_0^N=\rho_0^N dx$ is chaotic so that $\mu_0^{k,N} \to (\bar \rho_0 \, dx^1)^{\otimes k}$, and satisfies $\int \rho_0^N \, |\ln \rho_0^N| dx < C N$ and $\int |x^1|^2 \, \rho_0^{1,N} \, dx^1 \, < +\infty$.

Then $Q^N$ is chaotic  and for each $k$, $Q^{k,N}$ weakly converges to $(\bar Q)^{\otimes k}$ where $\bar Q$ is the unique solution of the non linear SDE \eqref{eqnldiff} with initial condition $\bar \rho_0 dx$ whose marginals flow is the (unique) free energy solution of the non linear PDE \eqref{eqmcv}.
\end{corollary}
\medskip

The proof is simple: thanks to the discussion at the beginning of the section and Proposition \ref{prophasto}, we may apply Theorem \ref{thmkrylorl} showing that the \emph{linear} SDE $$dX_t = \sqrt 2 \, dB_t - \nabla U(X_t) dt - (K*\bar \rho_t)(X_t) dt$$ has an unique weak solution which is of finite relative entropy (assuming that $\int |\nabla U|^2 \, \bar \rho_t dx < +\infty$.) Since any solution of the non linear SDE is a solution of the linear one, we may conclude for the Theorem. 

In particular, any weak limit $\bar Q$ of $Q^{1,N}$ in Theorem \ref{thmtight2} (2) has finite relative entropy, hence is unique according to Theorem \ref{thmexistnl1}. In addition, according to the discussion at the end of subsection \ref{subsecconvergence}, since $\bar Q$ is absolutely continuous w.r.t. $P$, propagation of chaos holds true. Hence the Corollary.
\bigskip

The previous Theorem holds with a quadratic confining potential, so that using the ``self similar'' coordinates as discussed in Remark \ref{remKSconfine} (at the process level) we can state
\begin{corollary}\label{corKSorlsansconf}
Corollary \ref{corKSorl} is still true for the Keller-Segel model without confinement ($U=0$).
\end{corollary}

\begin{remark}\label{remtightmargenough}
Notice that the knowledge of the convergence of the marginals flow is not enough for the previous argument. Indeed, if we know the uniqueness of the solution of the linear SDE, we do not know the uniqueness of the solution of the associated Fokker-Planck equation.
\hfill $\diamondsuit$
\end{remark}
\medskip

\subsection{More on the Keller-Segel model. \\ \\}

The limitation to $\chi<2$ in Corollary \ref{corKSorl} is due to the moment estimate (2) in Lemma \ref{lemKSFJ} where we can choose $\gamma>1$ only if $\chi <2$, yielding the same limitation for the tightness result (2) in Theorem \ref{thmtight2}. We did not succeed in improving this tightness property. 
\medskip

We shall thus look at the marginals flow level and to this end use Theorem \ref{propKSconfine}.
\begin{theorem}\label{thmksfinalchi}
Consider the Keller-Segel case $d=2$, $0<\chi<4$ with a confining potential $U$ such that $\int e^{- U} dx < +\infty$, i.e. an additional drift $b=\nabla U$ and assuming that $b$ is Lipschitz, or with $U=0$.

Assume that the initial condition $\mu_0^N=\rho_0^N dx$ is chaotic so that $\mu_0^{k,N} \to (\bar \rho_0 \, dx^1)^{\otimes k}$, and satisfies $\int \rho_0^N \, |\ln \rho_0^N| dx < C N$ and $\int |x^1|^p \, \rho_0^{1,N} \, dx^1 \, < +\infty$ for some $p > 4\chi$.

Then $\rho^{1,N,U}_{t \in [0,T]}$ weakly converges as $N$ goes to infinity to the unique free energy solution of the Keller-Segel equation.
\end{theorem}
\begin{proof}

If $\rho_0^{N}$ is chaotic and satisfies, for some constant $c$, $$\int \, \rho_0^{N} \, |\ln \rho_0^{N}| \,  dx \leq c N   \, , \quad  \int \, \rho_0^{N} \, |\ln (|x^1-x^2|)| \,  dx < c  \quad \textrm{and} \quad \int |x^1|^2 \, \rho_0^{N} \, dx \leq c $$ it holds 
\begin{equation}\label{eqfishKS0} 
\int_0^T \, \int \, \left(\sum_i \, \left|\nabla_i \ln \rho^{N,U}_t + \nabla \tilde U \, + \frac{\chi}{N} \, \sum_{j\neq i} \, \frac{x^i -x^j}{|x^i-x^j|^2} \right|^2\right) \, \rho^{N,U}_t(x) \, dx \, dt \, \leq \, C' \, N \, ,
\end{equation}
 where $C'$ is another constant. Using exchangeability, the fact that $\nabla \tilde U$ is at most linear and that $\rho_t^{N,U}$ is a density of probability, we deduce that for all $i$
\begin{equation}\label{eqfishKS1}
\int_0^T \, \int \, \left(\left|\nabla_i \ln \rho^{N,U}_t + \,  \frac{\chi}{N} \, \sum_{j\neq i} \, \frac{x^i -x^j}{|x^i-x^j|^2} \, \mathbf 1_{|x^i-x^j|\leq 1} \right|^2\right) \, \rho^{N,U}_t(x) \, dx \, dt \, \leq \, C \, ,
\end{equation}
for some $C$ depending on $c$ and $C'$ only.

Denote $$\omega_i(x) = \prod_{j\neq i} \, (1 \wedge |x^i-x^j|)^{\frac{\chi}{N}} \, ,$$  and  for $\theta>0$, $$\kappa_i(x)=  \prod_{j\neq i} \, (1 \vee |x^i-x^j|)^{\frac{\chi}{N}} \, .$$

Let $\rho$ be a non-negative function, $\alpha \in (0,1]$ and $0\leq M \leq +\infty$. Introduce 
\begin{eqnarray}\label{eqdeftilde}
\tilde \rho &:=& \rho \, \omega_1  \quad , \quad (\tilde \rho)^1(x^1) := \int \tilde \rho(x) \, dx^2 ... dx^N  \nonumber \\H(\alpha,M,\rho) &:=& \min(\rho,\rho^\alpha) \, \mathbf 1_{|x^1|\leq M} \, = \, (\rho \mathbf 1_{\rho \leq 1} + \rho^\alpha 1_{\rho \geq 1}) \,  \mathbf 1_{|x^1|\leq M} \, , \\ \rho^1(x^1,\alpha,M) &:=& \int H(\alpha,M,\rho)(x) \, dx^2 ... dx^N \, . \nonumber
\end{eqnarray}

If $\rho$ is a density of probability, $\tilde \rho$ and $(\tilde \rho)^1$ are sub-probability measures (total mass less or equal one)

According to Theorem 2, formula (9) in \cite{carlen} (also see \cite{Catpota} Theorem 3.8), it holds
\begin{equation}\label{eqtildefish}
\frac 14 \, I((\tilde \rho)^1) = \int \, \left|\nabla (\sqrt{(\tilde\rho)^1})\right|^2 \, dx^1 \leq \, \int \, \left|\nabla_1 (\tilde \rho)^{\frac 12}\right|^2 \, dx^1 \, dx^2 ... dx^N \, .
\end{equation}
Applying this inequality to $\rho^{N,U}_t$ we obtain, 
\begin{eqnarray}\label{eqtildefish2}
I((\widetilde{\rho_t^{N,U}})^1) &\leq& \, \int \, \left|\nabla_1 \ln \rho^{N,U}_t + \,  \frac{\chi}{N} \, \sum_{j\neq 1} \, \frac{x^1 -x^j}{|x^1-x^j|^2} \, \mathbf 1_{|x^1-x^j|\leq 1} \right|^2 \, \nonumber \\ && \quad \quad \quad \quad \rho^{N,U}_t(x) \, \prod_{j=2}^N \, (1 \wedge |x^1-x^j|^{\frac{\chi}{N}}) \, dx \nonumber \\ &\leq& \int \, \left|\nabla_1 \ln \rho^{N,U}_t + \,  \frac{\chi}{N} \, \sum_{j\neq i} \, \frac{x^i -x^j}{|x^i-x^j|^2} \, \mathbf 1_{|x^i-x^j|\leq 1} \right|^2 \, \rho^{N,U}_t(x) \, dx \, ,
\end{eqnarray}
and finally thanks to \eqref{eqfishKS1}
\begin{equation}\label{eqtildefish3}
\int_0^T \, I((\widetilde{\rho_t^{N,U}})^1) \, dt \, \leq \, C \, .
\end{equation}

Applying Lemma \ref{lemFHM} (2) and using that $(\widetilde{\rho_t^{N,U}})^1$ is a sub-probability measure (total mass less than 1), we have, for all $p \in [1,+\infty)$ 
\begin{equation}\label{eqtildefish4}
||(\widetilde{\rho_t^{N,U}})^1||_p \, \leq \, c(p) \, (I((\widetilde{\rho_t^{N,U}})^1))^{\frac{p-1}{p}} \, .
\end{equation}
\medskip

We will first deduce some integrability for the first marginal. 
\medskip

Indeed, let $g$ be a bounded function defined on $\mathbb R^2$. It holds for any $\alpha \in (0,1)$, any $\beta>0$, $\theta>0$, any $M < +\infty$ and any $\varepsilon >0$,
$$\int \, g(x^1) \,  (H(\alpha,M,\rho^{N,U}_t))^1(x^1) \, dx^1 = \int \, g(x^1) \, H(\alpha,M,\rho^{N,U}_t)(x) \,  dx  = $$
\begin{eqnarray}\label{eqfishbound}
 &=& \int \, g(x^1) \, \min(\rho_t^{N,U},(\rho_t^{N,U})^\alpha)(x) \, \mathbf 1_{|x^1|\leq M} \, \frac{(\varepsilon + \omega_1)^\alpha}{(\varepsilon + \omega_1)^\alpha} \, \frac{(\varepsilon + \kappa_1)^\theta}{(\varepsilon + \kappa_1)^\theta}dx  \nonumber \\ &\leq& A \, B
\end{eqnarray}
where for conjugate $q,p$ i.e. $\frac 1p + \frac 1q =1$, 
\begin{equation}\label{eqfishbound1}
A :=\left(\int g^q(x^1)   \, (\varepsilon + \omega_1)^{\beta q} \, (\varepsilon + \kappa_1)^{\theta q} \, \min((\rho_t^{N,U})^q,(\rho_t^{N,U})^{\alpha q})(x)\, dx\right)^{1/q}
\end{equation}
and
\begin{equation}\label{eqfishbound2}
 B := \left(\int \, \mathbf 1_{|x^1|\leq M} \, (\varepsilon + \omega_1)^{- \, \beta \, p} \, (\varepsilon + \kappa_1)^{- \theta p} \, dx\right)^{\frac 1p} \, .
\end{equation}
We may then take limits as $\varepsilon$ goes to $0$ in the final inequality.
\medskip

For the second factor $B$ we may write
\begin{eqnarray}\label{eqfishII}
B &=& \left( \int \, \mathbf 1_{|x^1|\leq M} \; \prod_{k=2}^N \, (1 \wedge |x^1-x^k|)^{- \, \beta p \chi/N} \, (1 \vee |x^1-x^k|)^{- \, \theta p \chi/N} dx^1 dx^2 ... dx^N\right)^{\frac 1p} \nonumber  \\ &\leq& \left( \int \, \mathbf 1_{|x^1|\leq M} \; \prod_{k=2}^N  \, (1 \wedge |y^k|)^{- \, \beta p \chi/N} \, (1 \vee |y^k|)^{- \, \theta p \chi/N} \, dx^1 dy^2 ... dy^N\right)^{\frac 1p} \, \\ &\leq& \left( \int \, \mathbf 1_{|x^1|\leq M} \;  \left(\int_{\mathbb R^2} (1 \wedge |y|)^{- \, \beta p \chi/N} \, (1 \vee |y|)^{- \, \theta p \chi/N} \, dy\right)^{N-1} dx^1\right)^{\frac 1p} \nonumber
\end{eqnarray}
Choose $$p=\lambda N \, ; \, \lambda \beta \chi <2 \, ; \, \lambda \theta \chi >2 \, .$$It follows that the intergral w.r.t $dy$ is finite, so that finally there exists a constant $c$ depending on $\lambda,\beta,\chi$ but not on $N$ such that
\begin{equation}\label{eqcontrolB}
B \leq c \, M^{2/\lambda N} \, .
\end{equation}
\medskip

For the first factor $A$ we use H\"{o}lder inequality for any $s>1$, furnishing $$A \leq A_1 \, A_2$$ with
\begin{eqnarray}\label{eqfishI}
A_1 &=& \left( \int \, g^{sq}(x^1) \, \omega_1^{\beta q s} \, \min((\rho_t^{N,U})^q,(\rho_t^{N,U})^{\alpha q})(x) \, dx\right)^{\frac{1}{qs}} \\ A_2 &=& \left( \int \, \kappa_1^{\theta q s/s-1} \, \min((\rho_t^{N,U})^q,(\rho_t^{N,U})^{\alpha q})(x) \, dx\right)^{\frac{s-1}{qs}}\nonumber 
\end{eqnarray}
We choose $$\alpha q =1 \, ; \, \beta q \, s=1$$ and use again $\rho^q \leq \rho$ when $\rho \leq 1$ so that for $u>1$
\begin{eqnarray}\label{eqa1}
A_1 &\leq& \left( \int \, g^{sq}(x^1) \, \omega_1 \, \rho_t^{N,U} \, dx\right)^{\frac{1}{qs}} = \left( \int \, g^{sq}(x^1) \, (\widetilde{\rho_t^{N,U}})^1(x^1) \, dx^1\right)^{\frac{1}{qs}} \nonumber \\ &\leq& c(q,s,u) \, ||g||_{qsu} \, (I((\widetilde{\rho_t^{N,U}})^1))^{\frac{1}{qsu}} \, .
\end{eqnarray}

Using convexity $$\kappa_1^{\theta q s/s-1}(x) \leq \frac{1}{N-1} \, \sum_{j=2}^N \, (1 \vee |x^1-x^j|)^{\theta qs \chi (N-1)/N (s-1)} \, \leq \, \frac{1}{N-1} \, \sum_{j=2}^N \, (1 \vee |x^1-x^j|)^{\theta q \chi s/s-1} \, .$$ It follows, using exchangeability
\begin{equation}\label{eqa2}
A_2 \leq \left(\int \, (1 \vee |x^1-x^2|)^{\theta q \chi s/s-1} \, \rho_t^{N,U} \, dx\right)^{\frac{s-1}{qs}} \, .
\end{equation}
Similarly to Lemma \ref{lemKSFJ} (1), or to lemma 5.6 in \cite{Tar} (written for $k=6$), for any $k >0$ and $t \in [0,T]$, 
\begin{equation}\label{eqmoment6}
\int \, (1 \vee |x^1-x^2|)^{k} \, \rho_t^{N,U} \, dx \leq C(k,T) \left(1 + \int \, (1 \vee |x^1-x^2|)^{k} \, \rho_0^{N} \, dx \right) \, .
\end{equation}
\medskip

Assuming that the right hand side in \eqref{eqmoment6} is finite, gathering all previous estimates we have obtained 
\begin{equation}\label{eqgather}
\int \, g(x^1) \,  (H(\alpha,M,\rho^{N,U}_t))^1(x^1) \, dx^1 \leq C \, ||g||_{qsu} \, M^{2/\lambda N} \, I((\widetilde{\rho_t^{N,U}})^1))^{\frac{1}{qsu}} \, ,
\end{equation}
where $C$ depends on $\chi$, $\rho_0^{1,2,N}$ the joint initial distribution of $(x^1,x^2)$, $\lambda,q,s,u,\theta,\beta,T$ but not on $N$. The parameters have to satisfy
\begin{eqnarray*}
p &=& \lambda N \quad ; \quad q=\frac{p}{p-1} = \frac{1}{1 -(1/\lambda N)} \quad ; \quad \alpha= \frac 1q = 1 -(1/\lambda N) \\ \beta &=& \frac{1}{qs} = s \, (1 -(1/\lambda N)) \\ \lambda &<& \frac{2}{\beta \chi} = \frac{2}{s \, (1 -(1/\lambda N)) \chi} \quad \textrm{ for instance } \; \lambda=\frac{2}{s\chi} \;  ;  \; \theta > \frac{2}{\lambda \chi} \quad \textrm{ for instance } \; \theta>s \, ;
\end{eqnarray*}
so that we still have two free parameters $u>1$ and $s>1$.
\medskip

Remark that 
\begin{equation}\label{eqgather2}
||\mathbf 1_{|x^1|\leq M} \, (\rho_t^{N,U})^1||_{r} \leq ||(H(\alpha,M,\rho^{N,U}_t))^1||_{r/\alpha}^{\frac 1\alpha} = ||(H(\alpha,M,\rho^{N,U}_t))^1||_{r/(1- \frac{1}{\lambda N})}^{1/(1- 1/\lambda N)}
\end{equation}
Combining \eqref{eqgather} and \eqref{eqgather2} we thus have obtained with $$\frac 43 = r = \frac{qsu}{qsu-1} \, (1- (1/\lambda N)) \quad i.e. \quad su = 4 \; \frac{1-(1/\lambda N)}{1+(3/\lambda N)}$$
\begin{equation}\label{eqgather3}
||\mathbf 1_{|x^1|\leq M} \, (\rho_t^{N,U})^1||_{r} \leq C \, M^{\frac{2}{\lambda N -1}} \; I((\widetilde{\rho_t^{N,U}})^1))^{\frac{1}{4(1+(3/\lambda N)}} \, .
\end{equation}
We may take first the limit as $N$ grows to infinity, use the l.s.c. of the $\mathbb L^{4/3}$ norm, and then take the limit as $M$ grows to infinity, use \eqref{eqtildefish3} and $1+I \geq I^\frac 12$ in order to obtain
\begin{equation}\label{eqgather4}
\int_0^T \, ||\bar \rho_t||^2_{4/3} \, dt \, \leq \, C \, ,
\end{equation}
for any weak limit $\bar \rho_.$ of $(\rho_.^{N,U})^1$. Recall that the previous inequality implies that $$\int_0^T \, (K*\bar \rho_t) \, \bar \rho_t \, dt \, < \, +\infty $$ so that, according to the superposition principle, there exists a solution to the non linear SDE \eqref{eqnldiff} with marginals flow given by $\bar \rho_t$. 

According to Theorem \ref{thmksconvorl}, there is only one such solution with finite entropy marginals flow. Since we know that the non linear process associated to the free energy solution of the Keller-Segel equation exists (Theorem \ref{thmexistnl1}), it is the unique solution of the non linear SDE. Consequently there is only one weak limit for $(\rho_.^{N,U})^1$ which is given by the free energy solution of the Keller-Segel equation.

It remains to check the integrability condition \eqref{eqmoment6} for $k=\theta q \chi s /s-1$. For $N$ large, it is enough that $k> s^2 \chi/s-1$. Optimizing on $s$ furnishes $s=2$ (and so $u$ close to 2).

We may finally cover the case $U=0$ using self similar coordinates as before.
\end{proof}

\end{document}